\documentclass[11pt,a4paper]{article}

\usepackage[utf8]{inputenc}
\usepackage[T1]{fontenc}
\usepackage[english]{babel} 
\usepackage{amsmath}
\usepackage{amsfonts}
\usepackage{amsthm}
\usepackage{amssymb}
\usepackage{amsbsy}
\usepackage{makeidx}
\usepackage{graphicx}
\usepackage{lmodern}
\usepackage[left=2cm,right=2cm,top=2cm,bottom=2cm]{geometry}
\usepackage{enumitem}

\usepackage[normalem]{ulem}
\usepackage{color}
\usepackage{comment}
\usepackage[dvipsnames]{xcolor}
\usepackage{graphicx}
\usepackage{framed}
\usepackage{tikz}
\usepackage{calrsfs}
\DeclareMathAlphabet{\pazocal}{OMS}{zplm}{m}{n}
\usepackage[scr=boondox, bfscr]{mathalpha}
\usepackage{stmaryrd} 
\SetSymbolFont{stmry}{bold}{U}{stmry}{m}{n}

\usepackage{scalerel}
\usepackage[usestackEOL]{stackengine}
\def\dashint{\,\ThisStyle{\ensurestackMath{%
  \stackinset{c}{.2\LMpt}{c}{.5\LMpt}{\SavedStyle-}{\SavedStyle\phantom{\int}}}%
  \setbox0=\hbox{$\SavedStyle\int\,$}\kern-\wd0}\int}

\usepackage{fourier-orns}
\usepackage{mathtools}
\usepackage{relsize}
\usepackage{dsfont}
\usepackage{comment}

\usepackage{hyperref}
\hypersetup{
	linktocpage = true,
	colorlinks = true,
	linkcolor = {RoyalBlue},
	urlcolor = {RoyalBlue},
	citecolor = {Green}}

\newcommand{\B}{\mathbb{B}}
\newcommand{\E}{\mathbb{E}}

\newcommand{\R}{\mathbb{R}}

\newcommand{\T}{\mathbb{T}}

\newcommand{\Apazo}{\pazocal{A}}

\newcommand{\Cpazo}{\pazocal{C}}

\newcommand{\Fpazo}{\pazocal{F}}

\newcommand{\Ipazo}{\pazocal{I}}
\newcommand{\Jpazo}{\pazocal{J}}

\newcommand{\Lpazo}{\pazocal{L}}
\newcommand{\Mpazo}{\pazocal{M}}

\newcommand{\Ppazo}{\pazocal{P}}

\newcommand{\Rpazo}{\pazocal{R}}

\newcommand{\Tpazo}{\pazocal{T}}

\newcommand{\Xpazo}{\pazocal{X}}

\newcommand{\Acal}{\mathcal{A}}
\newcommand{\Bcal}{\mathcal{B}}
\newcommand{\Ccal}{\mathcal{C}}

\newcommand{\Ecal}{\mathcal{E}}

\newcommand{\Gcal}{\mathcal{G}}

\newcommand{\Ical}{\mathcal{I}}
\newcommand{\Jcal}{\mathcal{J}}
\newcommand{\Kcal}{\mathcal{K}}
\newcommand{\Lcal}{\mathcal{L}}
\newcommand{\Mcal}{\mathcal{M}}

\newcommand{\Pcal}{\mathcal{P}}

\newcommand{\Scal}{\mathcal{S}}

\newcommand{\Kcalb}{\mathscr{K}}
\newcommand{\Xcalb}{\mathscr{X}}
\newcommand{\Vcalb}{\mathscr{V}}

\newcommand{\pfrak}{\mathfrak{p}}
\newcommand{\efrak}{\mathfrak{e}}

\newcommand{\BEfrak}{\boldsymbol{\mathfrak{E}}}

\newcommand{\Id}{\textnormal{Id}}

\newcommand{\Var}{\textnormal{Var}}
\newcommand{\supp}{\textnormal{supp}}
\newcommand{\sign}{\textnormal{sign}}
\newcommand{\Lip}{\textnormal{Lip}}
\newcommand{\AC}{\textnormal{AC}}

\newcommand{\Div}{\textnormal{div}}

\newcommand{\textbn}[1]{\textnormal{\textbf{#1}}}

\newcommand{\Adm}{\mathrm{Adm}}

\newcommand{\xb}{\boldsymbol{x}}

\newcommand{\vb}{\boldsymbol{v}}
\newcommand{\wb}{\boldsymbol{w}}

\newcommand{\Vb}{\boldsymbol{V} \hspace{-0.075cm}}

\newcommand{\Xb}{\boldsymbol{X}}

\newcommand{\dsf}{\textnormal{\textsf{d}}}

\newcommand{\Bgamma}{\boldsymbol{\gamma}}

\newcommand{\Bphi}{\boldsymbol{\phi}}
\newcommand{\BPhi}{\boldsymbol{\Phi}}

\newcommand{\Bnu}{\boldsymbol{\nu}}

\newcommand{\Beta}{\boldsymbol{\eta}}
\newcommand{\Bmu}{\boldsymbol{\mu}}
\newcommand{\BLambda}{\boldsymbol{\Lambda}}

\newcommand{\BPsi}{\boldsymbol{\Psi}}

\renewcommand{\epsilon}{\varepsilon}

\newcommand{\INTDom}[3]{\int_{#2} #1 \,\mathrm{d} #3}
\newcommand{\INTDomdash}[3]{\dashint_{#2} #1 \,\mathrm{d} #3}
\newcommand{\INTSeg}[4]{\int_{#3}^{#4} #1 \,\mathrm{d} #2}
\newcommand{\NormL}[3]{\parallel \hspace{-0.05cm} #1 \hspace{-0.05cm} \parallel_{L^{#2}(#3)}}
\newcommand{\NormC}[3]{\left\| #1  \right\| _ {C^{#2}(#3)}}
\newcommand{\Norm}[1]{\parallel \hspace{-0.1cm} #1 \hspace{-0.1cm} \parallel}
\newcommand{\derv}[2]{\frac{\textnormal{d} #1}{ \textnormal{d} #2}}

\newcommand{\elts}{\{1,\ldots,N\}}
\newcommand{\muAvLaw}{\INTDomdash{\mu^0_\omega}{\Omega_k^n}{ \pi(\omega)}}

\newcommand{\kl}{{(k-1)m+\ell}}
\newcommand{\klprime}{{(k'-1)m+\ell'}}
\newcommand{\bmu}{\bar{\Bmu}}
\newcommand{\nm}{{n,m}}
\newcommand{\bx}{\bar{x}}

\newtheorem{Def}{Definition}[section]
\newtheorem{thm}[Def]{Theorem}
\newtheorem{prop}[Def]{Proposition}
\newtheorem{rmk}[Def]{Remark}
\newtheorem{lem}[Def]{Lemma}
\newtheorem{cor}[Def]{Corollary}

\definecolor{greenn}{rgb}{0.0, 0.7, 0.0}

\newtheorem{taggedhypx}{Hypotheses}
\newenvironment{taggedhyp}[1]
    {\renewcommand\thetaggedhypx{#1}\taggedhypx}
    {\endtaggedhypx}

\newcommand{\changelocaltocdepth}[1]{%
  \addtocontents{toc}{\protect\setcounter{tocdepth}{#1}}%
  \setcounter{tocdepth}{#1}%
  }
  
\newif\ifDisplayFigures
 
\title{Structured Continuity Equations in Fibred Wasserstein Spaces}

\author{Benoît Bonnet-Weill\footnote{Laboratoire des Signaux et Systèmes, Université Paris-Saclay, CNRS, CentraleSupélec, 91190 Gif-sur-Yvette, France. \textit{Email:} \texttt{benoit.bonnet-weill@centralesupelec.fr}}\; and Nastassia Pouradier Duteil\footnote{Sorbonne Université, Université Paris Cité, CNRS, INRIA, Laboratoire Jacques-Louis Lions, LJLL, EPC MUSCLEES, F-75005 Paris, France. \textit{Email:} \texttt{nastassia.pouradier\_duteil@sorbonne-universite.fr}}}

\begin{document}

\maketitle

\begin{abstract}
In this article, we develop a comprehensive ODE-theory for structured continuity equations in fibred probability spaces, which represent a class of heterogeneous PDEs arising as the meanfield limit of nonexchangeable particle systems. After investigating in depth the topologies induced by the so-called fibred (or conditional) and classical Wasserstein metrics on such probability spaces, we establish quantitative Cauchy-Lipschitz and qualitative Carathéodory-Peano well-posedness results for structured continuity equations, along with precise correspondences between this class of evolutions, classical Lagrangian dynamics, and continuity equations. In keeping with what has long been known for exchangeable dynamics, we derive a general meanfield approximation result by solutions of nonexchangeable particle systems, along with a quantitative variant thereof under practically reasonable regularity assumptions on the driving field and initial data.  
\end{abstract}

{\footnotesize

\vspace{0.25cm}

\textbf{MSC2020 Subject Classification :} 35Q49, 46N20, 49Q22, 93A16
}

\tableofcontents


\changelocaltocdepth{2}

\section{Introduction}
\setcounter{equation}{0} \renewcommand{\theequation}{\thesection.\arabic{equation}}

The Vlasov equation, or \textit{continuity equation}, has long been the most prominent and common mathematical object to describe macroscopic approximations of large deterministic particle systems. Concretely, the latter is usually written in conservative form as 
\begin{equation}
\label{eq:Vlasov_intro}
\partial_t \mu(t) + \Div_x(v(t,\mu(t)) \mu(t)) = 0, 
\end{equation}
and describes the evolution through time of a population density $t \in [0,T] \mapsto \mu(t) \in \Pcal(\R^d)$, modelled here as a curve of probability measures that is transported under the action of a nonlocal vector field $v : [0,T] \times \Pcal(\R^d) \times \R^d \to \R^d$. When the dynamics is given as the convolution of the density with a sufficiently nice interaction function $\Psi : \R^{2d} \to \R^d$, e.g.
\begin{equation*}
v(t,\mu,x) := \INTDom{\Psi(x,y)}{\R^d}{\mu(y)},  
\end{equation*}
it has been known since the 1970s (see for instance \cite{BraunHepp77,Dobrushin1979,NeunzertWick80,JourdainMeleard98,Sznitzman91}) that the continuity equation \eqref{eq:Vlasov_intro} can be rigorously obtained as the meanfield limit of the interacting particle system
\begin{equation}
\label{eq:Particles_intro}
\dot x_i(t) = \frac1N \sum_{j=1}^N \Psi(x_i(t),x_j(t))
\end{equation} 
as the number of particles $N \to +\infty$. From a technical standpoint, the underlying convergence proofs classically relied on stability estimates for the continuity equation, involving either Wasserstein metrics or the Bounded-Lipschitz distance. Importantly, in \eqref{eq:Particles_intro}, the particles are supposed to be \textit{indistinguishable} -- or, one could say,  \textit{exchangeable} --, in the sense that the dynamics only depends on their positions, but not on their proper identities. Indeed, swapping any two particles (or equivalently permuting their labels) does not affect their trajectories. This modelling feature stems directly from the fact that, historically, the Vlasov equation and its particle description originated in the context of statistical physics \cite{BraunHepp77}, in which it is reasonable to assume that particles are identical. 

This exchangeability assumption, however, becomes highly limiting when it is expected that the identities of the particles should play an important role in the evolution of the system. This frequently occurs for \textit{multiagent systems} in the context of life and social sciences. Incidentally, an increasing body of work has recently shown not only that heterogeneity  has been selectively maintained in many biological systems, but also that it is a crucial mechanism for collective organisation, and thus group survival \cite{CoteFogartySih12,GuRohling19,JollesKingKillen20,PerezLeungDragottiGoodman21}. Complementarily, these so-called \textit{heterogeneous} dynamics have served as the bedrock for studying a wide range of models arising in life sciences during the last thirty years, including e.g. crowd dynamics \cite{CPT}, animal group coordination \cite{CS1,Ballerini}, opinion formation \cite{Hegselmann2002}, neuron synchronisation \cite{GuRohling19,PerezLeungDragottiGoodman21} and cell organisation in tissues \cite{Vicsek1995}. In such contexts, the evolution of the microscopic system is described by a \textit{nonexchangeable} particle dynamics of the form
\begin{equation}
\label{eq:NEParticles_intro}
\dot x_i(t) = \sum_{j=1}^N w_{ij} \Psi(x_i(t),x_j(t)),
\end{equation}
in which the labels -- or identities -- of the particles directly affect their interactions via the coupling coefficients $(w_{ij})_{i,j\in\{1,\cdots,N\}}\subset \R_+$. The latter can be seen as the edge weights of a directed interaction graph, in which a node $i \in \{1,\ldots,N\}$ is connected to another node $j \in \{1,\ldots,N\}$ if and only if $w_{ij}>0$. By leveraging results from graph theory, it was shown in the pioneering work \cite{Lovasz2006} that, under suitable scaling assumptions, the limit as $N \to +\infty$ of the graph sequence with adjacency matrices $(w_{ij})_{i,j \in \{1,\ldots,N\}}$ is represented by a function $w \in L^\infty([0,1]^2,\R_+)$ called a \textit{graphon}, in which $[0,1]$ plays the role of a continuous label space. The link between the discrete label set and the continuous one is heuristically given by the mapping $i\in\{1,\cdots,N\}\mapsto \omega \in [(i-1)/N,i/N)$. Then, as first shown in \cite{KaliuzhnyiMedvedev2018}, the meanfield limit of the nonexchangeable particle system \eqref{eq:NEParticles_intro} can be likewise characterised by a \textit{structured continuity equation} over the product space $[0,1]\times\R^d$, which takes the form
\begin{equation}
\label{eq:TransportGraphNonlocal_intro}
\partial_t \Bmu(t) + \Div_x( \Psi_w \star \Bmu(t) \Bmu(t)) = 0 
\end{equation}
where 
\begin{equation*}
(\Psi_w \star \Bmu)(\omega,x) := \INTDom{w(\omega,\theta) \Psi(x,y)}{[0,1] \times \R^d}{\Bmu(\theta,y)}.
\end{equation*}
Therein, the measures $\Bmu(t)\in \Pcal([0,1]\times\R^d)$ encode the density of the particles both in terms of their labels and positions in space, i.e. the evaluation $\Bmu(t)(A \times B)$ outputs the probability of finding at time $t \in [0,T]$ a particle with label $\omega\in A$ and position $x\in B$. It should be stressed that in \eqref{eq:TransportGraphNonlocal_intro}, only the space marginal of the density is transported, while the marginal $\pi := (\pfrak_{\Omega})_{\sharp} \Bmu(t) \in \Pcal(\Omega)$ encoding the distribution of the labels remains constant. The classical micro-to-macro convergence results mentioned above for exchangeable particle systems have been recently extended to this nonexchangeable setting, while taking into account various scaling assumptions on the coupling weights \cite{GkogkasKuehn2022,JabinPoyatoSoler2025,KuehnXu2022}, as well as co-evolving weights \cite{GkogkasKuehnXu2025} and non-binary interactions \cite{AyiPouradierDuteilPoyato2024} (see also the review article \cite{AyiPouradierDuteil2024} by the second author and references therein).

In summary, the discrete particle system \eqref{eq:NEParticles_intro} and its companion structured continuity equation \eqref{eq:TransportGraphNonlocal_intro} together allow for the extension of the theory of meanfield limits to nonexchangeable dynamics. However, two key hypotheses were implicitly made in the modelling process leading to both classes of equations. Firstly, the total influence of the population on any individual must take the form of a sum of binary (or potentially higher-order) interactions. Secondly, the dependence in the labels and positions of the particles should be decoupled as the product of two terms (isolating the underlying interaction topology from the spatial interaction function), which implies that the only source of heterogeneity comes from an underlying network prescribing the interactions. Although these two assumptions have allowed for major developments in the study of both microscopic and macroscopic nonexchangeable collective dynamics on graphs, many systems of high practical interest do not comply with them. For instance, heterogeneous neuronal network models are currently emerging, motivated by the recent finding that neural heterogeneity promotes robust learning \cite{GuRohling19,PerezLeungDragottiGoodman21}. Heterogeneity in such contexts can take many forms, since the membrane potential, the resting potential, the synaptic time constant and the reset potential (as well as several other parameters) can all be specific to each neuron. It is also known that phenotypic heterogeneities of cancer cells 
affect their velocity, and are likely to lead to phenotypic segregation within a tumour \cite{CarrilloLorenziMacfarlane25}. Another family of such examples is provided by recent models for heterogeneous enzymatic reactions which are described via the so-called Michaelis-Menten kinetics \cite{DouglasCarterWills24,GuRohling19}, in which both the Michaelis constants and limiting rates depend on the enzyme population. A last relevant example is given by collective dynamics with agent-dependent transmission and reaction delays \cite{Cicolani2026}. Due to the complex role played by heterogeneity in these models, they typically do not fit the framework of \eqref{eq:NEParticles_intro} and \eqref{eq:TransportGraphNonlocal_intro}.

In what follows, we propose to part with these assumptions and investigate as generally as possible the qualitative properties of the structured continuity equation 
\begin{equation}
\label{eq:IntroFibredCont}
\partial_t \Bmu(t) + \Div_x ( \vb(t,\Bmu(t)) \Bmu(t)) = 0 
\end{equation}
driven by some nonlocal vector field $\vb : [0,T] \times \Pcal(\Omega \times \R^d) \times \Omega \times \R^d \to \R^d$. Therein, we do not make any a priori assumption on the form of the velocity field, and let $\Bmu(t) \in \Pcal(\Omega \times \R^d)$ be a general measure over the labels and positions for all $t \in [0,T]$, where the encoding space $(\Omega,\dsf_{\Omega}(\cdot,\cdot))$ is simply assumed to be Polish. In this context, the dynamics \eqref{eq:IntroFibredCont} is understood in the sense of distributions
\begin{equation*}
\INTSeg{\INTDom{\Big( \partial_t \phi(t,\omega,x) + \big\langle \nabla_x \phi(t,\omega,x) , \vb(t,\Bmu(t),\omega,x) \big\rangle \Big)}{\Omega \times \R^d}{\Bmu(t)(\omega,x)}}{t}{0}{T} = 0
\end{equation*}
for each $\phi : [0,T] \times \Omega \times \R^d \to \R$ belonging to a suitable class of test functions with compact support in the time and space variables. From a modelling viewpoint, this general dynamics can be interpreted in a variety of ways. As it was the case in \eqref{eq:TransportGraphNonlocal_intro}, the variable $\omega \in \Omega$ may first be seen as a continuously indexed label for the particles. Another possible interpretation is to view \eqref{eq:IntroFibredCont} as a structured population model, i.e. as a heterogeneous population composed of individuals which can be distinguished from one another by a physical or biological attribute (such as age, size, phenotype etc). In that case,  
 $\omega \in \Omega$ plays the role of a \textit{structure variable}, see e.g. \cite{Perthame2007}, which affects the velocity of each particle but does not evolve in time. The measures $\Bmu(t) \in \Pcal(\Omega \times \R^d)$ can even more generally represent a continuum of interacting populations, which, thanks to the flexibility in the choice of the label marginal, may have very differentiated roles, such as being comprised of groups of leaders and followers \cite{FornasierPR2014}. Besides these considerations, the generality of the vector field $\vb : [0,T] \times \Pcal(\Omega \times \R^d) \times \Omega \times \R^d \to \R^d$ allows to encompass many of the previously studied frameworks in the literature, such as graphon-based continuity equations \eqref{eq:TransportGraphNonlocal_intro} and their various extensions (see e.g. \cite{AyiPouradierDuteilPoyato2024,GkogkasKuehn2022,JabinPoyatoSoler2025,KaliuzhnyiMedvedev2018,KuehnXu2022,Paul2024}), or the aforementioned models for heterogeneous neuron networks \cite{PerezLeungDragottiGoodman21} and heterogeneous enzymatic-type reactions \cite{DouglasCarterWills24,GuRohling19}. The last, but still pivotal motivation for studying \eqref{eq:IntroFibredCont} in such general a fashion is that it provides a setting subsuming both the so-called Eulerian and Lagrangian formulations of deterministic \textit{meanfield optimal control} (see e.g. \cite{Badreddine2022,SetValuedPMP,PMPWass,Burger2021,Jimenez2020,Pogodaev2016}), which were originally formalised in the groundbreaking work \cite{Cavagnari2022}, and led to a number of important outlets in the field, such as \cite{Averboukh2025,Jimenez2023,Jimenez2024} to quote a few recent ones.    

At this point, the following question naturally arises: which mathematical framework is more adequate to study \eqref{eq:IntroFibredCont}? To allow for the nonexchangeability of the particle system, the state space was extended from $\R^d$ for the classical continuity equation \eqref{eq:Vlasov_intro} to $\Omega\times\R^d$ in \eqref{eq:TransportGraphNonlocal_intro} and \eqref{eq:IntroFibredCont}. Thus, if the driving field were regular enough in both $(\omega,x) \in \Omega \times \R^d$, one could simply work in the classical Wasserstein space $(\Pcal_p(\Omega\times\R^d),W_p(\cdot,\cdot))$, and consider \eqref{eq:IntroFibredCont} as a standard continuity equation driven by the augmented field $\Vb(t,\Bmu,\omega,x) := (0,\vb(t,\Bmu,\omega,x))\in \Omega\times\R^d$. Then, the dynamics would be well-posed provided that the vector field is e.g. Lipschitz continuous in the label, space and measure variables, following the known results of e.g. \cite[Chapter 8]{AGS} and \cite{ContIncPp,Pedestrian}. 
However, this approach fails to render the fact that the label and space variables do not play a symmetrical role in \eqref{eq:IntroFibredCont}. Indeed, as already pointed out above, only the space marginal of the particle density is transported, whereas the label marginal $\pi : t \mapsto (\pfrak_{\Omega})_{\sharp} \Bmu(t) \in \Pcal(\Omega)$ remains invariant under the dynamics. This implies that solutions of \eqref{eq:IntroFibredCont} actually belong to the subspace 
\begin{equation*}
\Pcal_{\pi}(\Omega \times \R^d) := \Big\{ \Bmu \in \Pcal(\Omega \times \R^d) ~\,\textnormal{s.t.}~ (\pfrak_{\Omega})_{\sharp} \Bmu = \pi \Big\}
\end{equation*}
of all probability measures with a common first marginal $\pi \in \Pcal(\Omega)$. The latter is known as the space of \textit{fibred probability measures}, or \textit{Young measures}, and may be endowed with a variety of different metrics, see Remark \ref{rmk:Metrics} below. In this work, we consider the so-called \textit{fibred} -- or \textit{conditional} -- Wasserstein metrics, which are defined for each $p \in [1,+\infty)$ as
\begin{equation*}
W_{\pi, \,p}(\Bmu,\Bnu) := \bigg( \INTDom{W_p^p(\mu_{\omega},\nu_{\omega})}{\Omega}{\pi(\omega)} \bigg)^{1/p}
\end{equation*}
where the families of measures $\{\mu_\omega\}_{\omega \in \Omega},\{\nu_\omega\}_{\omega \in \Omega} \subset \Pcal(\R^d)$ stand for the respective disintegrations of $\Bmu,\Bnu \in \Pcal_{\pi}(\Omega \times \R^d)$ against their common first marginal $\pi\in\Pcal(\Omega)$. Note that such metrics are canonically well-defined and finite over the subspace 
\begin{equation*}
\Pcal_{\pi,\,p}(\Omega \times \R^d) := \bigg\{ \Bmu \in \Pcal_{\pi}(\Omega \times \R^d) ~\,\textnormal{s.t.}~ \INTDom{\INTDom{|x|^p}{\R^d}{\mu_{\omega}(x)}}{\Omega}{\pi(\omega)} < +\infty \bigg\}
\end{equation*}
of fibred measures with finite $p$-space moment. The choice of developing our analysis in the metric spaces $(\Pcal_{\pi,\,p}(\Omega \times \R^d),W_{\pi,\,p}(\cdot,\cdot))$ is supported by two key facts. Firstly, fibred Wasserstein metrics appear very naturally when searching for Lipschitz-like regularity conditions in nonexchangeable meanfield dynamics without imposing any regularity in $\omega \in \Omega$. Secondly, it was amply demonstrated in \cite{Peszek2023} which provides an in depth study of the gradient flow perspective to \eqref{eq:IntroFibredCont} (see also \cite{Barboni2024}) that they carry the same geometric Riemannian-flavoured structure as the classical Wasserstein space when $p=2$, in keeping with the seminal works \cite{McCann1997,Otto2001} and the theory developed in \cite{AGS}. What emanates from these observations is that structured continuity equations are essentially ODEs over the metric space $\Pcal_{\pi,\,p}(\Omega \times \R^d) \simeq L^p(\Omega,\Pcal_p(\R^d);\pi)$, much like what has been known  (and fruitfully put to use) for classical continuity equations over $\Pcal_p(\R^d)$ during the past twenty years.

\paragraph*{Objectives and bibliographical positioning.} In the aforedescribed context, a first aim of this article is to present a thorough topological study of the fibred probability spaces $\Pcal_{\pi,\, p}(\Omega \times \R^d)$ equipped with the fibred Wasserstein distance, and particularly to characterise its convergent sequences and compact sets. We also identify the key differences between the latter and those corresponding to the classical Wasserstein distance. Indeed, as shown e.g. in \cite{Peszek2023}, the fibred Wasserstein distance can be recast as an optimal transport problem in which moving between different fibres is penalised with an infinite cost, which prevents any transversal transfer of mass. Consequently, it is easy to verify that 
\begin{equation*}
W_{p}(\Bmu,\Bnu)\leq W_{\pi, \,p}(\Bmu,\Bnu),     
\end{equation*}
namely the fibred Wasserstein topology is finer than the classical Wasserstein topology. One of our notable findings in this context is that the latter respectively play the roles of canonical strong and weak topologies over $\Pcal_{\pi,\,p}(\Omega \times \R^d)$.  

Another goal of the present work is to identify minimal ODE-type regularity assumptions ensuring the existence and uniqueness of solutions to the Cauchy problem associated with \eqref{eq:IntroFibredCont}, so as to complement the gradient flow perspective to structured continuity equations developed independently in \cite{Barboni2024,Peszek2023}. In what follows, we establish that in addition to the classical Lipschitz regularity of $\vb : [0,T] \times \Pcal_{\pi,\,p}(\Omega \times \R^d) \times \Omega \times \R^d \to \R^d$ in the space variable, it is enough to require mere measurability in the label variable, and to relax the Lipschitz continuity with respect to the measure variable in terms of the fibred Wasserstein distance. We also show below that such regularity assumptions are tight for deriving stability estimates with respect to the initial data, which further supports the idea that fibred Wasserstein metrics are indeed better suited to studying the structured continuity equation \eqref{eq:IntroFibredCont} than classical ones, although the latter happen to be relevant when studying the weaker notion of Peano existence. These minimal regularity assumptions are in line with previous results for dynamics with graphon-type driving fields such as that exhibited in \eqref{eq:TransportGraphNonlocal_intro}, see for instance \cite{AyiPouradierDuteilPoyato2024,Castin2025,JabinPoyatoSoler2025,Paul2025}, although to the best of our knowledge, they were never investigated in such generality. 

Lastly, since the structured continuity equation \eqref{eq:IntroFibredCont} emerged from a modelling standpoint as a generalisation of \eqref{eq:TransportGraphNonlocal_intro} wherein the specific form of the vector field was inferred from the nonexchangeable particle system \eqref{eq:NEParticles_intro} through a meanfield limit process, it is natural to ask whether it may also be obtained as the meanfield limit, in some sense, of an interacting particle system. The third aim of this article is to answer positively to this question, by showing that when the label distribution $\pi \in \Pcal(\Omega)$ is nonatomic, solutions of \eqref{eq:IntroFibredCont} can indeed be approximated by those of well-chosen particle systems, which are constructed by discretising the vector field using adequate partitions of $\Omega$.


\paragraph*{Overview of contributions.}

We now provide a detailed outline of the main results of this paper, which are preceded by some preliminaries exposed in Section \ref{section:Preliminaries}. 

\begin{enumerate}[wide, labelindent=0pt]
\item[$\diamond$] In Section \ref{section:TopologyWass}, we begin by thoroughly studying the topologies induced by the classical and fibred Wasserstein metrics over $\Pcal_{\pi,\,p}(\Omega \times X)$, where $(X,\Norm{\cdot}_X)$ is a separable Banach space, picking up some of the questions and problems left open in \cite{Peszek2023}. In particular, we provide explicit and handy characterisations of relatively compact sets for both topologies, and sharply characterise their convergent sequences $(\Bmu_n) \subset \Pcal_{\pi,\,p}(\Omega \times X)$ in terms, respectively, of the weak and strong $L^1(\Omega,\R;\pi)$-convergence of the sequences of integrals    
\begin{equation*}
\omega \in \Omega \mapsto \INTDom{\phi(x)}{X}{\mu_{\omega}^n(x)} \in \R,
\end{equation*}
where $\phi \in C^0_p(X,\R)$ is an arbitrary continuous function with $p$-growth. Lastly, we prove a duality formula ``a la Kantorovich-Rubinstein'' for the $W_{\pi,1}$-distance, under the assumption that the fibred measures have uniformly compactly supported space marginals.
\item[$\diamond$] In Section \ref{section:CELocal}, we investigate the qualitative properties of structured continuity equations with local velocities over $\Pcal_{\pi}(\Omega \times \R^d)$, in the spirit of \cite[Chapter 8]{AGS}. Therein, given a Lebesgue-Borel field $\vb : [0,T] \times \Omega \times \R^d \to \R^d$, we begin by proving in Section \ref{subsection:Structure} that the distributional solutions of 
\begin{equation*}
\partial_t \Bmu(t) + \Div_x(\vb(t) \Bmu(t))=0 
\end{equation*}
indeed have a constant first marginal $\pi \in \Pcal(\Omega)$, and that they are given in terms of a sliced representation with respect to the latter. We also provide sufficient conditions for solutions to admit a (fibred) narrowly continuous representative, and prove a first existence and uniqueness result under Cauchy-Lipschitz regularity assumptions. Then, in Section \ref{subsection:LagrangianRepExistence}, we prove a variant of Ambrosio's superposition principle (see e.g. \cite[Theorem 3.4]{AmbrosioC2014}) for this class of dynamics, from which we subsequently infer various a priori bounds and regularity estimates on the solutions in the spirit of \cite{ContIncPp}, and a Carathéodory-Peano existence result. In Section \ref{subsection:CompactnessReachableSets}, we study the propagation of both weak and strong relative compactness of the initial data along the orbits of the dynamics.  
\item[$\diamond$] In Section \ref{section:CENonlocal}, we suppose that $p=1$ for the sake of exposition simplicity, and shift our focus to nonlocal dynamics over $\Pcal_{\pi,1}(\Omega \times \R^d)$. Therein, we begin in Section \ref{subsection:Lipschitz} by treating the case in which 
\begin{equation*}
 (\Bmu,x) \in \Pcal_{\pi,1}(\Omega \times \R^d) \times \R^d \mapsto \vb(t,\Bmu,\omega,x) \in \R^d
\end{equation*}
is Lipschitz continuous for the $W_{\pi,1} \times |\cdot|$-metric for $\Lcal^1 \times \pi$-almost every $(t,\omega) \in [0,T] \times \Omega$ and satisfies some basic growth conditions. Under these assumptions, we prove that solutions of \eqref{eq:IntroFibredCont} exist, are stable with respect to their initial data -- so in particular unique -- and represented by a generalised flow. In Section \ref{subsection:Peano}, we depart from the Cauchy-Lipschitz setting and prove an existence result à la Carathéodory-Peano for \eqref{eq:IntroFibredCont}, under the assumption that 
\begin{equation*}
 (\Bmu,x) \in \Pcal_{\pi,1}(\Omega \times \R^d) \times \R^d \mapsto \vb(t,\Bmu,\omega,x) \in \R^d
\end{equation*}
is continuous for the $W_1 \times |\cdot|$-metric for $\Lcal^1 \times \pi$-almost every $(t,\omega) \in [0,T] \times \Omega$. The proof is based on a variant for continuity equations of the classical semi-discrete Euler scheme from e.g. \cite[Chapter 1]{Filippov2013}, which was first explored in \cite{ContIncPp}. Finally, in Section \ref{subsection:Links}, we precisely discuss the various correspondences between structured continuity equations, Lagrangian dynamics and classical continuity equations. 
\item[$\diamond$] In Section \ref{Section:ParticleApprox}, we leverage most of the above results on structured continuity equations to show via a keen adaptation of the strategy developed in \cite{Medvedev2014} that every solution of \eqref{eq:IntroFibredCont} can be obtained as the uniform limit in $C^0([0,T],\Pcal_{\pi,1}(\Omega \times \R^d))$ of a sequence of empirical measures, which are concentrated on the solutions of a nonexchangeable particle system. The proof of this result is fairly long and technical, and spans Section \ref{subsection:AuxiliaryParticles}, Section \ref{subsection:InitialData} and Section \ref{subsection:ConvergenceProof}. Upon assuming that the initial data of the particle systems are sampled i.i.d. from the limit initial density, and that both the latter and the driving field enjoy some regularity in $\omega \in \Omega$ -- say, for instance, having bounded variations over $\Omega := [0,1]$ --, we also derive a quantitative convergence rate in Section \ref{subsection:QuantitativeConv}.
\item[$\diamond$] Finally, in Section \ref{section:Applications}, we provide a broad overview of the many different models which are encompassed in our framework, depending on the choice of reference measure $\pi \in \Pcal(\Omega)$, and the shape of the vector field $\vb : [0,T] \times \Pcal_{\pi,1}(\Omega \times \R^d) \times \Omega \times \R^d \to \R^d$.
\end{enumerate}

The paper is complemented by appendices to which we deferred the proofs of some results, mostly because they were (sometimes quite technical) adaptations of known arguments.

\changelocaltocdepth{2}
\section{Preliminaries}
\label{section:Preliminaries}

\setcounter{equation}{0} \renewcommand{\theequation}{\thesection.\arabic{equation}}

In this first section, we expose preliminary material on abstract integration and measure theory, classical and fibred optimal transport, and continuity equations over Euclidean spaces. 


\paragraph*{Function spaces, measure theory and Bochner integration.} 

In this first preliminary section, we recollect general notations and concepts of functional analysis and integration, for which we refer to the monographs \cite{Aubin1990,Bogachev,Bourbaki2,DiestelUhl}.

Given two Polish spaces $(X,\dsf_X(\cdot,\cdot))$ and $(Y,\dsf_Y(\cdot,\cdot))$, we let $C^0(X,Y)$ be the space of continuous mappings from $X$ into $Y$, and denote by $C^0_b(X,Y)$ the subspace of continuous bounded maps. For any $p \in [1,+\infty)$, we shall say that a function $\phi \in C^0(X,\R)$ has \textit{$p$-growth} provided that 
\begin{equation*}
|\phi(x)| \leq \alpha_{\phi} \Big(1+\dsf_X^p(x,x_0) \Big) 
\end{equation*}
for all $x \in X$ and some pair $(\alpha_{\phi},x_0) \in \R_+ \times X$, and denote by $C^0_p(X,\R)$ the collection of all such maps. We will also denote by $\Lip(\phi\,;K)$ the Lipschitz constant of a function over a subset $K \subset X$, and write $\Lip(K,\R)$ for the space of Lipschitz continuous functions. In the particular case in which $(X,\Norm{\cdot}_X)$ is a locally compact Banach space, given some $k \in \{0,1\}$, we will denote by $C^k_c(X,\R)$ the separable normed space of $k$-times continuously differentiable real-valued functions with compact support and by $C^k_0(X,\R)$ its standard norm completion. 

In the sequel, we let $\Pcal(\Omega) \subset \Mcal^+(\Omega)$ stand for the spaces of Borel probability measures and finite Radon measures over a Polish space $(\Omega,\dsf_{\Omega}(\cdot,\cdot))$ equipped with its usual Borel $\sigma$-algebra $\Bcal(\Omega)$. We also fix some $\pi \in \Mcal^+(\Omega)$ whose completed $\sigma$-algebra of $\pi$-measurable sets is denoted by $\Apazo$. 

\begin{Def}[Measurable and Carathéodory mappings]
A function $f : \Omega \to X$ is said to be Borel measurable whenever $f^{-1}(B) \in \Bcal(\Omega)$ for each $B \in \Bcal(X)$, and \textnormal{$\pi$-measurable} if $f^{-1}(B) \in \Apazo$. Similarly, a mapping $\varphi : \Omega \times X \to Y$ is termed \textnormal{Carathéodory} provided that $\omega \in \Omega \mapsto \varphi(\omega,x) \in Y$ is $\pi$-measurable for all $x \in X$ and $x \in X \mapsto \varphi(\omega,x) \in Y$ is continuous for $\pi$-almost every $\omega \in \Omega$.
\end{Def}

Below, we recollect the definition of integrability for metric-valued maps, following e.g. \cite[Chapter 8]{Aubin1990}, along with some basic measurability results borrowed from the latter reference and \cite{Papageorgiou1986}. 

\begin{Def}[Spaces of metric-valued and Bochner integrable functions]
\label{def:Integral}
Given some $p \in [1,+\infty]$, we say that a mapping $f : \Omega \to X$ valued in a Polish space $(X,\dsf_X(\cdot,\cdot))$ is \textnormal{$p$-integrable} if it is $\pi$-measurable and such that  
\begin{equation*}
\INTDom{\dsf_X^p(f(\omega),x_0)}{\Omega}{\pi(\omega)} < +\infty
\end{equation*}
for some $x_0 \in X$. We then denote by $L^p(\Omega,X;\pi)$ the complete metric space of all such maps endowed with the distance
\begin{equation*}
\dsf_{L^p(\Omega,X;\pi)}(f,g) := \bigg( \INTDom{\dsf_X^p(f(\omega),g(\omega))}{\Omega}{\pi(\omega)} \bigg)^{1/p}.
\end{equation*}
The latter is separable for $p \in [1,+\infty)$, and in the particular case in which $(X,\Norm{\cdot}_X)$ is a separable Banach space, the vector space $(L^p(\Omega,X;\pi), \NormL{\cdot}{p}{\Omega,X;\pi})$ coincides with the usual Bochner space.
\end{Def}

In the sequel, we will sometimes abuse notation and denote by $L^p(\Omega,E;\pi)$ the (not necessarily complete) metric space of $p$-integrable maps valued in a subset $E \subset X$ of a separable metric space.  

\begin{prop}[Some handy measurability results]
\label{prop:Measurability}
Let $(\Omega,\Apazo,\pi)$ be as above and suppose that $(X,\dsf_X(\cdot,\cdot)),(Y,\dsf_Y(\cdot,\cdot))$ are two Polish spaces. Then, the following holds. 
\begin{enumerate}
\item[$(a)$] For every Carathéodory map $\varphi : \Omega \times X \to Y$ and each $B \in \Bcal(Y)$, the set
\begin{equation*}
\Big\{ \omega \in \Omega ~\, \textnormal{s.t.}~ \varphi(\omega,x) \in B ~~ \text{for all $x \in X$} \Big\}
\end{equation*}
is $\pi$-measurable. Moreover, if $f : \Omega \to X$ is $\pi$-measurable, then the mapping $\omega \in \Omega \mapsto \varphi(\omega,f(\omega)) \in Y$ is also $\pi$-measurable.
\item[$(b)$] If both $(X,\dsf_X(\cdot,\cdot))$ and $(Y,\dsf_Y(\cdot,\cdot))$ are locally compact, then a map $\varphi : \Omega \times X \to Y$ is Carathéodory if and only if $\omega \mapsto \varphi(\omega) \in C^0(X,Y)$ is $\pi$-measurable for the topology of uniform convergence on compact sets.
\end{enumerate}
\end{prop}

In the next result, we recollect for the convenience of the reader a particular version of the far-reaching generalisation of Lusin's theorem provided in \cite[Theorem 7.1.13]{Bogachev}. The latter stems from the observation that Polish spaces are completely regular (see e.g. \cite[Theorem 2.46 and Lemma 3.20]{AliprantisB2006}).

\begin{thm}[A functional variant of Lusin's theorem]
\label{thm:Lusin}
Suppose that $(\Omega,\dsf_{\Omega}(\cdot,\cdot))$ is a Polish space, fix some $\pi \in \Mcal^+(\Omega)$ and let $(X,\Norm{\cdot}_X)$ be a separable Banach space. Then for every $\pi$-measurable map $\varphi : \Omega \to X$ and each $\epsilon > 0$, there exists $\varphi_{\epsilon} \in C^0_b(\Omega,X)$ such that 
\begin{equation*}
\pi \Big( \Big\{ \omega \in \Omega ~\,\textnormal{s.t.}~ \varphi(\omega) \neq \varphi_{\epsilon}(\omega)  \Big\} \Big) < \epsilon. 
\end{equation*}
Moreover if $\varphi : \Omega \to X$ is bounded, the latter can be chosen so that $\NormC{\varphi_{\epsilon}}{0}{\Omega,X} \,\leq\, \sup\limits_{\omega \in \Omega}\Norm{\varphi(\omega)}_X$.
\end{thm}

At this stage, recall that given two Polish spaces $(\Xb,\dsf_{\Xb}(\cdot,\cdot))$ and $(X,\dsf_X(\cdot,\cdot))$, the \textit{image} of an element $\Bmu \in \Pcal(\Xb)$ through a Borel (or more generally a $\Bmu$-measurable) map $\efrak : \Xb \to X$ is the unique probability measure satisfying
\begin{equation*}
\efrak_{\sharp} \Bmu(B) = \Bmu(\efrak^{-1}(B))
\end{equation*}
for each Borel set $B \in \Bcal(X)$. In the following theorem, we recollect a general version of the famed disintegration principle, which is borrowed from \cite[Theorem 5.3.1]{AGS}.

\begin{thm}[Disintegration theorem]
\label{thm:Disintegration}
Let $\Bmu \in \Pcal(\Xb)$ and $\mu := \efrak_{\sharp} \Bmu \in \Pcal(X)$ for some Borel map $\efrak : \Xb \to X$. Then, there exists a $\mu$-almost uniquely determined Borel family of probability measures $\{ \Bmu_x \}_{x \in X} \subset \Pcal(\Xb)$ which satisfy
\begin{equation}
\label{eq:DisintegrationCharac}
\left\{
\begin{aligned}
& \Bmu_{x}(\Xb \setminus \efrak^{-1}(x)) = 0 \hspace{0.85cm} \text{for $\mu$-almost every $x \in X$}, \\
& \INTDom{\Bphi(\xb)}{\Xb}{\Bmu(\xb)} = \INTDom{\bigg( \INTDom{\Bphi(\xb)}{\efrak^{-1}(x)}{\Bmu_{x}(\xb)} \bigg)}{X}{\mu(x)}, 
\end{aligned}
\right.
\end{equation}
for every bounded Borel map $\Bphi : \Xb \to \R$. In the particular case in which $\Xb := \Omega \times X$ and $\efrak := \pfrak_{\Omega} : \Omega \times X \to\Omega$ is the projection onto the first component, we will denote by $\Bmu = \INTDom{\mu_{\omega}}{\Omega}{\pi(\omega)}$ the disintegration of an element $\Bmu \in \Pcal(\Omega \times X)$ against its first marginal $\pi \in \Pcal(\Omega)$. 
\end{thm}

We conclude this first preliminary section by recalling a classical, yet seldom encountered characterisation of relative strong compactness in Bochner spaces, for which we refer e.g. to \cite[Theorem 4.7.28]{Bogachev} for a real-valued statement and \cite{Brooks1979} for its vector-valued counterpart. Therein, we denote 
\begin{equation*}
\INTDomdash{f(\omega)}{A}{\pi(\omega)} := \frac{1}{\pi(A)} \INTDom{f(\omega)}{A}{\pi(\omega)}
\end{equation*}
the average of any function $f \in L^p(\Omega,X;\pi)$ over a positive-measure set $A \in \Apazo$, and let
\begin{equation}
\label{eq:ConditionalExpectation}
\E_{\Ppazo}[f] := \sum_{i=1}^n \bigg( \INTDomdash{f(\omega)}{A_i}{\pi(\omega)}\bigg) \mathds{1}_{A_i} \in L^p(\Omega,X\;\pi)
\end{equation}
stand for its \textit{conditional expectation} over a finite partition $\Ppazo := (A_1,\ldots,A_n) \in \Apazo^n$ of $\Omega$ comprised of positive-measure sets.

\begin{thm}[Relative norm compactness in Bochner spaces]
\label{thm:CompactnessBochner}
Let $p \in [1,+\infty)$ and $(X,\Norm{\cdot}_X)$ be a separable Banach space. Then a subset $\Fpazo \subset L^p(\Omega,X;\pi)$ is relatively strongly compact if and only if the following holds. 
\begin{enumerate}
\item[$(i)$] For every $A \in \Apazo$ with positive measure, the set
\begin{equation*}
\bigg\{ \INTDomdash{f(\omega)}{A}{\pi(\omega)} ~\, \textnormal{s.t.}~  f \in \Fpazo \bigg\} \subset X
\end{equation*}
is relatively compact. 
\item[$(ii)$] For each $\epsilon > 0$, there exists $n_{\epsilon} \geq 1$ and a finite partition $\Ppazo_{\epsilon} := (A_1,\dots,A_{n_{\epsilon}})$ of $\Omega$ such that 
\begin{equation*}
\sup_{f \in \Fpazo} \NormL{f - \E_{\Ppazo_{\epsilon}}[f]}{p}{\Omega,X;\pi} \,\leq \epsilon. 
\end{equation*} 
\end{enumerate}
\end{thm}


\paragraph*{Optimal transport over Polish spaces} 

In this second preliminary section, we provide some basic optimal transport exposition, excerpted largely from \cite{AGS} (see also the references \cite{OTAM,villani1}). 

Given a Polish space $(X,\dsf_X(\cdot,\cdot))$, the space of probability measures $\Pcal(X)$ is naturally endowed with the so-called \textit{narrow topology}, which is determined by the notion of convergence
\begin{equation}
\label{eq:Narrow}
\mu_n ~\underset{n \to +\infty}{\rightharpoonup^*}~ \mu \qquad \textnormal{if and only if} \qquad \INTDom{\phi(x)}{X}{\mu_n(x)} ~\underset{n \to +\infty}{\longrightarrow}~ \INTDom{\phi(x)}{X}{\mu(x)}
\end{equation}
for each $\phi \in C^0_b(X,\R)$. The latter turns $\Pcal(X)$ into a separable and completely metrisable topological space (see e.g. \cite[Section 5.1]{AGS}). Given some $\phi \in C^0(X,Y)$, one may check that
\begin{equation}
\label{eq:ContinuityNarrow}
\mu \in \Pcal(X) \mapsto \phi_{\sharp \,} \mu \in \Pcal(Y)
\end{equation}
defines a narrowly continuous map, see for instance \cite[Section 5.2]{AGS}. Finally, we recall that a subset $\Kcal \subset \Pcal(X)$ is \textit{tight} if for every $\epsilon > 0$, there exists a compact set $K_{\epsilon} \subset X$ such that
\begin{equation*}
\sup_{\mu \in \Kcal} \mu(X \setminus K_{\epsilon}) < \epsilon.
\end{equation*}
Given some $p\in [1,+\infty)$, we denote by $\Pcal_p(X)$ the subset of all probability measures satisfying
\begin{equation*}
\INTDom{\dsf_X^p(x,x_0)}{X}{\mu(x)} < +\infty
\end{equation*}
for some (and thus all) $x_0 \in X$. The Wasserstein distance of order $p \in[1,+\infty)$ between two measures $\mu, \nu\in \Pcal_p(X)$ is then classically defined by 
\begin{equation*}
W_p(\mu,\nu) := \inf \bigg\{ \bigg(\INTDom{\dsf_X^p(x,y)}{X^2}{\gamma(x,y)}\bigg)^{1/p} ~\, \textnormal{s.t.}~ \gamma\in \Gamma(\mu,\nu) \bigg\}, 
\end{equation*}
where we denoted by  
\begin{equation*}
\Gamma(\mu,\nu) := \Big\{\gamma \in \Pcal(X^2) ~\, \textnormal{s.t.}~ \pfrak^1_{\sharp} \gamma = \mu ~~\text{and}~~ \pfrak^2_{\sharp} \gamma = \nu \Big\}
\end{equation*}
the set of so-called \textit{transport plans} between $\mu$ and $\nu$, with $\pfrak^1,\pfrak^2 : X \times X \to X$ standing respectively for the projections onto the first and second factors. In the sequel, we write $\Gamma_o(\mu,\nu)$ for the subset of optimal plans, which is always nonempty, see e.g. \cite[Section 7.1]{AGS}. In the following proposition, we compile several well-known results about Wasserstein spaces, for which we refer broadly to \cite{UsersGuidOT,AGS,Naldi2021}.

\begin{prop}[Miscellaneous facts about Wasserstein spaces]
\label{prop:Wass}
The spaces $(\Pcal_p(X),W_p(\cdot,\cdot))$ are complete and separable, and for every sequence $(\mu_n) \subset \Pcal_p(X)$ and any $\mu \in \Pcal_p(X)$, it holds that 
\begin{equation}
\label{eq:CharacWassConv1}
W_p(\mu_n,\mu) ~\underset{n \to +\infty}{\longrightarrow}~ 0 \qquad \textnormal{if and only if} \qquad \left\{
\begin{aligned}
\mu_n ~&\underset{n \to +\infty}{\rightharpoonup^*}~ \mu, \\
\INTDom{\dsf_X^p(x,x_0)}{X}{\mu_n(x)} ~&\underset{n \to +\infty}{\longrightarrow}~ \INTDom{\dsf_X^p(x,x_0)}{X}{\mu(x)}, 
\end{aligned}
\right.
\end{equation}
for some (and thus all) $x_0 \in X$, or equivalently if and only if
\begin{equation}
\label{eq:CharacWassConv2}
\INTDom{\phi(x)}{X}{\mu_n(x)} ~\underset{n \to +\infty}{\longrightarrow}~ \INTDom{\phi(x)}{X}{\mu(x)}
\end{equation}
for every $\phi \in C^0_p(X,\R)$. Alternatively, a subset $\Kcal \subset  \Pcal_p(X)$ is relatively compact for the $W_p$-topology if and only if it is tight and uniformly $p$-integrable, that is, if there exists some $x_0 \in X$ such that
\begin{equation*}
\sup_{\mu \in \Kcal} \INTDom{\dsf_X^p(x,x_0)}{\{x \;\textnormal{s.t.}\,\dsf_X(x,x_0) \geq k \;\}}{\mu(x)} ~\underset{k \to +\infty}{\longrightarrow}~ 0.
\end{equation*}
Finally, given some $\phi \in \Lip(X,Y)$, one has that 
\begin{equation}
\label{eq:LipEstWass}
W_p(\phi_{\sharp}\mu,\phi_{\sharp}\nu) \leq \Lip(\phi \, ; X) W_p(\mu,\nu)
\end{equation}
for all $\mu,\nu \in \Pcal_p(X)$, and in the particular case in which $\supp(\mu),\supp(\nu) \subset K$ for some compact set $K \subset X$, the following Kantorovich-Rubinstein duality formula
\begin{equation}
\label{eq:KantorovichRubinstein}
\begin{aligned}
W_1(\mu,\nu) = \sup \bigg\{ \INTDom{\phi(x)}{X}{(\mu-\nu)(x)} ~\,\textnormal{s.t.}~ \Lip(\phi\,;K) \leq 1 \bigg\}
\end{aligned}
\end{equation}
holds.
\end{prop}


\paragraph*{Optimal transport over fibred measure spaces.}

In this section, we recall some of the elementary definitions pertaining to fibred measure spaces and their corresponding Wasserstein distances, following \cite[Chapter 3]{Crauel2002} and \cite{Peszek2023}.

Given a complete probability space $(\Omega,\Apazo,\pi)$ over a Polish space and a separable Banach space $(X,\Norm{\cdot}_X)$, we define the space of \textit{fibred probability measures}, also known as \textit{Young measures}, over the product $\Omega \times X$ by
\begin{equation*}
\Pcal_{\pi}(\Omega \times X) := \Big\{ \Bmu \in \Pcal(\Omega \times X) ~\, \textnormal{s.t.}~ (\pfrak_{\Omega})_{\sharp} \Bmu = \pi \Big\}. 
\end{equation*}
In what follows, we shall write $C^0_{\pi,b}(\Omega \times X,\R)$ to refer to the Banach space of all Carathéodory maps $\varphi : \Omega \times X \to \R$ satisfying $\varphi(\omega) \in C^0_b(X,\R)$ for $\pi$-almost every $\omega \in \Omega$ and 
\begin{equation*}
\Norm{\varphi}_{C^0_{\pi,b}(\Omega \times X,\R)} \, := \INTDom{\, \NormC{\varphi(\omega)}{0}{X,\R}}{\Omega}{\pi(\omega)} < +\infty.
\end{equation*}
Given some $k \in \{0,1\}$, we likewise denote by $C^k_{\pi,c}(\Omega \times X,\R)$ the normed space of all Carathéodory mappings $\varphi : \Omega \times X \to \R$ such that $\varphi(\omega) \in C^k_c(X,\R)$ for $\pi$-almost every $\omega \in \Omega$ and 
\begin{equation*}
\Norm{\varphi}_{C^k_{\pi,c}(\Omega \times X,\R)} \, := \INTDom{\NormC{\varphi(\omega)}{k}{X,\R}}{\Omega}{\pi(\omega)} < +\infty. 
\end{equation*}
Observe that if $(X,\Norm{\cdot}_X)$ is locally compact, then for any $k \in \{0,1\}$ and each $\varphi \in C^k_{\pi,c}(\Omega \times X,\R)$, it follows from Proposition \ref{prop:Measurability} above that $\omega \in \Omega \mapsto \varphi(\omega) \in C^k_c(X,\R)$ is $\pi$-measurable. Moreover, the completion of $(C^k_{\pi,c}(\Omega \times X,\R),\Norm{\cdot}_{C^k_{\pi,c}(\Omega \times X,\R)})$ is then isomorphic to the Bochner space $L^1(\Omega,C^k_0(X,\R);\pi)$. Following \cite{Crauel2002,Valadier1990}, we now introduce relatives of the usual narrow topology adapted to $\Pcal_{\pi}(\Omega \times X)$.

\begin{Def}[Fibred narrow topology and stable topology]
\label{def:FibredNarrow}
We define the \textnormal{fibred narrow topology} as the coarsest topology for which the mapping
\begin{equation*}
\Bmu \in \Pcal_{\pi}(\Omega \times X) \mapsto \INTDom{\varphi(\omega,x)}{\Omega \times X}{\Bmu(\omega,x)} \in \R
\end{equation*}
is continuous for all $\varphi \in C^0_{\pi,b}(\Omega \times X,\R)$. Likewise, we call \textnormal{stable topology} the coarsest topology for which the previous mapping is continuous whenever $(\omega,x) \in \Omega \times X \mapsto \varphi(\omega,x) \in \R$ is bounded, Borel measurable in $\omega \in \Omega$ and continuous in $x \in X$.
\end{Def}

\begin{rmk}[Comparison between the stable, narrow and fibred narrow topologies]
\label{rmk:NarrowFibred}
Although they may seem different at first glance, it so happens that the stable, usual and fibred narrow topologies all coincide over $\Pcal_{\pi}(\Omega \times X)$ once a reference measure $\pi \in \Pcal(\Omega)$ has been fixed, see e.g. \cite{Beiglbock2018} and the discussions in \cite[Remark 2.15]{Peszek2023}. Hence, both the stable and fibred narrow topology should be seen as means of making the analysis more transparent and synthetic, as they allow one to directly work with test functions having merely measurable dependence in $\omega \in \Omega$, without systematically resorting to approximation arguments by continuous functions.
\end{rmk}


In the following proposition, we recollect basic topological properties about the fibred narrow topology borrowed from \cite{Crauel2002}, which hold e.g. under the assumption that $(\Omega,\Apazo,\pi)$ is a complete probability space over a Polish space.

\begin{prop}[Topological properties of the fibred narrow topology]
\label{prop:FibredTightness}
The space $\Pcal_{\pi}(\Omega \times X)$ endowed with the fibred narrow topology is separable and completely metrisable. Moreover, a subset $\Kcal \subset \Pcal_{\pi}(\Omega \times X)$ is relatively fibred narrowly compact if and only if it is \textnormal{fibrely tight}, namely if for each $\epsilon > 0$ there exists a compact set $K_{\epsilon} \subset X$ such that
\begin{equation*}
\sup_{\Bmu \in \Kcal} \Bmu \Big(\Omega \times (X \setminus K_{\epsilon}) \Big) < \epsilon, 
\end{equation*}
or equivalently if 
\begin{equation*}
\sup_{\Bmu \in \Kcal} \INTDom{\psi(x)}{\Omega \times X}{\Bmu(\omega,x)} < +\infty
\end{equation*}
for some function $\psi : X \to [0,+\infty]$ with compact sublevels. 
\end{prop}

\begin{proof}
Since we assumed that $(\Omega,\Apazo,\pi)$ is a complete probability space over a Polish space, it follows from \cite[Theorem 5.6]{Crauel2002} that $\Pcal_{\pi}(\Omega \times X)$ endowed with the fibred narrow topology is separable and completely metrisable. The equivalence between relative fibred narrow compactness and tightness can be found e.g. in \cite[Proposition 4.3]{Crauel2002} (see also the discussion in \cite[Remark 2.15]{Peszek2023}), whereas their characterisation in terms of uniformly bounded integrals against a coercive function follows from a simple adaptation e.g. of \cite[Remark 5.1.5]{AGS}. 
\end{proof}

By the disintegration principle recalled in Theorem \ref{thm:Disintegration} above, any element $\Bmu \in \Pcal_{\pi}(\Omega \times X)$ admits the representation $\Bmu = \INTDom{\mu_{\omega}}{\Omega}{\pi(\omega)}$ for some $\pi$-almost uniquely determined Borel family of measures $\{\mu_{\omega}\}_{\omega \in \Omega} \subset \Pcal(X)$, which can be equivalently seen as a Borel map $\omega \in \Omega \mapsto \mu_{\omega} \in \Pcal(X)$. Given a real number $p \in [1,+\infty)$, we define the \textit{fibred moment of order $p$} of some element $\Bmu \in \Pcal_{\pi, \,p}(\Omega \times X)$ by
\begin{equation*}
\Mpazo_{\pi, \,p}(\Bmu) := \bigg( \INTDom{\Mpazo_p^p(\mu_{\omega})}{\Omega}{\pi(\omega)} \bigg)^{1/p} := \, \bigg( \INTDom{\INTDom{\Norm{x}_X^p}{X}{\mu_{\omega}(x)}}{\Omega}{\pi(\omega)} \bigg)^{1/p}.
\end{equation*}
In keeping with standard notations in classical optimal transport theory, we let $\Pcal_{\pi, \,p}(\Omega \times X)$ be the corresponding space of measures with finite fibred moment of order $p \in[1,+\infty)$, and introduce the so-called \textit{fibred} -- or \textit{conditional} -- \textit{Wasserstein distance}.

\begin{Def}[Fibred Wasserstein distances]
Given two measures $\Bmu,\Bnu \in \Pcal_{\pi,\,p}(\Omega \times X)$, the \textnormal{fibred Wasserstein distance} of order $p \in [1,+\infty)$ between $\Bmu$ and $\Bnu$ is defined by   
\begin{equation}
\label{eq:FibredWasserstein}
W_{\pi, \,p}(\Bmu,\Bnu) := \bigg( \INTDom{W_p^p(\mu_{\omega},\nu_{\omega})}{\Omega}{\pi(\omega)} \bigg)^{1/p}
\end{equation}
where $\{\mu_{\omega}\}_{\omega \in \Omega},\{\nu_{\omega}\}_{\omega \in \Omega} \subset \Pcal_p(X)$ stand for their respective disintegrations against $\pi \in \Pcal(\Omega)$.
\end{Def}

Observe that the latter quantity is well-defined as a consequence of the Borel measurability of the families of probability measures $\{\mu_{\omega}\}_{\omega \in\Omega},\{\nu_{\omega}\}_{\omega \in \Omega} \subset \Pcal_p(X)$, see e.g. \cite[Remark 3.2]{Peszek2023} for more details. Following again the latter reference, it can be shown that 
\begin{equation*}
W_{\pi, \,p}(\Bmu,\Bnu) = \bigg( \INTDom{\INTDom{\Norm{x-y}_X^p}{X^2}{\gamma_{\omega}(x,y)}}{\Omega}{\pi(\omega)} \bigg)^{1/p}
\end{equation*}
for any Borel family of plans $\{\gamma_{\omega}\}_{\omega \in \Omega} \subset \Pcal(X^2)$ satisfying $\gamma_{\omega} \in \Gamma_o(\mu_{\omega},\nu_{\omega})$ for $\pi$-almost every $\omega \in \Omega$, whose existence follows from standard measurable selection principles, see e.g. \cite[Lemma 12.4.7]{AGS} or \cite[Lemma 3.6]{Peszek2023}. 

\begin{prop}[Basic topological properties of fibred Wasserstein spaces]
\label{prop:BasicTopo}
The metric spaces $(\Pcal_{\pi, \,p}(\Omega \times X),W_{\pi, \,p}(\cdot,\cdot))$ are complete and separable. 
\end{prop}

\begin{proof}
The completeness and separability of $(\Pcal_{\pi, \,p}(\Omega \times X),W_{\pi, \,p}(\cdot,\cdot))$ can be proven e.g. by a simple adaptation of the arguments in \cite[Appendix A]{Peszek2023} -- which deals with fibred measures over $\Omega \times X := \R^d \times \R^d$ -- to the case where $(\Omega,\Apazo,\pi)$ is a complete probability space over a Polish space and $(X,\Norm{\cdot}_X)$ is a separable Banach space.
\end{proof}

We close this preliminary section by noting that, upon adapting the lower semicontinuity arguments detailed e.g. in the proof of \cite[Proposition 3.12]{Peszek2023} for $X := \R^d$, one may show that 
\begin{equation}
\label{eq:LscFibred}
\Mpazo_{\pi,\,p}(\Bmu) \leq \liminf_{n \to +\infty} \Mpazo_{\pi,\, p}(\Bmu_n) \qquad \textnormal{and} \qquad W_{\pi,\,p}(\Bmu,\Bnu) \leq \liminf_{n \to +\infty} W_{\pi,\,p}(\Bmu_n,\Bnu_n)
\end{equation}
whenever $(\Bmu_n),(\Bnu_n) \subset \Pcal_{\pi,\,p}(\Omega \times X)$ are converging (fibred) narrowly towards $\Bmu,\Bnu \in \Pcal_{\pi,\,p}(\Omega \times X)$.

\begin{rmk}[On the various metrics one could endow $\Pcal_{\pi}(\Omega \times X)$ with]
\label{rmk:Metrics}
As already mentioned in the introduction, the space of fibred probability measures $\Pcal_{\pi}(\Omega\times X)$ can be equipped with a variety of distances other than fibred Wasserstein metrics. A small variant thereof is the $L^qW_p$-distances introduced in \cite{KitagawaTakatsu24}, and given by 
\begin{equation*}
W_{\pi, \,p, \, q}(\Bmu,\Bnu) := \bigg( \INTDom{W_p^q(\mu_{\omega},\nu_{\omega})}{\Omega}{\pi(\omega)} \bigg)^{1/q}.
\end{equation*}
Another common alternative to fibred Wasserstein metrics are the fibred bounded-Lipschitz distances  
\begin{equation*}
d_{p,\rm BL}(\Bmu,\Bnu) := \left\{
\begin{aligned}
& \bigg(\INTDom{\dsf_{\rm BL}(\mu_\omega,\nu_\omega)^p}{\Omega}{\pi(\omega)}\bigg)^{1/p} ~~ & \text{if $p\in [1,+\infty)$}, \\
& ~\sup_{\omega\in\Omega} \dsf_{\rm BL}(\mu_\omega,\nu_\omega) ~~ & \text{if $p=+\infty$}, 
\end{aligned}
\right.
\end{equation*}
which are stronger than 
the underlying narrow topology, and used e.g. in \cite{KaliuzhnyiMedvedev2018} for $p=1$, in \cite{KuehnXu2022} for $p=+\infty$ and in \cite{AyiPouradierDuteilPoyato2024} for a general $p\in[1,+\infty]$. A fairly different and interesting choice, inspired by the standard cut-distance for graphons (see e.g. \cite{Lovasz2012}), is the so-called cut-distance for probability graphons considered in \cite{AbrahamDelmasWeibel23} in the case where $\Omega:=[0,1]^2$, and defined by
\begin{equation*}
\dsf_{\square,\rm m}(\Bmu,\Bnu) = \sup_{S,T\subset [0,1]} \dsf_{\rm m}\bigg(\INTDom{\mu_{s,t}}{S \times T}{s\mathrm{d}t},\INTDom{\nu_{s,t}}{S \times T}{s\mathrm{d}t}\bigg).
\end{equation*}
Therein $\dsf_{\rm m}(\cdot,\cdot)$ can be any distance that metrises the weak topology on the space of Radon measures with mass less than 1, such as the Bounded-Lipschitz distance, the Fortet-Mourier distance or the Prokhorov metric.
\end{rmk}


\section{Classical and fibred Wasserstein topologies over \texorpdfstring{$\Pcal_{\pi,\,p}(\Omega \times X)$}{}}
\label{section:TopologyWass}

\setcounter{equation}{0} \renewcommand{\theequation}{\thesection.\arabic{equation}}

In this section, our goal is to study in greater detail the topologies induced by both the classical and fibred Wasserstein distances over $\Pcal_{\pi,\,p}(\Omega \times X)$. In order for all the relevant objects to be well-defined and make sense, we may sometimes assume that $\pi \in \Pcal_p(\Omega)$ in what follows. Note that, although some of the ensuing contributions could be (partly) derived from a combination of results from \cite{AGS,Attouch2005,Castaing2004,Peszek2023}, we deemed it relevant and insightful to present them in a self-contained and unified way.

In the sequel, we endow $\Omega \times X$ with the product metric given by 
\begin{equation*}
\dsf_{\Omega \times X}\Big((\omega,x),(\theta,y) \Big) := \Big( \, \dsf_{\Omega}^p(\omega,\theta) \,+ \Norm{x-y}_X^p \Big)^{1/p}
\end{equation*}
and begin our developments by recalling an elementary estimate between the classical and fibred Wasserstein metrics. 

\begin{lem}[Basic estimate on Wasserstein metrics]
\label{lem:MetricEstimate}
For every $\Bmu,\Bnu \in \Pcal_{\pi,\,p}(\Omega \times X)$, it holds that
\begin{equation}
\label{eq:WassDistComp}
W_p(\Bmu,\Bnu) \leq W_{\pi, \,p}(\Bmu,\Bnu).
\end{equation} 
\end{lem}

\begin{proof}
The metric estimate \eqref{eq:WassDistComp} follows from the observation that, given a Borel family of measures $\{\gamma_{\omega}\}_{\omega \in \Omega} \subset \Pcal_p(X^2)$ satisfying $\gamma_{\omega} \in \Gamma_o(\mu_{\omega},\nu_{\omega})$ for $\pi$-almost every $\omega \in \Omega$, the admissible plan $\Bgamma \in \Pcal_{\pi, \,p}((\Omega \times X)^2)$ defined by 
\begin{equation*}
\INTDom{\phi(\omega,x,\theta,y)}{(\Omega \times X)^2}{\Bgamma(\omega,x,\theta,y)} = \INTDom{\INTDom{\phi(\omega,x,\omega,y)}{X^2}{\gamma_{\omega}(x,y)}}{\Omega}{\pi(\omega)}
\end{equation*}
for every bounded Borel map $\phi : (\Omega \times X)^2 \to \R$ is an element of $\Gamma(\Bmu,\Bnu)$. Therefore
\begin{equation*}
\begin{aligned}
W_p(\Bmu,\Bnu) &\leq \bigg( \INTDom{\Big( \dsf_{\Omega}^p(\omega,\theta) \, + \Norm{x-y}_X^p \hspace{-0.05cm} \Big)}{(\Omega \times X)^2}{\Bgamma(\omega,x,\theta,y)} \bigg)^{1/p} \\
& = \bigg( \INTDom{\INTDom{\Norm{x-y}_X^p}{X^2}{\gamma_{\omega}(x,y)}}{\Omega}{\pi(\omega}) \bigg)^{1/p} \\
& = W_{\pi,\,p}(\Bmu,\Bnu).
\end{aligned}
\end{equation*}
which is the desired inequality. 
\end{proof}

We continue our comparison between both Wasserstein topologies by characterising the underlying relatively compact subsets of $\Pcal_{\pi,\,p}(\Omega \times X)$. Such results will prove relevant in our later developments, and we start by focusing on the topology induced by the $W_p$-metric.

\begin{prop}[Characterisation of relative compactness in the classical Wasserstein topology] 
\label{prop:ClassicalCompactness}
Suppose that $\pi \in \Pcal_p(\Omega)$. Then, a subset $\Kcal \subset \Pcal_{\pi,\,p}(\Omega \times X)$ is relatively compact in the $W_p$-metric \textnormal{if and only if} it is fibrely tight in the sense of Proposition \ref{prop:FibredTightness}, and such that
\begin{equation}
\label{eq:RelativeCompactnessWass}
\sup_{\Bmu \in \Kcal} \INTDom{\Norm{x}_X^p}{\{ (\omega,x) \; \textnormal{s.t.} \, \dsf_{\Omega}(\omega,\omega_0) + \| x \|_X \geq k \}}{\Bmu(\omega,x)} ~\underset{k \to +\infty}{\longrightarrow}~ 0,
\end{equation}
which can be equivalently recast as
\begin{equation}
\label{eq:RelativeCompactnessWassBis}
\sup_{\Bmu \in \Kcal} \bigg( \INTDom{\Mpazo_p^p(\mu_{\omega})}{\{ \omega \; \textnormal{s.t.} \, \dsf_{\Omega}(\omega,\omega_0) \geq k \}}{\pi(\omega)} + \INTDom{\Norm{x}_X^p}{\{ (\omega,x) \; \textnormal{s.t.} \, \| x \|_X \geq k \}}{\Bmu(\omega,x)}\bigg) ~\underset{k \to +\infty}{\longrightarrow}~ 0
\end{equation}
for some (and thus all) $\omega_0 \in \Omega$
\end{prop}

\begin{proof}
We begin by noting that, as a consequence e.g. of \cite[Theorem 4.4]{Crauel2002}, a subset $\Kcal \subset \Pcal_{\pi, \,p}(\Omega \times X)$ is fibrely tight if and only if it is narrowly relatively compact, which by the standard Prokhorov theorem (see for instance \cite[Theorem 5.1.3]{AGS}) amounts to the set also being tight in the usual sense. Hence, following Proposition \ref{prop:Wass} above, there only remains to show that \eqref{eq:RelativeCompactnessWass} is equivalent to the usual uniform $p$-integrability of the set $\Kcal \subset \Pcal_{\pi,\,p}(\Omega \times X)$. 

In the case in which $\Kcal \subset \Pcal_{\pi,\,p}(\Omega \times X)$ is uniformly $p$-integrable, it is clear that \eqref{eq:RelativeCompactnessWass} holds, since
\begin{equation*}
\begin{aligned}
& \sup_{\Bmu \in \Kcal} \INTDom{\Norm{x}_X^p}{\{ (\omega,x) \; \textnormal{s.t.} \, \dsf_{\Omega}(\omega,\omega_0) + \| x \|_X \geq k \}}{\Bmu(\omega,x)} \\
& \hspace{1.1cm} \leq \sup_{\Bmu \in \Kcal} \INTDom{\Big(\, \dsf_{\Omega}^p(\omega,\omega_0) + \Norm{x}_X^p \Big)}{\{ (\omega,x) \; \textnormal{s.t.} \, \dsf_{\Omega}(\omega,\omega_0) + \| x \|_X \geq k \}}{\Bmu(\omega,x)} ~\underset{k \to +\infty}{\longrightarrow}~ 0.
\end{aligned}
\end{equation*}
To prove the converse implication, observe that for each $\Bmu \in \Pcal_{\pi,\,p}(\Omega \times X)$, it holds that
\begin{equation*}
\begin{aligned}
& \INTDom{\dsf_{\Omega}^p(\omega,\omega_0)}{\{ (\omega,x) \; \textnormal{s.t.} \, \dsf_{\Omega}(\omega,\omega_0) + \| x \|_X \geq k \}}{\Bmu(\omega,x)} \\
& \leq \INTDom{\dsf_{\Omega}^p(\omega,\omega_0)}{\{ \omega \; \textnormal{s.t.} \, \dsf_{\Omega}(\omega,\omega_0) \geq k/2 \}}{\pi(\omega)} + \INTDom{\INTDom{\dsf_{\Omega}^p(\omega,\omega_0)}{\{ x \;\textnormal{s.t.}\, \|x\|_X \geq k/2 \}}{\mu_{\omega}(x)}}{\{ \omega \; \textnormal{s.t.} \, \dsf_{\Omega}(\omega,\omega_0) < k/2\}}{\pi(\omega)} \\
& \leq \INTDom{\dsf_{\Omega}^p(\omega,\omega_0)}{\{ \omega \; \textnormal{s.t.} \, \dsf_{\Omega}(\omega,\omega_0) \geq k/2 \}}{\pi(\omega)} + \Big( \frac{k}{2}\Big)^p \INTDom{\mu_{\omega} \Big( \Big\{ x \in X ~\,\textnormal{s.t}~ \Norm{x}_X \geq k/2 \Big\}\Big)}{\Omega}{\pi(\omega)} \\
& \leq \INTDom{\dsf_{\Omega}^p(\omega,\omega_0)}{\{ \omega \; \textnormal{s.t.} \, \dsf_{\Omega}(\omega,\omega_0) \geq k/2\}}{\pi(\omega)} + \INTDom{\Norm{x}_X^p}{\{ (\omega,x) \;\textnormal{s.t.}\, \|x\|_X \geq k/2 \}}{\Bmu(\omega,x)}
\end{aligned}
\end{equation*}
by a simple application of Chebyshev's inequality. Since we assumed that $\pi \in \Pcal_p(\Omega)$, it follows from \eqref{eq:RelativeCompactnessWass} together with the previous estimates that 
\begin{equation*}
\sup_{\Bmu \in \Kcal} \INTDom{\Big( \, \dsf_{\Omega}^p(\omega,\omega_0) + \Norm{x}_X^p \Big)}{\{ (\omega,x) \; \textnormal{s.t.} \, \dsf_{\Omega}(\omega,\omega_0) + \| x \|_X \geq k \}}{\Bmu(\omega,x)} ~\underset{k \to +\infty}{\longrightarrow}~ 0.
\end{equation*}
The equivalence between \eqref{eq:RelativeCompactnessWass} and \eqref{eq:RelativeCompactnessWassBis} may then be inferred from the basic observation that 
\begin{equation*}
\begin{aligned}
\Big\{(\omega,x) ~\,\textnormal{s.t.}~ \dsf_{\Omega}(\omega,\omega_0) \geq k ~\,\text{or}~ \Norm{x}_X \geq k \Big\} & \subset  \Big\{(\omega,x) ~\,\textnormal{s.t.}~ \dsf_{\Omega}(\omega,\omega_0) + \Norm{x}_X \geq k \Big\} \\
& \subset \Big\{(\omega,x) ~\,\textnormal{s.t.}~ \dsf_{\Omega}(\omega,\omega_0) \geq k/2 ~\,\text{or}~ \Norm{x}_X \geq k/2 \Big\}
\end{aligned}
\end{equation*}
for each $k \geq 1$, and the conclusion simply follows by monotonicity of the integral. 
\end{proof}

In line with what precedes, we exhibit in the following theorem a characterisation of relatively compact subsets of $\Pcal_{\pi,\,p}(\Omega \times X)$ for the $W_{\pi,\,p}$-metric, which answers an open question raised in \cite[Remark 3.14]{Peszek2023}. Our result builds upon the observation that
\begin{equation*}
\Pcal_{\pi, \,p}(\Omega \times X) \simeq L^p \big( \Omega,\Pcal_p(X);\pi \big),
\end{equation*}
and is inspired by the classical oscillation-based relative compactness criterion for subsets of Lebesgue-Bochner spaces recollected in Theorem \ref{thm:CompactnessBochner} above. In this context, we let $\INTDomdash{\mu_{\omega}}{A}{\pi(\omega)} \in \Pcal_p(X)$ stand for the average of a Borel family $\{ \mu_{\omega}\}_{\omega \in \Omega} \subset \Pcal_p(X)$ over a positive-measure set $A \in \Apazo$, defined by 
\begin{equation}
\label{eq:AverageMeasure}
\INTDom{\phi(x)}{X}{\Big( \mathsmaller{\INTDomdash{\mu_{\omega}}{A}{\pi(\omega)}} \Big)(x)} := \INTDomdash{\INTDom{\phi(x)}{X}{\mu_{\omega}(x)}}{A}{\pi(\omega)}
\end{equation}
for every bounded Borel map $\phi : X \to \R$. Given a finite partition $\Ppazo := (A_1,\dots,A_n)$ of $\Omega$ into disjoint subsets of positive measure, we also denote by 
\begin{equation}
\label{eq:CondExpMeasure}
\E_{\Ppazo}[\Bmu]  := \INTDom{\, \sum_{i=1}^n \bigg( \INTDomdash{\mu_{\theta}}{A_i}{\pi(\theta)} \bigg) \mathds{1}_{A_i}(\omega)}{\Omega}{\pi(\omega)} \in \Pcal_{\pi, \,p}(\Omega \times X)
\end{equation}
the associated conditional expectation of an element $\Bmu \in \Pcal_{\pi, \,p}(\Omega \times X)$ with respect to $\Ppazo$. 

\begin{thm}[Characterisation of relative compactness in the fibred Wasserstein metric]
\label{thm:FibredCompactness}
A subset $\Kcalb \subset \Pcal_{\pi,\, p}(\Omega \times X)$ is relatively compact for the $W_{\pi,\,p}$-metric \textnormal{if and only if} the following holds.
\begin{enumerate}
\item[$(i)$] For every $A \in \Apazo$ with positive measure, the set 
\begin{equation*}
\bigg\{ \INTDomdash{\mu_{\omega}}{A}{\pi(\omega)} ~\, \textnormal{s.t.}~ \Bmu \in \Kcalb \bigg\} \subset \Pcal_p(X)
\end{equation*}
is relatively compact. 
\item[$(ii)$] For each $\epsilon > 0$, there exists a finite partition $\Ppazo_{\epsilon} = (A_1,\dots,A_{n_{\epsilon}})$ of $\Omega$ such that 
\begin{equation*}
\sup_{\Bmu \in \Kcalb} W_{\pi, \,p}( \Bmu , \E_{\Ppazo_{\epsilon}}[\Bmu]) \leq \epsilon. 
\end{equation*}
\end{enumerate}
\end{thm}

\begin{proof}
We start by assuming that $\Kcalb \subset \Pcal_{\pi, \,p}(\Omega \times X)$ is a relatively compact set and establish the necessity of conditions $(i)$ and $(ii)$. Regarding item $(i)$, it follows from the observation that, given a subset $A \in \Apazo$ with positive measure and a Borel family of plans satisfying $\gamma_{\omega} \in \Gamma_o(\mu_{\omega},\nu_{\omega})$ for $\pi$-almost every $\omega \in \Omega$, it holds that
\begin{equation}
\label{eq:AverageWassEst}
\begin{aligned}
W_p \Big( \mathsmaller{\INTDomdash{\mu_{\omega}}{A}{\pi(\omega)}} , \mathsmaller{\INTDomdash{\nu_{\omega}}{A}{\pi(\omega)}} \Big) & \leq \bigg( \frac{1}{\pi(A)} \INTDom{\INTDom{\|x-y\|_X^p}{X^2}{\gamma_{\omega}(x,y)}}{A}{\pi(\omega)} \bigg)^{1/p} \\
& \leq \frac{W_{\pi, \,p}(\Bmu,\Bnu)}{\pi(A)^{1/p}}
\end{aligned}
\end{equation}
for every $\Bmu,\Bnu \in \Kcalb$. Then, because $\Kcalb \subset \Pcal_{\pi,\,p}(\Omega \times X)$ is relatively compact, there exists for each $\epsilon > 0$ an integer $m_{\epsilon} \geq 1$ along with measures $\{\Bmu_k\}_{k=1}^{m_{\epsilon}} \subset \Pcal_{\pi,\,p}(\Omega \times X)$ such that
\begin{equation*}
\Kcalb \subset \bigcup_{k=1}^{m_{\epsilon}} \B_{\Pcal_{\pi, \,p}(\Omega \times X)} \Big( \Bmu_k \, , \pi(A)^{1/p} \epsilon \Big), 
\end{equation*}
where we denoted by $\B_{\Pcal_{\pi, \,p}(\Omega \times X)}(\Bmu,r) \subset \Pcal_{\pi,\,p}(\Omega \times X)$ the metric ball of radius $r >0$ centred at some $\Bmu \in \Pcal_{\pi,\,p}(\Omega \times X)$. This, combined with the estimate from \eqref{eq:AverageWassEst}, implies that 
\begin{equation*}
\bigg\{ \INTDomdash{\mu_{\omega}}{A}{\pi(\omega)} ~\, \textnormal{s.t.}~ \Bmu \in \Kcalb \bigg\} \subset \bigcup_{k=1}^{m_{\epsilon}} \B_{\Pcal_p(X)} \Big( \mathsmaller{\INTDomdash{\mu_{\omega}^k}{A}{\pi(\omega)}} \, , \epsilon \Big)
\end{equation*}
which precisely amounts to this set being relatively compact. We now establish item $(ii)$ by noting again that for every $\epsilon > 0$, there exists an integer $m_{\epsilon} \geq 1$ and a collection $\{\Bmu_k\}_{k=1}^{m_{\epsilon}} \subset \Kcalb$
such that
\begin{equation*}
\Kcalb \subset \bigcup_{k=1}^{m_{\epsilon}} \B_{\Pcal_{\pi, \,p}(\Omega \times X)} \Big( \Bmu_k \, ,\frac{\epsilon}{4} \Big). 
\end{equation*}
Denoting by $\{\mu_{\omega}^k\}_{\omega \in \Omega} \subset \Pcal_p(X)$ their respective disintegrations against $\pi \in \Pcal(\Omega)$, which can be equivalently seen as Borel maps $\omega \in \Omega  \mapsto \mu_{\omega}^k \in \Pcal_p(X)$, it follows e.g. from \cite[Chapter 8]{Aubin1990} (see also \cite[Appendix A]{Peszek2023} for a proof) that there exist simple Borel functions $\omega \in \Omega \mapsto \nu_{\omega}^k \in \Pcal_p(X)$ such that 
\begin{equation*}
\bigg( \INTDom{W_p^p(\mu_{\omega}^k,\nu_{\omega}^k)}{\Omega}{\pi(\omega)} \bigg)^{1/p} \leq \frac{\epsilon}{4}
\end{equation*}
for each $k \in \{1,\dots,m_{\epsilon}\}$. In particular, for every $\Bmu \in \Kcalb$, there exists some $k \in \{1,\dots,m_{\epsilon}\}$ for which 
\begin{equation*}
W_{\pi, \,p}(\Bmu,\Bnu_k) \leq \frac{\epsilon}{2}, 
\end{equation*}
where $\Bnu_k := \INTDom{\nu_{\omega}^k}{\Omega}{\pi(\omega)}$. Letting now $\Ppazo_{\epsilon}$ be any partition of $\Omega$ into positive-measure sets over which the simple functions $\omega \in \Omega \mapsto \nu_{\omega}^k \in \Pcal_p(X)$ are constant, which implies in turn that $\E_{\Ppazo_{\epsilon}}[\Bnu_k] = \Bnu_k$ for each $k \in \{1,\dots,m_{\epsilon}\}$, it follows that 
\begin{equation*}
\begin{aligned}
W_{\pi, \,p}( \Bmu , \E_{\Ppazo_{\epsilon}}[\Bmu]) & \leq W_{\pi, \,p}(\Bmu,\Bnu_k) + W_{\pi, \,p}( \Bnu_k , \E_{\Ppazo_{\epsilon}}[\Bnu_k]) + W_{\pi,\,p} \Big(\E_{\Ppazo_{\epsilon}}[\Bnu_k] , \E_{\Ppazo_{\epsilon}}[\Bmu] \Big) \\
& \leq 2 W_{\pi, \,p}(\Bmu,\Bnu_k) \\
& \leq \epsilon
\end{aligned}
\end{equation*}
uniformly with respect to $\Bmu \in \Kcalb$, where we leveraged the basic identity 
\begin{equation*}
W_{\pi,\,p} \Big(\E_{\Ppazo_{\epsilon}}[\Bmu],\E_{\Ppazo_{\epsilon}}[\Bnu_k] \Big) \leq W_{\pi,\,p}(\Bmu,\Bnu_k)
\end{equation*}
that one can readily obtain by raising the first inequality in \eqref{eq:AverageWassEst} to the power $p \in [1,+\infty)$ and summing over the elements of $\Ppazo_{\epsilon}$. 

We now prove that conditions $(i)$ and $(ii)$ of Theorem \ref{thm:FibredCompactness} entail the relative compactness of $\Kcalb \subset \Pcal_{\pi, \,p}(\Omega \times X)$. To begin with, observe that for each $\epsilon > 0$, there exists some $n_{\epsilon} \geq 1$ and a finite partition $\Ppazo_{\epsilon} := (A_1,\dots,A_{n_{\epsilon}})$ of $\Omega$ whose elements have positive measure, such that 
\begin{equation}
\label{eq:BallCover1}
\sup_{\Bmu \in \Kcalb} \bigg( \INTDom{W_p^p \big(\mu_{\omega},\E_{\Ppazo_{\epsilon}}[\Bmu]_{\omega} \big)}{\Omega}{\pi(\omega)} \bigg)^{1/p} \leq \frac{\epsilon}{2}
\end{equation}
as a consequence of assumption $(ii)$. At this stage, note that under assumption $(i)$, the sets 
\begin{equation*}
\Acal_k := \bigg\{ \INTDomdash{\mu_{\omega}}{A_k}{\pi(\omega)} ~\, \textnormal{s.t.}~  \Bmu \in \Kcalb \bigg\} \subset \Pcal_p(X)
\end{equation*}
are relatively compact for each $k \in \{1,\dots,n_{\epsilon}\}$, which means that there exist integers $m_{\epsilon,k} \geq 1$ along with finite collections $\{ \nu_{k,\ell}\}_{\ell=1}^{m_{\epsilon,k}} \subset \Acal_k$ such that 
\begin{equation}
\label{eq:BallCover2}
\Acal_k \subset \bigcup_{\ell=1}^{m_{\epsilon,k}} \B_{\Pcal_p(X)} \Big( \nu_{k,\ell} \, , \frac{\epsilon}{2} \Big). 
\end{equation}
Thence, one may verify upon combining \eqref{eq:BallCover1} and \eqref{eq:BallCover2} that for each $\Bmu \in \Kcalb$, there exists a choice of indices $\ell_k(\Bmu) \in \{1,\dots,m_{\epsilon,k}\}$ and a measure of the form
\begin{equation*}
\Bnu := \INTDom{ \bigg( \sum_{k=1}^{n_{\epsilon}} \nu_{k,\ell_k(\Bmu)} \mathds{1}_{A_k}(\omega) \bigg)}{\Omega}{\pi(\omega)} \in \Pcal_{\pi, \,p}(\Omega \times X)
\end{equation*}
such that $W_{\pi, \,p}(\Bmu,\Bnu) \leq \epsilon$. As the collection of all such measures is finite by construction, we can conclude that the set $\Kcalb \subset \Pcal_{\pi, \,p}(\Omega \times X)$ is relatively compact. 
\end{proof}

In the following theorem, we characterise convergent sequences for both the classical and fibred Wasserstein topologies in terms of the convergence of sequences of integrals against suitable classes of test functions, in the spirit of the classical results listed in Proposition \ref{prop:Wass} above.

\begin{thm}[Integral characterisation of classical and fibred Wasserstein convergence]
\label{thm:ComparisonWass}
Let $(\Bmu_n) \subset \Pcal_{\pi,\,p}(\Omega \times X)$ and $\Bmu \in \Pcal_{\pi,\,p}(\Omega \times X)$ be given. Then, it holds that 
\begin{equation}
\label{eq:WassTopoComp1}
W_{\pi, \,p}(\Bmu_n,\Bmu) ~\underset{n \to +\infty}{\longrightarrow}~ 0 ~~\text{if and only if}~~ \INTDom{\phi(x)}{X}{\mu_{\omega}^n(x)} ~\underset{n \to +\infty}{\longrightarrow}~ \INTDom{\phi(x)}{X}{\mu_{\omega}(x)} ~~ \text{in $L^1(\Omega,\R;\pi)$}
\end{equation}
for every $\phi \in C^0_p(X,\R)$. Assuming in addition that $\pi \in \Pcal_p(\Omega)$, it also holds that
\begin{equation}
\label{eq:WassTopoComp2}
W_p(\Bmu_n,\Bmu) ~\underset{n \to +\infty}{\longrightarrow}~ 0 ~~\text{if and only if}~~ \INTDom{\phi(x)}{X}{\mu_{\omega}^n(x)} ~\underset{n \to +\infty}{\rightharpoonup}~ \INTDom{\phi(x)}{X}{\mu_{\omega}(x)} ~~ \text{in $L^1(\Omega,\R;\pi)$}
\end{equation}
for every $\phi \in C^0_p(X,\R)$.
\end{thm}

\begin{proof}
The proof of this theorem is split into two steps, corresponding to the respective characterisations of fibred and classical Wasserstein convergence. 

\paragraph*{Step 1 -- Characterisation of fibred Wasserstein convergence.}

In this first step,  we begin by proving the characterisation \eqref{eq:WassTopoComp1} of fibred Wasserstein convergence in $\Pcal_{\pi,\,p}(\Omega \times X)$, starting with the direct implication. To this end, we assume at first that $\pi \in \Pcal_p(\Omega)$, and observe following Lemma \ref{lem:MetricEstimate} and Proposition \ref{prop:ClassicalCompactness} that the sequence $(\Bmu_n) \subset \Pcal_{\pi,\,p}(\Omega \times X)$ is fibrely tight and such that
\begin{equation*}
\sup_{n \geq 1} \INTDom{\Norm{x}_X^p}{\{ (\omega,x) \; \textnormal{s.t.} \,\dsf(\omega,\omega_0) + \| x \|_X \geq k \}}{\Bmu_n(\omega,x)} ~\underset{k \to +\infty}{\longrightarrow}~ 0
\end{equation*}
for some (hence all) $\omega_0 \in \Omega$. In particular, there exists for every $\epsilon > 0$ a compact set $K_{\epsilon} \subset X$ such that
\begin{equation*}
\sup_{n \geq 1} \Bmu_n \Big( \Omega \times (X \setminus K_{\epsilon}) \Big) < \epsilon.  
\end{equation*}
Moreover, given some $\phi \in C^0_p(X,\R)$, it follows from the Moreau-Yosida construction detailed e.g. at the beginning of \cite[Section 5.1]{AGS} applied to its positive and negative parts combined with Dini's theorem (see e.g. \cite[Theorem 7.13]{Rudin1976}) that there exists a sequence of Lipschitz functions $(\phi_m) \subset \Lip(X,\R)$ such that $|\phi_m| \leq \alpha_{\phi}(1+\|\cdot\|_X^p)$ for each $m \geq 1$ and 
\begin{equation}
\label{eq:FibredWassersteinConv2}
\NormC{\phi - \phi_m}{0}{K_{\epsilon},\R} ~\underset{m \to +\infty}{\longrightarrow}~ 0.  
\end{equation}
Whence, for any given $k \geq 1$, from all the above combined with the variant of Kantorovich duality for non-compactly supported measures (see e.g. \cite[Theorem 1.14]{villani2}), we may infer that
\begin{equation}
\label{eq:FibredWassersteinConv1}
\begin{aligned}
\INTDom{\bigg| \INTDom{\phi(x)}{X}{(\mu_{\omega}-\mu_{\omega}^n)(x)}  \bigg|}{\Omega}{\pi(\omega)} & \leq \INTDom{\bigg| \INTDom{(\phi(x) - \phi_m(x))}{X}{(\mu_{\omega}-\mu_{\omega}^n)(x)}  \bigg|}{\Omega}{\pi(\omega)} \\
& \hspace{1.15cm} + \Lip(\phi_m \, ; X) W_{\pi,\,p}(\Bmu_n,\Bmu) \\
& \leq \INTDom{\bigg| \INTDom{(\phi(x) - \phi_m(x))}{\{ x \;\textnormal{s.t.}\, \|x\|_X < k \}}{(\mu_{\omega}-\mu_{\omega}^n)(x)} \bigg|}{\Omega}{\pi(\omega)} \\
& \hspace{1.15cm} + 4 \alpha_{\phi} \, \sup_{n \geq 1} \INTDom{(1+\Norm{x}_X^p)}{\{ (\omega,x) \; \textnormal{s.t.} \, \dsf(\omega,\omega_0) + \| x \|_X \geq k \}}{\Bmu_n(\omega,x)} \\
& \hspace{1.15cm} + \Lip(\phi_m \, ; X) W_{\pi,\,p}(\Bmu_n,\Bmu) \\
& \leq \INTDom{\bigg| \INTDom{(\phi(x) - \phi_m(x))}{\{ x \in K_{\epsilon} \;\textnormal{s.t.}\, \|x\|_X < k \}}{(\mu_{\omega}-\mu_{\omega}^n)(x)} \bigg|}{\Omega}{\pi(\omega)}  \\
& \hspace{1.15cm} + 4 \alpha_{\phi} (1+k^p) \, \sup_{n \geq 1} \Bmu_n \Big( \Omega \times (X \setminus K_{\epsilon}) \Big) \\
& \hspace{1.15cm}  + 4 \alpha_{\phi} \, \sup_{n \geq 1} \INTDom{(1+\Norm{x}_X^p)}{\{ (\omega,x) \; \textnormal{s.t.} \, \dsf(\omega,\omega_0) + \| x \|_X \geq k \}}{\Bmu_n(\omega,x)} \\
& \hspace{1.15cm} + \Lip(\phi_m \, ; X) W_{\pi,\,p}(\Bmu_n,\Bmu) \\
& \leq 2 \NormC{\phi - \phi_m}{0}{K_{\epsilon},\R}  + 4 \alpha_{\phi} (1+k^p) \epsilon \\
& \hspace{1.15cm} + 4 \alpha_{\phi} \, \sup_{n \geq 1} \INTDom{(1+\Norm{x}_X^p)}{\{ (\omega,x) \; \textnormal{s.t.} \, \dsf(\omega,\omega_0) + \| x \|_X \geq k \}}{\Bmu_n(\omega,x)} \\
& \hspace{1.15cm} + \Lip(\phi_m \, ; X) W_{\pi,\,p}(\Bmu_n,\Bmu) . 
\end{aligned}
\end{equation}
Therefore, by letting first $n \to +\infty$, then $m \to +\infty$, $\epsilon \to 0^+$ and lastly $k \to +\infty$, we conclude that 
\begin{equation*}
\INTDom{\bigg| \INTDom{\phi(x)}{X}{\mu_{\omega}(x)} - \INTDom{\phi(x)}{X}{\mu_{\omega}^n(x)}  \bigg|}{\Omega}{\pi(\omega)} ~\underset{n \to +\infty}{\longrightarrow}~ 0
\end{equation*}
for any $\phi \in C^0_p(X,\R)$. For a general $\pi \in \Pcal(\Omega)$, the tightness and uniform integrability properties supplied by Proposition \ref{prop:ClassicalCompactness} follow from the convergence of fibred moments combined with the Lebesgue-Vitali dominated convergence theorem (see e.g. \cite[Theorem 4.5.4]{Bogachev}), and the rest of the proof remains unchanged.

We now establish the converse implication by showing that, if $(\Bmu_n) \subset \Pcal_{\pi,\,p}(\Omega \times X)$ is any sequence for which the right-hand side of \eqref{eq:WassTopoComp1} holds, then any of its subsequences itself admits a further subsequence converging towards $\Bmu \in \Pcal_{\pi,\,p}(\Omega \times X)$ in the $W_{\pi,\,p}$-metric. To do so, we begin by recalling following again \cite[Section 5.1]{AGS} that there exists a countable family of functions $\{\phi_m\}_{m=1}^{+\infty} \subset \Lip(X,\R)$ such that a sequence $(\mu_n) \subset \Pcal(X)$ converges narrowly towards $\mu \in \Pcal(X)$ if and only if
\begin{equation*}
\INTDom{\phi_m(x)}{X}{\mu_n(x)}  ~\underset{n \to +\infty}{\longrightarrow}~ \INTDom{\phi_m(x)}{X}{\mu(x)} 
\end{equation*}
for each $m \geq 1$. At this stage, fix an arbitrary subsequence $(\Bmu_{n_k}) \subset (\Bmu_n)$, let $\phi_0 := \Norm{\cdot}_X^p \in C^0_p(X,\R)$, and observe that by standard results from Lebesgue integration, there exists a subsequence $(\Bmu_{n_k}^0)$ adjoined with a $\pi$-measurable set $A_0 \in \Apazo$ such that $\pi(A_0) = 1$ and
\begin{equation*}
\INTDom{\Norm{x}_X^p}{X}{\mu_{\omega}^{0,n_k}(x)} ~\underset{k \to +\infty}{\longrightarrow}~ \INTDom{\Norm{x}_X^p}{X}{\mu_{\omega}(x)} 
\end{equation*}
for all $\omega \in A_0$. Then, by induction, there exists a nested family of extracted subsequences $\{(\Bmu_{n_k}^m)\}_{m=0}^{+\infty}$ of $(\Bmu_{n_k})$ along with a decreasing family of $\pi$-measurable sets $\{A_m\}_{m=0}^{+\infty} \subset \Apazo$ such that $\pi(A_m) = 1$ and 
\begin{equation*}
\INTDom{\phi_{\ell}(x)}{X}{\mu_{\omega}^{m,n_k}(x)}  ~\underset{k \to +\infty}{\longrightarrow}~ \INTDom{\phi_{\ell}(x)}{X}{\mu_{\omega}(x)} 
\end{equation*}
for all $\omega \in A_m$, each $m \geq 0$ and every $\ell \in \{0,\dots,m\}$. Therefore, upon setting $A := \bigcap_{m=0}^{+\infty} A_m$, it can be checked that $\pi(A) = 1$ and that the diagonal extraction $(\Bmu_{n_k}^k) \subset (\Bmu_{n_k})$ satisfies
\begin{equation*}
\left\{
\begin{aligned}
\mu_{\omega}^{k,n_k} ~&\underset{k \to +\infty}{\rightharpoonup^*}~ \mu_{\omega}, \\
\, \INTDom{\Norm{x}_X^p}{X}{\mu_{\omega}^{k,n_k}(x)} ~&\underset{k \to +\infty}{\longrightarrow}~ \INTDom{\Norm{x}_X^p}{X}{\mu_{\omega}(x)}
\end{aligned}
\right.
\end{equation*}
for all $\omega \in A$. By Proposition \ref{prop:Wass} above along with what precedes, this implies in particular that  
\begin{equation*}
W_p(\mu_{\omega}^{k,n_k},\mu_{\omega}) ~\underset{k \to +\infty}{\longrightarrow}~ 0
\end{equation*}
for $\pi$-almost every $\omega \in \Omega$. Since the fibred moments of the subsequence $(\Bmu_{n_k}^k) \subset (\Bmu_{n_k})$ are convergent -- hence uniformly integrable -- whereas $\pi \in \Pcal(\Omega)$, we may apply Vitali's dominated convergence theorem and infer that the latter sequence converges towards $\Bmu \in \Pcal_{\pi,\,p}(\Omega \times X)$ in the $W_{\pi,\,p}$-metric. Since $(\Bmu_{n_k}) \subset (\Bmu_n)$ was an arbitrary subsequence, this allows us to conclude.

\paragraph*{Step 2 -- Characterisation of classical Wasserstein convergence.} 

In this second step,  we establish the characterisation \eqref{eq:WassTopoComp2} of classical Wasserstein convergence in $\Pcal_{\pi,\,p}(\Omega \times X)$. We start with the direct implication and begin by observing that, by Proposition \ref{prop:ClassicalCompactness} above, one has that
\begin{equation*}
\begin{aligned}
\sup_{n \geq 1} \INTDom{\Mpazo_p^p(\mu_{\omega}^n)}{\{\omega \;\textnormal{s.t.}\, \dsf_{\Omega}(\omega,\omega_0) \geq k\}}{\pi(\omega)} & =\, \sup_{n \geq 1} \INTDom{\Norm{x}_X^p}{\{(\omega,x) \;\textnormal{s.t.}\, \dsf_{\Omega}(\omega,\omega_0) \geq k\}}{\Bmu_n(\omega,x)} \\
& \leq \, \sup_{n \geq 1} \INTDom{\Norm{x}_X^p}{\{(\omega,x) \;\textnormal{s.t.}\, \dsf_{\Omega}(\omega,\omega_0) + \|x\|_X \geq k\}}{\Bmu_n(\omega,x)} ~\underset{k \to +\infty}{\longrightarrow}~ 0, 
\end{aligned}
\end{equation*}
for some (and thus all) $\omega_0 \in \Omega$. By \cite[Proposition 4.5.3]{Bogachev}, the latter uniform integrability property amounts to the boundedness of the sequence and the uniform absolute continuity of its integrals, i.e.
\begin{equation}
\label{eq:LusinApprox1}
\sup_{n \geq 1} \INTDom{\Mpazo_p^p(\mu_{\omega}^n)}{\Omega}{\pi(\omega)} < +\infty \qquad \text{and} \qquad \sup_{n \geq 1} \INTDom{\Mpazo_p^p(\mu_{\omega}^n)}{A}{\pi(\omega)} ~\underset{\pi(A) \to 0^+}{\longrightarrow}~ 0.
\end{equation}
Now given any $\xi \in L^{\infty}(\Omega,\R;\pi)$, there exists by the variant of Lusin's fundamental result recalled in Theorem \ref{thm:Lusin} above a sequence of functions $(\xi_m) \subset C^0(\Omega,\R)$ such that
\begin{equation}
\label{eq:LusinApprox2}
\sup_{m \geq 1} \NormC{\xi_m}{0}{\Omega,\R} \leq~ \NormL{\xi}{\infty}{\Omega,\R;\pi} \qquad \text{and} \qquad \pi\Big( \Big\{ \omega \in \Omega ~\, \textnormal{s.t.}~ \xi_m(\omega) \neq \xi(\omega) \Big\}\Big) ~\underset{m \to +\infty}{\longrightarrow}~ 0.
\end{equation}
Thence, given any $\phi \in C^0_p(X,\R)$ and some $m,n \geq 1$, it holds that 
\begin{equation*}
\begin{aligned}
& \bigg| \INTDom{\xi(\omega) \bigg( \INTDom{\phi(x)}{X}{\mu_{\omega}(x)} \bigg)}{\Omega}{\pi(\omega)} - \INTDom{\xi(\omega) \bigg( \INTDom{\phi(x)}{X}{\mu_{\omega}^n(x)} \bigg)}{\Omega}{\pi(\omega)}  \bigg| \\
& \hspace{0.9cm} \leq \bigg| \INTDom{\xi_m(\omega) \bigg( \INTDom{\phi(x)}{X}{\mu_{\omega}(x)} \bigg)}{\Omega}{\pi(\omega)} - \INTDom{\xi_m(\omega) \bigg(\INTDom{\phi(x)}{X}{\mu_{\omega}^n(x)} \bigg)}{\Omega}{\pi(\omega)}  \bigg| \\
& \hspace{2.2cm} + 2 \NormL{\xi}{\infty}{\Omega,\R;\pi}  \sup_{n \geq 1} \INTDom{\alpha_{\phi} \bigg(2 + \Mpazo_p^p(\mu_{\omega}^n) + \Mpazo_p^p(\mu_{\omega}) \bigg)}{\{\omega \;\textnormal{s.t.}\, \xi_m(\omega) \neq \xi(\omega) \}}{\pi(\omega)}.
\end{aligned}
\end{equation*}
Upon noting that the first term in the right-hand side of the previous expression converges to 0 as $n \to +\infty$ by the characterisation \eqref{eq:CharacWassConv2} of Wasserstein convergence whereas the second one vanishes as $m \to +\infty$ as a consequence of \eqref{eq:LusinApprox1} and \eqref{eq:LusinApprox2}, one may infer from what precedes that 
\begin{equation*}
\INTDom{\xi(\omega) \bigg( \INTDom{\phi(x)}{X}{\mu_{\omega}^n(x)} \bigg)}{\Omega}{\pi(\omega)} ~\underset{n \to +\infty}{\longrightarrow}~ \INTDom{\xi(\omega) \bigg( \INTDom{\phi(x)}{X}{\mu_{\omega}(x)} \bigg)}{\Omega}{\pi(\omega)}
\end{equation*}
for all $\xi \in L^{\infty}(\Omega,\R;\pi)$ and each $\phi \in C^0_p(X,\R)$.

To establish the converse implication, we follow the same path as in the proof of \cite[Theorem 4.3.1]{Attouch2005} and note that, upon using the characterisation \eqref{eq:CharacWassConv2} of Wasserstein convergence, one gets that 
\begin{equation}
W_p \Big( \mathsmaller{\INTDomdash{\mu_{\omega}^n}{\Omega}{\pi(\omega)}} , \mathsmaller{\INTDomdash{\mu_{\omega}}{\Omega}{\pi(\omega)}} \Big) ~\underset{n \to +\infty}{\longrightarrow}~ 0.
\end{equation}
By Proposition \ref{prop:Wass} and the definition \eqref{eq:AverageMeasure} of average measure, this implies in particular that the sequence $(\Bmu_n) \subset \Pcal_{\pi,\, p}(\Omega \times X)$ is fibrely tight in the sense of Proposition \ref{prop:FibredTightness}, and that it satisfies
\begin{equation}
\label{eq:UnifIntegEst1}
\sup_{n \geq 1} \INTDom{\Norm{x}_X^p}{\{(\omega,x) \;\text{s.t.}\, \|x\|_X \geq k\}}{\Bmu_n(\omega,x)} \,=\, \sup_{n \geq 1} \INTDom{\Norm{x}_X^p}{\{ x \;\text{s.t.}\, \|x\|_X \geq k\}}{\Big( \mathsmaller{\INTDomdash{\mu_{\omega}^n}{\Omega}{\pi(\omega)}} \Big)(x)}  ~\underset{k \to +\infty}{\longrightarrow}~ 0.
\end{equation}
Moreover, it follows from choosing $\phi := \|\cdot\|_X^p \in C^0_p(X,\R)$ in the right-hand side of \eqref{eq:WassTopoComp2} that
\begin{equation*}
\quad \Mpazo_p^p(\mu_{\omega}^n) ~\underset{n \to +\infty}{\rightharpoonup}~ \Mpazo_p^p(\mu_{\omega}) \quad \text{in $L^1(\Omega,\R;\pi)$}, 
\end{equation*}
which implies by the classical Dunford-Pettis characterisation of relative weak $L^1$-compactness (see for instance \cite[Theorem 4.7.18]{Bogachev}) that the latter sequence is bounded and uniformly integrable. By the equivalence between uniform integrability and uniform absolute continuity of the integrals, it follows 
\begin{equation}
\label{eq:UnifIntegEst2}
\sup_{n \geq 1} \INTDom{\Norm{x}_X^p}{\{(\omega,x) \;\text{s.t.}\, \dsf_{\Omega}(\omega,\omega_0) \geq k\}}{\Bmu_n(\omega,x)} = \sup_{n \geq 1} \INTDom{\Mpazo_p^p(\mu_{\omega}^n)}{\{\omega \;\textnormal{s.t.}\, \dsf_{\Omega}(\omega,\omega_0) \geq k\}}{\pi(\omega)} ~\underset{k \to +\infty}{\longrightarrow}~ 0, 
\end{equation}
for some (and thus all) $\omega_0 \in \Omega$. Thus, by combining \eqref{eq:UnifIntegEst1} and \eqref{eq:UnifIntegEst2}, we obtain that
\begin{equation*}
\sup_{n \geq 1} \INTDom{\Norm{x}_X^p}{\{(\omega,x) \;\text{s.t.}\, \dsf_{\Omega}(\omega,\omega_0) + \|x\|_X \geq k\}}{\Bmu_n(\omega,x)} ~\underset{k \to +\infty}{\longrightarrow}~ 0. 
\end{equation*}
Therefore, owing to Proposition \ref{prop:ClassicalCompactness} above, there only remains to show that
\begin{equation*}
\Bmu_n ~\underset{n \to +\infty}{\rightharpoonup^*}~ \Bmu 
\end{equation*}
in order to conclude. To this end, given some $\varphi \in C^0_b(\Omega \times X,\R)$, it follows from the fibred tightness of the sequence $(\Bmu_n) \subset \Pcal_{\pi,\,p}(\Omega \times X)$ that for each $\epsilon > 0$, there exists a compact set $K_{\epsilon} \subset X$ such that 
\begin{equation}
\label{eq:WeakImpliesNarrow1}
\sup_{n \geq 1} \Bmu_n \Big(\Omega \times (X \setminus K_{\epsilon}) \Big) \leq \frac{\epsilon}{12 \NormC{\varphi}{0}{\Omega \times X,\R}}.
\end{equation}
Observe then that the map $\omega \in \Omega \mapsto \varphi(\omega) \in C^0(K_{\epsilon},\R)$ is continuous and therefore Borel measurable, so that by \cite[Corollary 1.1.7]{AnalysisBanachSpaces}, there exists a sequence $(\varphi_m)$ of simple Borel functions for which 
\begin{equation*}
\sup_{m \geq 1}\NormC{\varphi_m(\omega)}{0}{K_{\epsilon},\R} \leq \NormC{\varphi(\omega)}{0}{K_{\epsilon},\R} \qquad \text{and} \qquad \varphi_m(\omega) ~\underset{m \to +\infty}{\longrightarrow}~ \varphi(\omega) 
\end{equation*}
for every $\omega \in \Omega$. In particular, there exists some $m_{\epsilon} \geq 1$, a family of disjoint Borel sets $\{A_k\}_{k=1}^{m_{\epsilon}} \subset \Bcal(\Omega)$ and maps $\{\phi_k\}_{k=1}^{m_{\epsilon}} \subset C^0(K_{\epsilon},\R)$ such that the simple function $\varphi_{\epsilon} := \sum_{k=1}^{m_\epsilon} \mathds{1}_{A_k} \phi_k$ satisfies
\begin{equation}
\label{eq:WeakImpliesNarrow2}
\max_{k \in \{1,\dots,m_{\epsilon}\}}\NormC{\phi_k}{0}{K_{\epsilon},\R} \leq \NormC{\varphi}{0}{\Omega \times X,\R} \qquad \text{and} \qquad \INTDom{\NormC{\varphi(\omega) - \varphi_{\epsilon}(\omega)}{0}{K_{\epsilon},\R}}{\Omega}{\pi(\omega)} \leq \frac{\epsilon}{6}, 
\end{equation}
where we used Lebesgue's dominated convergence theorem to obtain the second inequality. Upon observing now that by Tietze's extension theorem (see e.g. \cite[Theorem 5.1]{Dugundji1966}), each of the maps $\phi_k \in C^0(K_{\epsilon},\R)$ admits an extension $\widetilde{\phi}_k \in C^0_b(X,\R)$ satisfying 
\begin{equation*}
\| \widetilde{\phi}_k \|_{C^0(X,\R)} \leq \NormC{\varphi}{0}{\Omega \times X,\R},
\end{equation*}
we may finally set $\widetilde{\varphi}_{\epsilon} := \sum_{k=1}^{m_\epsilon} \mathds{1}_{A_k} \widetilde{\phi}_k$ and infer from \eqref{eq:WassTopoComp2} that, for $n \geq 1$ sufficiently large, one has
\begin{equation*}
\max_{k \in \{1,\dots,m_{\epsilon}\}} \bigg| \INTDom{\mathds{1}_{A_k}(\omega) \bigg( \INTDom{\widetilde{\phi}_k(x)}{X}{\mu_{\omega}(x)} \bigg)}{\Omega}{\pi(\omega)} - \INTDom{\mathds{1}_{A_k}(\omega) \bigg( \INTDom{\widetilde{\phi}_k(x)}{X}{\mu_{\omega}^n(x)} \bigg)}{\Omega}{\pi(\omega)} \bigg| \leq \frac{\epsilon}{3 m_{\epsilon}}. 
\end{equation*}
By combining the latter inequality with the estimates of \eqref{eq:WeakImpliesNarrow1} and \eqref{eq:WeakImpliesNarrow2}, we finally obtain that
\begin{equation}
\label{eq:NarrowConvergenceBigEst}
\begin{aligned}
& \bigg| \INTDom{\varphi(\omega,x)}{\Omega \times X}{\Bmu(\omega,x)} - \INTDom{\varphi(\omega,x)}{\Omega \times X}{\Bmu_n(\omega,x)} \bigg| \\
& \hspace{0.7cm} \leq \bigg| \INTDom{\Big( \varphi(\omega,x) - \widetilde{\varphi}_{\epsilon}(\omega,x) \Big)}{\Omega \times X}{\Bmu(\omega,x)} - \INTDom{\Big( \varphi(\omega,x) - \widetilde{\varphi}_{\epsilon}(\omega,x) \Big)}{\Omega \times X}{\Bmu_n(\omega,x)} \bigg| \\
& \hspace{1.2cm} + \sum_{k=1}^{m_{\epsilon}} \bigg| \INTDom{\mathds{1}_{A_k}(\omega) \bigg( \INTDom{\widetilde{\phi}_k(x)}{X}{\mu_{\omega}(x)} \bigg)}{\Omega}{\pi(\omega)} - \INTDom{\mathds{1}_{A_k}(\omega) \bigg( \INTDom{\widetilde{\phi}_k(x)}{X}{\mu_{\omega}^n(x)} \bigg)}{\Omega}{\pi(\omega)} \bigg| \\
& \hspace{0.7cm} \leq \bigg| \INTDom{\Big( \varphi(\omega,x) - \widetilde{\varphi}_{\epsilon}(\omega,x) \Big)}{\Omega \times K_{\epsilon}}{\Bmu(\omega,x)} - \INTDom{\Big( \varphi(\omega,x) - \widetilde{\varphi}_{\epsilon}(\omega,x) \Big)}{\Omega \times K_{\epsilon}}{\Bmu_n(\omega,x)}  \bigg| \\
& \hspace{1.2cm} + 2 \bigg( \NormC{\varphi}{0}{\Omega \times X,\R} + \sup_{\omega \in \Omega} \NormC{\widetilde{\varphi}_{\epsilon}(\omega)}{0}{K_{\epsilon},\R}  \bigg) \, \sup_{n \geq 1} \Bmu_n \Big( \Omega \times (X \setminus K_{\epsilon}) \Big) \\
& \hspace{1.2cm} + \sum_{k=1}^{m_{\epsilon}} \bigg| \INTDom{\mathds{1}_{A_k}(\omega) \bigg( \INTDom{\widetilde{\phi}_k(x)}{X}{\mu_{\omega}(x)} \bigg)}{\Omega}{\pi(\omega)} - \INTDom{\mathds{1}_{A_k}(\omega) \bigg( \INTDom{\widetilde{\phi}_k(x)}{X}{\mu_{\omega}^n(x)} \bigg)}{\Omega}{\pi(\omega)} \bigg| \\
& \hspace{0.7cm} \leq 2 \INTDom{\NormC{\varphi(\omega) - \varphi_{\epsilon}(\omega)}{0}{K_{\epsilon},\R}}{\Omega}{\pi(\omega)} + 4 \NormC{\varphi}{0}{\Omega \times X,\R} \, \sup_{n \geq 1} \Bmu_n \Big( \Omega \times (X \setminus K_{\epsilon}) \Big) \\
& \hspace{1.2cm} + \sum_{k=1}^{m_{\epsilon}} \bigg| \INTDom{\mathds{1}_{A_k}(\omega) \bigg( \INTDom{\widetilde{\phi}_k(x)}{X}{\mu_{\omega}(x)} \bigg)}{\Omega}{\pi(\omega)} - \INTDom{\mathds{1}_{A_k}(\omega) \bigg( \INTDom{\widetilde{\phi}_k(x)}{X}{\mu_{\omega}^n(x)} \bigg)}{\Omega}{\pi(\omega)} \bigg| \\
& \hspace{0.7cm} \leq \epsilon
\end{aligned}
\end{equation}
for $n \geq 1$ sufficiently large, which is precisely the sought-after claim. 
\end{proof}


\begin{rmk}[Illustration and comparison of both topologies on some examples]\label{rmk:closeness}
It is clear that the classical and fibred Wasserstein distances are usually quite different. The most basic example illustrating this fact is that of a pair of empirical measures given by 
\begin{equation*}
\Bmu_1 := \tfrac{1}{2} \big(\delta_{(\omega_1,x_1)} + \delta_{(\omega_2,x_2)} \big) \qquad \text{and} \qquad \Bmu_2 := \tfrac{1}{2} \big( \delta_{(\omega_1,x_2)} + \delta_{(\omega_2,x_1)} \big)
\end{equation*}
for some $\omega_1,\omega_2 \in \Omega$ and $x_1,x_2 \in X$, for which one may readily check that
\begin{equation}
\label{eq:RmkWassComp}
W_p(\Bmu_1,\Bmu_2) = \min\Big\{ \dsf_{\Omega}(\omega_1,\omega_2) , \Norm{x_1-x_2}_X \hspace{-0.1cm} \Big\} \qquad \text{whereas} \qquad W_{\pi, \, p}(\Bmu_1,\Bmu_2) = \, \Norm{x_1 - x_2}_X. 
\end{equation}
The latter, however, does not provide much insight on the gap between the topologies induced by both metrics. Indeed given a sequence $(\Bmu_n) \subset \Pcal_{\pi,\,p}(\Omega \times X)$ with general term $\Bmu_n := \tfrac{1}{2} \big(\delta_{(\omega_1,x_1^n)} + \delta_{(\omega_2,x_2^n)} \big)$ for each $n \geq 1$ and some $\Bmu := \tfrac{1}{2} \big(\delta_{(\omega_1,x_1)} + \delta_{(\omega_2,x_2)} \big)$, it follows from \eqref{eq:RmkWassComp} above that 
\begin{equation*}
W_p(\Bmu_n,\Bmu) ~\underset{n \to +\infty}{\longrightarrow}~ 0 \qquad \text{if and only if} \qquad  W_{\pi ,\,p}(\Bmu_n,\Bmu) ~\underset{n \to +\infty}{\longrightarrow}~ 0. 
\end{equation*}
The same argument remains valid whenever $\pi \in \Pcal(\Omega)$ is made of a finite number of atoms, and the topologies induced by the classical and fibred Wasserstein metrics coincide in that case as well. Theorem \ref{thm:ComparisonWass} above provides a general certificate for this statement, since in that case $L^1(\Omega,\R;\pi)$ is isomorphic to a finite dimensional vector space, whose strong and weak topologies must then coincide.    

Consider now the opposite configuration in which $\pi \in \Pcal(\Omega)$ has no atoms. Then, it can be shown (see e.g. \cite[Theorem 2.4.3]{Castaing2004}) that the set of measures supported on the graphs of $p$-integrable maps
\begin{equation*}
\Gcal_{\pi,\,p}(\Omega \times X) := \Big\{ (\Id,x)_{\sharp} \pi \in \Pcal_{\pi, \, p}(\Omega \times X) ~\, \textnormal{s.t.}~ x \in L^p(\Omega,X;\pi) \Big\}
\end{equation*}
is dense in $\Pcal_{\pi, \,p }(\Omega \times X)$ for the $W_p$-topology. However, it can also be checked that $\Gcal_{\pi,\,p}(\Omega \times X)$ is a closed set for the $W_{\pi,\,p}$-topology.  Indeed, for every $(x_n) \subset L^p(\Omega,X;\pi)$ such that the sequence with general term $\Bmu_n := (\Id,x_n)_{\sharp} \pi \in \Gcal_{\pi,\,p}(\Omega \times X)$ is converging to some $\Bmu \in \Pcal_{\pi,\,p}(\Omega \times X)$, one has that 
\begin{equation*}
W_{\pi,\,p}(\Bmu_n , \Bmu) ~\underset{n \to +\infty}{\longrightarrow}~ 0 \qquad \text{if and only if }\qquad \INTDom{W_p^p \big(\delta_{x_n(\omega)},\mu_{\omega} \big)}{\Omega}{\pi(\omega)} ~\underset{n \to +\infty}{\longrightarrow}~ 0.
\end{equation*}
In particular, it holds that $W_p(\delta_{x_n(\omega)},\mu_{\omega}) \to 0$ as $n \to +\infty$ for $\pi$-almost every $\omega \in \Omega$, up to a subsequence, which in turn entails the existence of a map $x \in L^p(\Omega,X;\pi)$ such that $\Bmu = (\Id,x)_{\sharp} \pi$. 
\end{rmk}

We close this section by establishing a Kantorovich-Rubinstein duality formula for the fibred Wasserstein distance of order $1$, which is similar in spirit to that e.g. of \cite[Theorem 3.2.1]{Castaing2004} formulated in terms of bounded-Lipschitz distances.

\begin{prop}[A duality formula à la Kantorovich-Rubinstein]
\label{prop:Kantorovich}
Suppose that $K \subset X$ is a compact set. Then, for every pair of measures $\Bmu,\Bnu \in \Pcal_{\pi}(\Omega \times K)$, it holds that
\begin{equation}
\label{eq:KantorovichDuality}
W_{\pi,1}(\Bmu,\Bnu) = \sup \bigg\{ \INTDom{\INTDom{\varphi(\omega,x)}{K}{(\mu_{\omega}-\nu_{\omega})(x)}}{\Omega}{\pi(\omega)} ~\,\textnormal{s.t.}~ \varphi \in \Lip_{\pi,1}(\Omega \times K) \bigg\}
\end{equation}
wherein $\Lip_{\pi,1}(\Omega \times K,\R)$ refers to the collection of all Carathéodory mappings $\varphi : \Omega \times K \to \R$ satisfying $\Lip(\varphi(\omega)\, ; K) \leq 1$ for $\pi$-almost every $\omega \in \Omega$
\end{prop}

\begin{proof}
This proof is postponed to Appendix \ref{section:AppendixKanto}, as it revolves around arguments borrowed from \cite[Theorem 14.60]{Rockafellar} and \cite[Chapter VII]{Castaing1977} which somewhat differ from the rest of the paper.
\end{proof}

\begin{rmk}[Kantorovich duality for fibred measures with general supports]
Although we believe that the latter formula should still hold for general elements of $\Pcal_{\pi,1}(\Omega \times X)$, the arguments supporting our proof detailed in Appendix \ref{section:AppendixKanto} rely on measurable selection principles to interchange the supremum and the integral. These become quite subtle and much more intricate in the noncompact case, since then the set of 1-Lipschitz maps can no longer be endowed with a nice separable topology. 
\end{rmk}


\section{Continuity equations in fibred probability spaces}
\label{section:CELocal}

\setcounter{equation}{0} \renewcommand{\theequation}{\thesection.\arabic{equation}}

Throughout this section and the remainder of the manuscript, we assume that the probability space $(\Omega,\Apazo,\pi)$ encoding the agent labels is a complete probability space over a Polish space.


\subsection{Definition of solutions, basic results and Cauchy-Lipschitz well-posedness}
\label{subsection:Structure}

In this section, we lay the mathematical foundations for studying structured continuity equations in fibred measure spaces. Given a time horizon $T>0$, we consider dynamics of the form 
\begin{equation}
\label{eq:TransportGraph}
\partial_t \Bmu(t) + \Div_x (\vb(t) \Bmu(t)) = 0
\end{equation}
where the vector field $\vb : [0,T] \times \Omega \times \R^d \to \R^d$ is jointly $\Lcal^1 \times \pi$-measurable in $(t,\omega) \in [0,T] \times \Omega$ and Borel in $x \in \R^d$ -- in the sequel we will simply say that it is measurable or Lebesgue-Borel --, and the dynamics in \eqref{eq:TransportGraph} is understood in the sense of distributions against the \textit{admissible test functions}
\begin{equation*}
\begin{aligned}
\Adm((0,T) \times \Omega \times \R^d) := \Bigg\{ & \phi \in C^0_b([0,T] \times \Omega \times \R^d,\R) ~\,\textnormal{s.t.}~ \phi(\cdot,\omega,\cdot) \in C^1_c((0,T) \times \R^d)~~  \text{with} \\
& \hspace{0.1cm} \supp(\phi(\cdot,\omega,\cdot)) \subset K_{\phi} ~\text{compact and}~~ \underset{\omega \in \Omega}{\sup} \NormC{\phi(\omega)}{1}{[0,T] \times \R^d,\R} < +\infty \Bigg\}. 
\end{aligned}
\end{equation*}
%

\begin{Def}[Distributional solutions of structured continuity equations]
We say that a Borel curve $\Bmu : [0,T] \to \Pcal(\Omega \times \R^d)$ is a distributional solution of \eqref{eq:TransportGraph} provided that 
\begin{equation}
\label{eq:DistribCE1}
\INTSeg{\INTDom{\Big( \partial_t \phi(t,\omega,x) + \big\langle \nabla_x \phi(t,\omega,x) , \vb(t,\omega,x) \big\rangle \Big)}{\Omega \times \R^d}{\Bmu(t)(\omega,x)}}{t}{0}{T} = 0
\end{equation}
for every admissible test function $\phi \in \Adm((0,T) \times \Omega \times \R^d)$. 
\end{Def}

Throughout this first section, we will start by assuming that the measure curve and the underlying velocity field comply with the joint integrability bounds 
\begin{equation}
\label{eq:IntBound}
\INTSeg{\INTDom{{\frac{|\vb(t,\omega,x)|}{1+|x|}}}{\Omega \times \R^d}{\Bmu(t)(\omega,x)}}{t}{0}{T} < +\infty.
\end{equation}
Note that the distributional formulation \eqref{eq:DistribCE1} would still make sense under milder conditions, in the spirit e.g. of \cite[Remark 1.2]{AmbrosioC2014}, but we will need \eqref{eq:IntBound} or stronger variants thereof in most of what follows. In the next proposition, we begin with a preliminary result which rigorously establishes the time-invariance of the first marginal of any solution of \eqref{eq:TransportGraph}. 

\begin{prop}[Marginal invariance]
\label{prop:Marginal}
Let $\Bmu : [0,T] \to \Pcal(\Omega \times \R^d)$ be a Borel solution of \eqref{eq:TransportGraph} satisfying \eqref{eq:IntBound}. Then its first marginal
\begin{equation*}
\pi : t \in [0,T] \mapsto (\pfrak_{\Omega})_{\sharp} \Bmu(t) \in \Pcal(\Omega)
\end{equation*}
is constant outside of an $\Lcal^1$-negligible set.
\end{prop}

\begin{proof}
Let us consider admissible test functions of the form 
\begin{equation*}
\phi_n(t,\omega,x) := \zeta(t)\xi(\omega)\psi_n(x),
\end{equation*}
where $(\zeta,\xi) \in C^1_c((0,T),\R) \times C^0_b(\Omega,\R)$ are arbitrary while $(\psi_n) \subset C^1_c(\R^d,\R)$ is a sequence of radially monotone non-increasing functions such that
\begin{equation}
\label{eq:CutOff}
\psi_n(x) = 
\left\{
\begin{aligned}
& 1 \qquad \text{if \, $|x| \leq n$}, \\
& 0 \qquad \text{if \, $|x| \geq 2n$}, 
\end{aligned}
\right.
\qquad \text{and} \qquad \NormC{\nabla \psi_n}{0}{\R^d,\R^d} \leq \frac{2}{n}
\end{equation}
for all $n\geq 1$. By plugging the latter in the distributional identity \eqref{eq:DistribCE1}, we obtain that
\begin{equation}
\label{eq:MarginalProofEst}
\begin{aligned}
& \bigg| \INTSeg{\zeta'(t) \INTDom{\xi(\omega) \psi_n(x)}{\Omega \times \R^d}{\Bmu(t)(\omega,x)}}{t}{0}{T} \, \bigg| \\
& \leq \NormC{\zeta}{1}{[0,T],\R} \NormC{\xi}{0}{\Omega,\R} \NormC{ \nabla \psi_n}{0}{\R^d,\R^d} \INTSeg{\INTDom{|\vb(t,\omega,x)|}{\Omega \times \{x \; \textnormal{s.t.} \, n \leq |x| \leq 2n \}}{\Bmu(t)(\omega,x)}}{t}{0}{T} \\
& \leq \NormC{\zeta}{1}{[0,T],\R} \NormC{\xi}{0}{\Omega,\R} 6 \INTSeg{\INTDom{\frac{|\vb(t,\omega,x)|}{1+|x|}}{\Omega \times \{x \; \textnormal{s.t.} \, n \leq |x| \leq 2n \}}{\Bmu(t)(\omega,x)}}{t}{0}{T}. 
\end{aligned}
\end{equation}
As a consequence of the integrability bound \eqref{eq:IntBound}, it holds that 
\begin{equation*}
\INTSeg{\INTDom{\frac{|\vb(t,\omega,x)|}{1+|x|}}{\Omega \times \{x \; \textnormal{s.t.} \, n \leq |x| \leq 2n \}}{\Bmu(t)(\omega,x)}}{t}{0}{T} ~\underset{n \to +\infty}{\longrightarrow}~ 0,
\end{equation*}
and thus letting $n \to +\infty$ while using the monotone convergence theorem in \eqref{eq:MarginalProofEst} allows us to obtain 
\begin{equation*}
\INTSeg{\zeta'(t) \INTDom{\xi(\omega)}{\Omega \times \R^d}{\Bmu(t)(\omega,x)}}{t}{0}{T} = 0
\end{equation*}
for all $(\zeta,\xi) \in C^1_c((0,T),\R) \times C^0_b(\Omega,\R)$. In particular, this implies that the maps $t \in [0,T] \mapsto \INTDom{\xi(\omega)}{\Omega \times \R^d}{\Bmu(t)(\omega,x)}$ 
have zero distributional derivative, and are thus constant outside of an $\Lcal^1$-negligible set $\Apazo_{\xi} \subset [0,T]$ for each test function $\xi \in C^0_b(\Omega,\R)$. Hence, up to considering a countable family $\{\xi_m\}_{m=1}^{+\infty} \subset C^0_b(\Omega,\R)$ generating the narrow topology (see e.g. \cite[Remark 5.1.1]{AGS}), the first marginal $\pi : t \in [0,T] \mapsto (\pfrak_{\Omega})_{\sharp} \Bmu(t) \in \Pcal(\Omega)$ is constant over the set $\Apazo := [0,T] \setminus \bigcup_{m=1}^{+\infty} \Apazo_{\xi_m}$. 
\end{proof}

It stems from the previous proposition that a Borel curve $\Bmu : [0,T] \to \Pcal(\Omega \times \R^d)$ solving \eqref{eq:TransportGraph} is actually valued in $\Pcal_{\pi}(\Omega \times \R^d)$ outside of an $\Lcal^1$-negligible set, and may thus be represented via the disintegration formula
\begin{equation}
\label{eq:DisintegratedCurve}
\Bmu(t) = \INTDom{\mu_{\omega}(t)}{\Omega}{\pi(\omega)} 
\end{equation}
for $\Lcal^1$-almost every $t \in [0,T]$, where we used Theorem \ref{thm:Disintegration}. Therein, the collections $\{\mu_{\omega}(t)\}_{\omega \in \Omega} \subset \Pcal(\R^d)$ are $\pi$-almost uniquely determined Borel families of measures. In the following theorem, we show that these disintegrated curves solve a system of coupled continuity equations.

\begin{thm}[Disintegrated dynamics]
\label{thm:Equivalence}
Let $\Bmu : [0,T] \to \Pcal_{\pi}(\Omega \times \R^d)$ be a solution of \eqref{eq:TransportGraph}, denote by $\{\mu_{\omega}(\cdot)\}_{\omega \in \Omega}$ the $\pi$-almost uniquely determined Borel collection of curves defined via \eqref{eq:DisintegratedCurve}, and suppose that \eqref{eq:IntBound} holds. 

Then for $\pi$-almost every $\omega \in \Omega$, the curve $\mu_{\omega} : [0,T] \to \Pcal(\R^d)$ solves the continuity equation
\begin{equation}
\label{eq:DisintegratedTransport}
\partial_t \mu_{\omega}(t) + \Div_x(\vb(t,\omega) \mu_{\omega}(t)) = 0.\end{equation}
Conversely, if $\{\mu_{\omega}(\cdot)\}_{\omega \in \Omega} \subset C^0([0,T],\Pcal(\R^d))$ is a Borel collection of solutions of \eqref{eq:DisintegratedTransport}, then the curve of measures $\Bmu : [0,T] \to \Pcal_{\pi}(\Omega \times \R^d)$ defined via \eqref{eq:DisintegratedCurve} solves \eqref{eq:TransportGraph}. 
\end{thm}

\begin{proof}
We begin by showing the direct implication. Given a solution $\Bmu : [0,T] \to \Pcal_{\pi}(\Omega \times \R^d)$ of \eqref{eq:TransportGraph}, one may deduce from the disintegration formula \eqref{eq:DisintegratedCurve} that for all times $t\in [0,T]$, there exists a $\pi$-almost uniquely determined family of probability measures $\{\mu_\omega(t)\}_{\omega\in\Omega}$ for which
\begin{equation}
\label{eq:TransportDis1}
\INTSeg{\INTDom{\phi(t,\omega,x)}{\Omega \times \R^d}{\Bmu(t)(\omega,x)}}{t}{0}{T} = \INTSeg{\INTDom{\INTDom{\phi(t,\omega,x)}{\R^d}{\mu_{\omega}(t)(x)}}{\Omega}{\pi(\omega)}}{t}{0}{T}
\end{equation}
for every bounded Borel map $\phi : [0,T] \times \Omega \times \R^d \to \R$. Thence, plugging in \eqref{eq:DistribCE1} any admissible test function of the particular form 
\begin{equation*}
\phi(t,\omega,x) := \xi(\omega) \varphi(t,x)
\end{equation*}
for some $\xi \in C^0_b(\Omega,\R)$ and $\varphi \in C^1_c((0,T) \times \R^d,\R)$, it follows from \eqref{eq:TransportDis1} along with Fubini's theorem that
\begin{equation}
\label{eq:ProofEquivDistrib}
\INTDom{\xi(\omega) \bigg( \INTSeg{\INTDom{\Big( \partial_t \varphi(t,x) + \big\langle \nabla_x \varphi(t,x) , \vb(t,\omega,x) \big\rangle \Big)}{\R^d}{\mu_{\omega}(t)(x)}}{t}{0}{T} \bigg)}{\Omega}{\pi(\omega)} = 0.
\end{equation}
Now for any $\pi$-measurable set $A \in \Apazo$, one may find thanks to Theorem \ref{thm:Lusin} a sequence of continuous functions $(\xi_m) \subset C^0_b(\Omega,\R)$ such that 
\begin{equation}
\label{eq:ApproxIndicatorMeasure}
\sup_{m \geq 1}\NormC{\xi_m}{0}{\Omega,\R} \leq 1 \qquad \text{and} \qquad \pi \Big( \Big\{ \omega \in \Omega ~\,\textnormal{s.t.}~ \xi_m(\omega) \neq \mathds{1}_A \Big\} \Big) ~\underset{m \to +\infty}{\longrightarrow}~ 0. 
\end{equation}
Thence, upon noting that the map
\begin{equation*}
\omega \in \Omega \mapsto \INTSeg{\INTDom{\Big( \partial_t \varphi(t,x) + \big\langle \nabla_x \varphi(t,x) , \vb(t,\omega,x) \big\rangle \Big)}{\R^d}{\mu_{\omega}(t)(x)}}{t}{0}{T} \in \R
\end{equation*}
is integrable under \eqref{eq:IntBound}, it follows from \eqref{eq:ProofEquivDistrib} combined with \eqref{eq:ApproxIndicatorMeasure} and the absolute continuity of the Lebesgue integral (see e.g. \cite[Proposition 4.5.3]{Bogachev}) that
\begin{equation*}
\INTDom{\bigg( \INTSeg{\INTDom{ \Big( \partial_t \varphi(t,x) + \big\langle \nabla_x \varphi(t,x) , \vb(t,\omega,x) \big\rangle \Big)}{\R^d}{\mu_{\omega}(t)(x)}}{t}{0}{T} \bigg)}{A}{\pi(\omega)} = 0.
\end{equation*}
Noting then that $A \subset \Omega$ was an arbitrary $\pi$-measurable set, we infer the existence of some $A_{\varphi} \in \Apazo$ such that $\pi(A_{\varphi}) = 1$ and 
\begin{equation*}
\INTSeg{\INTDom{ \Big( \partial_t \varphi(t,x) + \big\langle \nabla_x \varphi(t,x) , \vb(t,\omega,x) \big\rangle \Big)}{\R^d}{\mu_{\omega}(t)(x)}}{t}{0}{T} = 0
\end{equation*}
for all $\omega \in A_{\varphi}$. Considering an exhaustion $(0,T) \times \R^d$ e.g. by balls and leveraging the separability of $(C^1(K,\R),\NormC{\cdot}{1}{K,\R})$ for each compact set $K \subset (0,T) \times \R^d$, we may argue by density using \eqref{eq:IntBound} and find a common full-measure set $\Theta \in \Apazo$ over which the latter identity holds for all $\varphi \in C^1_c((0,T) \times \R^d,\R)$. Therefore, we conclude that $\mu_{\omega} : [0,T] \to \Pcal(\R^d)$ solves \eqref{eq:DisintegratedTransport} for $\pi$-almost every $\omega \in \Omega$.

We conclude our proof by showing the converse statement in Theorem \ref{thm:Equivalence}. To this end, observe that given a family $\{\mu_{\omega}(\cdot)\}_{\omega \in \Omega} \subset C^0([0,T],\Pcal(\R^d))$ of solutions of \eqref{eq:DisintegratedTransport}, we can define $\Bmu : [0,T] \to \Pcal_{\pi}(\Omega \times \R^d)$ in a unique way via \eqref{eq:DisintegratedCurve}. We may then directly deduce that $\Bmu(\cdot) \in C^0([0,T],\Pcal_{\pi}(\Omega \times \R^d))$ by an elementary application of Lebesgue's dominated convergence theorem, since 
\begin{equation*}
\INTDom{\varphi(\omega,x)}{\R^d}{\mu_{\omega}(s)(x)} ~\underset{s \to t}{\longrightarrow}~ \INTDom{\varphi(\omega,x)}{\R^d}{\mu_{\omega}(t)(x)}
\end{equation*}
for each $\varphi \in C^0_b(\Omega \times \R^d,\R)$ and every $\omega \in \Omega$. For any $\phi \in \Adm((0,T) \times \Omega \times \R^d)$, it follows from similar arguments as those detailed in Step 2 of the proof of Theorem \ref{thm:ComparisonWass} that there exists a sequence of simple Borel functions $(\phi_m)$ from $\Omega$ into $C^1_c((0,T) \times \R^d,\R)$ such that $\supp(\phi_m(\omega)) \subset K_{\phi}$ with 
\begin{equation*}
\sup_{m \geq 1}\NormC{\phi_m(\omega)}{1}{[0,T] \times \R^d,\R} \leq \NormC{\phi(\omega)}{1}{[0,T] \times \R^d,\R} \quad \text{and} \quad \NormC{\phi(\omega) -\phi_m(\omega)}{1}{[0,T] \times \R^d,\R} ~\underset{m \to +\infty}{\longrightarrow}~ 0 
\end{equation*}
for each $\omega \in \Omega$. As we assumed that the curves $\{\mu_{\omega}(\cdot)\}_{\omega \in \Omega} \subset C^0([0,T],\Pcal(\R^d))$ solve \eqref{eq:DisintegratedTransport}, it clearly holds that 
\begin{equation*}
\INTSeg{\INTDom{\Big( \partial_t \phi_m(t,\omega,x) + \big\langle \nabla_x \phi_m(t,\omega,x) , \vb(t,\omega,x) \big\rangle \Big)}{\Omega \times \R^d}{\Bmu(t)(\omega,x)}}{t}{0}{T} = 0
\end{equation*}
for each $m \geq 1$, from whence the result follows by letting $m \to +\infty$ and resorting yet again to Lebesgue's dominated convergence theorem. 
\end{proof}

In what follows, we show that the stronger integrability condition 
\begin{equation}
\label{eq:IntBoundBis}
\INTSeg{\sup_{\omega \in \Omega} \INTDom{|\vb(t,\omega,x)|}{\R^d}{\mu_{\omega}(t)(x)}}{t}{0}{T} < +\infty
\end{equation}
entails the existence of a (fibred) narrowly continuous representative for solutions of \eqref{eq:TransportGraph}, whose dynamics may be equivalently written in differential form involving a more general -- and perhaps natural -- family of test functions.  

\begin{prop}[Fibred narrow continuous representative and differential form of the dynamics]
\label{prop:Representative}
Let $\Bmu : [0,T] \to \Pcal_{\pi}(\Omega \times \R^d)$ be a solution of \eqref{eq:DistribCE1} satisfying the integrability condition \eqref{eq:IntBoundBis}. Then $\Bmu(\cdot) \in C^0([0,T], \Pcal_{\pi}(\Omega \times \R^d))$ up to a modification on an $\Lcal^1$-negligible set. Moreover, it holds that
\begin{equation}
\label{eq:DistribCE2}
\derv{}{t}{} \INTDom{\varphi(\omega,x)}{\Omega \times \R^d}{\Bmu(t)(\omega,x)} = \INTDom{\big\langle \nabla_x \varphi(\omega,x) , \vb(t,\omega,x) \big\rangle}{\Omega \times \R^d}{\Bmu(t)(\omega,x)}
\end{equation}
for $\Lcal^1$-almost every $t \in [0,T]$ and each $\varphi \in C^1_{\pi,c}(\Omega  \times \R^d,\R)$. Conversely, if such a curve satisfies \eqref{eq:DistribCE2}, then it is a solution of \eqref{eq:DistribCE1}.
\end{prop}

\begin{proof}
To establish that $\Bmu : [0,T] \to \Pcal_{\pi}(\Omega \times \R^d)$ admits a fibred narrowly continuous representative, it is sufficient to adapt the arguments of \cite[Lemma 8.1.2]{AGS}, showing first that under \eqref{eq:IntBoundBis}, the mappings
\begin{equation*}
\Bmu_{\varphi} : t \in [0,T] \mapsto \INTDom{\varphi(\omega,x)}{\Omega \times \R^d}{\Bmu(t)(\omega,x)} \in \R
\end{equation*}
 are equi-absolutely continuous for each $\varphi \in C^1_{\pi,c}(\Omega \times \R^d,\R)$. Indeed, one has that
\begin{equation}
\label{eq:IntermediateAC}
\begin{aligned}
|\Bmu_{\varphi}(t_2) - \Bmu_{\varphi}(t_1)| & \leq \INTSeg{\INTDom{\big|\langle \nabla_x \varphi(\omega,x) , \vb(s,\omega,x) \rangle \big|}{\Omega \times \R^d}{\Bmu(s)(\omega,x)}}{s}{t_1}{t_2} \\ 
& \leq \bigg( \INTDom{\sup_{x \in \R^d} |\nabla_x \varphi(\omega,x)|}{\Omega}{\pi(\omega)} \bigg) \INTSeg{\sup_{\omega \in \Omega} \, \INTDom{|\vb(s,\omega,x)|}{\R^d}{\mu_{\omega}(s)(x)}}{s}{t_1}{t_2} \\
& \leq \; \Norm{\varphi}_{C^1_{\pi,c}(\Omega \times \R^d,\R)} \INTSeg{\sup_{\omega \in \Omega} \, \INTDom{|\vb(s,\omega,x)|}{\R^d}{\mu_{\omega}(s)(x)}}{s}{t_1}{t_2}
\end{aligned}
\end{equation}
for all times $0 \leq t_1 \leq t_2 \leq T$. Then, denote by $\Apazo_{\varphi} \subset [0,T]$ the subset of full $\Lcal^1$-measure made of all Lebesgue points of $\Bmu_{\varphi}(\cdot) \in \AC([0,T],\R)$, and fix a countable dense set $\{\varphi_n\}_{n=1}^{+\infty} \subset C^1_{\pi,c}(\Omega \times \R^d,\R)$. Upon letting $\Apazo := \cap_{n=1}^{+\infty} \Apazo_{\varphi_n} \subset (0,T)$, it follows from what precedes that $\Bmu : \Apazo \to \Pcal_{\pi}(\Omega \times \R^d)$ admits a continuous extension $\widetilde{\Bmu} : [0,T] \to C^1_{\pi,c}(\Omega \times \R^d,\R)'$. To conclude, there only remains to show that the collection of measures $\{\Bmu(t)\}_{t \in \Apazo}$ is fibrely tight. This latter fact follows from a routine application of the characterisation of relative fibred compactness provided in Proposition \ref{prop:FibredTightness} combined with the integrability condition \eqref{eq:IntBoundBis} in the spirit of \cite[Lemma 8.1.2]{AGS}.

We now prove that under \eqref{eq:IntBoundBis}, every solution of \eqref{eq:TransportGraph} satisfies \eqref{eq:DistribCE2}, and vice versa. For the direct implication, we simply take some $\varphi \in C^1_{\pi,c}(\Omega \times \R^d,\R)$ and note that if $\Bmu(\cdot)$ solves \eqref{eq:TransportGraph}, then by Theorem \ref{thm:Equivalence} the disintegrated curves $\{\mu_{\omega}(\cdot)\}_{\omega \in \Omega}$ satisfy 
\begin{equation*}
\INTSeg{\INTDom{\Big( \zeta'(t) \varphi(\omega,x) + \big\langle \zeta(t)\nabla_x \varphi(\omega,x) , \vb(t,\omega,x) \big\rangle \Big)}{\R^d}{\mu_{\omega}(t)(x)}}{t}{0}{T} = 0
\end{equation*}
for every $\zeta \in C^1_c((0,T),\R)$ and $\pi$-almost every $\omega \in \Omega$. Under \eqref{eq:IntBoundBis}, it is licit to integrate the above identity over $\Omega$, and it then follows from Fubini's theorem that
\begin{equation*}
\derv{}{t} \INTDom{\varphi(\omega,x)}{\Omega \times \R^d}{\Bmu(t)(\omega,x)} = \INTDom{\big\langle \nabla_x \varphi(\omega,x) , \vb(t,\omega,x) \big\rangle}{\Omega \times \R^d}{\Bmu(t)(\omega,x)}
\end{equation*}
in the sense of distributions, and consequently for $\Lcal^1$-almost every $t \in [0,T]$ since the right-hand side is Lebesgue integrable.
Conversely, if $\Bmu(\cdot) \in C^0([0,T],\Pcal_{\pi}(\Omega \times \R^d))$ solves \eqref{eq:DistribCE2} for each $\varphi \in C^1_{\pi,c}(\Omega \times \R^d)$, then it satisfies
\begin{equation*}
\INTSeg{\INTDom{\Big( \partial_t( \zeta(t)\xi(\omega)\psi(x)) + \big\langle \nabla_x(\zeta(t)\xi(\omega)\psi(x)) , \vb(t,\omega,x) \big\rangle \Big)}{\Omega \times \R^d}{\Bmu(t)(\omega,x)}}{t}{0}{T} = 0
\end{equation*}
for every $\zeta \in C^1_c((0,T),\R)$, all $\xi \in C^0_b(\Omega,\R)$ and each $\psi \in C^1_c(\R^d,\R)$ by a simple integration by parts. We may then conclude by combining the Stone-Weierstrass theorem (see e.g. \cite[Chapter X -- Section 3.4, Corollary 2]{Bourbaki2}) and the approximation argument by simple functions detailed in the proof of Theorem \ref{thm:Equivalence} above that the latter identity holds for each test function $\phi \in \Adm((0,T) \times \Omega \times \R^d,\R) $.
\end{proof}

\begin{rmk}[Comparing the integrability conditions \eqref{eq:IntBound} and \eqref{eq:IntBoundBis}] 
At the beginning of this section, we defined distributional solutions of fibred continuity equations in duality with the class $\Adm((0,T) \times \Omega \times \R^d,\R)$ of admissible test functions. Its elements, which are continuously differentiable in $(t,x) \in (0,T) \times \R^d$ with norm and support uniformly bounded in $\omega \in \Omega$, are sufficiently constrained to accommodate the fairly loose condition \eqref{eq:IntBound}. On the other hand, when considering elements of the (much) larger space $C^1_{\pi,c}(\Omega \times \R^d,\R)$, one needs to strengthen the integrability condition to handle both the fact that the support of each $\varphi(\omega) \in C^1_c(\R^d,\R)$ may depend in a non-trivial way on $\omega \in \Omega$, and that the underlying space norm is merely integrable. Such needs are reflected in \eqref{eq:IntBoundBis}, which requires uniform integrability in $\omega \in \Omega$ of the full unweighted field.  
\end{rmk}

We end this first section by proving a basic existence and uniqueness result in the spirit of standard Cauchy-Lipschitz theory for initial value problems of the form
\begin{equation}
\label{eq:FibredCauchyBis}
\left\{
\begin{aligned}
& \partial_t \Bmu(t) + \Div_x(\vb(t) \Bmu(t)) = 0, \\
& \Bmu(\tau) = \Bmu^{\tau}, 
\end{aligned}
\right.
\end{equation}
where $(\tau,\Bmu^{\tau}) \in [0,T] \times \Pcal_{\pi}(\Omega \times \R^d)$ is a fixed pair of initial conditions. 

\begin{prop}[Well-posedness in the Cauchy-Lipschitz framework]
\label{prop:Wellposed}
Suppose that the vector fields $(t,x) \in [0,T] \times \R^d \mapsto \vb(t,\omega,x) \in \R^d$ are Carathéodory for $\pi$-almost every $\omega \in \Omega$ and comply with the following integrability and regularity bounds
\begin{equation}
\label{eq:CauchyLipElemBound2}
\INTSeg{\sup_{(\omega,x) \in \Omega \times \R^d} \frac{|\vb(t,\omega,x)|}{1 + |x|} \,}{t}{0}{T} < +\infty \qquad \text{and} \qquad \INTSeg{\Lip(\vb(t,\omega) \, ; K)}{t}{0}{T} < +\infty
\end{equation}
for every compact set $K \subset \R^d$. Then for each pair $(\tau,\Bmu^{\tau}) \in [0,T] \times \Pcal_{\pi}(\Omega \times \R^d)$, the Cauchy problem \eqref{eq:FibredCauchyBis} admits a unique globally defined solution, given by 
\begin{equation}
\label{eq:RepresentationCauchy}
\Bmu(t) = \INTDom{\Phi_{(\tau,t) \sharp \,}^{\omega} \mu_{\omega}^{\tau}}{\Omega}{\pi(\omega)}
\end{equation}
for all times $t\in [\tau,T]$. Therein, the maps $(t,\omega,x) \in [\tau,T] \times \Omega \times \R^d \mapsto \Phi_{(\tau,t)}^{\omega}(x) \in \R^d$ are the unique solutions of the Cauchy problems
\begin{equation}
\label{eq:FlowDef}
\Phi_{(\tau,t)}^{\omega}(x) = x + \INTSeg{\vb \Big( s , \omega , \Phi_{(\tau,s)}^{\omega}(x) \Big)}{s}{\tau}{t}.   
\end{equation}
In particular, they are continuous in $(t,x) \in [\tau,T] \times \R^d$ and $\pi$-measurable in $\omega \in \Omega$.
\end{prop}

\begin{proof}
It is a classical result from the theory of ordinary differential equations (see for instance \cite[Chapter 1]{Filippov2013}) that, under the regularity conditions posited in  \eqref{eq:CauchyLipElemBound2} above, the flow equations \eqref{eq:FlowDef} admit for $\pi$-almost every $\omega \in \Omega$ a unique and globally defined solution issued from every $x \in \R^d$. By folklore results excerpted from that same reference, one may further show that the mappings $(t,x) \in [\tau,T] \times \R^d \mapsto \Phi_{(\tau,t)}^{\omega}(x) \in \R^d$ are continuous. We now prove that the maps $\omega \in \Omega \mapsto \Phi_{(\tau,t)}^{\omega}(x) \in \R^d$ are $\pi$-measurable for all $(t,x) \in [\tau,T] \times \R^d$. This follows e.g. from the fact that, owing to \cite[Chapter 1, Theorem 1]{Filippov2013}, the solution of \eqref{eq:FlowDef} can be obtained 
as the uniform limit of a sequence $(\Phi_{(\tau,\cdot)}^{\omega,n}(x)) \subset \AC([\tau,T],\R^d)$ whose elements are the unique solutions of the delayed flow equations
\begin{equation*}
\Phi_{(\tau,t)}^{\omega,n}(x) = x + \INTSeg{\vb \Big( s , \omega , \Phi_{(\tau,s-T/n)}^{\omega,n}(x) \Big)}{s}{\tau}{t},
\end{equation*}
with the convention that $\Phi_{(\tau,s)}^{\omega,n}(x) = x$ for $s \in [\tau-T/n,\tau]$. By construction, such maps are clearly $\pi$-measurable in $\omega \in \Omega$ under our regularity assumptions on $\vb : [0,T] \times \Omega \times \R^d \to \R^d$, so that $\omega \mapsto \Phi_{(\tau,t)}^{\omega}(x) \in \R^d$ is in turn $\pi$-measurable for all $(t,x) \in [\tau,T] \times \R^d$ as the pointwise limit of a sequence of measurable functions.

We now prove that the formula \eqref{eq:RepresentationCauchy} is licit and defines the unique solution of \eqref{eq:TransportGraph}. To this end, begin by noting that $\omega \in \Omega \mapsto \Phi_{(\tau,\cdot)}^{\omega}(\cdot) \in C^0([\tau,T] \times \R^d,\R^d)$ is $\pi$-measurable as a consequence of Proposition \ref{prop:Measurability}-$(b)$, which implies together with \eqref{eq:ContinuityNarrow} that 
\begin{equation*}
(\omega,\mu) \in \Omega \times \Pcal(\R^d) \mapsto \Phi_{(\tau,t) \sharp \,}^{\omega} \mu \in \Pcal(\R^d)
\end{equation*}
is a Carathéodory map for all times $t \in [\tau,T]$. This allows us in turn to define the curves of measures 
\begin{equation*}
\mu_{\omega}(t) := \Phi_{(\tau,t) \sharp \,}^{\omega} \mu_{\omega}^{\tau}, 
\end{equation*}
for all times $t \in [\tau,T]$ and $\pi$-almost every $\omega \in \Omega$, which under \eqref{eq:CauchyLipElemBound2} are the unique solutions of
\begin{equation*}
\left\{
\begin{aligned}
& \partial_t \mu_{\omega}(t) + \Div_x(\vb(t,\omega) \mu_{\omega}(t)) = 0, \\
& \mu_{\omega}(\tau) = \mu_{\omega}^{\tau}
\end{aligned}
\right.
\end{equation*}
as a consequence e.g. of the standard results of \cite[Section 8.1]{AGS}. To conclude, there remains to observe that the curve $\Bmu : [\tau,T] \to \Pcal_{\pi}(\Omega \times \R^d)$ defined via \eqref{eq:RepresentationCauchy} is a solution of the Cauchy problem \eqref{eq:FibredCauchyBis}, and that the latter is in fact unique owing to the representation result of Theorem \ref{thm:Equivalence}. 
\end{proof}


\subsection{Lagrangian representation, a priori estimates and Peano existence}
\label{subsection:LagrangianRepExistence}

In this second section, we prove the counterpart for structured continuity equations of Ambrosio's famed superposition principle, see e.g. \cite{AmbrosioPDE,AmbrosioC2014}. We then leverage this result to derive quantitative stability estimates on solutions of Cauchy problems of the form
\begin{equation}
\label{eq:FibredCauchyTer}
\left\{
\begin{aligned}
& \partial_t \Bmu(t) + \Div_x(\vb(t) \Bmu(t)) = 0, \\
& \Bmu(\tau) = \Bmu^{\tau}, 
\end{aligned}
\right.
\end{equation}
for some arbitrary starting pair $(\tau,\Bmu^{\tau}) \in [0,T] \times \Pcal_{\pi}(\Omega \times \R^d)$. In what ensues, we let $\Sigma_T := C^0([0,T],\R^d)$ be the space of continuous curves from $[0,T]$ into $\R^d$, denote by $\efrak_t : (x,\sigma) \in \R^d \times \Sigma_T \mapsto \sigma(t) \in \R^d$ the standard evaluation map, and let $\BEfrak_t : (\omega,x,\sigma) \in \Omega \times \R^d \times \Sigma_T \mapsto (\omega,\sigma(t)) \in \Omega \times \R^d$ stand for its natural lift to fibred spaces. 

\begin{thm}[Superposition principle for structured continuity equations]
\label{thm:Superposition}
Let $\vb : [0,T] \times \Omega \times \R^d \to \R^d$ be a Lebesgue-Borel velocity field and $\Bmu(\cdot) \in C^0([\tau,T],\Pcal_{\pi}(\Omega \times \R^d))$ be a solution of \eqref{eq:FibredCauchyTer} complying with the integrability bound
\begin{equation}
\label{eq:IntBoundSuperposition}
\INTSeg{\INTDom{\frac{|\vb(t,\omega,x)|}{1+|x|} \,}{\Omega \times \R^d}{\Bmu(t)(\omega,x)}}{t}{\tau}{T} < +\infty. 
\end{equation}
Then, there exists a superposition measure $\Beta \in \Pcal(\Omega \times \R^d \times \Sigma_T)$ such that $(\pfrak_{\Omega})_{\sharp}\Beta = \pi$ and $(\BEfrak_t)_{\sharp} \Beta = \Bmu(t)$ for all times $t \in [\tau,T]$, whose disintegration $\{\Beta_{\omega}\}_{\omega \in \Omega} \subset \Pcal(\R^d \times \Sigma_T)$ is concentrated for  $\pi$-almost every $\omega \in \Omega$ on the set of pairs $(x_{\omega},\sigma_{\omega}) \in \R^d \times \AC([\tau,T],\R^d)$ satisfying
\begin{equation}
\label{eq:ODECharac}
\sigma_{\omega}(t) = x_{\omega} + \INTSeg{\vb(s,\omega,\sigma_{\omega}(s))}{s}{\tau}{t}
\end{equation}
for all times $t \in [\tau,T]$. 
\end{thm}

\begin{proof}
The proof of this result being based on a rather long and technical adaptation of the classical strategy detailed in \cite{AmbrosioC2014}, its presentation is postponed to Appendix \ref{section:AppendixSuperposition} below. 
\end{proof}

\begin{rmk}[On the integrability condition \eqref{eq:IntBoundSuperposition}]
Let us recall that the classical superposition principle for curves of measures $\mu(\cdot) \in C^0([0,T],\Pcal(\R^d))$ solving a standard continuity equation with velocity $v : [0,T] \times \R^d \to \R^d$ holds e.g. under the usual integrability condition 
\begin{equation*}
\INTSeg{\INTDom{\frac{|v(t,x)|}{1+|x|}}{\R^d}{\mu(t)(x)}}{t}{0}{T} < +\infty, 
\end{equation*}
see for instance \cite[Theorem 3.4]{AmbrosioC2014}. Heuristically, the latter is a relaxed sublinearity estimate on the driving field, which holds on average with respect to the solution and ensures the global well-posedness of characteristics for almost every initial data. In this context, the integrability bound \eqref{eq:IntBoundSuperposition} subtending our superposition principle imposes a completely similar condition on the fibred field. 
\end{rmk}

The previous theorem allows for the elementary derivation of useful a priori estimates for solutions of \eqref{eq:FibredCauchyTer} starting from an initial datum $\Bmu^{\tau} \in \Pcal_{\pi,\,p}(\Omega \times \R^d)$ for some $p \in [1,+\infty)$, under fairly mild regularity assumptions. 

\begin{prop}[A priori moment, equi-integrability, regularity and support bounds]
\label{prop:Bounds}
Fix some $(\tau,\Bmu^{\tau}) \in [0,T] \times \Pcal_{\pi,\,p}(\Omega \times \R^d)$, and let $\Bmu(\cdot) \in C^0([\tau,T],\Pcal_{\pi}(\Omega \times \R^d))$ be a solution of \eqref{eq:FibredCauchyTer} driven by a Lebesgue-Borel vector field $\vb : [0,T] \times \Omega \times \R^d \to \R^d$ complying with the sublinearity estimate 
\begin{equation}
\label{eq:SublinearityEst}
|\vb(t,\omega,x)| \leq m(t)( 1 + |x|)
\end{equation}
for $\Lcal^1$-almost every $t \in [0,T]$ and $\Bmu(t)$-almost every $(\omega,x) \in \Omega \times \R^d$, where $m(\cdot) \in L^1([0,T],\R_+)$ is given. Then, it holds that
\begin{equation}
\label{eq:MomentEst}
\Mpazo_p(\mu_{\omega}(t)) \leq \bigg( \Mpazo_p(\mu^{\tau}_{\omega}) + \INTSeg{m(s)}{s}{\tau}{t} \bigg) \exp \bigg( \INTSeg{m(s)}{s}{\tau}{t} \bigg)
\end{equation}
for all times $t \in [\tau,T]$ and $\pi$-almost every $\omega \in \Omega$. Moreover, there exists a positive constant $C_T > 0$ depending only on the magnitude of $\Norm{m(\cdot)}_1$ such that
\begin{equation}
\label{eq:UnifIntEst}
\INTDom{|x|^p}{\{x \; \textnormal{s.t.} \, |x| \geq R \}}{\mu_{\omega}(t)(x)} \leq C_T^p \INTDom{(1+|x|)^p}{\{x \; \textnormal{s.t.} \, |x| \geq R/C_T-1\}}{\mu^{\tau}_{\omega}(x)} 
\end{equation}
for $\pi$-almost every $\omega \in \Omega$ and each $R > 0$, and 
\begin{equation}
\label{eq:ACEst}
W_p(\mu_{\omega}(t_1),\mu_{\omega}(t_2)) \leq (1+\Mpazo_p(\mu_{\omega}^{\tau})) (1+C_T) \INTSeg{m(s)}{s}{t_1}{t_2}
\end{equation}
for all times $\tau \leq t_1 \leq t_2 \leq T$ and $\pi$-almost every $\omega \in \Omega$. This implies in particular that $\Bmu(\cdot) \in \AC([\tau,T],\Pcal_{\pi,\,p}(\Omega \times \R^d))$, with 
\begin{equation*}
W_{\pi,\,p}(\Bmu(t_1),\Bmu(t_2)) \leq (1+\Mpazo_{\pi,\,p}(\Bmu^{\tau})) (1+C_T) \INTSeg{m(s)}{s}{t_1}{t_2}.
\end{equation*}
Besides, if $\supp(\Bmu^\tau) \subset \Omega \times B(0,r)$ for some $r>0$, then 
\begin{equation}
\label{eq:SuppEst}
\supp(\Bmu(t)) \subset \Omega \times B(0,R_r)
\end{equation}
for all times $t \in [\tau,T]$, where $R_r := \big( r + \|m(\cdot)\|_{L^1([\tau,T],\R)}) \exp(\|m(\cdot)\|_{L^1([\tau,T],\R)} \big)$.
\end{prop}

\begin{proof}
To establish first the a priori moment bound \eqref{eq:MomentEst}, we start by noting that owing to \eqref{eq:SublinearityEst} above combined with Theorem \ref{thm:Superposition}, there exists a superposition measure $\Beta^{\tau} = \INTDom{\Beta_{\omega}^{\tau}}{\Omega}{\pi(\omega)} \in \Pcal(\Omega \times \R^d \times \Sigma_T)$ whose disintegrations concentrate on the sets of pairs $(x_{\omega},\sigma_{\omega}) \in \R^d \times \AC([\tau,T],\R^d)$ satisfying
\begin{equation*}
\sigma_{\omega}(t) = x_{\omega} + \INTSeg{\vb(s,\omega,\sigma_{\omega}(s))}{s}{\tau}{t} 
\end{equation*}
for all times $t \in [\tau,T]$ and $\pi$-almost every $\omega \in \Omega$, and that is such that $\mu_{\omega}(t) = (\efrak_t)_{\sharp} \Beta_{\omega}^{\tau}$ . Besides, by a simple application of Gr\"onwall's lemma, one can verify that 
\begin{equation}
\label{eq:CharacEstProof}
|\sigma_{\omega}(t)| \leq \bigg( |x_{\omega}| + \INTSeg{m(s)}{s}{\tau}{t} \bigg) \exp \bigg(\INTSeg{m(s)}{s}{\tau}{t} \bigg)
\end{equation}
for all times $t \in [\tau,T]$ and $\Beta_{\omega}^{\tau}$-almost every $(x_{\omega},\sigma_{\omega}) \in \R^d \times \Sigma_T$, and the latter identity in turn yields 
\begin{equation*}
\begin{aligned}
\Mpazo_p(\mu_{\omega}(t)) & = \bigg( \INTDom{|x|^p}{\R^d}{\mu_{\omega}(t)(x)} \bigg)^{1/p} \\
& = \bigg( \INTDom{|\efrak_t(x_{\omega},\sigma_{\omega})|^p}{\R^d \times \Sigma_T}{\Beta_{\omega}^{\tau}(x_{\omega},\sigma_{\omega})} \bigg)^{1/p} \\
& \leq \bigg( \INTDom{\bigg( |x_{\omega}| + \INTSeg{m(s)}{s}{\tau}{t} \bigg)^p}{\R^d \times \Sigma_T}{\Beta_{\omega}^{\tau}(x_{\omega},\sigma_{\omega})} \bigg)^{1/p} \exp \bigg(\INTSeg{m(s)}{s}{\tau}{t} \bigg) \\
& \leq \bigg( \Mpazo_p(\mu_{\omega}^{\tau}) + \INTSeg{m(s)}{s}{\tau}{t} \bigg) \exp \bigg(\INTSeg{m(s)}{s}{\tau}{t} \bigg)
\end{aligned}
\end{equation*}
for all times $t \in [\tau,T]$, which is precisely the desired estimate. In order to derive the equi-integrability bound of \eqref{eq:UnifIntEst}, we start by noting that following \eqref{eq:CharacEstProof} above, one has 
\begin{equation*}
|\sigma_{\omega}(t)| \leq C_T(1+|x_{\omega}|)
\end{equation*}
for all times $t \in [\tau,T]$ and $\Beta_{\omega}^{\tau}$-almost every $(x_{\omega},\sigma_{\omega}) \in \R^d \times \Sigma_T$, where we introduced the constant $C_T := \max\{1,\Norm{m(\cdot)}_1\} \exp(\Norm{m(\cdot)}_1)$. Given some $R>0$, this implies in particular  
\begin{equation*}
\begin{aligned}
\INTDom{|x|^p}{\{ x \; \textnormal{s.t.} \, |x| \geq R \}}{\mu_{\omega}(t)(x)} & = \INTDom{|\efrak_t(x_{\omega},\sigma_{\omega})|^p}{\{ (x_{\omega},\sigma_{\omega}) \; \textnormal{s.t.} \, |\efrak_t(x_{\omega},\sigma_{\omega})| \geq R\}}{\Beta_{\omega}^{\tau}(x_{\omega},\sigma_{\omega})} \\
& \leq C_T^p \INTDom{(1+|x_{\omega}|)^p}{\{ (x_{\omega},\sigma_{\omega}) \; \textnormal{s.t.} \, C_T(1+|x_{\omega}|) \geq R\}}{\Beta_{\omega}^{\tau}(x_{\omega},\sigma_{\omega})} \\
& = C_T^p \INTDom{(1+|x|)^p}{\{ x \; \textnormal{s.t.} \, |x| \geq R/C_T-1\}}{\mu^{\tau}_{\omega}(x)}
\end{aligned}
\end{equation*}
for all times $t \in [\tau,T]$. We can now shift our attention towards the absolute continuity estimate of \eqref{eq:ACEst}. Upon noting that $(\efrak_{t_1},\efrak_{t_2})_{\sharp} \Beta_{\omega}^{\tau} \in \Gamma(\mu_{\omega}(t_1),\mu_{\omega}(t_2))$ for $\pi$-almost every $\omega \in \Omega$, one has that 
\begin{equation}
\label{eq:ACEstProof}
\begin{aligned}
W_p(\mu_{\omega}(t_1),\mu_{\omega}(t_2)) & \leq \bigg( \INTDom{|\efrak_{t_2}(x_{\omega},\sigma_{\omega}) - \efrak_{t_1}(x_{\omega},\sigma_{\omega})|^p}{\R^d \times \Sigma_T}{\Beta_{\omega}^{\tau}(x_{\omega},\sigma_{\omega})} \bigg)^{1/p} \\
& \leq \bigg( \INTDom{\bigg( \INTSeg{|\vb(s,\omega,\sigma_{\omega}(s))|}{s}{t_1}{t_2} \bigg)^p}{\R^d \times \Sigma_T}{\Beta_{\omega}^{\tau}(x_{\omega},\sigma_{\omega})} \bigg)^{1/p} \\
& \leq \bigg( \INTDom{\bigg( \INTSeg{m(s)(1+C_T)(1+|x|)}{s}{t_1}{t_2} \bigg)^p}{\R^d}{\mu_{\omega}^{\tau}(x)} \bigg)^{1/p} \\
& \leq (1+C_T) (1+\Mpazo_p(\mu_{\omega}^{\tau})) \INTSeg{m(s)}{s}{t_1}{t_2}
\end{aligned}
\end{equation}
where we used Fubini's theorem. The absolute continuity of $\Bmu(\cdot) \in C^0([\tau,T],\Pcal_{\pi,\,p}(\Omega \times \R^d))$ simply follows by integrating the previous identity, since then
\begin{equation*}
\begin{aligned}
W_{\pi,\,p}(\Bmu(t_1),\Bmu(t_2)) \leq (1+C_T)(1+\Mpazo_{\pi,\,p}(\Bmu^{\tau})) \INTSeg{m(s)}{s}{t_1}{t_2}
\end{aligned}
\end{equation*}
for all times $\tau \leq t_1 \leq t_2 \leq T$. Finally, the support estimate claimed in \eqref{eq:SuppEst} in the case where $\Bmu^{\tau}\in \Pcal_{\pi}(\Omega \times B(0,r))$ for some $r > 0$ follows from the observation that, upon letting $R_r > 0$ be as above, one may deduce from \eqref{eq:CharacEstProof} that
\begin{equation*}
\mu_{\omega}(t)(\R^d \backslash B(0,R_r)) = \Beta_{\omega} \Big( \Big\{ (x_{\omega},\sigma_{\omega}) \in \R^d \times \Sigma_T ~\, \textnormal{s.t.}~ |\sigma_{\omega}(t)| > R_r \Big\} \Big) = 0
\end{equation*}
for $\pi$-almost every $\omega \in \Omega$, which then implies that $\supp(\Bmu(t)) \subset \Omega \times B(0,R_r)$ for all times $t \in [\tau,T]$.
\end{proof}

We end this section by proving an existence result à la Carathéodory-Peano for structured continuity equations, which will prove instrumental in the remainder of the manuscript. 

\begin{thm}[Existence in the Carathéodory framework]
\label{thm:PeanoLocal}
Suppose that the vector fields $\vb : [0,T] \times \Omega \times \R^d \to \R^d$ are Carathéodory and comply with the sublinearity estimate 
\begin{equation}
\label{eq:SublinearityEstBis}
|\vb(t,\omega,x)| \leq m(t)( 1 + |x|)
\end{equation}
for $\Lcal^1 \times  \pi$-almost every $(t,\omega) \in [0,T] \times \Omega$ and all $x \in \R^d$, where $m(\cdot)\in L^1([0,T],\R_+)$ is given. Then for every pair $(\tau,\Bmu^{\tau}) \in [0,T] \times \Pcal_{\pi,\,p}(\Omega \times \R^d)$, the Cauchy problem \eqref{eq:FibredCauchyTer} admits a forward solution $\Bmu(\cdot) \in \AC([\tau,T],\Pcal_{\pi,\,p}(\Omega \times \R^d))$.
\end{thm}

\begin{proof}
To begin with, one may note that under our working assumptions, it follows from classical existence results for ordinary differential equations (see e.g. \cite[Chapter 1]{Filippov2013}) that for $\pi$-almost every $\omega \in \Omega$ and any initial condition $x \in \R^d$, the characteristic equations
\begin{equation}
\label{eq:ExistenceCharacProof}
\sigma(t) = x + \INTSeg{\vb(s,\omega,\sigma(s))}{s}{\tau}{t} 
\end{equation}
admit globally defined forward solutions $\sigma_{\omega}(\cdot) \in \AC([\tau,T],\R^d)$. Then, letting $\Bmu^{\tau} = \INTDom{\mu_{\omega}^{\tau}}{\Omega}{\pi(\omega)}$, it follows e.g. from \cite[Section 3]{AmbrosioC2014} that any superposition measure $\Beta_{\omega}^{\tau} \in \Pcal(\R^d \times \Sigma_T)$ satisfying $(\pfrak_{\R^d})_{\sharp} \Beta_{\omega}^{\tau} = \mu_{\omega}^{\tau}$ and concentrated on such pairs $(x_{\omega},\sigma_{\omega}) \in \R^d \times \Sigma_T$ generates via the evaluation operation $t \in [\tau,T] \mapsto (\efrak_t)_{\sharp} \Beta_{\omega}^{\tau} \in \Pcal(\R^d)$ a forward solution to 
\begin{equation}
\label{eq:DisintegratedCauchyProof}
\left\{
\begin{aligned}
& \partial_t \mu_{\omega}(t) + \Div_x(\vb(t,\omega) \mu_{\omega}(t)) = 0, \\
& \mu_{\omega}(\tau) = \mu_{\omega}^{\tau}.  
\end{aligned}
\right.
\end{equation}
In order to conclude, owing to Theorem \ref{thm:Equivalence} above, there only remains to show that one may choose such a collection $\{\Beta_{\omega}^{\tau}\}_{\omega \in \Omega} \subset \Pcal(\R^d \times \Sigma_T)$ in a measurable way, which then entails the measurability of the families $\{\mu_{\omega}(t)\}_{\omega \in \Omega}$ for all times $t \in [\tau,T]$. To this end, observe that
\begin{equation}
\Ecal : \Beta^{\tau} \in \Pcal(\R^d \times \Sigma_T) \mapsto (\pfrak_{\R^d})_{\sharp} \Beta^{\tau} \in \Pcal(\R^d)
\end{equation}
defines a narrowly continuous map as a consequence e.g. of \eqref{eq:ContinuityNarrow} above, whereas 
\begin{equation*}
\Scal :(\omega,\Beta^{\tau}) \in \Omega \times \Pcal(\R^d \times \Sigma_T) \mapsto \INTDom{\bigg( \INTSeg{\frac{\big| \efrak_t(x,\sigma) - x - \INTSeg{\vb(s,\omega,\efrak_s(x,\sigma))}{s}{\tau}{t} \big|}{1 + |x| + 2 \sup_{s \in [\tau,T]} |\efrak_s(x,\sigma)|}}{t}{\tau}{T} \bigg)}{\R^d \times \Sigma_T}{\Beta^{\tau}(x,\sigma)} \in \R
\end{equation*}
is $\pi$-measurable in $\omega \in \Omega$ as well as narrowly continuous in $\Beta^{\tau} \in \Pcal(\R^d \times \Sigma_T)$ under our working regularity assumptions, where the continuity of the denominator in the integral follows from Berge's maximum theorem (see e.g. \cite[Theorem 17.31]{AliprantisB2006}). Recalling that the narrow topology of $\Pcal(\R^d \times \Sigma_T)$ is separable and completely metrisable, it follows e.g. from the measurable selection principle of \cite[Theorem 8.2.9]{Aubin1990} that there exists a $\pi$-measurable map $\omega \in \Omega \mapsto \Beta_{\omega}^{\tau} \in \Pcal(\R^d \times \Sigma_T)$ such that 
\begin{equation*}
\Ecal(\Beta_{\omega}^{\tau}) = \mu_{\omega}^{\tau} \quad \quad \text{and} \quad \quad \Scal(\omega,\Beta_{\omega}^{\tau}) = 0.  
\end{equation*}
At this stage, note that the second of these conditions amounts to having
\begin{equation}
\label{eq:IntegralEtaZero}
\INTDom{\frac{\big| \efrak_t(x_{\omega},\sigma_{\omega}) - x_{\omega} - \INTSeg{\vb(s,\omega,\efrak_s(x_{\omega},\sigma_{\omega}))}{s}{\tau}{t} \big|}{1 + |x_{\omega}| + 2\sup_{s \in [\tau,T]} |\efrak_s(x_{\omega},\sigma_{\omega})|} \,}{\R^d \times \Sigma_T}{\Beta_{\omega}^{\tau}(x_{\omega},\sigma_{\omega})} = 0
\end{equation}
for $\Lcal^1$-almost every $t \in [\tau,T]$ by Fubini's theorem. From the continuity of the evaluation map $t \in [\tau,T] \mapsto \efrak_t(x,\sigma) \in \R^d$ for all $(x,\sigma) \in \R^d \times \Sigma_T$ along with the observation that 
\begin{equation*}
\frac{\big| \efrak_t(x_{\omega},\sigma_{\omega}) - x_{\omega} - \INTSeg{\vb(s,\omega,\efrak_s(x_{\omega},\sigma_{\omega}))}{s}{\tau}{t} \big|}{1 + |x_{\omega}| + 2 \sup_{s \in [\tau,T]} |\efrak_s(x_{\omega},\sigma_{\omega})|} \leq \Big( 1+\NormL{m(\cdot)}{1}{[\tau,T],\R} \hspace{-0.05cm} \Big)
\end{equation*}
for all times $t \in [\tau,T]$ as a consequence of the sublinearity assumption \eqref{eq:SublinearityEstBis}, it follows by taking a countable dense sequence for which \eqref{eq:IntegralEtaZero} holds and applying Lebesgue's dominated convergence theorem that the family $\{\Beta_{\omega}^{\tau}\}_{\omega \in \Omega} \subset \Pcal(\R^d \times \Sigma_T)$ concentrates on the solutions of \eqref{eq:ODECharac}. 
\end{proof}


\subsection{Classical and fibred relative compactness of the reachable sets} \label{subsection:CompactnessReachableSets}

In this section, we establish compactness results for solutions of the Cauchy problem \eqref{eq:FibredCauchyTer} from some set of initial data $\Kcal_{\tau} \subset \Pcal_{\pi,\,p}(\Omega \times \R^d)$. In what follows, we denote by $\Rpazo_{(\tau,t)}(\Kcal_{\tau}) \subset \Pcal_{\pi,\,p}(\Omega \times \R^d)$ the underlying \textit{reachable set} at time $t \in [\tau,T]$ , defined by  
\begin{equation*}
\Rpazo_{(\tau,t)}(\Kcal_{\tau}) := \bigg\{ \Bmu(t) \in \Pcal_{\pi,\,p}(\Omega \times \R^d) ~\, \textnormal{s.t.}~ \text{$\Bmu(\cdot)$ solves \eqref{eq:FibredCauchyTer} for some $\Bmu^{\tau} \in \Kcal_{\tau}$} \bigg\}. 
\end{equation*}
In the following theorem, we show that if $\Kcal_{\tau} \subset \Pcal_{\pi,\,p}(\Omega \times \R^d)$ is relatively compact either in the $W_p$-topology or $W_{\pi,\,p}$-topology, then under suitable assumptions (see the discussion in Remark \ref{rmk:CompactnessOrbits} below), the reachable sets $\Rpazo_{(\tau,t)}(\Kcal_{\tau})$ are all contained within a relatively compact set for that same topology for all times $t\in [\tau,T]$. 

\begin{thm}[Classical and fibred relative compactness of reachable sets]
\label{thm:CompactnessReach}
Let $\vb : [0,T] \times \Omega \times \R^d \to \R^d$ be a Carathéodory vector field complying with the sublinearity estimate
\begin{equation*}
|\vb(t,\omega,x)| \leq m(t)(1+|x|)
\end{equation*}
for $\Lcal^1 \times \pi$-almost every $(t,\omega) \in [0,T] \times \Omega$ and all $x \in \R^d$, where $m(\cdot) \in L^1([0,T],\R_+)$ is given. Then, if $\pi \in \Pcal_p(\Omega)$ and $\Kcal_{\tau} \subset \Pcal_{\pi,\,p}(\Omega \times \R^d)$ is a relatively compact set for the $W_p$-topology, there exists a relatively compact set $\Kcal \subset \Pcal_{\pi,\,p}(\Omega \times \R^d)$ for that same topology such that 
\begin{equation*}
\Rpazo_{(\tau,t)}(\Kcal_{\tau}) \subset \Kcal
\end{equation*}
for all times $t \in [\tau,T]$. Alternatively, if $\pi \in \Pcal(\Omega)$ and for every $R>0$ there exists a map $L_R(\cdot) \in L^1([0,T],\R_+)$ such that
\begin{equation*}
\Lip(\vb(t,\omega) \, ; B(0,R)) \leq L_R(t)
\end{equation*}
for $\Lcal^1 \times \pi$-almost every $(t,\omega) \in [0,T] \times \Omega$, then for each $r >0$ and every relatively compact set $\Kcalb_{\tau} \subset \Pcal_{\pi,\,p}(\Omega \times B(0,r))$ in the $W_{\pi,\,p}$-topology, there exists a relatively compact set $\Kcalb \subset \Pcal_{\pi,\,p}(\Omega \times \R^d)$ in that same topology such that 
\begin{equation*}
\Rpazo_{(\tau,t)}(\Kcalb_{\tau}) \subset \Kcalb
\end{equation*}
for all times $t \in [\tau,T]$.
\end{thm}

\begin{proof}
Following Proposition \ref{prop:ClassicalCompactness}, one may observe at first that the relative compactness in the $W_p$-topology of a set $\Kcal \subset \Pcal_{\pi,\,p}(\Omega \times \R^d)$ is equivalently characterised by the fact that
\begin{equation}
\label{eq:CompactnessCharacProof}
\sup_{\Bmu \in \Kcal} \Bigg( \INTDom{\Mpazo_p^p(\mu_{\omega})}{\{ \omega \; \textnormal{s.t.} \, \dsf_{\Omega}(\omega,\omega_0) \geq k \}}{\pi(\omega)} + \INTDom{|x|^p}{\{ (\omega,x) \; \textnormal{s.t.} \, |x| \geq k \}}{\Bmu(\omega,x)}\Bigg) ~\underset{k \to +\infty}{\longrightarrow}~ 0
\end{equation}
for some (and thus all) $\omega_0 \in \Omega$. Indeed, the convergence to zero of the second term in the previous equation implies that $\Kcal_{\tau} \subset \Pcal_{\pi,\,p}(\Omega \times \R^d)$ is fibrely tight because closed balls in $(\R^d,|\cdot|)$ are compact. Then, for every solution $\Bmu(\cdot) \in \AC([\tau,T],\Pcal_{\pi,\,p}(\Omega \times \R^d))$ of \eqref{eq:FibredCauchyTer} with $\Bmu^{\tau} \in \Kcal_{\tau}$, one may infer from the moment bound \eqref{eq:MomentEst} of Proposition \ref{prop:Bounds} that
\begin{equation}
\label{eq:WeakCompactOrbit1}
\begin{aligned}
& \INTDom{\Mpazo_p^p(\mu_{\omega}(t))}{\{ \omega \;\textnormal{s.t.}\, \dsf_{\Omega}(\omega,\omega_0) \geq k \}}{\pi(\omega)} \\
& \hspace{1.25cm} \leq \exp \bigg( p \INTSeg{m(s)}{s}{\tau}{t} \bigg) \INTDom{\bigg( \Mpazo_p(\mu_{\omega}^{\tau}) + \INTSeg{m(s)}{s}{\tau}{t}\bigg)^p}{\{ \omega \;\textnormal{s.t.}\, \dsf_{\Omega}(\omega,\omega_0) \geq k \}}{\pi(\omega)}.
\end{aligned}
\end{equation}
Similarly, the equi-integrability estimate \eqref{eq:UnifIntEst} of Proposition \ref{prop:Bounds} also yields
\begin{equation}
\label{eq:WeakCompactOrbit2}
\begin{aligned}
\INTDom{|x|^p}{\{ (\omega,x) \;\textnormal{s.t.}\, |x| \geq k \}}{\Bmu(t)(\omega,x)} & = \INTDom{\INTDom{|x|^p}{\{ x \;\textnormal{s.t.}\, |x| \geq k \}}{\mu_{\omega}(t)(x)}}{\Omega}{\pi(\omega)} \\
& \leq C_T^p \INTDom{\INTDom{(1+|x|)^p}{\{ x \;\textnormal{s.t.}\, |x| \geq k/C_T-1 \}}{\mu_{\omega}^{\tau}(x)}}{\Omega}{\pi(\omega)} \\
& \leq C_T^p \INTDom{(1+|x|)^p}{\{ (\omega,x) \;\textnormal{s.t.}\, |x| \geq k/C_T-1 \}}{\Bmu^{\tau}(\omega,x)} 
\end{aligned}
\end{equation}
for all times $t \in [\tau,T]$. Hence, it stems from \eqref{eq:WeakCompactOrbit1} and \eqref{eq:WeakCompactOrbit2} together with the fact that $\pi \in \Pcal_p(\Omega)$ and the relative compactness of $\Kcal_{\tau}$ in the $W_p$-topology that 
\begin{equation*}
\sup_{\Bmu(t) \in \Rpazo_{(\tau,t)}(\Kcal^{\tau})} \Bigg( \INTDom{\Mpazo_p^p(\mu_{\omega})}{\{ \omega \; \textnormal{s.t.} \, \dsf_{\Omega}(\omega,\omega_0) \geq k \}}{\pi(\omega)} + \INTDom{(1+|x|)^p}{\{ (\omega,x) \; \textnormal{s.t.} \, | x | \geq k \}}{\Bmu(\omega,x)}\Bigg) ~\underset{k \to +\infty}{\longrightarrow}~ 0
\end{equation*}
uniformly with respect to $t \in [\tau,T]$, which equivalently means by Proposition \ref{prop:ClassicalCompactness} that there exists a relatively compact set $\Kcal \subset \Pcal_{\pi,\,p}(\Omega \times \R^d)$ in that same topology for which $\Rpazo_{(\tau,t)}(\Kcal_{\tau}) \subset \Kcal$.

We consider now the case in which $\Kcalb_{\tau} \subset \Pcal_{\pi,\,p}(\Omega \times B(0,r))$ for some $r >0$ is relatively compact in the $W_{\pi,\,p}$-topology, a property which has been characterised in Theorem \ref{thm:FibredCompactness} above. Let then $\Bmu(\cdot) \in \AC([\tau,T],\Pcal_{\pi,\,p}(\Omega \times \R^d))$ be a solution of \eqref{eq:FibredCauchyTer} from $\Bmu^{\tau} \in \Kcalb_{\tau}$ and fix some $A \in \Apazo$ with positive measure. Then, it follows from \eqref{eq:UnifIntEst} of Proposition \ref{prop:Bounds} that 
\begin{equation*}
\begin{aligned}
\INTDom{|x|^p}{\{x \;\textnormal{s.t.} \, |x| \geq k \}}{\Big( \mathsmaller{\INTDomdash{\mu_{\omega}(t)}{A}{\pi(\omega)}} \Big)(x)} & = \INTDomdash{\INTDom{|x|^p}{\{x \;\textnormal{s.t.} \, |x| \geq k \}}{\mu_{\omega}(t)(x)}}{A}{\pi(\omega)} \\
& \leq C_T^p \INTDomdash{\INTDom{(1+|x|)^p}{\{x \;\textnormal{s.t.} \, |x| \geq k/C_T-1 \}}{\mu_{\omega}^{\tau}(x)}}{A}{\pi(\omega)} \\
& = C_T^p \INTDom{(1+|x|)^p}{\{x \;\textnormal{s.t.}\, |x| \geq k/C_T-1\}}{\Big( \mathsmaller{\INTDomdash{\mu_{\omega}^{\tau}}{A}{\pi(\omega)}} \Big)(x)}
\end{aligned}
\end{equation*}
from whence we may deduce in conjunction with \eqref{eq:CompactnessCharacProof} and the relative compactness of $\Kcalb_{\tau}$ in the $W_{\pi,\,p}$-topology that there exists a relatively compact set $\Kcal_{\tau}^A \subset \Pcal_p(\R^d)$ such that 
\begin{equation}
\label{eq:FibredCompactnessProof1}
\bigg\{ \INTDomdash{\mu_{\omega}(t)}{A}{\pi(\omega)} ~\,\textnormal{s.t.}~ \Bmu(t) \in \Rpazo_{(\tau,t)}(\Kcalb_{\tau}) \bigg\} \subset \Kcal_{\tau}^A
\end{equation}
for all times $t \in [\tau,T]$, which is precisely a version of Theorem \ref{thm:FibredCompactness}-$(i)$ that is uniform with respect to $\Bmu^{\tau} \in \Kcalb_{\tau}$. At this stage, recall that under our working assumptions, it follows from Proposition \ref{prop:Wellposed} that every solution $\Bmu(\cdot) \in \AC([\tau,T],\Pcal_{\pi,\,p}(\Omega \times \R^d))$ of \eqref{eq:FibredCauchyTer} is given explicitly by 
\begin{equation*}
\Bmu(t) = \INTDom{\Phi_{(\tau,t) \sharp \,}^{\omega} \mu_{\omega}^{\tau}}{\Omega}{\pi(\omega)}.  
\end{equation*}
Moreover, by the support estimate \eqref{eq:SuppEst} of Proposition \ref{prop:Bounds}, there exists a radius $R_r > 0$ for which
\begin{equation}
\label{eq:CompactnessSuppEst}
\supp(\Bmu(t)) \subset \Omega \times B(0,R_r)
\end{equation}
for all times $t \in [\tau,T]$ and each $\Bmu^{\tau} \in \Kcalb_{\tau}$. Then, given a set $A \in \Apazo$ of positive measure and some $\omega \in A$, one may readily verify that 
\begin{equation}
\label{eq:OscillationEst1}
\begin{aligned}
W_p \Big( \mu_{\omega}(t) , \mathsmaller{\INTDomdash{\mu_{\theta}(t)}{A}{\pi(\theta)}} \Big) & = W_p \Big( \Phi_{(\tau,t) \sharp \,}^{\omega} \mu_{\omega}^{\tau} \, , \mathsmaller{\INTDomdash{\Phi_{(\tau,t) \sharp \,}^{\theta} \mu_{\theta}^{\tau}}{A}{\pi(\theta)}} \Big) \\
& \leq W_p \bigg( \Phi_{(\tau,t) \sharp \,}^{\omega} \mu_{\omega}^{\tau} \, , \Phi_{(\tau,t) \sharp }^{\omega} \Big( \mathsmaller{\INTDomdash{\mu_{\theta}^{\tau}}{A}{\pi(\theta)}} \Big) \bigg) \\
& \hspace{0.45cm} + W_p \bigg( \Phi_{(\tau,t) \sharp }^{\omega} \Big( \mathsmaller{\INTDomdash{\mu_{\theta}^{\tau}}{A}{\pi(\theta)}} \Big) , \mathsmaller{\INTDomdash{\Phi_{(\tau,t) \sharp \,}^{\theta} \mu_{\theta}^{\tau}}{A}{\pi(\theta)}} \bigg).
\end{aligned}
\end{equation}
The first term in the right-hand side of \eqref{eq:OscillationEst1} can be bounded from above quite straightforwardly as 
\begin{equation}
\label{eq:OscillationEst2}
\begin{aligned}
W_p \bigg( \Phi_{(\tau,t) \sharp \,}^{\omega} \mu_{\omega}^{\tau} \, , \Phi_{(\tau,t) \sharp }^{\omega} \Big( \mathsmaller{\INTDomdash{\mu_{\theta}^{\tau}}{A}{\pi(\theta)}} \Big) \bigg) %
& \leq \Lip \Big( \Phi_{(\tau,t)}^{\omega} \, ; B(0,r) \Big) W_p \Big( \mu_{\omega}^{\tau} , \mathsmaller{\INTDomdash{\mu_{\theta}^{\tau}}{A}{\pi(\theta)}} \Big) \\
& \leq \exp \big( \hspace{-0.05cm} \NormL{L_{R_r}(\cdot)}{1}{[\tau,T],\R} \hspace{-0.1cm} \big) \, W_p \Big( \mu_{\omega}^{\tau} , \mathsmaller{\INTDomdash{\mu_{\theta}^{\tau}}{A}{\pi(\theta)}} \Big), 
\end{aligned}
\end{equation}
where we leveraged standard regularity estimates on flows of ordinary differential equations (see e.g. \cite[Theorem 16.2]{AmbrosioBS2021}) along with the basic optimal transport inequality \eqref{eq:LipEstWass}. Concerning the second term in the right-hand side of \eqref{eq:OscillationEst1}, one may begin by observing that
\begin{equation*}
\INTDomdash{\Big( \Phi_{(\tau,t)}^{\omega} , \Phi_{(\tau,t)}^{\theta} \Big)_{\raisebox{4pt}{$\scriptstyle{\sharp}$}} \, \mu_{\theta}^{\tau}}{A}{\pi(\theta)} \in \Gamma \bigg( \Phi_{(\tau,t) \sharp }^{\omega} \Big( \mathsmaller{\INTDomdash{\mu_{\theta}^{\tau}}{A}{\pi(\theta)}} \Big) , \mathsmaller{\INTDomdash{\Phi_{(\tau,t) \sharp \,}^{\theta} \mu_{\theta}^{\tau}}{A}{\pi(\theta)}} \bigg),  
\end{equation*}
which allows to derive the preliminary estimate 
\begin{equation}
\label{eq:OscillationEst31}
W_p \bigg( \Phi_{(\tau,t) \sharp }^{\omega} \Big( \mathsmaller{\INTDomdash{\mu_{\theta}^{\tau}}{A}{\pi(\theta)}} \Big) , \mathsmaller{\INTDomdash{\Phi_{(\tau,t) \sharp \,}^{\theta} \mu_{\theta}^{\tau}}{A}{\pi(\theta)}} \bigg) \leq \bigg( \INTDomdash{\INTDom{\Big| \Phi_{(\tau,t)}^{\omega}(x) - \Phi_{(\tau,t)}^{\theta}(x)  \Big|^p}{\R^d}{\mu_{\theta}^{\tau}(x)}}{A}{\pi(\theta)} \bigg)^{1/p}. 
\end{equation}
At this stage, one may leverage standard stability estimates on characteristic flows driven by different vector fields in the spirit e.g. of \cite[Appendix A]{ContInc} to obtain that
\begin{equation}
\label{eq:OscillationEst32}
\Big| \Phi_{(\tau,t)}^{\omega}(x) - \Phi_{(\tau,t)}^{\theta}(x) \Big| \leq \bigg( \INTSeg{ \NormC{\vb(s,\omega) - \vb(s,\theta)}{0}{B(0,R_r),\R^d}}{s}{\tau}{t} \bigg) \exp \big( \hspace{-0.05cm} \NormL{L_{R_r}(\cdot)}{1}{[\tau,T],\R} \hspace{-0.1cm} \big)
\end{equation}
for all $(t,x) \in [\tau,T] \times B(0,r)$ and $\pi$-almost every $\theta \in A$. Hence, it stems from combining \eqref{eq:OscillationEst31} with \eqref{eq:OscillationEst32} and subsequently plugging the resulting expression together with \eqref{eq:OscillationEst2} into \eqref{eq:OscillationEst1} that 
\begin{equation*}
\begin{aligned}
& W_p \Big( \mu_{\omega}(t) , \mathsmaller{\INTDomdash{\mu_{\theta}(t)}{A}{\pi(\theta)}} \Big) \\
& \leq \Bigg( W_p \Big( \mu_{\omega}^{\tau} , \mathsmaller{\INTDomdash{\mu_{\theta}^{\tau}}{A}{\pi(\theta)}} \Big)  + \bigg(\INTDomdash{\bigg(\INTSeg{\NormC{\vb(s,\omega) - \vb(s,\theta)}{0}{B(0,R_r),\R^d}}{s}{\tau}{t} \bigg)^p}{A}{\pi(\theta)} \bigg)^{1/p} \, \Bigg) \\
& \hspace{10.2cm} \times \exp \Big( \hspace{-0.05cm} \NormL{L_{R_r}(\cdot)}{1}{[\tau,T],\R} \hspace{-0.1cm} \Big).
\end{aligned}
\end{equation*}
Consequently, given any partition $\Ppazo := (A_1,\dots,A_n)$ of $\Omega$ into positive-measure sets and recalling the definition \eqref{eq:CondExpMeasure} of the underlying conditional expectation, it then holds that 
\begin{equation*}
\begin{aligned}
& W_{\pi,\,p}\big(\Bmu(t),\E_{\Ppazo}[\Bmu(t)] \big)  \\
& = \bigg( \sum_{i=1}^n \INTDom{W_p^p \Big( \mu_{\omega}(t) , \mathsmaller{\INTDomdash{\mu_{\theta}(t)}{A_i}{\pi(\theta)}} \Big)}{A_i}{\pi(\omega)} \bigg)^{1/p} \\
& \leq 2^{(p-1)/p} \Bigg( \sum_{i=1}^n \INTDom{W_p^p \Big( \mu_{\omega}^{\tau} , \mathsmaller{\INTDomdash{\mu_{\theta}^{\tau}}{A_i}{\pi(\theta)}} \Big)}{A_i}{\pi(\omega)} \\
& \hspace{2.15cm} + \sum_{i=1}^n \INTDom{\INTDomdash{\bigg( \INTSeg{\NormC{\vb(s,\omega) - \vb(s,\theta)}{0}{B(0,R_r),\R^d}}{s}{\tau}{t} \bigg)^p}{A_i}{\pi(\theta)}}{A_i}{\pi(\omega)} \Bigg)^{1/p} \\
& \hspace{9.75cm} \times \exp \Big( \hspace{-0.05cm} \NormL{L_{R_r}(\cdot)}{1}{[\tau,T],\R} \hspace{-0.1cm} \Big) \\
& \leq 2^{(p-1)/p} \Bigg( W_{\pi,\,p}(\Bmu^{\tau},\E_{\Ppazo}[\Bmu^{\tau}]) \\
& \hspace{2.05cm}+ 2 \bigg( \INTDom{\NormL{\vb(\omega) - \E_{\Ppazo}[\vb](\omega)}{1}{[0,T],C^0(B(0,R_r),\R^d)}^p}{\Omega}{\pi(\omega)} \bigg)^{1/p} \, \Bigg) \exp \Big( \hspace{-0.05cm} \NormL{L_{R_r}(\cdot)}{1}{[\tau,T],\R} \hspace{-0.05cm} \Big)
\end{aligned}
\end{equation*}
for all times $t \in [\tau,T]$. In particular, one may deduce from the compactness criterion of Theorem \ref{thm:CompactnessBochner} applied to the Bochner space $L^p\big(\Omega,L^1([\tau,T],C^0(B(0,R_r),\R^d));\pi\big)$ along with that of Theorem \ref{thm:FibredCompactness} that for each $\epsilon > 0$, there exists a finite partition $\Ppazo_{\epsilon}$ of $\Omega$ into positive-measure sets such that 
\begin{equation}
\label{eq:FibredCompactnessProof2}
W_{\pi,\,p}\big(\Bmu(t),\E_{\Ppazo_\varepsilon}[\Bmu(t)] \big) \leq \epsilon
\end{equation}
for all times $t \in [\tau,T]$ and every $\Bmu^{\tau} \in \Kcalb_{\tau}$. 
\end{proof}

\begin{rmk}[Concerning the relative compactness of the reachable sets]
\label{rmk:CompactnessOrbits}
Following the geometric approach to optimal transport developed in \cite{AGS} and its careful transposition to fibred Wasserstein spaces proposed in \cite{Peszek2023}, one can think of structured continuity equations as ODEs over the metric space $\Pcal_{\pi,\,p}(\Omega \times \R^d) \simeq L^p(\Omega,\Pcal_p(\R^d);\pi)$. Concomitantly, Theorem \ref{thm:ComparisonWass} above suggests that the classical and fibred Wasserstein convergences respectively play the role of weak and strong topologies over this space.

In this context, the fact that the reachable sets of a structured continuity equation driven by a Carathéodory field are contained within a relatively compact set for the $W_p$-topology is reminiscent of standard results on relative weak compactness of orbits for ODEs in Banach spaces, see e.g. \cite[Chapter 2]{Deimling1977}. On the other hand, it is well-known (see for instance \cite{Dieudonne1950}) that such reachable sets are not relatively strongly compact in general, unless one makes some additional assumptions on the velocity field, a possibility being to impose a stronger Lipschitz regularity, see e.g. \cite[Section 2.2]{Deimling1977}. 
\end{rmk}


\section{Nonlocal continuity equations in fibred Wasserstein spaces}
\label{section:CENonlocal}

\setcounter{equation}{0} \renewcommand{\theequation}{\thesection.\arabic{equation}}

Throughout the remainder of the paper, we derive general existence, uniqueness and representation results for solutions of Cauchy problems of the form
\begin{equation}
\label{eq:TransportGraphNonlocal}
\left\{
\begin{aligned}
& \partial_t \Bmu(t) + \Div_x \big( \vb(t,\Bmu(t)) \Bmu(t) \big) = 0, \\
& \Bmu(0) = \Bmu^0.
\end{aligned}
\right.
\end{equation}
Therein, we assume as before that $(\Omega,\Apazo,\pi)$ is a complete probability space over a Polish space, and that we are given some initial datum $\Bmu^0 \in \Pcal_{\pi,\,p}(\Omega \times \R^d)$. For the sake of readability, we will confine our analysis to the case $p=1$, but stress that all our results remain valid up to minor technical adaptations for any $p \in [1,+\infty)$. Following the preliminary result of Section \ref{subsection:Structure}, we recall that the dynamics in \eqref{eq:TransportGraphNonlocal} is understood in the sense of distributions as 
\begin{equation}
\label{eq:DistribCEBis}
\INTSeg{\INTDom{\bigg( \partial_t \phi(t,\omega,x) + \big\langle \nabla_x \phi(t,\omega,x) , \vb(t,\Bmu(t),\omega,x) \big\rangle \bigg)}{\Omega \times \R^d}{\Bmu(t)(\omega,x)}}{t}{0}{T} = 0
\end{equation}
for every admissible test function $\phi \in \Adm((0,T) \times \Omega \times \R^d,\R)$. In what follows, we will frequently use the condensed notation $\|\cdot\|_p := \|\cdot\|_{L^p([0,T],\R)}$ for the $L^p$-norm of a function over $[0,T]$.


\subsection{Existence, uniqueness and representation in the Cauchy-Lipschitz framework}
\label{subsection:Lipschitz}

In what follows, we study the existence, uniqueness, stability and representation of solutions to the Cauchy problem \eqref{eq:TransportGraphNonlocal} under the following (local) Cauchy-Lipschitz assumptions. 

\begin{taggedhyp}{\textbn{(CL)}} \hfill
\label{hyp:CL} 
\begin{enumerate}
\item[$(i)$] The mapping $(t,\omega) \in [0,T] \times \Omega \mapsto \vb(t,\Bmu,\omega,x) \in \R^d$ is $\Lcal^1 \times \pi$-measurable for all $(\Bmu,x) \in \Pcal_{\pi,1}(\Omega \times \R^d) \times \R^d$. 
\item[$(ii)$] There exists a map $m(\cdot) \in L^1([0,T],\R_+)$ such that for each $\Bmu \in \Pcal_{\pi,1}(\Omega \times \R^d)$, it holds that 
\begin{equation*}
|\vb(t,\Bmu,\omega,x)| \leq m(t) \bigg( 1 + |x| + \Mpazo_{\pi,1}(\Bmu) \bigg)
\end{equation*}
for $\Lcal^1 \times \pi$-almost every $(t,\omega) \in [0,T] \times \Omega$ and all $x \in \R^d$. 
\item[$(iii)$] For each $R > 0$, there exists a map $L_R(\cdot) \in L^1([0,T],\R_+)$ such that for $\Lcal^1 \times \pi$-almost every $(t,\omega) \in [0,T] \times \Omega$, it holds that 
\begin{equation*}
\big|\vb(t,\Bmu,\omega,x) - \vb(t,\Bnu,\omega,y) \big| \leq L_{R}(t) \Big( W_{\pi,1}(\Bmu,\Bnu) + |x-y| \Big)
\end{equation*}
for each $\Bmu,\Bnu \in \Pcal_{\pi}(\Omega \times B(0,R))$ and all $x,y \in B(0,R)$.
\end{enumerate}
\end{taggedhyp}

We would like to underline that, while we choose here to work with locally Lipschitz vector fields at the price of confining our study to curves of compactly supported measures (which nonetheless arise frequently in applications), the coming results generalise straightforwardly to measures with finite moment, up to globalising the regularity assumptions on the dynamics, see \cite{ContIncPp} for more details. 

We begin our investigations by transposing the a priori estimates of Proposition \ref{prop:Bounds} to solutions of \eqref{eq:TransportGraphNonlocal} driven by a nonlocal field satisfying Hypotheses \ref{hyp:CL}. 

\begin{lem}[Some a priori estimates]
\label{lem:NonlocalBounds}
Let  $\Bmu^0 \in \Pcal_{\pi,1}(\Omega \times \R^d)$ and suppose that $\vb : [0,T] \times \Pcal_{\pi,1}(\Omega \times \R^d) \times \Omega \times \R^d \to \R^d$ satisfies Hypotheses \ref{hyp:CL}-$(i)$, $(ii)$. Then, every solution of \eqref{eq:TransportGraphNonlocal} complies with the a priori moment, uniform-integrability and regularity bounds of Proposition \ref{prop:Bounds}, for a given constant $C_T>0$ depending only on the magnitudes of $\Mpazo_{\pi,1}(\Bmu^0)$ and $\Norm{m(\cdot)}_1$. In addition, if $\Bmu^0 \in \Pcal_{\pi}(\Omega \times B(0,r))$ for some $r > 0$, then
\begin{equation*}
\supp(\Bmu(t)) \subset \Omega \times B(0,R_r),
\end{equation*}
for all times $t \in [0,T]$, where $R_r := (r+\|m(\cdot)\|_1) \exp(2\|m(\cdot)\|_1)$.
\end{lem}

In the following proposition, we leverage Hypothesis \ref{hyp:CL}-$(iii)$ to prove a general stability estimate for solutions of \eqref{eq:TransportGraphNonlocal} in terms of the fibred Wasserstein metric. 

\begin{prop}[Distance estimates between solutions] 
\label{prop:Stability}
Let $\Bmu^0 \in \Pcal_{\pi}(\Omega \times B(0,r))$ for some $r > 0$ and $\Bmu(\cdot) \in \AC([0,T],\Pcal_{\pi,1}(\Omega \times \R^d))$ be a solution of \eqref{eq:TransportGraphNonlocal} driven by a nonlocal vector field $\vb : [0,T] \times \Pcal_{\pi,1}(\Omega \times \R^d) \times \Omega \times \R^d \to \R^d$ satisfying Hypotheses \ref{hyp:CL}. Moreover, let $\Bnu(\cdot) \in \AC([0,T],\Pcal_{\pi,1}(\Omega \times \R^d))$ be a solution of the Cauchy problem
\begin{equation*}
\left\{
\begin{aligned}
& \partial_t \Bnu(t) + \Div_y ( \wb(t) \Bnu(t)) = 0, \\
& \Bnu(0) = \Bnu^0, 
\end{aligned}
\right.
\end{equation*}
from some initial datum $\Bnu^0 \in \Pcal_{\pi}(\Omega \times B(0,r))$, and driven by a Lebesgue-Borel vector field $\wb : [0,T] \times \Omega \times \R^d \to \R^d$ complying with the sublinearity estimate 
\begin{equation}
\label{eq:SublinEstw}
|\wb(t,\omega,y)| \leq m(t) \big(1 + |y|) 
\end{equation}
for $\Lcal^1$-almost every $t \in [0,T]$ and $\Bnu(t)$-almost every $(\omega,y) \in \Omega \times \R^d$. Then, it holds that  
\begin{equation}
\label{eq:DistanceEst}
\begin{aligned}
W_{\pi,1}(\Bmu(t),\Bnu(t)) & \leq \bigg( W_{\pi,1}(\Bmu^0,\Bnu^0)  + \INTSeg{\INTDom{\NormL{\vb(s,\Bmu(s),\omega)-\wb(s,\omega)}{\infty}{\R^d,\R^d ; \, \nu_{\omega}(s)} \hspace{-0.05cm}}{\Omega}{\pi(\omega)}}{s}{0}{t} \bigg) \\
& \hspace{10cm} \times \exp \big( \|L_{R_r}(\cdot)\|_1 \big)
\end{aligned}
\end{equation}
for all times $t \in [0,T]$, where $R_r := (r \, + \Norm{m(\cdot)}_1) \exp(2 \Norm{m(\cdot)}_1)$.
\end{prop}

\begin{proof}
As a starter, notice that owing to Hypothesis \ref{hyp:CL}-$(ii)$ and \eqref{eq:SublinEstw} along with the fact that $\Bmu^0,\Bnu^0 \in \Pcal_{\pi}(\Omega \times B(0,r))$, it stems from the a priori bounds recollected in Lemma \ref{lem:NonlocalBounds} that 
\begin{equation}
\label{eq:SupportDistProof}
\supp(\mu_{\omega}(t)) \cup \supp(\nu_{\omega}(t)) \subset B(0,R_r)
\end{equation}
for all times $t \in [0,T]$ and $\pi$-almost every $\omega \in \Omega$, with $R_r > 0$ being defined as above. Besides, by Theorem \ref{thm:Superposition}, there exist two superposition measures $\Beta_{\Bmu},\Beta_{\Bnu} \in \Pcal(\Omega \times \R^d \times \Sigma_T)$ admitting the disintegration representations 
\begin{equation*}
\Beta_{\Bmu} = \INTDom{\Beta_{\Bmu}^{\omega}}{\Omega}{\pi(\omega)} \qquad \text{and} \qquad \Beta_{\Bnu} = \INTDom{\Beta_{\Bnu}^{\omega}}{\Omega}{\pi(\omega)},
\end{equation*}
for which it holds that $(\pfrak_{\R^d})_{\sharp} \Beta_{\Bmu}^{\omega} = \mu^0_{\omega}$ and $(\pfrak_{\R^d})_{\sharp} \Beta_{\Bnu}^{\omega} = \nu^0_{\omega}$ as well as $(\efrak_t)_{\sharp} \Beta_{\Bmu}^{\omega} = \mu_{\omega}(t)$ and $(\efrak_t)_{\sharp} \Beta_{\Bnu}^{\omega} = \nu_{\omega}(t)$ for all times $t \in [0,T]$ and $\pi$-almost every $\omega \in \Omega$. Moreover, these latter are concentrated on the sets of pairs satisfying
\begin{equation}
\label{eq:CharacDefDist}
\sigma_{\mu_{\omega}}(t) = x_{\omega} + \INTSeg{\vb \Big(s,\Bmu(s),\omega,\sigma_{\mu_{\omega}}(s) \Big)}{s}{0}{t} \qquad \text{and} \qquad \sigma_{\nu_{\omega}}(t) = y_{\omega} + \INTSeg{\wb(s,\omega,\sigma_{\nu_{\omega}}(s))}{s}{0}{t}
\end{equation}
for $\Beta_{\Bmu}^{\omega} \times \Beta_{\Bnu}^{\omega}$-almost every $(x_{\omega},\sigma_{\mu_{\omega}},y_{\omega},\sigma_{\nu_{\omega}}) \in (\R^d \times \Sigma_T)^2$ and $\pi$-almost every $\omega \in \Omega$. At this stage, given some measurable family $\{\gamma_{\omega}^0 \}_{\omega \in \Omega} \subset \Pcal(\R^{2d})$ satisfying $\gamma^0_{\omega} \in \Gamma_o(\mu_{\omega}^0,\nu_{\omega}^0)$ for $\pi$-almost every $\omega \in \Omega$, we define the transport plan $\hat{\Beta}_{\Bmu,\Bnu}^{\omega} \in \Gamma(\Beta_{\Bmu}^{\omega},\Beta_{\Bnu}^{\omega})$ by 
\begin{equation*}
\hat{\Beta}_{\Bmu,\Bnu}^{\omega} := \INTDom{\big( \Beta_{\Bmu}^{\omega,x} \times \Beta_{\Bnu}^{\omega,y} \big)}{\R^{2d}}{\gamma_{\omega}^0(x,y)}
\end{equation*}
wherein $\{ \Beta_{\Bmu}^{\omega,x}\}_{x \in \R^d},\{\Beta_{\Bnu}^{\omega,y}\}_{y \in \R^d} \subset \Pcal(\Sigma_T)$ respectively stand for the disintegrations of $\Beta_{\Bmu}^{\omega},\Beta_{\Bnu}^{\omega} \in \Pcal(\R^d \times \Sigma_T)$ against their first marginals $\mu_{\omega}^0,\nu_{\omega}^0 \in \Pcal(\R^d)$. Then, the latter construction permits to derive the following elementary estimate
\begin{equation}
\label{eq:DistanceEstProof1}
\begin{aligned}
W_1(\mu_{\omega}(t),\nu_{\omega}(t)) \leq \INTDom{|\sigma_{\mu_{\omega}}(t) - \sigma_{\nu_{\omega}}(t)|}{(\R^d \times \Sigma_T)^2}{\hat{\Beta}_{\Bmu,\Bnu}^{\omega}(x_{\omega},\sigma_{\mu_{\omega}},y_{\omega},\sigma_{\nu_{\omega}})}
\end{aligned}
\end{equation}
for all times $t \in [0,T]$. Observe now that as a consequence of \eqref{eq:CharacDefDist} along with the support estimate of \eqref{eq:SupportDistProof} and Hypothesis \ref{hyp:CL}-$(iii)$, one further has that
\begin{equation*}
\begin{aligned}
|\sigma_{\mu_{\omega}}(t) - \sigma_{\nu_{\omega}}(t)| \leq |x-y| & + \INTSeg{\Big| \vb(s,\Bmu(s),\omega,\sigma_{\mu_{\omega}}(s)) - \wb(s,\omega,\sigma_{\nu_{\omega}}(s))  \Big|}{s}{0}{t} \\
\leq |x-y| & + \INTSeg{L_{R_r}(s) |\sigma_{\mu_{\omega}}(s) - \sigma_{\nu_{\omega}}(s)|}{s}{0}{t} \\
& + \INTSeg{\Big| \vb(s,\Bmu(s),\omega,\sigma_{\nu_{\omega}}(s)) - \wb(s,\omega,\sigma_{\nu_{\omega}}(s)) \Big|}{s}{0}{t} \\
\leq |x-y| & + \INTSeg{L_{R_r}(s) |\sigma_{\mu_{\omega}}(s) - \sigma_{\nu_{\omega}}(s)|}{s}{0}{t} \\
& + \INTSeg{\NormL{\vb(s,\Bmu(s),\omega) - \wb(s,\omega)}{\infty}{\R^d,\R^d ; \, \nu_{\omega}(s)} \hspace{-0.05cm}}{s}{0}{t}
\end{aligned}
\end{equation*}
for $\hat{\Beta}_{\Bmu,\Bnu}^{\omega}$-almost every $(x_{\omega},\sigma_{\mu_{\omega}},y_{\omega},\sigma_{\nu_{\omega}}) \in (\R^d \times \Sigma_T)^2$, where we made use of the observation that
\begin{equation*}
\big| \vb(s,\Bmu(s),\omega,\sigma_{\nu_{\omega}}(s)) - \wb(s,\omega,\sigma_{\nu_{\omega}}(s)) \big| \leq \, \NormL{\vb(s,\Bmu(s),\omega) - \wb(s,\omega)}{\infty}{\R^d,\R^d ; \, \nu_{\omega}(s)}
\end{equation*}
for all times $s \in [0,t]$ and $\Beta_{\Bnu}^{\omega}$-almost every $(y_{\omega},\sigma_{\nu_{\omega}}) \in \R^d \times \Sigma_T$, owing to the definition of the essential supremum and since $\nu_{\omega}(s) = (\efrak_s)_{\sharp} \Beta_{\Bnu}^{\omega}$ for $\pi$-almost every $\omega \in \Omega$. Then, by a simple application of Gr\"onwall's lemma, the latter identities imply that
\begin{equation}
\label{eq:DistanceEstProof2}
|\sigma_{\mu_{\omega}}(t) - \sigma_{\nu_{\omega}}(t)| \leq \bigg( |x_{\omega}-y_{\omega}| + \INTSeg{\NormL{\vb(s,\Bmu(s),\omega) - \wb(s,\omega)}{\infty}{\R^d,\R^d ; \, \nu_{\omega}(s)}}{s}{0}{t} \bigg) \exp \bigg( \INTSeg{L_{R_r}(s)}{s}{0}{t} \bigg)
\end{equation}
for all times $t \in [0,T]$ and $\hat{\Beta}_{\Bmu,\Bnu}^{\omega}$-almost every $(x_{\omega},\sigma_{\mu_{\omega}},y_{\omega},\sigma_{\nu_{\omega}}) \in (\R^d \times \Sigma_T)^2$. In particular, upon merging the estimates \eqref{eq:DistanceEstProof1} and \eqref{eq:DistanceEstProof2}, one finally obtains  
\begin{equation*}
\begin{aligned}
& W_1(\mu_{\omega}(t),\nu_{\omega}(t)) \\
& \leq \INTDom{|x_{\omega}-y_{\omega}|}{(\R^d \times \Sigma_T)^2}{\hat{\Beta}_{\Bmu,\Bnu}^{\omega}(x_{\omega},\sigma_{\mu_{\omega}},y_{\omega},\sigma_{\nu_{\omega}})} \exp \bigg( \INTSeg{L_{R_r}(s)}{s}{0}{t} \bigg) \\
& \hspace{0.45cm} + \INTDom{\bigg( \INTSeg{\NormL{\vb(s,\Bmu(s),\omega) - \wb(s,\omega)}{\infty}{\R^d,\R^d ; \, \nu_{\omega}(s)}}{s}{0}{t} \bigg)}{(\R^d \times \Sigma_T)^2}{\hat{\Beta}_{\Bmu,\Bnu}^{\omega}(x_{\omega},\sigma_{\mu_{\omega}},y_{\omega},\sigma_{\nu_{\omega}})} \\
& \hspace{11.2cm} \times \exp \bigg( \INTSeg{L_{R_r}(s)}{s}{0}{t} \bigg) \\
& = \bigg( W_1(\mu^0_{\omega},\nu^0_{\omega}) + \INTSeg{\NormL{\vb(s,\Bmu(s),\omega) - \wb(s,\omega)}{\infty}{\R^d,\R^d ; \, \nu_{\omega}(s)}}{s}{0}{t} \bigg) \exp \bigg( \INTSeg{L_{R_r}(s)}{s}{0}{t} \bigg)
\end{aligned}
\end{equation*}
for all times $t \in [0,T]$ and $\pi$-almost every $\omega \in \Omega$, where we used the fact that $\gamma_{\omega}^0 \in \Gamma_o(\mu^0_{\omega},\nu^0_{\omega})$. Hence, we may conclude that
\begin{equation*}
\begin{aligned}
& W_{\pi,1}(\Bmu(t),\Bnu(t)) \\
& = \INTDom{W_1(\mu_{\omega}(t),\nu_{\omega}(t))}{\Omega}{\pi(\omega)} \\
& \leq \bigg( \INTDom{W_1(\mu_{\omega}^0,\nu_{\omega}^0)}{\Omega}{\pi(\omega)} + \INTSeg{\INTDom{\NormL{\vb(s,\Bmu(s),\omega) - \wb(s,\omega)}{\infty}{\R^d,\R^d ; \, \nu_{\omega}(s)}}{\Omega}{\pi(\omega)}}{s}{0}{t} \bigg) \exp \big( \|L_{R_r}(\cdot)\|_1 \big) \\
& = \bigg( W_{\pi,1}(\Bmu^0,\Bnu^0) + \INTSeg{\INTDom{\NormL{\vb(s,\Bmu(s),\omega) - \wb(s,\omega)}{\infty}{\R^d,\R^d ; \, \nu_{\omega}(s)}}{\Omega}{\pi(\omega)}}{s}{0}{t} \bigg) \exp \big( \|L_{R_r}(\cdot)\|_1 \big)
\end{aligned}
\end{equation*}
for all times $t \in [0,T]$, up to an ultimate application of Fubini's theorem.
\end{proof}

\begin{rmk}[On the necessity of working with fibred Wasserstein metrics to get stability estimates]
\label{rmk:Stability}
Here, we exemplify why the stability estimate of Proposition \ref{prop:Stability} formulated in terms of fibred Wasserstein distance cannot hold with the classical one, due to the lack of regularity of the vector field in the variable $\omega \in \Omega$. To this end, set $d=1$, take some $\varepsilon\in (0,1)$ and suppose $\Omega := [0,1]$ is endowed with the measure $\pi:=\tfrac{1}{2} \big(\delta_0+\delta_\varepsilon\big) \in \Pcal([0,1])$. Then, consider the two initial data given by 
\begin{equation*}
\Bmu^0 := \tfrac{1}{2} \big(\delta_{(0,0)} + \delta_{(\varepsilon,1)} \big) \qquad \text{and} \qquad \Bnu^0 := \tfrac{1}{2} \big( \delta_{(0,1)} + \delta_{(\varepsilon,0)} \big),
\end{equation*}
and let   
\begin{equation*}
\vb(t,\Bmu,\omega,x) := \left\{
\begin{aligned}
& 0 ~~ & \text{if~ $\omega = 0$},\\
& 1 ~~& \text{if~ $\omega = \epsilon$}.
\end{aligned}
\right.
\end{equation*}
Notice that $\vb : [0,T] \times \Pcal_{\pi,1}(\Omega \times \R) \times \Omega \times \R$ clearly satisfies Hypotheses \ref{hyp:CL}, and that it is even globally Lipschitz in its measure variable with respect to the classical Wasserstein metric. The unique solutions of \eqref{eq:TransportGraphNonlocal} from $\Bmu^0,\Bnu^0 \in \Pcal_{\pi,1}([0,1] \times \R)$ may then be computed explicitly as 
\begin{equation*}
\Bmu(t) := \tfrac{1}{2} \big(\delta_{(0,0)} + \delta_{(\varepsilon,1+t)} \big) \qquad \text{and} \qquad \Bnu(t) := \tfrac{1}{2} \big( \delta_{(0,1)} + \delta_{(\varepsilon,t)} \big).
\end{equation*}
for all times $t \in [0,T]$. Now, as seen in Remark \ref{rmk:closeness}, 
it holds that 
\begin{equation*}
W_{\pi,1}(\Bmu(t),\Bnu(t)) = 1 = W_{\pi,1}(\Bmu^0,\Bnu^0),    
\end{equation*}
which is compatible with the expected stability estimate \eqref{eq:DistanceEst}. On the other hand, recalling that $\Omega \times \R$ is endowed with the product metric, observe that
\begin{equation*}
W_{1}(\Bmu(t),\Bnu(t)) = \min \big\{ 1,\epsilon + t \big\} >  \varepsilon = W_{1}(\Bmu^0,\Bnu^0)    
\end{equation*}
for all $t \in (0,T]$. This shows that there is no hope to get stability estimates for \eqref{eq:TransportGraphNonlocal} in terms of the classical Wasserstein distance without imposing more regularity in $\omega \in \Omega$ to the vector field. 
\end{rmk}

We are now ready to state a general well-posedness result for \eqref{eq:TransportGraphNonlocal} under Hypotheses \ref{hyp:CL}. 

\begin{thm}[Existence, uniqueness and representation formula]
\label{thm:CauchyLipschitz}
Suppose that $\Bmu^0 \in \Pcal_{\pi}(\Omega \times B(0,r))$ for some $r>0$ and assume that $\vb : [0,T] \times \Pcal_{\pi,1}(\Omega \times \R^d) \times \Omega \times \R^d \to \R^d$ satisfies Hypotheses \ref{hyp:CL}. 

Then, the Cauchy problem \eqref{eq:TransportGraphNonlocal} admits a unique solution $\Bmu(\cdot) \in \AC([0,T],\Pcal_{\pi,1}(\Omega \times \R^d))$ given by the formula
\begin{equation}
\label{eq:FlowRepresentation}
\Bmu(t) = \BPhi_{(0,t)}[\Bmu^0]_{\sharp} \Bmu^0
\end{equation}
for all $t \in [0,T]$. Therein, the maps $(\BPhi_{(0,t)}[\Bmu^0])_{t \in [0,T]} \subset L^{\infty}(\Omega,C^0(B(0,r), \Omega \times \R^d);\pi)$ have the form
\begin{equation*}
\BPhi_{(0,t)} [\Bmu^0](\omega,x) = \Big( \omega , \Phi_{(0,t)}^{\omega}[\Bmu^0](x) \Big)
\end{equation*}
for $\pi$-almost every $\omega \in \Omega$ and all $(t,x) \in [0,T] \times \R^d$, where $(\Phi_{(0,t)}^{\omega}[\Bmu^0])_{t \in [0,T]} \subset C^0(B(0,r),\R^d)$ stand for the unique classical flow solutions of the Cauchy problems 
\begin{equation}
\label{eq:NonlocalFlowDef}
\Phi_{(0,t)}^{\omega}[\Bmu^0](x) = x + \INTSeg{\vb \Big( s, \BPhi_{(0,s)}[\Bmu^0]_{\sharp} \Bmu^0, \omega , \Phi_{(0,s)}^{\omega}[\Bmu^0](x) \Big)}{s}{0}{t}
\end{equation}
for all $(t,x) \in [0,T] \times \R^d$ and $\pi$-almost every $\omega \in \Omega$.
\end{thm}

\begin{proof}
We start by establishing the uniqueness of solutions to \eqref{eq:TransportGraphNonlocal}. To this end, one simply needs to apply Proposition \ref{prop:Stability} to two candidate solutions $\Bmu(\cdot),\Bnu(\cdot) \in \AC([0,T],\Pcal_{\pi,1}(\Omega \times \R^d))$ of the Cauchy problem \eqref{eq:TransportGraphNonlocal} with the same initial condition $\Bmu^0 \in \Pcal_{\pi}(\Omega \times B(0,r))$. In that case, it holds that 
\begin{equation*}
\begin{aligned}
& W_{\pi,1}(\Bmu(t),\Bnu(t)) \leq \bigg( \INTSeg{\INTDom{\NormL{\vb(s,\Bmu(s),\omega) - \vb(s,\Bnu(s),\omega)}{\infty}{\R^d,\R^d ; \, \nu_{\omega}(s)}}{\Omega}{\pi(\omega)}}{s}{0}{t} \bigg) \\
& \hspace{10cm} \times \exp \big( \|L_{R_r}(\cdot)\|_1 \big)
\end{aligned}
\end{equation*}
for all times $t \in [0,T]$, with $R_r >0$ being given as in Lemma \ref{lem:NonlocalBounds}. In particular, it then follows from Hypothesis \ref{hyp:CL}-$(iii)$ that 
\begin{equation*}
W_{\pi,1}(\Bmu(t),\Bnu(t)) \leq \bigg( \INTSeg{L_{R_r}(s) W_{\pi,1}(\Bmu(s),\Bnu(s))}{s}{0}{t} \bigg) \exp \big( \|L_{R_r}(\cdot)\|_1 \big),
\end{equation*}
from whence we can conclude by a simple application of Gr\"onwall's lemma.

To prove the existence of solutions admitting a flow representation, we adapt an argument proposed in \cite[Section 2.4]{PMPWass} to show that the operator $\BLambda : \Xpazo \to \Xpazo$ defined over the Banach space $\Xpazo := L^{\infty}(\Omega,C^0([0,T] \times B(0,r),\R^d);\pi)$ by 
\begin{equation}
\label{eq:LambdaDef}
\BLambda(\Phi)(t,\omega,x) := x + \INTSeg{\vb \bigg( s , \INTDom{\Phi(s,\theta)_{\sharp} \mu_{\theta}^0}{\Omega}{\pi(\theta)} , \omega , \Phi(s,\omega,x) \bigg)}{s}{0}{t}
\end{equation}
admits a fixed point. To do so, owing to the localisation of the Lipschitz constants of the dynamics posited in Hypotheses \ref{hyp:CL}, we first need to consider a truncated version of the operator. Given an arbitrary radius $R >0$, we thus define
\begin{equation*}
\BLambda_R := \pfrak_{B(0,R)} \circ \BLambda
\end{equation*}
which maps the closed subset $\Xpazo_R := L^{\infty}(\Omega,C^0([0,T] \times B(0,r),B(0,R));\pi)$ of $\Xpazo$ into itself. Note that the latter can be made into a complete metric space upon being endowed with the distance
\begin{equation*}
\dsf_{\Xpazo_R}(\Phi_1,\Phi_2) ~:= \sup_{(t,\omega,x) \in [0,T] \times \Omega \times B(0,r)} \exp \bigg( \hspace{-0.1cm} -4 \INTSeg{L_{R}(s)}{s}{0}{t} \bigg) \big| \Phi_1(t,\omega,x) - \Phi_2(t,\omega,x) \big|,
\end{equation*}
where the supremum with respect to the variable $\omega \in \Omega$ is taken to be $\pi$-essential. At this stage, it follows from a series of elementary computations outlined e.g. in the proof of \cite[Theorem 2.1.1]{BressanPiccoli} (see also \cite[Appendix A]{SemiSensitivity}) combined with Hypotheses \ref{hyp:CL}-$(iii)$ that 
\begin{equation*}
\dsf_{\Xpazo_R} \Big( \BLambda_R(\Phi_1),\BLambda_R(\Phi_2) \Big) \leq \tfrac{1}{2} \dsf_{\Xpazo_R}(\Phi_1,\Phi_2)
\end{equation*}
for all $\Phi_1,\Phi_2 \in \Xpazo_R$. Thence, for each $R>0$, Banach's fixed point theorem (see e.g. \cite[Theorem 5.7]{Brezis}) yields the existence of a unique element $\Phi_R \in \Xpazo_R$ such that $\BLambda_R(\Phi_R) = \Phi_R$. 

Our goal now is to prove that, upon choosing $R>0$ to be sufficiently large, the map $\Phi_R \in \Xpazo_R$ is also a fixed point of the untruncated operator $ \BLambda : \Xpazo \to \Xpazo$. To this end, note first that by Hypothesis \ref{hyp:CL}-$(ii)$ combined with the sublinearity of the Euclidean projection, it holds for any $R > 0$ that
\begin{equation}
\label{eq:FlowREst}
\begin{aligned}
\INTDom{\big| \Phi_R(t,\omega,x) \big|}{\Omega \times \R^d}{\Bmu^0(\omega,x)} & = \INTDom{\big| \BLambda_R(\Phi_R)(t,\omega,x) \big|}{\Omega \times \R^d}{\Bmu^0(\omega,x)} \\
& \leq \INTDom{\big| \BLambda(\Phi_R)(t,\omega,x) \big|}{\Omega \times \R^d}{\Bmu^0(\omega,x)} \\
& \leq r \,+ \INTSeg{m(s) \bigg( 1 + 2\INTDom{|\Phi_R(s,\omega,x)|}{\Omega \times \R^d}{\Bmu^0(\omega,x)} \bigg)}{s}{0}{t}.
\end{aligned}
\end{equation}
%
%
Hence, by applying Gr\"onwall's lemma in \eqref{eq:FlowREst}, there exists some constant $C_T > 0$ such that 
\begin{equation}
\label{eq:SublinRFlow}
\bigg| \vb \bigg( t , \INTDom{\Phi_R(t,\theta)_{\sharp} \mu_{\theta}^0}{\Omega}{\pi(\theta)} , \omega , \Phi_R(t,\omega,x) \bigg) \bigg| \leq (1+C_T) m(t) \Big( 1 + \big| \Phi_R(t,\omega,x) \big| \Big)
\end{equation}
for $\Lcal^1 \times \pi$-almost every $(t,\omega) \in [0,T] \times \Omega$ and all $x \in B(0,r)$. In particular, we may deduce from \eqref{eq:SublinRFlow} together with Hypothesis \ref{hyp:CL}-$(ii)$ and yet another application of Gr\"onwall's lemma that
\begin{equation*}
\big| \Phi_R(t,\omega,x) \big| \leq \bigg( r + (1+C_T) \INTSeg{m(s)}{s}{0}{t} \bigg) \exp \bigg( (1+C_T) \INTSeg{m(s)}{s}{0}{t} \bigg)
\end{equation*}
for all $(t,x) \in [0,T] \times B(0,r)$ and $\pi$-almost every $\omega \in \Omega$. Since this bound is independent of $R>0$, there exists some $\bar{R} > 0$ such that for all $R \geq \bar{R}$, it holds
\begin{equation*}
\Phi_R(t,\omega,x) = \Phi_{(0,t)}^{\omega}[\Bmu^0](x) \in \Xpazo
\end{equation*}
for all $(t,x) \in [0,T] \times B(0,r)$ and $\pi$-almost every $\omega \in \Omega$, with $\Phi[\Bmu^0] = \BLambda_{R}(\Phi[\Bmu^0]) = \BLambda(\Phi[\Bmu^0])$. To conclude, there remains to check that the curve of measures defined by 
\begin{equation*}
\Bmu(t) := \BPhi_{(0,t)}[\Bmu^0]_{\sharp} \Bmu^0 = \INTDom{\Phi_{(0,t)}^{\omega}[\Bmu^0]_{\sharp} \mu^0_{\omega}}{\Omega}{\pi(\omega)}
\end{equation*}
for all times $t \in [0,T]$ is an element of $\AC([0,T],\Pcal_{\pi,1}(\Omega \times \R^d))$ and that it solves \eqref{eq:TransportGraphNonlocal}, which can be done in an elementary way following e.g. Theorem \ref{thm:Equivalence} above. 
\end{proof}


\subsection{Existence of solutions in the Carathéodory-Peano framework}
\label{subsection:Peano}

In this second section, we depart from the Cauchy-Lipschitz setting and prove a general existence result ``à la Peano'' for solutions of \eqref{eq:TransportGraphNonlocal}, under the following Carathéodory-type regularity assumptions.  

\begin{taggedhyp}{\textbn{(P)}} \hfill
\label{hyp:P}
\begin{enumerate}
\item[$(i)$] The mapping $(t,\omega) \in [0,T] \times \Omega \mapsto \vb(t,\Bmu,\omega,x) \in \R^d$ is $\Lcal^1\times \pi$-measurable for all $(\Bmu,x) \in \Pcal_{\pi,1}(\Omega \times \R^d) \times \R^d$.
\item[$(ii)$] There exists a map $m(\cdot) \in L^1([0,T],\R_+)$ such that for each $\Bmu \in \Pcal_{\pi,1}(\Omega \times \R^d)$, it holds that 
\begin{equation*}
|\vb(t,\Bmu,\omega,x)| \leq m(t) \bigg( 1 + |x| + \Mpazo_{\pi,1}(\Bmu) \bigg)
\end{equation*}
for $\Lcal^1 \times \pi$-almost every $(t,\omega) \in [0,T] \times \Omega$ and all $x \in \R^d$. 
\item[$(iii)$] The mapping $(\Bmu,x) \in \Pcal_{\pi,1}(\Omega \times \R^d) \times \R^d \mapsto \vb(t,\Bmu,\omega,x) \in \R^d$ is continuous in the $W_1 \times |\cdot|$-metric for $\Lcal^1 \times \pi$-almost every $(t,\omega) \in [0,T] \times \Omega$.  
\end{enumerate}
\end{taggedhyp}

Before moving on, notice that Hypotheses~\ref{hyp:P}-$(i)$, $(ii)$ are the same as Hypotheses~\ref{hyp:CL}-$(i)$, $(ii)$, so the only difference between the two settings is the regularity we require in the measure and space variable of the velocity field in Hypothesis \ref{hyp:P}-$(iii)$. This assumption will be properly put into perspective in Remark \ref{rmk:PeanoAssumption} below. Nevertheless, one should note that following e.g. Proposition \ref{prop:ClassicalCompactness} and Theorem \ref{thm:ComparisonWass}, the latter encompasses familiar convolution-type fields of the form
\begin{equation*}
\vb(t,\Bmu,\omega,x) = \INTDom{w(t,\omega,\theta) \Psi(x,y)}{\Omega \times \R^d}{\Bmu(\theta,y)}
\end{equation*}
for some $w(\cdot) \in L^{\infty}([0,T] \times \Omega \times \Omega,\R_+;\Lcal^1 \times \pi \times \pi)$ and $\Psi \in C^0_1(\R^{2d},\R^d)$. However, it should be stressed that a vector field exhibiting a nonlinear dependence on the measures $\{\mu_{\omega}\}_{\omega \in \Omega} \subset \Pcal_1(\R^d)$ will typically fail to satisfy Hypothesis \ref{hyp:P}-$(iii)$, even though it may be continuous for the $W_{\pi,1}$-topology, as illustrated below.

\begin{rmk}[Examples of functions violating Hypothesis \ref{hyp:P}-$(iii)$]
\label{rmk:WeakContinuity}
Let $\Omega := [0,1]$ be endowed with $\pi := \Lcal^1_{\llcorner [0,1]}$ and $(r_n) \subset L^{\infty}([0,1],\R)$ be the so-called Rademacher sequence, defined by 
\begin{equation*}
r_n(\omega) := \sign(\sin(2^n \omega)) \in \{-1,1\}
\end{equation*} 
for every $\omega\in \Omega$ and each $n \geq 1$. Then, it is well-known that the sequence $(\Bmu_n) \subset \Pcal_{\pi,1}([0,1] \times \R)$ with general term $\Bmu_n := (\Id,r_n)_{\sharp} \Lcal^1_{\llcorner [0,1]}$ converges narrowly towards $\Bmu := \tfrac{1}{2} \Lcal^1_{\llcorner[0,1]} \times (\delta_{-1} + \delta_1)$. Thus, for every map $\phi \in C^0_b([-1,1],\R)$, it may be easily checked that 
\begin{equation*}
\begin{aligned}
\INTDom{ \phi\bigg( \INTDom{x}{\R}{\mu_{\omega}^n(x)} \bigg)}{\Omega}{\pi(\omega)} ~\underset{n \to +\infty}{\longrightarrow}~ \tfrac{1}{2}( \phi(-1) + \phi(1)) \quad \text{whereas} \quad \INTDom{ \phi\bigg( \INTDom{x}{\R}{\mu_{\omega}(x)} \bigg)}{\Omega}{\pi(\omega)} = \phi(0). 
\end{aligned}
\end{equation*}
This shows that the mapping $\Bmu \in \Pcal_{\pi,1}([0,1] \times \R) \mapsto \INTDom{\phi(\INTDom{x}{\R}{\mu_{\omega}(x)})}{\Omega}{\pi(\omega)} \in \R$ is generally not continuous for the $W_1$-topology. On the other hand, it may be shown by a standard argument relying on the Lebesgue-Vitali dominated convergence theorem (see e.g. \cite[Theorem 4.5.4]{Bogachev}) that it is continuous over $\Pcal_{\pi,1}([0,1] \times \R)$ for the $W_{\pi,1}$-topology. 
\end{rmk}

We are now ready to prove our general existence result for solutions of \eqref{eq:TransportGraphNonlocal}, under Hypotheses \ref{hyp:P}. 

\begin{thm}[Existence of solutions to nonlocal structured continuity equations]
\label{thm:Existence}
Suppose that $\pi \in \Pcal_1(\Omega)$, fix some $\Bmu^0 \in \Pcal_{\pi,1}(\Omega \times \R^d)$ and assume that $\vb : [0,T] \times \Pcal_{\pi,1}(\Omega \times \R^d) \times \Omega \times \R^d \to \R^d$ satisfies Hypotheses \ref{hyp:P}. Then, the Cauchy problem \eqref{eq:TransportGraphNonlocal} admits a solution $\Bmu(\cdot) \in \AC([0,T],\Pcal_{\pi,1}(\Omega \times \R^d))$. 
%
\end{thm}

Before moving to the proof of this result, we state an elementary Gr\"onwall-type estimate that will prove useful in the sequel. 

\begin{lem}[A shifted Gr\"onwall estimate]
\label{lem:DelayedGronwall}
Let $\beta > 0$ and $f(\cdot) \in C^0([-\beta,T],\R_+)$ be a map such that
\begin{equation*}
f(t) \leq \alpha(t) \bigg( f(0) + \INTSeg{m(s) f(s-\beta)}{s}{0}{t} \bigg) 
\end{equation*}
for all times $t \in [0,T]$, where $m(\cdot) \in L^1([0,T],\R_+)$ and $\alpha(\cdot) \in L^{\infty}([0,T],\R_+)$. Then 
\begin{equation*}
f(t) \, \leq ~ \Norm{\alpha(\cdot)}_{\infty} \bigg( f(0) + \sup_{s \in [-\beta,0]} f(s) \Norm{m(\cdot)}_1 \bigg) \exp \Big( \Norm{\alpha(\cdot)}_{\infty} \Norm{m(\cdot)}_1 \Big)
\end{equation*}
for all times $t \in [0,T]$. 
\end{lem}

\begin{proof}
This result follows directly from Gr\"onwall's lemma applied to the map $h(t) := \sup_{s \in [0,t]} f(s)$. Indeed, it can be verified that 
\begin{equation*}
\begin{aligned}
h(t) & \leq \sup_{s \in [0,t]} \alpha(s) \bigg( f(0) + \INTSeg{m(\sigma) f(\sigma-\beta)}{\sigma}{0}{s} \bigg) \\
& \leq~ \Norm{\alpha(\cdot)}_{\infty} \bigg( f(0) + \sup_{s \in [-\beta,0]} f(s) \|m(\cdot)\|_1 + \INTSeg{m(s) h(s)}{s}{0}{t} \bigg),
\end{aligned}
\end{equation*}
and the conclusion directly follows.
\end{proof}

\begin{proof}[Proof of Theorem \ref{thm:Existence}]
To establish the existence of solutions to \eqref{eq:TransportGraphNonlocal} under Hypotheses \ref{hyp:P}, we leverage a compactness argument inspired by the explicit Euler scheme developed in \cite[Section 3]{ContIncPp}. 

\paragraph*{Step 1 -- Explicit Euler scheme and moment bounds.}

Given some $n \geq 1$, we introduce the notation $\tau_n^k := kT/n$ for each $k \in \{0,\dots,n\}$ and define inductively the curves of measures $\Bmu_n^{k+1}(\cdot) \in \AC([\tau_n^k,\tau_n^{k+1}],\Pcal_{\pi,1}(\Omega \times \R^d))$ as being one of the solutions of the Cauchy problem
\begin{equation}
\label{eq:CauchyInduction}
\left\{
\begin{aligned}
& \partial_t \Bmu_n^{k+1}(t) + \Div_x \Big( \vb(t,\Bmu_n^k(t-\tfrac{T}{n})) \Bmu_n^{k+1}(t) \Big) = 0, \\
& \Bmu_n^{k+1}(\tau_n^k) = \Bmu_n^k(\tau_n^k), 
\end{aligned}
\right.
\end{equation}
for each $k \in \{0,\dots,n-1\}$, with the convention that $\Bmu_n^0(t) = \Bmu^0$ for each $t \in [-\tfrac{T}{n},0]$. To see why the scheme is well-defined, observe that for each $k \in \{0,\dots,n-1\}$, the (local) structured vector fields 
\begin{equation*}
(t,\omega,x) \in [\tau_n^k,\tau_n^{k+1}] \times \Omega \times \R^d \mapsto \vb \Big( t , \Bmu_n^k \big(t - \tfrac{T}{n} \big) , \omega , x \Big) \in \R^d
\end{equation*}
are Carathéodory and satisfy the assumptions of Theorem \ref{thm:PeanoLocal}, so that the Cauchy problems formulated in \eqref{eq:CauchyInduction} all admit at least one solution. Then, by the gluing result e.g. of \cite[Lemma 4.4]{Dolbeault2009}, one can readily check that the concatenated curve $\Bmu_n(\cdot) \in \AC([0,T],\Pcal_{\pi,1}(\Omega \times \R^d))$ given by $\Bmu_n(t) := \Bmu_n^{k+1}(t)$ when $t \in [\tau_n^k,\tau_n^{k+1}]$ is a solution of \eqref{eq:TransportGraph} driven by the Carathéodory field
\begin{equation}
\label{eq:ExistenceVelocity}
\vb_n(t,\omega,x) := \sum_{k=0}^{n-1} \mathds{1}_{[\tau_n^k,\tau_n^{k+1}]}(t) \,  \vb \Big(t,\Bmu_n^k(t - \tfrac{T}{n}),\omega,x \Big) =  \vb \Big(t,\Bmu_n(t - \tfrac{T}{n}),\omega,x \Big)
\end{equation}
defined for $\Lcal^1 \times \pi$-almost every $(t,\omega) \in [0,T] \times \Omega$ and all $x \in \R^d$. By Hypothesis \ref{hyp:P}-$(ii)$, it further holds that
\begin{equation*}
|\vb_n \big(t,\omega,x)| \leq m(t) \Big( 1 + |x| + \Mpazo_{\pi,1} \big( \Bmu_n(t-\tfrac{T}{n}) \big) \Big)
\end{equation*}
for all times $t \in [0,T]$ and each $k \in \{0,\dots,n-1\}$, which upon reproducing the computations in the proof of Proposition \ref{prop:Bounds} above leads to the following moment inequality
\begin{equation*}
\Mpazo_{\pi,1}(\Bmu_n(t)) \leq \bigg( \Mpazo_{\pi,1}(\Bmu_n^0) + \INTSeg{m(s) \Big( 1+\Mpazo_{\pi,1}(\Bmu_n \big( s - \tfrac{T}{n}) \big) \Big)}{s}{0}{t} \bigg) \exp \bigg( \INTSeg{m(s)}{s}{0}{t} \bigg) 
\end{equation*}
which holds for all times $t \in [0,T]$ and each $n \geq 1$. Then, by Lemma \ref{lem:DelayedGronwall}, there exists a constant $\Cpazo_T > 0$ depending only on the magnitudes of $\Mpazo_{\pi,1}(\Bmu^0)$ and $\Norm{m(\cdot)}_1$ such that 
\begin{equation}
\label{eq:ExistenceMomentBound}
\sup_{n \geq 1} \Mpazo_{\pi,1}(\Bmu_n(t)) \leq \Cpazo_T
\end{equation}
for all times $t \in [0,T]$.

\paragraph*{Step 2 -- Compactness of trajectories and properties of the limit curve.} 

In this second step, we show that the sequence of curves built in Step 1 is relatively compact in $C^0([0,T],\Pcal_{\pi,1}(\Omega \times \R^d))$ for the $W_1$-topology. To this end, notice that due to Hypothesis \ref{hyp:P}-$(ii)$ combined with the moment bound derived in \eqref{eq:ExistenceMomentBound}, the nonlocal velocity field $\vb : [0,T] \times \Pcal_{\pi}(\Omega \times \R^d) \times \Omega \times \R^d \to \R^d$ is such that
\begin{equation}
\label{eq:ExistenceSublin}
\Big|\vb \Big(t,\Bmu_n(t-\tfrac{T}{n}),\omega,x \Big)\Big| \leq (1+\Cpazo_T) m(t) (1+|x|)
\end{equation}
for $\Lcal^1 \times \pi$-almost every $(t,\omega) \in [0,T] \times \Omega$, all $x \in \R^d$ and each $n \geq 1$. Thence, by Theorem \ref{thm:CompactnessReach}, there exists a compact set $\Kcal \subset \Pcal_{\pi,1}(\Omega \times \R^d)$ for the $W_1$-topology such that
\begin{equation}
\label{eq:CompactInc}
\Bmu_n(t) \in \Kcal
\end{equation}
for all times $t \in [0,T]$. Similarly, it follows from \eqref{eq:ExistenceSublin} combined with the regularity estimate \eqref{eq:ACEst} of Proposition \ref{prop:Bounds} and the distance inequality of Lemma \ref{lem:MetricEstimate} that 
\begin{equation}
\label{eq:ExistenceACEst}
W_1(\Bmu_n(t_1),\Bmu_n(t_2)) \leq W_{\pi,1}(\Bmu_n(t_1),\Bmu_n(t_2)) \leq c_T \INTSeg{m(s)}{s}{t_1}{t_2}
\end{equation}
for all times $0 \leq t_1 \leq t_2 \leq T$ and some $c_T >0$ depending only on the magnitudes of $\Mpazo_{\pi,1}(\Bmu^0)$ and $\Norm{m(\cdot)}_1$. Therefore, observing that $(\Pcal_{\pi,1}(\Omega \times \R^d),W_1(\cdot,\cdot))$ is a Polish space since it is a closed subset of $(\Pcal_1(\Omega \times \R^d),W_1(\cdot,\cdot))$, it follows from the Ascoli-Arzel\`a compactness criterion (see e.g. \cite[Chapter 7 -- Theorem 18]{Kelley1975}) that
\begin{equation}
\label{eq:ExistenceUniformConv}
\sup_{t \in [0,T]} W_1(\Bmu_n(t),\Bmu(t)) ~\underset{n \to +\infty}{\longrightarrow}~ 0
\end{equation}
for some limit curve $\Bmu : [0,T] \to \Pcal_{\pi,1}(\Omega \times \R^d)$, along a subsequence that we do not relabel. Moreover, it can be easily checked by using the lower-semicontinuity estimates of \eqref{eq:LscFibred} above that 
\begin{equation*}
\Mpazo_{\pi,1}(\Bmu(t)) \leq \Cpazo_T \qquad \text{and} \qquad W_{\pi,1}(\Bmu(t_1),\Bmu(t_2)) \leq c_T \INTSeg{m(s)}{s}{t_1}{t_2}
\end{equation*}
for all times $t \in [0,T]$ and $0 \leq t_1 \leq t_2 \leq T$. 

\paragraph*{Step 3 -- Dynamics of the limit curve.} 

In this third and last step, we show that the curve $\Bmu(\cdot) \in \AC([0,T],\Pcal_{\pi,1}(\Omega \times \R^d))$ is a solution of the Cauchy problem \eqref{eq:TransportGraphNonlocal}. First, recall that by construction, the curves $(\Bmu_n(\cdot)) \subset \AC([0,T],\Pcal_{\pi,1}(\Omega \times \R^d))$ are such that 
\begin{equation}
\label{eq:ExistenceCauchyProbSeq}
\INTSeg{\INTDom{\bigg( \partial_t \phi(t,\omega,x) + \Big\langle \nabla_x \phi(t,\omega,x) \, , \, \vb \Big( t,\Bmu_n \big( t-\tfrac{T}{n} \big) ,\omega,x \Big) \Big\rangle \bigg)}{\Omega \times \R^d}{\Bmu_n(t)(\omega,x)}}{t}{0}{T} = 0
\end{equation}
for each test function $\phi \in \Adm((0,T) \times \Omega \times \R^d,\R)$. Observing now that due to \eqref{eq:ExistenceSublin}, the integrands
\begin{equation*}
(\omega,x) \in \Omega \times \R^d \mapsto \bigg( \partial_t \phi(t,\omega,x) + \Big\langle \nabla_x \phi(t,\omega,x) \, , \, \vb \Big( t,\Bmu_n \big( t-\tfrac{T}{n} \big) ,\omega,x \Big) \Big\rangle \bigg) \in \R
\end{equation*}
belong to $C^0_{\pi,b}(\Omega \times \R^d,\R)$ and are bounded for all times $t \in [0,T]$, it follows from \eqref{eq:ExistenceUniformConv} combined with the characterisation of narrow convergence over $\Pcal_{\pi,1}(\Omega \times \R^d)$ recalled in Remark \ref{rmk:NarrowFibred} and Lebesgue's dominated convergence theorem that
\begin{equation}
\label{eq:ExistenceConv1}
\INTSeg{\INTDom{\Big( \partial_t \phi(t,\omega,x) + \big\langle \nabla_x \phi(t,\omega,x) \, , \, \vb \big( t,\Bmu(t) ,\omega,x \big) \big\rangle \Big)}{\Omega \times \R^d}{\big(  \Bmu(t) - \Bmu_n(t) \big)(\omega,x)}}{t}{0}{T} ~\underset{n \to +\infty}{\longrightarrow}~ 0.
\end{equation}
At this stage, by merging \eqref{eq:ExistenceACEst} with \eqref{eq:ExistenceUniformConv}, it can also be inferred that
\begin{equation*}
\sup_{t \in [0,T]} W_1 \Big( \Bmu_n(t-\tfrac{T}{n}),\Bmu(t) \Big) ~\underset{n \to +\infty}{\longrightarrow}~ 0,
\end{equation*}
along the same subsequence. Observe then that $(\Bmu,x) \in \Kcal \times B(0,R) \mapsto \vb(t,\Bmu,\omega,x) \in \R^d$ is continuous, hence uniformly continuous for $\Lcal^1 \times \pi$-almost every $(t,\omega) \in [0,T] \times \Omega$ and each $R > 0$ as a consequence of Hypothesis \ref{hyp:P}-$(iii)$, which implies that the mapping $\Bmu \in \Kcal \mapsto \vb(t,\Bmu,\omega) \in C^0(B(0,R),\R^d)$ is continuous as well. By combining this fact together with \eqref{eq:CompactInc} and Hypothesis \ref{hyp:P}-$(ii)$, we may apply  Lebesgue's dominated convergence theorem to obtain that 
\begin{equation*}
\INTSeg{\INTDom{\NormC{\, \vb(t,\Bmu(t),\omega) - \vb \Big(t,\Bmu_n\big(t - \tfrac{T}{n} \big), \omega \Big)}{0}{B(0,R),\R^d}}{\Omega}{\pi(\omega)}}{t}{0}{T} ~\underset{n \to +\infty}{\longrightarrow}~ 0 
\end{equation*}
for each $R>0$. In particular, by considering test functions of the form $\phi(t,\omega,x) := \zeta(t) \xi(\omega) \psi(x)$ for some $\zeta \in C^1_c((0,T),\R)$, $\xi \in C^0_b(\Omega,\R)$ and $\psi \in C^1_c(\R^d,\R)$ while letting $R_{\psi} > 0$ be such that $\supp(\psi) \subset B(0,R_{\psi})$, it necessarily holds that 
\begin{equation}
\label{eq:ExistenceConv2}
\begin{aligned}
& \, \bigg| \INTSeg{\INTDom{\Big\langle \nabla_x \phi(t,\omega,x) \, , \, \vb\big( t,\Bmu(t),\omega,x \big) - \vb \Big( t,\Bmu_n \big( t-\tfrac{T}{n} \big) ,\omega,x \Big) \Big\rangle}{\Omega \times \R^d}{\Bmu_n(t)(\omega,x)}}{t}{0}{T} \, \bigg| \\
& \hspace{1cm} \leq \INTSeg{\INTDom{ \Big| \, \zeta(t) \xi(\omega) \, \Big\langle \nabla \psi(x) \, , \, \vb\big( t,\Bmu(t),\omega,x \big) - \vb \Big( t,\Bmu_n \big( t-\tfrac{T}{n} \big) ,\omega,x \Big) \Big\rangle \Big|}{\Omega \times \R^d}{\Bmu_n(t)(\omega,x)}}{t}{0}{T} \\
& \hspace{1cm} \leq \, \NormC{\zeta}{0}{[0,T],\R} \NormC{\xi}{0}{\Omega,\R} \NormC{\nabla \psi}{0}{B(0,R_{\psi}),\R^d} \\
& \hspace{1.75cm} \times \INTSeg{\INTDom{\NormC{\, \vb(t,\Bmu(t),\omega) - \vb \Big(t,\Bmu_n\big(t - \tfrac{T}{n} \big), \omega \Big)}{0}{B(0,R_{\psi}),\R^d}}{\Omega}{\pi(\omega)}}{t}{0}{T} ~\underset{n \to +\infty}{\longrightarrow}~ 0.
\end{aligned}
\end{equation}
Thence, by plugging the convergence results of \eqref{eq:ExistenceConv1} and \eqref{eq:ExistenceConv2} in \eqref{eq:ExistenceCauchyProbSeq} while noting that $\Bmu(0) = \Bmu^0$ as a consequence of \eqref{eq:ExistenceUniformConv}, we finally obtain, up to a standard density argument in the spirit of Theorem \ref{thm:Equivalence}, that $\Bmu(\cdot) \in \AC([0,T],\Pcal_{\pi,1}(\Omega \times \R^d))$ is a solution of \eqref{eq:TransportGraphNonlocal}. 
\end{proof}

\begin{rmk}[Concerning our regularity assumptions for Peano existence]
\label{rmk:PeanoAssumption}
In Remark \ref{rmk:CompactnessOrbits} above, we alluded to the fact that the reachable sets of ODEs in infinite dimensional spaces are generally not relatively strongly compact, which precludes any naive usage of the classical Peano method in this context. There are many standard workarounds to this issue, involving e.g. measures of noncompactness or stronger notions of continuity, a fairly exhaustive list of which can be found e.g. in \cite{Deimling1977}.

In Hypotheses \ref{hyp:P}-$(iii)$, we assumed that the driving field is continuous for the $W_1$-topology in its measure variable $\Bmu\in \Pcal_{\pi,1}(\Omega \times \R^d)$. Following Theorem \ref{thm:ComparisonWass} and the discussion around Remark \ref{rmk:WeakContinuity}, this requirement can be understood as a form of weak continuity over $\Pcal_{\pi,1}(\Omega \times \R^d)$, which is incidentally one of the standard conditions allowing to recover Peano's theorem in infinite dimension, see e.g. \cite[Theorem 8.3]{Deimling1977}. While we could have opted for something more general inspired by the latter reference, we preferred to stick to this simpler framework, as it already encompasses most of the basic applications the theory aims to cover. Nonetheless, it would be interesting to see if some practical examples could require this more sophisticated machinery. 
\end{rmk}


\subsection{Links with Lagrangian dynamics and classical continuity equations}
\label{subsection:Links}

In this section, we exhibit various connections between nonlocal structured continuity equations and the two limit dynamical models they contain, namely Lagrangian (or continuum) dynamics in $L^1(\Omega,\R^d;\pi)$ on the one hand, and classical continuity equations over $\Pcal_1(\R^d)$ on the other.

We start by proving that under Hypotheses \ref{hyp:CL}, a curve of measures solves a structured continuity equation starting from an initial datum that is concentrated on a graph if and only if it remains concentrated on the graph of the solution to the underlying Lagrangian dynamics.

\begin{prop}[Link between structured continuity equations and Lagrangian dynamics]
Suppose that $\Bmu^0 \in \Pcal_{\pi,1}(\Omega \times \R^d)$ is of the form $\Bmu^0 := (\Id,x^0)_{\sharp} \pi$ for some $x^0 \in L^{\infty}(\Omega,\R^d;\pi)$, and let $\vb : [0,T] \times \Pcal_{\pi,1}(\Omega \times \R^d) \times \Omega \times \R^d \to \R^d$ be a nonlocal velocity field satisfying Hypotheses \ref{hyp:CL}. Then, the unique solution $\Bmu(\cdot) \in \AC([0,T],\Pcal_{\pi,1}(\Omega \times \R^d))$ of \eqref{eq:TransportGraphNonlocal} can be expressed as 
\begin{equation*}
\Bmu(t) = (\Id,x(t))_{\sharp} \pi
\end{equation*}
for all times $t \in [0,T]$, where $x(\cdot) \in \AC([0,T],L^1(\Omega,\R^d;\pi))$ is the unique solution of the classical Lagrangian dynamics 
\begin{equation}
\label{eq:ClassicalGraph}
\left\{
\begin{aligned}
& \partial_t x(t,\omega) = \vb \Big(t, (\Id,x(t))_{\sharp} \pi , \omega, x(t,\omega) \Big), \\
& x(0,\omega) = x^0(\omega), 
\end{aligned}
\right.
\end{equation}
for $\pi$-almost every $\omega \in \Omega$. 
\end{prop}

\begin{proof}
Given a curve $x(\cdot) \in \AC([0,T],L^1(\Omega,\R^d))$ solution of the classical Lagrangian dynamics \eqref{eq:ClassicalGraph}, it easily follows from \eqref{eq:DistribCE2} that the curve of measures $\Bmu(\cdot) \in \AC([0,T],\Pcal_{\pi,1}(\Omega \times \R^d))$ defined by $\Bmu(t):= (\Id,x(t))_{\sharp} \pi$  for all times $t \in [0,T]$ is a solution of the structured continuity equation \eqref{eq:TransportGraphNonlocal}. The converse implication follows straightforwardly from Theorem \ref{thm:CauchyLipschitz}, upon remarking that $\Bmu^0 \in \Pcal_{\pi}(\Omega \times B(0,r))$ with $r:= \|x^0\|_{L^{\infty}(\Omega,\R^d;\pi)}$, and 
\begin{equation*}
\begin{aligned}
\Bmu(t) & = \BPhi_{(0,t)}[\Bmu^0]_{\sharp} \Bmu^0 \\
& = \BPhi_{(0,t)}[\Bmu^0]_{\sharp} ((\Id,x^0)_{\sharp} \pi) \\
& = \BPhi_{(0,t)}[\Bmu^0](\Id,x^0)_{\sharp} \pi, 
\end{aligned}
\end{equation*}
for all times $t \in [0,T]$, where
\begin{equation}
\label{eq:NonlocalFlowTrim}
\BPhi_{(0,t)}[\Bmu^0](\omega,x^0(\omega)) = \Big( \omega , \Phi_{(0,t)}^{\omega}[\Bmu^0](x^0(\omega)) \Big)
\end{equation}
for $\pi$-almost every $\omega \in \Omega$. It then follows from what precedes combined with the representation formula of Theorem \ref{thm:CauchyLipschitz} and standard Cauchy-Lipschitz uniqueness arguments (see e.g. \cite[Theorem 3.2]{Medvedev2014}) that the curve $x(\cdot) \in \AC([0,T],L^{\infty}(\Omega,\R^d))$ defined by 
\begin{equation*}
x(t,\omega) := \Phi_{(0,t)}^{\omega}[\Bmu^0](x^0(\omega))
\end{equation*}
for all times $t \in [0,T]$ and $\pi$-almost every $\omega \in \Omega$ is the unique solution of the Cauchy problem \eqref{eq:ClassicalGraph}, which concludes the proof by virtue of \eqref{eq:NonlocalFlowTrim}.
\end{proof}

In the following proposition, we prove a result corresponding to the other limit scenario, stating that a curve of measures solves a structured continuity equation from a product measure and driven by a homogeneous vector field satisfying Hypotheses \ref{hyp:CL} if and only if it can be written as a product involving the solution of a classical continuity equation.

\begin{prop}[Correspondence between classical and structured continuity equations]
Suppose that $\Bmu^0 \in \Pcal_{\pi,1}(\Omega \times \R^d)$ is of the form $\Bmu^0 := \pi \times \mu^0$ for some $\mu^0 \in \Pcal(B(0,r))$ with $r > 0$. Moreover, let $\vb : [0,T] \times \Pcal_{\pi,1}(\Omega \times \R^d) \times \Omega \times \R^d \to \R^d$ be a nonlocal velocity field satisfying Hypotheses \ref{hyp:CL} and given in the particular form 
\begin{equation}
\label{eq:ClassicalVelocity}
\vb(t,\Bmu,\omega,x) := v \big(t, (\pfrak_{\R^d})_{\sharp} \Bmu , x \big)
\end{equation}
for $\Lcal^1$-almost every $t \in [0,T]$, any $\Bmu \in \Pcal_{\pi}(\Omega \times \R^d)$ and all $(\omega,x) \in \Omega \times \R^d$,
where $v:[0,T] \times \Pcal_1(\R^d) \times \R^d \to \R^d$. Then, the unique solution $\Bmu(\cdot) \in \AC([0,T],\Pcal_{\pi,1}(\Omega \times \R^d))$ of \eqref{eq:TransportGraphNonlocal} can be expressed as 
\begin{equation*}
\Bmu(t) = \pi \times \mu(t)
\end{equation*}
for all times $t \in [0,T]$, where $\mu(\cdot) \in \AC([0,T],\Pcal_1(\R^d))$ is the unique solution of the Cauchy problem
\begin{equation}
\label{eq:ClassicalTransport}
\left\{
\begin{aligned}
& \partial_t \mu(t) + \Div_x \big( v(t,\mu(t)) \mu(t) \big) = 0, \\
& \mu(0) = \mu^0.
\end{aligned}
\right.
\end{equation}
\end{prop}

\begin{proof}
Given an absolutely continuous solution of \eqref{eq:ClassicalTransport}, it is a matter of elementary computations to check that the measures defined by 
\begin{equation*}
\Bmu(t) := \pi \times \mu(t) \in \Pcal_{\pi}(\Omega \times \R^d)
\end{equation*}
for all times $t \in [0,T]$ form a fibred absolutely continuous curve which is a solution of \eqref{eq:TransportGraphNonlocal} driven by $\vb : [0,T] \times \Pcal_{\pi}(\Omega \times \R^d) \times \Omega \times \R^d \to \R^d$ defined as in \eqref{eq:ClassicalVelocity}. On the other hand, given a solution $\Bmu(\cdot) \in \AC([0,T],\Pcal_{\pi,1}(\Omega \times \R^d))$ driven by such a velocity field, it follows from Theorem \ref{thm:Equivalence} along with the support bound from Lemma \ref{lem:NonlocalBounds} that $\supp(\Bmu(t)) \subset \Omega \times B(0,R_r)$ for all times $t \in [0,T]$ and
\begin{equation*}
\Bmu(t) = \INTDom{\mu_{\omega}(t)}{\Omega}{\pi(\omega)},
\end{equation*}
where the Borel family of curves $\{\mu_{\omega}(\cdot)\}_{\omega \in \Omega} \subset \AC([0,T],\Pcal_1(\R^d))$ solves the system of coupled continuity equations
\begin{equation}
\label{eq:ClassicalDisintegration}
\left\{
\begin{aligned}
& \partial_t \mu_{\omega}(t) + \Div_x(v(t,\mu(t)) \mu_{\omega}(t)) = 0, \\
& \mu_{\omega}(0) = \mu^0, 
\end{aligned}
\right.
\end{equation}
in which $\mu(t) = (\pfrak_{\R^d})_{\sharp} \Bmu(t)$ for all times $t \in [0,T]$. Moreover, it can be checked that under Hypotheses \ref{hyp:CL}, the vector field $(t,x) \in [0,T] \times \R^d \mapsto v(t,\mu(t),x) \in \R^d$ is Carathéodory and satisfies
\begin{equation*}
|v(t,\mu(t),x)| \leq C_T\, m(t)(1+|x|) \qquad \text{and} \qquad \Lip \Big(v(t,\mu(t)) \, ; B(0,R_r) \Big) \leq L_{R_r}(t)
\end{equation*}
for $\Lcal^1$-almost every $t \in [0,T]$, all $x \in \R^d$ and some constant $C_T >0$, so that solutions of \eqref{eq:ClassicalDisintegration} are unique as a consequence e.g. of \cite[Proposition 8.1.7]{AGS}. This implies in particular that the curves $\mu_{\omega}(\cdot) \in \AC([0,T],\Pcal_1(\R^d))$ must all be identical for $\pi$-almost every $\omega \in \Omega$, from where we deduce that  
\begin{equation*}
\mu_{\omega}(t) = \mu(t)
\end{equation*}
for all times $t \in [0,T]$. At this point, there simply remains to note that 
\begin{equation*}
\Lip \Big(v(t,\cdot,x) \, ; \Pcal(B(0,R_r)) \Big) \leq L_{R_r}(t)
\end{equation*}
for $\Lcal^1$-almost every $t \in [0,T]$ and all $x \in B(0,R_r)$, which allows us to conclude by classical uniqueness results for nonlocal transport equations (see e.g. \cite[Theorem 4]{ContInc}) that $\mu(\cdot) \in \AC([0,T],\Pcal_1(\R^d))$ is the unique solution of \eqref{eq:ClassicalTransport}. 
\end{proof}

In the previous results, we showed that under suitable assumptions on the initial data and driving field -- and in particular that it should comply with Hypotheses \ref{hyp:CL} --, solutions of structured continuity equations coincide exactly with those of classical Lagrangian dynamics and continuity equations. In what comes next, we provide complementary statements exhibiting fully general one-way links between these objects, without asking for any regularity on the driving field. These, however, may hold at the price of relinquishing the closed form of the dynamics, see the related discussions in \cite{Paul2024}.

\begin{thm}[From structured continuity equations to continuum and meanfield dynamics]
Suppose that $\Bmu^0 \in \Pcal_{\pi,1}(\Omega \times \R^d)$ and let $\Bmu(\cdot) \in \AC([0,T],\Pcal_{\pi,1}(\Omega \times \R^d))$ be a solution of \eqref{eq:TransportGraphNonlocal} driven by a Lebesgue-Borel vector field $\vb : [0,T] \times \Pcal_{\pi,1}(\Omega \times \R^d) \times \Omega \times \R^d \to \R^d$ satisfying Hypothesis \ref{hyp:P}-$(ii)$. Then, the following holds. 
\begin{enumerate}
\item[$(a)$] The barycentric projection $(t,\omega) \in [0,T] \times \Omega \mapsto x(t,\omega) := \INTDom{x \,}{\R^d}{\mu_{\omega}(t)(x)} \in \R^d$ solves the Lagrangian dynamics
\begin{equation}
\label{eq:BarycentricDynamics}
\partial_t x(t,\omega) := \INTDom{\vb(t,\Bmu(t),\omega,y)}{\R^d}{\mu_{\omega}(t)(y)}
\end{equation}
for $\Lcal^1$-almost every $t \in [0,T]$ and $\pi$-almost every $\omega \in \Omega$. In particular, if it holds that
\begin{equation}
\label{eq:LinearField1}
\vb(t,\Bmu,\omega,x) := \INTDom{\Big( a(t,\omega,\theta) x + b(t,\omega,\theta) y \Big)}{\Omega \times \R^d}{\Bmu(\theta,y)}
\end{equation}
for some $a(\cdot),b(\cdot) \in L^{\infty}([0,T] \times \Omega \times \Omega,\R_+;\Lcal^1 \times \pi \times \pi)$, then the above dynamics is closed. 
\item[$(b)$] The space marginal $t \in [0,T] \mapsto \mu(t) := (\pfrak_{\R^d})_{\sharp} \Bmu(t) \in \Pcal_1(\R^d)$ solves the continuity equation 
\begin{equation}
\label{eq:ClassicalContinuity}
\partial_t \mu(t) + \Div_x(v(t)\mu(t)) = 0  
\end{equation}
driven by the Lebesgue-Borel vector field $v : [0,T] \times \R^d \to \R^d$ given by 
\begin{equation}
\label{eq:LinearField2}
v(t,x) := \INTDom{\vb(t,\Bmu(t),\omega,x)}{\Omega}{\pi_{t,x}(\omega)}. 
\end{equation}
Therein, for all times $t \in [0,T]$, the collection $\{\pi_{t,x}\}_{x \in \R^d} \subset \Pcal(\Omega)$ is the $\mu(t)$-almost uniquely-determined Borel family of probability measures such that
\begin{equation}
\label{eq:DisintegrationStatement}
\INTDom{\varphi(\omega,x)}{\Omega \times \R^d}{\Bmu(t)(\omega,x)} = \INTDom{\bigg( \INTDom{\varphi(\omega,x)}{\Omega}{\pi_{t,x}(\omega)} \bigg)}{\R^d}{\mu(t)(x)}, 
\end{equation}
for every Borel map $\varphi : \Omega \times \R^d \to [0,+\infty]$. In particular, if it holds that
\begin{equation}
\label{eq:ClosedFormField}
\vb(t,\Bmu,\omega,x) := v(t,(\pfrak_{\R^d})_{\sharp} \Bmu,x)
\end{equation}
for some Lebesgue-Borel field $v : [0,T] \times \Pcal(\R^d) \times \R^d \to \R^d$, then the above dynamics is closed.
\end{enumerate}

\end{thm}

\begin{proof}
To begin with, note that up to a minor adaptation of the a priori moment bound of Proposition \ref{prop:Bounds}, there exist constants $C_T,c_T > 0$ such that 
\begin{equation*}
\sup_{t \in [0,T]} \Mpazo_1(\mu_{\omega}(t)) \leq C_T (1+ \Mpazo_1(\mu_{\omega}^0)) \qquad \text{and} \qquad W_1(\mu_{\omega}(t_1),\mu_{\omega}(t_2)) \leq c_T  (1+ \Mpazo_1(\mu_{\omega}^0)) \INTSeg{m(s)}{s}{t_1}{t_2}
\end{equation*}
for all times $0 \leq t_1 \leq t_2 \leq T$ and $\pi$-almost every $\omega \in \Omega$, which implies that $x(\cdot) \in \AC([0,T],L^1(\Omega,\R^d))$ by a simple application of Jensen's inequality. Under Hypothesis \ref{hyp:P}-$(ii)$, these estimates clearly entail 
\begin{equation*}
\INTSeg{\INTDom{\frac{|\vb(t,\Bmu(t),\omega,x)|}{1+|x|}}{\Omega \times \R^d}{\Bmu(t)(\omega,x)}}{t}{0}{T} < +\infty,
\end{equation*}
and up to a simple adaptation of the truncation argument detailed in the proof of Proposition \ref{prop:Marginal} above, one may choose in the distributional formulation \eqref{eq:DistribCE1} test functions of the form
\begin{equation*}
\varphi(\omega,x) := \psi_n(x)\langle e_i , x \rangle
\end{equation*}
where $\{e_i\}_{i=1}^d$ is an orthonormal basis of $\R^d$ and the sequence $(\psi_n) \subset C^1_c(\R^d,\R)$ is as in \eqref{eq:CutOff}. From this, we deduce after letting $n \to +\infty$ that 
\begin{equation*}
\partial_t x(t,\omega) = \derv{}{t} \INTDom{x \,}{\R^d}{\mu_{\omega}(t)(x)} = \INTDom{\vb(t,\Bmu(t),\omega,x)}{\R^d}{\mu_{\omega}(t)(x)} 
\end{equation*}
for $\Lcal^1$-almost every $t \in [0,T]$ and $\pi$-almost every $\omega \in \Omega$. The fact that the dynamics becomes closed when the driving field is of the form \eqref{eq:LinearField1} is then an easy consequence of the following observation 
\begin{equation*}
\begin{aligned}
\INTDom{\vb(t,\Bmu(t),\omega,x)}{\R^d}{\mu_{\omega}(t)(x)} & = \INTDom{\bigg( \INTDom{\Big( a(t,\omega,\theta)x + b(t,\omega,\theta) y \Big)}{\Omega \times \R^d}{\Bmu(t)(\theta,y)} \bigg)}{\R^d}{\mu_{\omega}(t)(x)} \\
& = \bigg( \INTDom{a(t,\omega,\theta)}{\Omega}{\pi(\theta)} \bigg) x(t,\omega) + \INTDom{b(t,\omega,\theta) x(t,\theta)}{\Omega}{\pi(\theta)}
\end{aligned}
\end{equation*}
by Fubini's theorem, which settles the proof of item $(a)$. 

Regarding item $(b)$, we note first that owing to the general disintegration principle provided in Theorem \ref{thm:Disintegration}, there exists for all times $t \in [0,T]$ a $\mu(t)$-almost uniquely determined Borel family of measures $\{\pi_{t,x}\}_{x \in \R^d} \subset \Pcal(\Omega)$ such that \eqref{eq:DisintegrationStatement} holds. In particular, this implies that  
\begin{equation*}
\begin{aligned}
\derv{}{t} \INTDom{\varphi(x)}{\R^d}{\mu(t)(x)} & = \derv{}{t} \INTDom{\varphi(x)}{\Omega \times \R^d}{\Bmu(t)(\omega,x)} \\
& = \INTDom{\langle \nabla \varphi(x) \, , \vb(t,\Bmu(t),\omega,x) \rangle}{\Omega \times \R^d}{\Bmu(t)(\omega,x)} \\
& = \INTDom{\bigg\langle \nabla \varphi(x) \, , \INTDom{\vb(t,\Bmu(t),\omega,x)}{\Omega}{\pi_{t,x}(\omega)} \bigg\rangle}{\R^d}{\mu(t)(x)} \\
& = \INTDom{\big\langle \nabla \varphi(x) , v(t,x) \big\rangle}{\R^d}{\mu(t)(x)}
\end{aligned}
\end{equation*} 
for each $\varphi \in C^1_c(\R^d,\R)$, which amounts to \eqref{eq:ClassicalContinuity}. The fact that the latter dynamics has closed form under \eqref{eq:ClosedFormField} stems from the observation that
\begin{equation*}
\INTDom{v(t,(\pfrak_{\R^d})_{\sharp}\Bmu(t),x)}{\Omega}{\pi_{t,x}(\omega)} = v(t,\mu(t),x) 
\end{equation*}
for $\Lcal^1$-almost every $t \in [0,T]$ and $\mu(t)$-almost every $x \in \R^d$, which closes the proof. 
\end{proof}


\section{Discretisation and particle approximation}
\setcounter{equation}{0} \renewcommand{\theequation}{\thesection.\arabic{equation}}
\label{Section:ParticleApprox}


In this section, our goal is to show that solutions of structured continuity equations of the form \eqref{eq:TransportGraphNonlocal} can be obtained as the meanfield limit of nonexchangeable interacting particle dynamics. Such convergence results have recently been established for some particular choices of vector fields $\vb : [0,T] \times \Pcal_{\pi}(\Omega \times \R^d) \times \Omega \times \R^d \to \R^d$, the most prominent of which being those where the label and space variables are separated, as it is the case for kernels of the form 
\begin{equation*}
 \vb(t,\Bmu, \omega, x) := \INTDom{w(t,\omega,\theta)\psi(x,y)}{\Omega\times\R^d}{\Bmu(\theta,y)},    
\end{equation*}
see e.g. the recent literature \cite{AyiPouradierDuteilPoyato2024,ChibaMedvedev19,JabinPoyatoSoler2025,KaliuzhnyiMedvedev2018,KuehnXu2022,Paul2024}. Here, our goal is to extend such results to more general families of vector fields, and to provide quantitative versions of said convergence results in some particular cases of interest.  

Throughout this section, we shall assume that the probability space $(\Omega,\Apazo,\pi)$ is nonatomic, and that $\vb : [0,T] \times \Pcal_{\pi,1}(\Omega \times \R^d) \times \Omega \times \R^d \to \R^d$ satisfies Hypotheses \ref{hyp:CL}. We also fix an initial datum $\Bmu^0\in\Pcal_{\pi}(\Omega \times B(0,r))$ for some $r>0$, so that according to Lemma \ref{lem:NonlocalBounds} and Theorem \ref{thm:CauchyLipschitz}, there exists a unique solution $\Bmu(\cdot)\in\AC([0,T],\Pcal_{\pi,1}(\Omega \times B(0,R_r)))$ to the Cauchy problem
\begin{equation}
\label{eq:TransportGraphParticles}
\left\{
\begin{aligned}
& \partial_t \Bmu(t) + \Div_x(\vb(t,\Bmu(t)) \Bmu(t)) = 0, \\
& \Bmu(0) = \Bmu^0. 
\end{aligned}
\right.
\end{equation}
Our goal is to prove that this curve can be approximated by the solution of a particle system, whose construction relies on associating each particle (or each subset of particles) with an adequate subset of $\Omega$. This hinges upon the following proposition, which itself relies on the general equivalence between standard atomless probability spaces (see e.g. \cite[Chapter 15]{Royden88} and \cite[Proposition D1]{Cavagnari2022}). 

\begin{prop}[Canonical equipartitions of standard probability spaces]
\label{prop:partition}
Suppose that $(\Omega,\Apazo,\pi)$ is a complete, nonatomic probability space over a Polish space. Then, there exists a pair of Borel sets $\Omega_0\subset \Omega$ and $I_0\subset [0,1]$ satisfying $\pi(\Omega_0)=0$ and $\Lcal^1(I_0)=0$, along with a Borel bijection with Borel inverse $\psi:\Omega\setminus\Omega_0\rightarrow [0,1]\setminus I_0$ such that 
\begin{equation*}
\psi_{\sharp}\pi = \Lcal^1 \qquad \text{and} \qquad \psi^{-1}_{\sharp}\Lcal^1 = \pi.
\end{equation*}
Moreover, let $I_i^N  := \big[\tfrac{i-1}{N},\tfrac{i}{N}\big)$ and $\Omega_i^N:=\psi^{-1}(I_i^N)$ for each $i\in\{1,\ldots,N\}$ and all $N \geq 1$. Then, the collection $\Ppazo_N := (\Omega_i^N)_{i \in \{1,\dots,N\}}$ forms a disjoint equipartition of $\Omega$, namely $\bigcup_{i=1}^N \Omega_i^N = \Omega$ with 
\begin{equation*}
\Omega_i^N\cap \Omega_{j}^N = \emptyset \qquad \text{and} \qquad \pi(\Omega_i^N) = \frac{1}{N} 
\end{equation*}
for every $i,j \in \{1,\dots,N\}$ with $i \neq j$. Moreover, if $N := mn$ for some $m,n \geq 1$, then $(\Omega_i^N)_{i\in\{1,\ldots,N\}}$ is a refinement of $(\Omega_k^n)_{k\in\{1,\cdots,n\}}$, in the sense that 
\begin{equation}\label{eq:refinement}
\bigcup_{\ell=1}^m\Omega_{(k-1)m+\ell}^N = \Omega_k^n
\end{equation}
for each $k\in\{1,\cdots,n\}$. 
\end{prop}

Importantly, the above proposition combined e.g. with \cite[Lemma D3]{Cavagnari2022} grants access to a general approximation property stated in the following proposition, which will be used throughout this section. For the sake of self-containedness, we sketch its proof in Appendix \ref{section:AppendixEquivalenceProba} below.

\begin{prop}[Universal discrete approximations by conditional expectations]
\label{prop:Conv_CondExp}
For all $N \geq 1 $, let $\Ppazo_N := (\Omega_i^N)_{i\in\{1,\ldots,N\}}$ be an equipartition of $\Omega$ given as in Proposition \ref{prop:partition}. 
Fix also some $p\in [1,+\infty)$ and let $(X,\Norm{\cdot}_X)$ be a separable Banach space. Then for each $f\in L^p(\Omega,X;\pi)$, it holds that
\begin{equation*}
\NormL{f - \E_{\Ppazo_{N}}[f]}{p}{\Omega,X;\pi} ~\underset{N \to +\infty}{\longrightarrow}~ 0,
\end{equation*}
where $\E_{\Ppazo_{N}}[f] \in L^p(\Omega,X;\pi)$ is the conditional expectation with respect to $\Ppazo_N$, as defined in \eqref{eq:ConditionalExpectation}.
\end{prop}


Moving forward, we shall suppose that $N:=m n$ for some $m,n \geq 1$ and let  $(\Omega_i^N)_{i\in\{1,\ldots,N\}}$ and $(\Omega_k^n)_{k \in \{1,\cdots,n\}}$ be the corresponding canonical equipartitions of $\Omega$ built as in Proposition~\ref{prop:partition}. Henceforth, we consider the nonexchangeable particle system 
\begin{equation}\label{eq:Mic}
\begin{cases}
\dot x_i^\nm(t) = v_i^N(t,\Bmu^{\nm}(t),x_i^\nm(t))  \\
x_{i}^\nm(0) = x_{i}^{0,\nm}, 
\end{cases}
\end{equation}
where for each $i\in\{1,\ldots,N\}$, the vector field $v_i^N : [0,T] \times \Pcal_{\pi,1}(\Omega \times \R^d) \times \R^d \to \R^d$ is defined by
\begin{equation}
\label{eq:viNDef}
v_i^N(t,\Bmu,x) := \INTDomdash{\vb(t,\Bmu,\omega,x) }{\Omega_i^N}{\pi(\omega)}
\end{equation}
for $\Lcal^1$-almost every $t \in [0,T]$ and all $(\Bmu,x) \in \Pcal_{\pi,1}(\Omega \times \R^d) \times \R^d$. Heuristically, the microscopic velocity stirring the agent with label $i \in \{1,\ldots,N\}$ is obtained simply by averaging the macroscopic field over the corresponding subset $\Omega_i^N$. Similarly, the curve of empirical measures $\Bmu^{\nm}:[0,T]\rightarrow\Pcal_{\pi,1}(\Omega \times \R^d)$ is built from the solutions $(x_i^\nm(\cdot))_{i\in\elts}$ of \eqref{eq:Mic} via the formula
\begin{equation}
\label{eq:emp_Mic}
\Bmu^{\nm}(t) := \sum_{k=1}^n \pi_{\llcorner \Omega_k^n} \times \bigg( \frac{1}{m} \sum_{\ell=1}^m \delta_{x_{(k-1)m+\ell}^\nm(t)} \bigg)
\end{equation}
for all times $t \in [0,T]$, where $\pi_{\llcorner \Omega_k^n} \in \Mcal^+(\Omega_k^n)$ stands for the usual restriction of $\pi \in \Pcal(\Omega)$ to the subset $\Omega_k^n$. We are now ready to state our main convergence result. 

\begin{thm}[Particle approximation of solutions of structured continuity equations]
\label{thm:ConvergenceMicroMacro}
Suppose that $\Bmu^0 \in \Pcal_{\pi}(\Omega \times B(0,r))$ for some $r>0$, assume that $\vb : [0,T] \times \Pcal_{\pi,1}(\Omega \times \R^d) \times \Omega \times \R^d \to \R^d$ satisfies Hypotheses \ref{hyp:CL}, and denote by $\Bmu(\cdot)\in \AC([0,T],\Pcal_{\pi,1}(\Omega\times B(0,R_r)))$ the unique solution of the Cauchy problem~\eqref{eq:TransportGraphParticles}. For each $m,n \geq 1$, let $(x_\kl^{0,\nm})_{(k,\ell)\in\{ 1,\ldots,n \}\times\{ 1,\ldots,m \}}\subset \R^d$ be independent random variables distributed according to the law 
\begin{equation*}
\Lpazo \Big(x_\kl^{0,\nm} \Big) = \muAvLaw\in \Pcal(B(0,r))
\end{equation*}
for each $k \in \{1,\cdots,n\}$, and denote by $(x_\kl^{\nm}(\cdot))_{(k,\ell)\in\{ 1,\ldots,n \}\times \{ 1,\ldots,m \}} \subset \AC([0,T],\R^d)$ the corresponding solutions of the particle system \eqref{eq:Mic}. Then, it holds that
\begin{equation*}
\lim_{m,n \to +\infty} \E\bigg[\hspace{0.02cm} \sup_{t \in [0,T]}W_{\pi,1}(\Bmu^{n,m}(t),\Bmu(t)) \bigg] = 0,  
\end{equation*}
where the empirical measure curves $(\Bmu^{n,m}(\cdot)) \subset \AC([0,T],\Pcal_{\pi,1}(\Omega \times B(0,R_r)))$ are given as in \eqref{eq:emp_Mic}. 
\end{thm}


\begin{rmk}[On the construction of the empirical measures and discrete fields]
Notice that the partition $(\Omega_k^n)_{k \in \{1,\ldots,n\}}$ used to build the empirical measure $\Bmu^{\nm}(\cdot)$ in~\eqref{eq:emp_Mic} is coarser than $(\Omega_i^N)_{\{1,\ldots,N\}}$, which is used to define the vector fields $(v_i^N(\cdot))_{i\in\elts}$. Then, according to this construction, the particles whose labels belong to the subset $\{(k-1)m+1,\ldots,km\}$ are ``vertically'' grouped as Dirac masses in the same subset $\Omega_k^n$, as illustrated in Figure \ref{fig:mu2D} below. Although it may come across as unnecessary, this double discretisation procedure is central to our convergence theorem. Indeed, a simpler and more natural alternative would have been to consider empirical measures of the form 
\begin{equation*}
\widetilde\Bmu^{N}(t) = \sum_{i=1}^N \pi_{\llcorner \Omega_i^N} \times \delta_{x_{i}(t)}
\end{equation*}
for all times $t\in[0,T]$, which map each subset of the partition $(\Omega_i^N)_{\{1,\ldots,N\}}$ to a unique particle trajectory.
Notice, however, that this would imply that $\widetilde\Bmu^N(t) \in \Gcal_{\pi,1}(\Omega \times \R^d)$ for all times $t \in [0,T]$ and each $N \geq 1$, where $\Gcal_{\pi,1}(\Omega \times \R^d)$ stands for the subset of fibred measures which are supported on the graph of an integrable map, introduced in Remark~\ref{rmk:closeness} above. However, as explained therein, the set $\Gcal_{\pi,\,1}(\Omega \times \R^d)$ is closed in the $W_{\pi,\,1}$-topology, which means that such sequences $(\widetilde{\Bmu}^N(t))$ may only converge towards an element of $\Gcal_{\pi,\,1}(\Omega \times \R^d)$, but never towards a general measure $\Bmu(t)\in \Pcal_{\pi,1}(\Omega\times \R^d)$. This may only be achieved by considering a weaker mode of convergence e.g. in classical Wasserstein metric, see for instance \cite{Paul2024}, which comes at the price of imposing some regularity in the index variable of the vector field. On the contrary, according to the construction of $\Bmu^\nm : [0,T] \to \Pcal_{\pi,1}(\Omega \times \R^d)$ detailed in~\eqref{eq:emp_Mic} above, it indeed holds that $\Bmu^\nm(t)\not\in \Gcal_{\pi,\,1}(\Omega \times \R^d)$ for all times $t \in [0,T]$.
\end{rmk}

The remainder of this section is dedicated to proving Theorem \ref{thm:ConvergenceMicroMacro}. The proof is based on the introduction of several auxiliary particle systems and empirical measures, which we define in Section \ref{subsection:AuxiliaryParticles}. Section \ref{subsection:InitialData} is then dedicated to proving the convergence of the initial data, while in Section \ref{subsection:ConvergenceProof}, we provide the full proof of Theorem \ref{thm:ConvergenceMicroMacro}. Lastly, we discuss in Section \ref{subsection:QuantitativeConv} some ways of quantifying the convergence result of Theorem \ref{thm:ConvergenceMicroMacro} which are illustrated by a relevant class of examples.  


\subsection{Auxiliary particle system and empirical measures}
\label{subsection:AuxiliaryParticles}

The main difference between our configuration and that encountered for meanfield limits of exchangeable particle systems is that the empirical curves $\Bmu^{\nm}(\cdot)$ defined in \eqref{eq:emp_Mic} do not solve the limit equation \eqref{eq:TransportGraphParticles}. For this reason, we will resort to a two-step limit process, based on the product decomposition $N=mn$. To do so, we consider the auxiliary particle system 
\begin{equation}
\label{eq:Mic-aux}
\begin{cases}
\dot \bx_\kl^\nm(t) = v_k^n(t,\bmu^{n,m}(t),\bx_\kl^\nm(t)), \\
\bx_{\kl}^\nm(0) = x_{\kl}^{0,\nm}, 
\end{cases}
\end{equation}
where for each $k \in \{1,\ldots,n\}$, the discrete vector fields $v_k^n : [0,T] \times \Pcal_{\pi,1}(\Omega \times \R^d) \times \R^d \to \R^d$ are given by 
\begin{equation}
\label{eq:vknDef}
v_k^n(t,\Bmu,x) = \INTDomdash{\vb(t,\Bmu,\omega,x) }{\Omega_k^n}{\pi(\omega)}
\end{equation}
for $\Lcal^1$-almost every $t \in [0,T]$, all $(\Bmu,x) \in \Pcal_{\pi,1}(\Omega \times \R^d) \times \R^d$ and each $k \in \{1,\cdots,n\}$, and
\begin{equation}
\label{eq:emp_Mic-aux}
\bar{\Bmu}^{n,m}(t) := \sum_{k=1}^n \pi_{\llcorner \Omega_k^n} \times \bigg( \frac{1}{m} \sum_{\ell=1}^m \delta_{\bx_\kl^\nm(t)} \bigg)
\end{equation}
for all times $t \in [0,T]$. Contrarily to the original particle system \eqref{eq:Mic}, the evolution of the particles $(\bx_\kl^\nm)_{(k,l)\in\{1,\ldots,n\}\times\{ 1,\ldots,m\}} $ solution of \eqref{eq:Mic-aux} is prescribed by only $n$ distinct vector fields $(v_k^n(\cdot))_{k\in\{1,\cdots,n\}}$, each obtained by averaging the macroscopic vector field $\vb : [0,T] \times \Pcal_{\pi,1}(\Omega \times \R^d) \times \Omega \times \R^d \to \R^d$ over the corresponding subset of the coarse partition $(\Omega_k^n)_{k\in\{1,\cdots,n\}}$. 

In the next proposition, we show that both systems of differential equations \eqref{eq:Mic} and \eqref{eq:Mic-aux} are well-posed, and provide handy stability and regularity estimates on their solutions. 

\begin{prop}[Well-posedness of the particle systems and a priori estimates]
\label{prop:ExistUniq-Mic}
Suppose that $\Bmu^0 \in \Pcal_{\pi}(\Omega \times B(0,r))$ for some $r>0$ and assume that $\vb : [0,T] \times \Pcal_{\pi,1}(\Omega \times \R^d) \times \Omega \times \R^d \to \R^d$ satisfies Hypotheses \ref{hyp:CL}. Then, for any initial data $(x_{i}^{0,\nm})_{i\in\{1,\ldots,N\}}\in (\R^d)^N$, there exist unique solutions $(x_{i}^{\nm})_{i\in\{1,\ldots,N\}},(\bx_{i}^{\nm}(\cdot))_{i\in\{1,\ldots,N\}} \in \AC([0,T],(\R^d)^N)$ to the particle systems \eqref{eq:Mic} and \eqref{eq:Mic-aux}. Moreover for each $i\in \{1,\ldots,N\}$ and $k \in \{1,\ldots,n\}$, the vector fields $(v_i^N(\cdot))_{i\in\{1,\ldots,N\}}$ and $(v_k^n(\cdot))_{k\in\{1,\cdots,n\}}$ comply with the growth estimates
\begin{equation}
\label{eq:AuxSystemBound}
\big| v^N_i(t,\Bmu,x) \big| \leq m(t) \bigg( 1 + |x| + \Mpazo_{\pi,1}(\Bmu) \bigg) \quad \text{and} \quad 
\big| v^n_k(t,\Bmu,x) \big| \leq m(t) \bigg( 1 + |x| + \Mpazo_{\pi,1}(\Bmu) \bigg),
\end{equation}
for $\Lcal^1$-almost every $t \in [0,T]$ and all $(\Bmu,x) \in \Pcal_{\pi,1}(\Omega \times \R^d) \times \R^d$. Similarly for each $R>0$, the vector fields satisfy the Lipschitz estimates
\begin{equation}
\label{eq:AuxSystemLip}
\left\{
\begin{aligned}
\big|v^N_i(t,\Bmu,x) - v^N_i(t,\Bnu,y) \big| & \leq L_{R}(t) \Big( W_{\pi,1}(\Bmu,\Bnu) + |x-y| \Big), \\
\big|v^n_k(t,\Bmu,x) - v^n_k(t,\Bnu,y) \big| & \leq L_{R}(t) \Big( W_{\pi,1}(\Bmu,\Bnu) + |x-y| \Big),
\end{aligned}
\right.
\end{equation}
for $\Lcal^1$-almost every $t \in [0,T]$, all $\Bmu,\Bnu \in \Pcal_{\pi}(\Omega \times B(0,R))$ and every $x,y \in B(0,R)$. In addition
\begin{equation}
\label{eq:AuxSystemMax}
\left\{
\begin{aligned}
\max_{i\in\{1,\ldots,N\}} |x_i^\nm(t)| & \leq \bigg(\max_{i\in\{1,\ldots,N\}} |x_i^{0,\nm}| + \|m(\cdot)\|_1 \bigg) \exp \big(2 \|m(\cdot)\|_1 \big), \\
\max_{i\in\{1,\ldots,N\}} |\bx_i^\nm(t)| & \leq \bigg(\max_{i\in\{1,\ldots,N\}} |\bx_i^{0,\nm}| + \|m(\cdot)\|_1 \bigg) \exp \big(2 \|m(\cdot)\|_1 \big),
\end{aligned}
\right.
\end{equation}
for all times $t\in [0,T]$.
\end{prop}

\begin{proof}
These statements being quite elementary, we postpone their proof to the Appendix \ref{sec:AppendixWellPosednessMic} below.
\end{proof}

In addition to the auxiliary particle system \eqref{eq:Mic-aux}, we consider the auxiliary transport equation
\begin{equation}
\label{eq:Transport-aux}
\partial_t \bmu(t) + \Div_x (\vb^n(t,\bmu(t)) \bmu(t)) = 0
\end{equation}
where the vector field $\vb^n:[0,T]\times \mathcal{P}_{\pi,1}(\Omega \times \R^d)\times \Omega\times\R^d\rightarrow\R^d$ is defined by 
\begin{equation}
\label{eq:vn}
\vb^n(t,\Bmu,\omega,x) := \sum_{k=1}^n v^n_k(t,\Bmu,x) \mathds{1}_{\Omega_k^n}(\omega) = \sum_{k=1}^n  \bigg( \INTDomdash{\vb(t,\Bmu,\theta,x)}{\Omega_k^n}{\pi(\theta)} \bigg) \mathds{1}_{\Omega_k^n}(\omega).
\end{equation}
One can easily show that the latter satisfies Hypotheses \ref{hyp:CL}, whence, given any initial data $\bmu^{0,n}\in\Pcal_{\pi}(\Omega \times B(0,r))$, it follows from Theorem \ref{thm:CauchyLipschitz} that the Cauchy problem associated with \eqref{eq:Transport-aux} admits a unique solution $\bmu^n(\cdot) \in \AC([0,T],\Pcal_{\pi,1}(\Omega \times \R^d))$. In the following lemma, we show that the curves of empirical measures $\bar{\Bmu}^{n,m}(\cdot) \in \AC([0,T],\Pcal_{\pi,1}(\Omega \times \R^d))$ defined in \eqref{eq:emp_Mic-aux} also solve the auxiliary continuity equation \eqref{eq:Transport-aux}. 

\begin{lem}[Empirical solutions of the auxiliary structured continuity equation]
\label{lem:SolutionAuxiliary}
Suppose that $\Bmu^0 \in \Pcal_{\pi}(\Omega \times B(0,r))$ for some $r>0$, assume that $\vb : [0,T] \times \Pcal_{\pi,1}(\Omega \times \R^d) \times \Omega \times \R^d \to \R^d$ satisfies Hypotheses \ref{hyp:CL}. Let  $(\bx_{i}^{\nm}(\cdot))_{i\in\{1,\ldots,N\}}\in \AC([0,T],(\R^d)^N)$ be the unique solution to the auxiliary particle system \eqref{eq:Mic-aux} with initial condition given by $(x_{i}^{0,\nm})_{i\in\{1,\ldots,N\}}\in (\R^d)^N$, and consider the curve of empirical measures $\bar{\Bmu}^{n,m} : [0,T] \to \Pcal_{\pi,1}(\Omega \times \R^d)$ defined by \eqref{eq:emp_Mic-aux}. Then $\bar{\Bmu}^{n,m}(\cdot) \in \AC([0,T],\Pcal_{\pi,1}(\Omega \times \R^d))$ is a distributional solution of \eqref{eq:Transport-aux}.
\end{lem}

\begin{proof}
Given any admissible test function $\phi \in \Adm((0,T) \times \Omega \times \R^d,\R)$, it holds that  
\begin{equation*}
\begin{split}
\INTSeg{ & \INTDom{ \partial_t\phi(t,\omega,x) }{\Omega\times\R^d}{\bmu^\nm(t)(\omega,x)} }{t}{0}{T} \\
& = \INTSeg{ \sum_{k=1}^n \INTDom{ \frac{1}{m} \sum_{\ell=1}^m \partial_t \phi \Big(t,\omega,\bx^\nm_\kl(t) \Big)}{\Omega_k^n}{\pi(\omega) }}{t}{0}{T}\\
& = - \INTSeg{ \sum_{k=1}^n \INTDom{ \bigg( \frac{1}{m} \sum_{\ell=1}^m \Big\langle \nabla_x \phi \Big(t,\omega,\bx^\nm_\kl(t) \Big) \, , \, v_k^n \Big(t,\bmu^\nm(t),\bx^\nm_\kl(t) \Big) \Big\rangle \bigg) }{\Omega_k^n}{\pi(\omega) }}{t}{0}{T} \\
& = - \INTSeg{ \sum_{k=1}^n \bigg( \frac{1}{m} \sum_{\ell=1}^m \INTDom{ \Big\langle \nabla_x \phi \Big(t,\omega,\bx^\nm_\kl(t) \Big) \, , \, v_k^n \Big(t,\bmu^\nm(t),\bx^\nm_\kl(t) \Big) \Big\rangle}{\Omega_k^n}{ \pi(\omega)} \bigg)}{t}{0}{T} \\
& = - \INTSeg{ \INTDom{\Big\langle \nabla_x \phi(t,\omega,x) , \vb^n\Big(t,\bmu^\nm(t),\omega,x \Big) \Big\rangle}{\Omega\times\R^d}{\bmu^\nm(t)(\omega,x)}}{t}{0}{T},
\end{split}
\end{equation*}
which is the definition of distributional solution of \eqref{eq:Transport-aux}.
\end{proof}

In the following definition, we recollect the definitions of the various curves of fibred measures introduced so far.

\begin{Def}[List of relevant fibred measure curves]
\label{def:mumn}
Let $\Bmu^0\in \Pcal_{\pi,1}(\Omega\times B(0,r))$ for some $r >0$ and for $N := mn$ with $n,m\geq 1$, consider any element $(x_\kl^{0,\nm})_{(k,\ell)\in\{ 1,\ldots,n \}\times\{ 1,\ldots,m \}} \in B(0,r)^N$. 
\begin{itemize}
\item $\Bmu^\nm(\cdot)\in \AC([0,T],\Pcal_{\pi,1}(\Omega\times\R^d))$ denotes  the curve of empirical measures \eqref{eq:emp_Mic} associated with the solution to the original particle system~\eqref{eq:Mic} with initial conditions given by $(x^{0,\nm}_i)_{i\in\{1,\ldots,N\}}$.
\item $\bmu^\nm(\cdot)\in \AC([0,T],\Pcal_{\pi,1}(\Omega\times\R^d))$ denotes the curve of empirical measures \eqref{eq:emp_Mic-aux} associated with the solution to the auxiliary particle system~\eqref{eq:Mic-aux} with  initial conditions given by $(x^{0,\nm}_i)_{i\in\{1,\ldots,N\}}$. As shown in Lemma \ref{lem:SolutionAuxiliary} above, it is also a solution of the auxiliary continuity equation~\eqref{eq:Transport-aux}.
\item $\bmu^n(\cdot) \in \AC([0,T],\Pcal_{\pi,1}(\Omega\times\R^d))$ denotes the solution of the auxiliary transport equation~\eqref{eq:Transport-aux} whose initial condition is defined through the formula
\begin{equation}\label{eq:mun0}
\bmu^{0,n} := \sum_{k=1}^n \pi_{\llcorner \Omega_k^n} \times \bigg( \INTDomdash{\mu_{\omega}^0}{\Omega_k^n}{\pi(\omega)} \bigg) \in \Pcal_{\pi}(\Omega \times B(0,r)).
\end{equation}
\item $\Bmu(\cdot) \in \AC([0,T],\Pcal_{\pi,1}(\Omega\times\R^d))$ denotes the solution to the original structured equation~\eqref{eq:TransportGraphNonlocal} with initial condition $\Bmu^0 \in \Pcal_{\pi}(\Omega \times B(0,r))$.
\end{itemize}
Moreover, note that 
\begin{equation*}
\Big( \supp(\Bmu^\nm(t)) \cup \supp(\bmu^\nm(t)) \cup \supp(\bmu^n(t)) \cup \supp(\Bmu(t)) \Big) \subset \Omega \times B(0,R_r) 
\end{equation*}
for all times $t \in [0,T]$, where $R_r > 0$ is given as in Lemma \ref{lem:NonlocalBounds}.
\end{Def}

We provide illustrations of these objects in Figures~\ref{fig:mu3D} and \ref{fig:mu2D}, in the particular case in which $\Omega := [0,1]$ and $(\Omega_k^n)_{k \in \{1\ldots,n\}} := (I_k^n)_{k \in \{1\ldots,n\}}$, and where both $\pi \in \Pcal([0,1])$ and  $\Bmu^0\in\Pcal_\pi([0,1]\times \R)$ are absolutely continuous with respect to the Lebesgue measure. In Figure~\ref{fig:mu3D}, we propose a 3D visualisation of a certain choice of $\Bmu^0\in\Pcal_\pi([0,1]\times\R)$ (left) and of the corresponding measure $\bar{\Bmu}^{0,n}\in\Pcal_\pi([0,1]\times \R)$ (right), which is built by averaging the fibres of the former over each subset of the coarse partition $(I_k^n)_{k \in \{1\ldots,n\}}$. In Figure~\ref{fig:mu2D}, we provide 2D colour representations of the densities of $\Bmu^0$ (left) and $\bar{\Bmu}^{0,n}$ (centre) with respect to the Lebesgue measure, along with a particular outcome of the random probability measure $\Bmu^{0,n,m}$ (right), which is constructed by sampling $m$ i.i.d. points $(x_\kl^{0,\nm})_{(k,\ell) \in \{ 1,\ldots,n \} \times \{1,\ldots,m\}}$ from the averages $\INTDomdash{\mu^0_\omega}{I_k^n}{ \pi(\omega)}$ of the initial measure over $(I_k^n)_{k \in \{1\ldots,n\}}$.

\begin{figure}[!ht]
\centering
\includegraphics[width=0.48\textwidth]{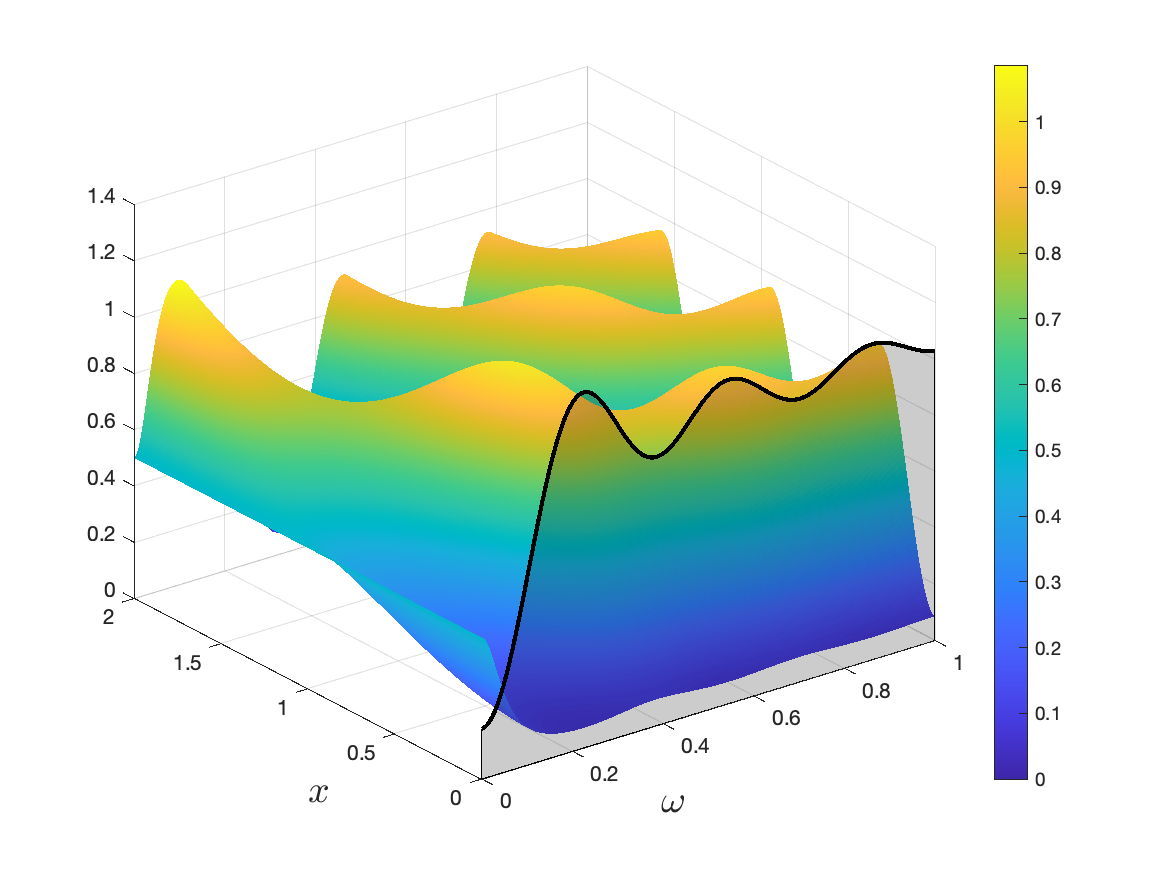}
\includegraphics[width=0.48\textwidth]{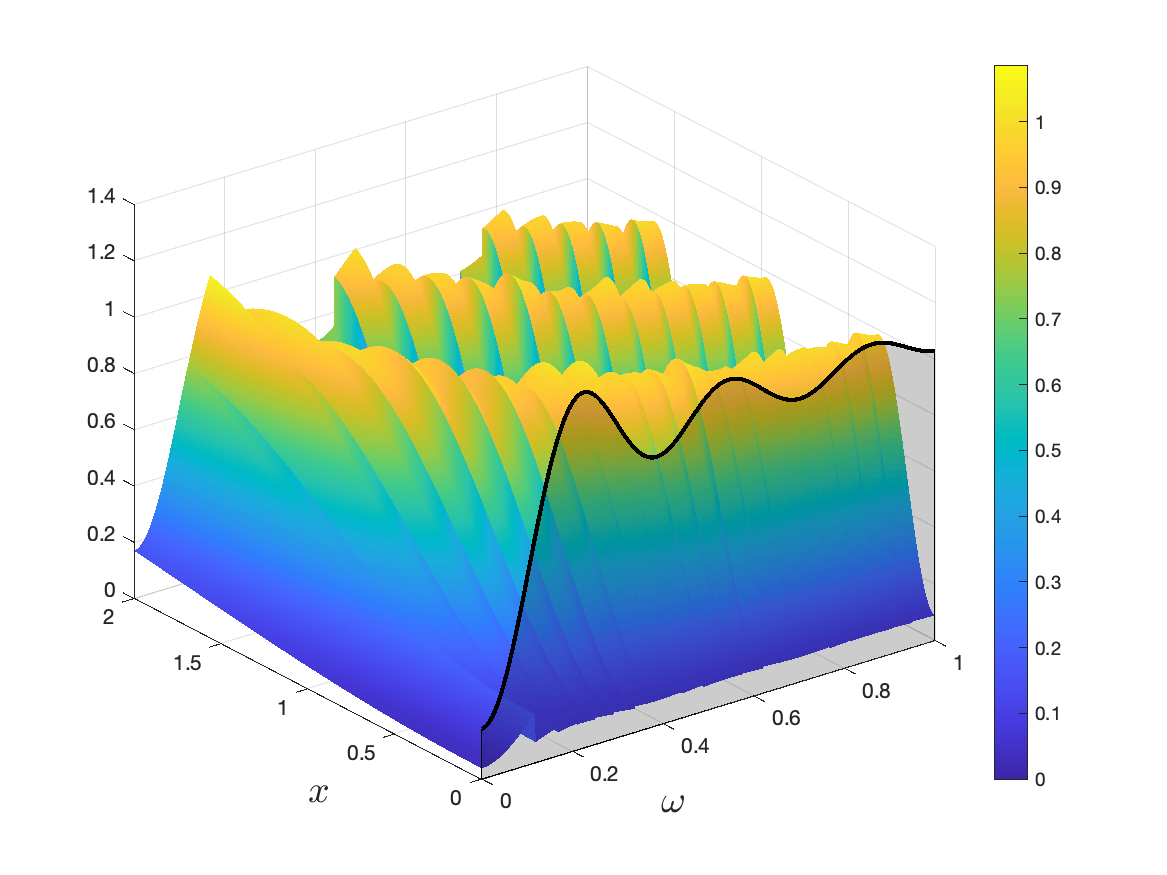}
\caption{{\small \textit{We display a 3D illustration of a certain choice $\Bmu^0 \in \Pcal_{\pi,1}([0,1] \times [0,2])$ (left) and of the corresponding piecewise averaged measure $\bmu^{0,n} \in \Pcal_{\pi,1}([0,1] \times [0,2])$ (right) for $n=20$. Their common marginal $\pi \in \Pcal([0,1])$ is represented by the black curve in the plane $x=0$. Here $\Bmu^0$ is taken to be absolutely continuous with respect to the Lebesgue measure on the product space, and the colour and height of the surfaces represent the densities of both measures with respect to the latter.}}}
\label{fig:mu3D}
\end{figure}

\begin{figure}[!ht]
\centering
\includegraphics[trim =1cm 0cm 1cm 1cm, clip=true, height=0.24\textwidth]{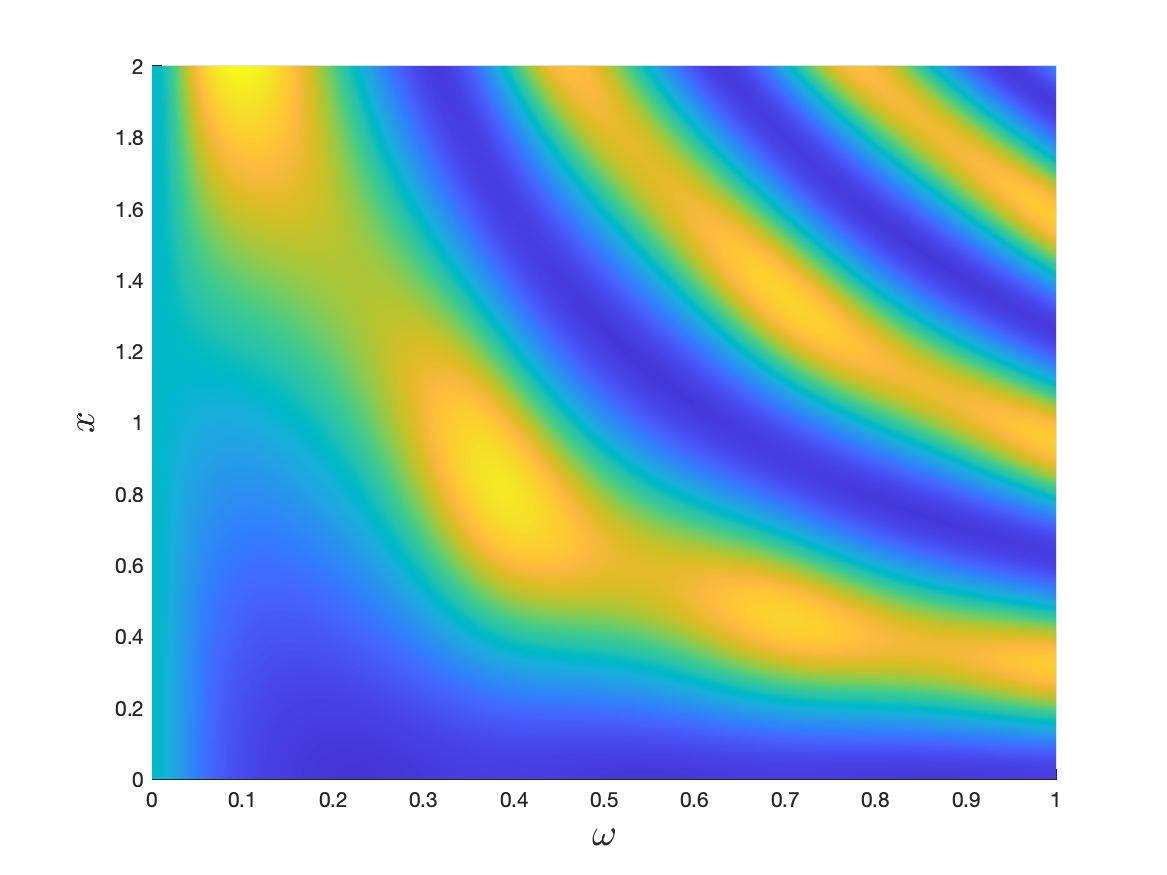}
\includegraphics[trim =1cm 0cm 1cm 1cm, clip=true, height=0.24\textwidth]{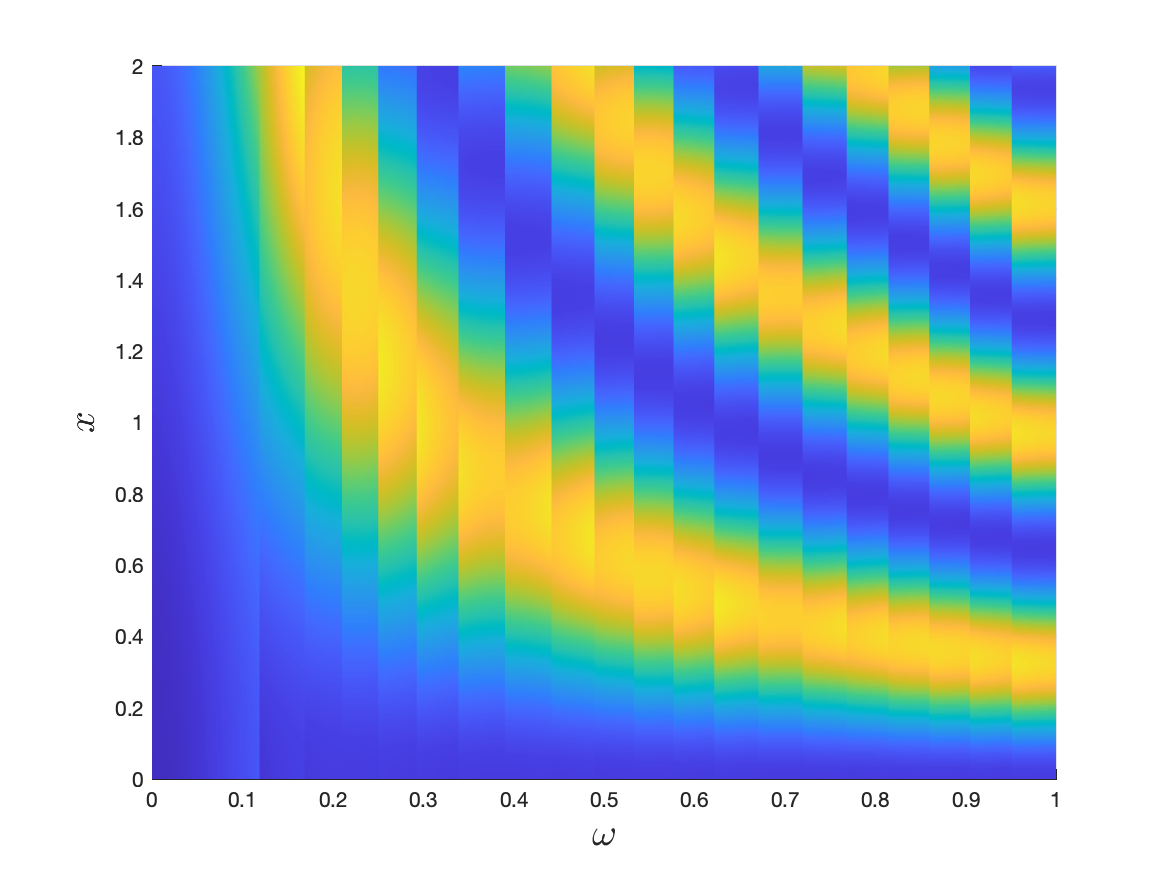}
\includegraphics[trim =17cm 0cm 1cm 1cm, clip=true, height=0.24\textwidth]{mu3D.eps}
\includegraphics[trim =1cm 0cm 1cm 1cm, clip=true, height=0.24\textwidth]{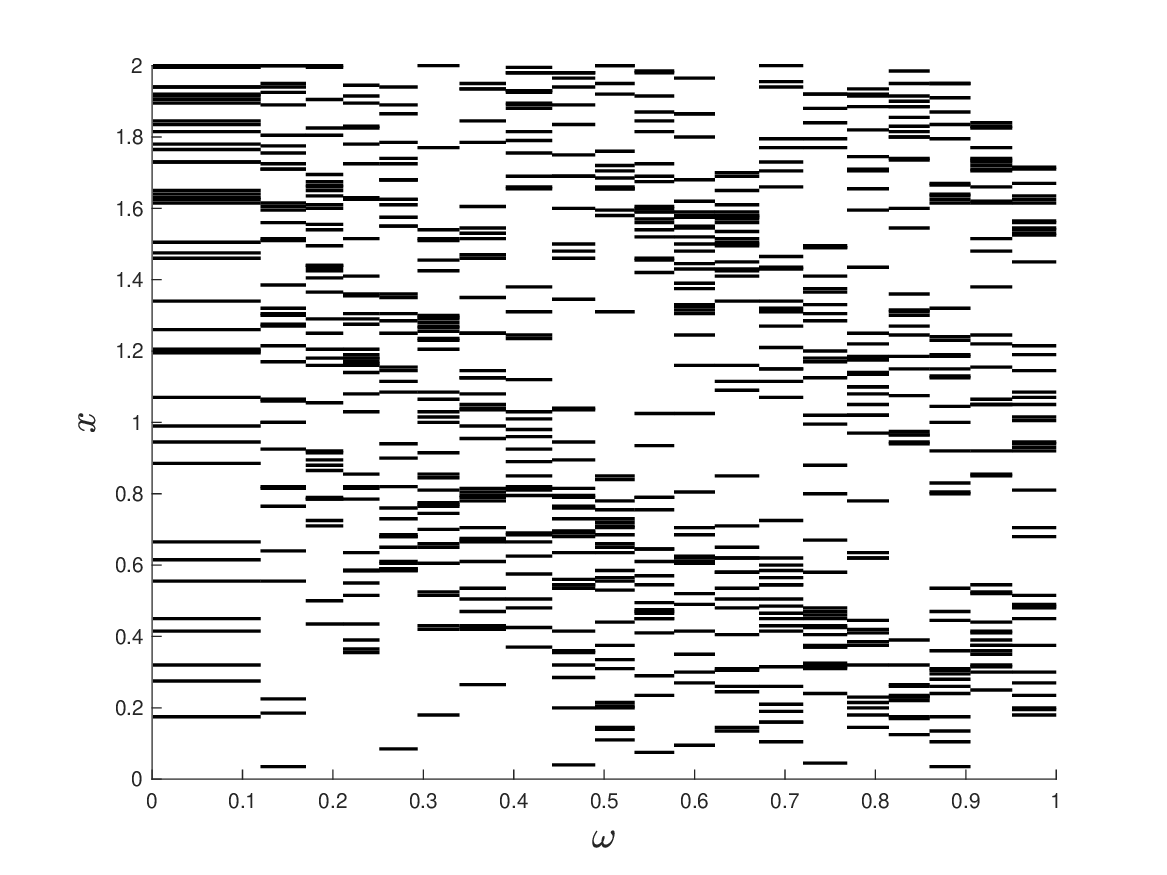}
 \caption{{\small \textit{2D illustration of the same choice $\Bmu^0$ (left) and the corresponding $\bmu^{0,n}$ (centre) with $n=20$, along with a sampling of the random empirical measure $\Bmu^{0,n,m} \in \Pcal_{\pi,1}([0,1] \times [0,2])$ (right) with $m=40$. In the first two plots, the colour represents the densities of $\Bmu^0$ and $\bmu^{0,n}$ with respect to the Lebesgue measure, whereas in the third plot the horizontal black lines symbolise measures of the form $\pi_{\llcorner I_k^n} \times \delta_{x_\kl^{0,\nm}} \in \Mcal^+(I_k^n \times \R)$.}}}
\label{fig:mu2D}
\end{figure}

In order to prove Theorem \ref{thm:ConvergenceMicroMacro}, that is to establish the uniform $W_{\pi,1}$-convergence in expectation of $\Bmu^{n,m}(\cdot)$ towards $\Bmu(\cdot)$ as both $m,n \to +\infty$, we will use the following decomposition
\begin{equation}
\label{eq:WassDecomp}
W_{\pi,1}(\Bmu^\nm(t),\Bmu(t)) \leq W_{\pi,1}(\Bmu^\nm(t),\bmu^\nm(t)) + W_{\pi,1}(\bmu^\nm(t),\bmu^n(t)) +  W_{\pi,1}(\bmu^n(t),\Bmu(t)),
\end{equation}
which is valid for all times $t \in [0,T]$.


\subsection{Convergence of the initial data}
\label{subsection:InitialData}

The aim of this subsection is to prove the convergence of each of the terms appearing in \eqref{eq:WassDecomp} at time $t=0$. Note that since the initial conditions of the two particle systems \eqref{eq:Mic} and \eqref{eq:Mic-aux} are identical, it directly holds that 
\begin{equation*}
W_{\pi,1}(\Bmu^{0,\nm},\bmu^{0,\nm}) = 0,    
\end{equation*}
and we are left with proving the convergence of the terms $W_{\pi,1}(\bmu^{0,\nm},\bmu^{0,n})$ and $W_{\pi,1}(\bmu^{0,n},\Bmu^{0})$, which appear respectively in Lemma \ref{lem:conv-init-nm-n} and
Lemma \ref{lem:conv-init-n} below. To this end, recall that 
\begin{equation}
\label{def:munm0}
\bmu^{0,n} = \sum_{k=1}^n \pi_{\llcorner \Omega_k^n} \times \bigg( \INTDomdash{\mu_{\omega}^0}{\Omega_k^n}{\pi(\omega)} \bigg) \qquad \text{and} \qquad \bmu^{0,n,m} = \sum_{k=1}^n \pi_{\llcorner \Omega_k^n} \times \bigg( \frac{1}{m} \sum_{\ell=1}^m \delta_{\bx_\kl^{0,\nm}} \bigg)
\end{equation}
where for each $k\in \{ 1,\ldots,n\}$, the points $(\bx_\kl^{0,\nm})_{\ell\in\{ 1,\ldots,m\}}$ are i.i.d. random variables with 
\begin{equation*}
\Lpazo \Big( \bx_\kl^{0,\nm} \Big) = \muAvLaw.
\end{equation*}
By construction, both measures have piecewise-constant fibres on the partition $(\Omega_k^n)_{k\in\{1,\cdots,n\}}$, so that 
\begin{equation}
\label{eq:FibredMeasureDiscrepancy}
\begin{aligned}
W_{\pi,1}(\bmu^{0,\nm},\bmu^{0,n}) & = \sum_{k=1}^n \INTDom{W_1(\bar{\mu}_\omega^{0,\nm},\bar{\mu}_\omega^{0,n})}{\Omega_k^n}{\pi(\omega)} \\
& = \sum_{k=1}^n \INTDom{ W_1\left(\frac{1}{m}\sum_{\ell=1}^m \delta_{\bx_\kl^{0,\nm}} \, , \INTDomdash{\mu_\theta^0}{\Omega_k^n}{\pi(\theta)} \right) }{\Omega_k^n}{\pi(\omega)}\\
& = \frac1n \sum_{k=1}^n  W_1 \bigg( \frac{1}{m}\sum_{\ell=1}^m \delta_{\bx_\kl^{0,\nm}}, \muAvLaw \bigg).
\end{aligned}
\end{equation}
Hence, our task boils down to showing that the 1-Wasserstein distance between the probability distribution $\muAvLaw$ and an empirical measure associated with an $m$-sample of the probability distribution vanishes in expectation for each $k \in\{1,\ldots,n\}$. 

The problem of quantifying the convergence of an empirical measure to its associated probability distribution is well-known and has been studied since the 1990s. As such, many quantitative results are available in this context, and we point the reader e.g. to \cite{BoissardLeGouic, BolleyGuillinVillani, FournierGuillin} and the references therein for a survey. Notably in \cite{FournierGuillin}, quantitative rates were established upon assuming moment bounds on the underlying probability distribution. In our context, applying these results to the averages $\muAvLaw \in \mathcal{P}(B(0,r))$ for each $k\in\{1,\cdots,n\}$ takes the following particular form. 

\begin{lem}[Quantitative approximation of the averaged measures in expectation] 
\label{lem:conv-init-nm-n}
Under our working assumptions, there exists a constant $C_d>0$ depending only on $d \geq 1$ such that
\begin{equation*}
\mathbb{E}\Big[W_{\pi,1}(\bmu^{0,\nm},\bmu^{0,n})\Big] \leq 
r C_d 
\left\{
\begin{aligned}
& m^{-1/2} ~~ & \text{if $d = 1$}, \\
& m^{-1/2} \ln(1+m) ~~ & \text{if $d=2$}, \\
& m^{-1/d} ~~ & \text{if $d > 2$}
\end{aligned} 
\right.
\end{equation*}
for each $n \geq 1$. 
\end{lem}

\begin{proof}
Fix an index $k \in \{1,\cdots,n\}$, and notice that since $\Bmu^0\in\mathcal{P}_{\pi}(\Omega\times B(0,r))$ is assumed to have bounded support, it holds that $\muAvLaw\in \mathcal{P}_q( B(0,r))$ for every $q>0$, with 
\begin{equation*}
\begin{aligned}
\Mpazo_q \bigg( \muAvLaw \bigg) & = \bigg( \INTDomdash{\INTDom{|x|^q}{\R^d}{\mu^0_\omega(x)}}{\Omega_k^n}{\pi(\omega)} \bigg)^{1/q} \leq r   
\end{aligned}
\end{equation*}
being bounded independently of $k\in\{1,\cdots,n\}$ and $n \geq 1$. Then, from \cite[Theorem 1]{FournierGuillin}, there exists a constant $C_{d}>0$ depending only on the ambient dimension $d \geq 1$ such that
\begin{equation*}
\mathbb{E}\left[ W_1\bigg(\frac{1}{m}\sum_{\ell=1}^m \delta_{x_\kl^{0,\nm}}, \muAvLaw \bigg) \right] \leq 
r C_d 
\left\{
\begin{aligned}
& m^{-1/2} ~~ & \text{if $d = 1$}, \\
& m^{-1/2} \ln(1+m) ~~ & \text{if $d=2$}, \\
& m^{-1/d} ~~ & \text{if $d > 2$}.
\end{aligned} 
\right.
\end{equation*}
This, together with the estimates from \eqref{eq:FibredMeasureDiscrepancy}, implies that  
\begin{equation*}
\begin{aligned}
\mathbb{E}\Big[ W_{\pi,1}(\bmu^{0,\nm},\bmu^{0,n}) \Big] & = \frac{1}{n} \sum_{k=1}^n \mathbb{E}\left[W_1\left(\frac{1}{m}\sum_{\ell=1}^m \delta_{x_\kl^{0,\nm}}, \INTDomdash{\mu^0_\omega}{\Omega_k^n}{\pi(\omega)} \right)\right]\\
& \leq r C_d \left\{
\begin{aligned}
& m^{-1/2} ~~ & \text{if $d = 1$}, \\
& m^{-1/2} \ln(1+m) ~~ & \text{if $d=2$}, \\
& m^{-1/d} ~~ & \text{if $d > 2$}
\end{aligned} 
\right.
\end{aligned}
\end{equation*}
which concludes the proof. 
\end{proof}

We now prove that the distance between the initial data $\Bmu^0$ and $\bmu^{0,n}$ vanishes as $n \to +\infty$. Note that contrary to the previous situation, both measures are deterministic.  

\begin{lem}[Approximation via conditional expectations in $\Pcal_{\pi,1}(\Omega \times \R^d)$]
\label{lem:conv-init-n}
Given $\Bmu^0\in \mathcal{P}_{\pi}(\Omega\times B(0,r))$ for some $r >0$ and $\bmu^{0,n} \in \Pcal_{\pi}(\Omega \times B(0,r))$ defined as in \eqref{eq:mun0}, it holds that
\begin{equation*}
W_{\pi,1}(\Bmu^0,\bmu^{0,n}) ~\underset{n \to +\infty}{\longrightarrow}~ 0.
\end{equation*}
\end{lem}

\begin{proof}
By definition of the partition $\Ppazo_n :=(\Omega_k^n)_{k\in\{1,\cdots,n\}}$, there exists for each $\omega\in\Omega$ a unique index $\kappa(\omega) \in\{1,\cdots,n\}$ such that $\omega\in\Omega_{\kappa(\omega)}^n$. Then, upon letting 
\begin{equation*}
\Ccal_r := \Big\{\phi\in C^0(B(0,r),\R) ~\,\textnormal{s.t.}~ \Lip(\phi\,;B(0,r)) \leq 1 ~\text{and}~ \phi(0)=0 \Big\},
\end{equation*}
it follows from the definition of the fibred Wasserstein distance and the usual Kantorovich-Rubinstein duality formula recalled in \eqref{eq:KantorovichRubinstein} that
\begin{equation*}
\begin{split}
 W_{\pi,1}(\Bmu^0,\bmu^{0,n}) & = \INTDom{W_1(\mu^0_\omega,\bar\mu^{0,n}_\omega)}{\Omega}{\pi(\omega)} \\
& = \INTDom{ \sup \bigg\{\INTDom{ \phi(x) }{\R^d}{\Big(\mu_{\omega}^0 - \mathsmaller{\INTDomdash{\mu^0_\theta}{\Omega^n_{\kappa(\omega)}}{\pi(\theta)}}\Big)(x)}  ~\,\textnormal{s.t.}~ \phi \in \Ccal_r \Bigg\}}{\Omega}{\pi(\omega)} \\
& = \INTDom{ \sup \bigg\{ \INTDom{ \phi(x) }{\R^d}{\mu^0_\omega(x)}-  \INTDomdash{ \bigg( \INTDom{ \phi(x) }{\R^d}{\mu^0_\theta(x)} \bigg)}{\Omega^n_{\kappa(\omega)}}{\pi(\theta)} ~\,\textnormal{s.t.}~ \phi \in \Ccal_r \bigg\}}{\Omega}{\pi(\omega)}.
\end{split}
\end{equation*}
We now consider the functional $\Ical_{\Bmu} : \Omega \to C^0(\Ccal_r,\R)$ defined by  
\begin{equation*}
\Ical_{\Bmu}(\omega) : \phi \in \Ccal_r \mapsto \INTDom{\phi(x)}{\R^d}{\mu_{\omega}^0(x)} \in \R
\end{equation*}
for $\pi$-almost every $\omega \in \Omega$, and whose average over any subset $\Omega_k^n \subset \Omega$ is given by 
\begin{equation*}
\INTDomdash{\Ical_{\Bmu}(\omega)}{\Omega_k^n}{\pi(\omega)} : \phi \in \Ccal_r \mapsto \INTDomdash{\INTDom{\phi(x)}{\R^d}{\mu_{\omega}^0(x)}}{\Omega_k^n}{\pi(\omega)}  \in \R. 
\end{equation*}
At this stage, one may observe that $\Ccal_r \subset C^0(B(0,r),\R)$ is a compact set by the usual Ascoli-Arzelà theorem, which means that $C^0(\Ccal_r,\R)$ is a Banach space. Thence, it follows from Proposition \ref{prop:Conv_CondExp} that
\begin{equation*}
 W_{\pi,1}(\Bmu^0,\bmu^{0,n}) = \INTDom{\NormC{\Ical_{\Bmu}-\E_{\Ppazo_n}[\Ical_{\Bmu}]}{0}{\Ccal_r,\R}}{\Omega}{\pi(\omega)} ~\underset{n\to\infty}{\longrightarrow}~ 0, 
\end{equation*}
which is precisely the sought-after claim. 
\end{proof}

\begin{rmk}[A more general approximation result via conditional expectations in $\Pcal_{\pi,\,p}(\Omega \times X)$]
\label{rmk:ConditionalApproxWass}
While we chose to provide a fairly elementary proof of the above lemma relying on the Kantorovich-Rubinstein duality formula and Proposition \ref{prop:Conv_CondExp}, one may actually show that 
\begin{equation*}
W_{\pi,\,p}(\Bmu,\E_{\Ppazo_n}[\Bmu]) ~\underset{n \to +\infty}{\longrightarrow}~ 0
\end{equation*}
for each $p \geq 1$ and every $\Bmu \in \Pcal_{\pi,\,p}(\Omega \times X)$, where $(X,\Norm{\cdot}_X)$ is a separable Banach space. To do so, given $\omega \in \Omega$, choose a family $\{\gamma_{\omega,\theta}\}_{\theta \in \Omega_{\kappa(\omega)}^n} \subset \Pcal_p(X^2)$ satisfying $\gamma_{\omega,\theta} \in \Gamma_o(\mu_{\omega},\mu_{\theta})$ for $\pi$-almost every $\theta \in \Omega_{\kappa(\omega)}^n$ and consider the admissible plans 
\begin{equation*}
\bar{\gamma}_{\omega} \in \Gamma\Big( \mu_{\omega} , \mathsmaller{\INTDomdash{\mu_{\theta}}{\Omega_{\kappa(\omega)}^n}{\pi(\theta)}} \Big)
\end{equation*}
defined for $\pi$-almost every $\omega \in \Omega$ through the formula 
\begin{equation*}
\INTDom{\phi(x,y)}{X^2}{\bar{\gamma}_{\omega}(x,y)} := \INTDomdash{\INTDom{\phi(x,y)}{X^2}{\gamma_{\omega,\theta}(x,y)}}{\Omega_{\kappa(\omega)}^n}{\pi(\theta)}   
\end{equation*}
for every Borel map $\phi : X^2 \to [0,+\infty]$. Then, it is enough to observe that 
\begin{equation}
\label{eq:DistanceAverageMeasure}
\begin{aligned}
W_p^p \Big( \mu_{\omega} , \mathsmaller{\INTDomdash{\mu_{\theta}}{\Omega_{\kappa(\omega)}^n}{\pi(\theta)}} \Big) & \leq \INTDom{\|x-y\|_X^p}{X^2}{\bar{\gamma}_{\omega}(x,y)} \\
& = \INTDomdash{\INTDom{\|x-y\|_X^p}{X^2}{\gamma_{\omega,\theta}(x,y)}}{\Omega_{\kappa(\omega)}^n}{\pi(\theta)} \\
& = \INTDomdash{W_p^p(\mu_{\omega},\mu_{\theta})}{\Omega_{\kappa(\omega)}^n}{\pi(\theta)},
\end{aligned}
\end{equation}
and to follow the general strategy subtending Proposition \ref{prop:Conv_CondExp} recalled in Appendix \ref{section:AppendixEquivalenceProba} combined with the abstract metric-valued Lebesgue differentiation theorem of \cite[Theorem 3.4]{Lucic2023} in order to conclude. 
\end{rmk}


\subsection{Convergence of the solutions of the particle systems and structured equations}
\label{subsection:ConvergenceProof}

As mentioned previously, the proof of Theorem \ref{thm:ConvergenceMicroMacro} relies on showing that the three quantities, which correspond to each of the three terms appearing in the right-hand side of \eqref{eq:WassDecomp}, converge to zero. 

The following proposition focuses first on estimating the distance between the curves of empirical measures $\Bmu^{n,m}(\cdot),\bar{\Bmu}^{n,m}(\cdot) \in \AC([0,T],\Pcal_{\pi,1}(\Omega \times B(0,R_r)))$ corresponding to the solutions to the main and auxiliary particle systems \eqref{eq:Mic}-\eqref{eq:Mic-aux}, whose construction is recalled in Definition \ref{def:mumn} above.

\begin{prop}[Distance estimate between the empirical measure curves]
\label{prop:mumn-mumnbar}
Suppose that the assumptions of Theorem \ref{thm:ConvergenceMicroMacro} are satisfied, and let $\Bmu^\nm(\cdot),\bmu^\nm(\cdot) \in \AC([0,T],\Pcal_{\pi,1}(\Omega \times B(0,R_r)))$ be the curves of empirical measures defined in \eqref{eq:emp_Mic} and \eqref{eq:emp_Mic-aux} respectively. Then, it holds that
\begin{equation*}
W_{\pi,1}(\Bmu^\nm(t),\bmu^\nm(t)) \leq \bigg(2 \INTSeg{L_{R_{r}}(s) W_{\pi,1}(\Bmu^\nm(s),\Bmu(s))}{s}{0}{t} + U^n_r \bigg) C_T^2, 
\end{equation*}
for all times $t \in [0,T]$, where $C_T := \exp(\|L_{R_r}(\cdot)\|_1)$. Moreover, it holds that $\lim\limits_{n\to +\infty} U^n_r =0$.
\end{prop}

\begin{proof}
Recalling the definitions \eqref{eq:emp_Mic} and \eqref{eq:emp_Mic-aux} of both curves of empirical measures, their $W_{\pi,1}$-distance can be estimated as 
\begin{equation}
\label{eq:Wpi1_to_l1}
\begin{aligned}
W_{\pi, 1}(\Bmu^{\nm}(t),\bmu^\nm(t)) & = \INTDom{ W_1(\mu^{\nm}_\omega(t),\bar{\mu}^\nm_\omega(t)) }{\Omega}{\pi(\omega)} \\ 
& = \frac{1}{n} \sum_{k=1}^n W_1 \bigg( \frac{1}{m} \sum_{\ell=1}^m \delta_{x_\kl^\nm(t)} \, , \frac{1}{m} \sum_{\ell=1}^m \delta_{\bx_\kl^\nm(t)} \bigg) \\
& \leq \frac{1}{mn}  \sum_{k=1}^n  \sum_{\ell=1}^m \big|x_\kl^\nm(t) - \bx_\kl^\nm(t) \big|
\end{aligned}
\end{equation}
for all times $t  \in [0,T]$, where the third inequality follows simply by considering the suboptimal plans
\begin{equation*}
\frac{1}{m} \sum_{\ell=1}^m \Big( \delta_{x_\kl^\nm(t)} \times \delta_{\bx_\kl^\nm(t)} \Big) \in \Gamma \bigg( \frac{1}{m} \sum_{\ell=1}^m \delta_{x_\kl^\nm(t)} \, , \frac{1}{m} \sum\limits_{\ell=1}^m \delta_{\bx_\kl^\nm(t)} \bigg)
\end{equation*}
for each $m,n \geq 1$. Thus, we are now left with bounding the difference between the two solutions of the microscopic systems \eqref{eq:Mic} and \eqref{eq:Mic-aux}. For each 
$k\in\{1,\cdots,n\}$ and $\ell\in\{ 1,\ldots,m\}$, it holds that 
\begin{equation*}
\begin{aligned}
& \big| x_\kl^\nm(t)  - \bx_\kl^\nm(t) \big| \\
& \hspace{1cm} = \bigg| \INTSeg{v_\kl^N \Big(s,\Bmu^\nm(s),x_\kl^\nm(s)\Big) - v_k^n \Big(s,\bmu^\nm(s),\bx_\kl^\nm(s)\Big)}{s}{0}{t} \, \bigg| \\
& \hspace{1cm} \leq \INTSeg{\Big|\, v_\kl^N \Big(s,\Bmu^\nm(s),x_\kl^\nm(s) \Big) - v_\kl^N\Big(s,\Bmu^\nm(s),\bx_\kl^\nm(s) \Big) \Big|}{s}{0}{t} \\
& \hspace{1.5cm} + \INTSeg{\Big|\, v_\kl^N \Big(s,\Bmu^\nm(s),\bx_\kl^\nm(s) \Big) - v_\kl^N \Big(s,\Bmu(s),\bx_\kl^\nm(s) \Big) \Big|}{s}{0}{t} \\
& \hspace{1.5cm} + \INTSeg{\Big|\, v_\kl^N(s,\Bmu(s),\bx_\kl^\nm(s) \Big) - v_k^n \Big(s,\Bmu(s),\bx_\kl^\nm(s)\Big) \Big|}{s}{0}{t} \\
& \hspace{1.5cm} + \INTSeg{\Big|\, v_k^n(s,\Bmu(s),\bx_\kl^\nm(s) \Big) - v_k^n \Big(s,\Bmu^\nm(s),\bx_\kl^\nm(s) \Big) \Big|}{s}{0}{t} \\
& \hspace{1.5cm} + \INTSeg{\Big|\, v_k^n \Big(s,\Bmu^\nm(s),\bx_\kl^\nm(s) \Big) - v_k^n \Big(s,\bmu^\nm(s),\bx_\kl^\nm(s) \Big) \Big|}{s}{0}{t}.
\end{aligned}
\end{equation*}
At this stage, recall that $x^{\nm}_\kl(t)\in B(0,R_r)$ for all times $t\in [0,T]$ as a consequence of the uniform bound \eqref{eq:AuxSystemMax} of Proposition \ref{prop:ExistUniq-Mic}, where $R_r > 0$ is defined as in Lemma \ref{lem:NonlocalBounds}. Besides, it follows from the local Lipschitz regularity estimate \eqref{eq:AuxSystemLip} of Proposition~\ref{prop:ExistUniq-Mic} combined with \eqref{eq:Wpi1_to_l1} that
\begin{equation}
\label{eq:bound_diffx}
\begin{aligned}
& \big|x_\kl^\nm(t) - \bx_\kl^\nm(t) \big| \\
& \hspace{1cm} \leq \INTSeg{L_{R_r}(s)  \, \big| x_\kl^\nm(s) - \bx_\kl^\nm(s) \big|}{s}{0}{t} \\
& 
\hspace{1.5cm} + \INTSeg{L_{R_r}(s) \bigg( \frac{1}{mn}  \sum_{k'=1}^n  \sum_{\ell'=1}^m \big|x_\klprime^\nm(s) - \bx_\klprime^\nm(s) \big| \bigg)}{s}{0}{t} \\
& \hspace{1.5cm} + \INTSeg{2L_{R_r}(s) W_{\pi, 1}(\Bmu^{\nm}(s),\Bmu(s))}{s}{0}{t} \\
& \hspace{1.5cm} + \INTSeg{ \Big| v_\kl^N \Big(s,\Bmu(s),\bx_\kl^\nm(s) \Big) - v_k^n\Big(s,\Bmu(s),\bx_\kl^\nm(s) \Big) \Big|}{s}{0}{t}.
\end{aligned}
\end{equation}
By summing now over all the indices $(k,\ell) \in \{1,\ldots,n\} \times \{1,\ldots,m\}$, dividing by $N \geq 1$, applying Gr\"onwall's lemma and plugging the resulting estimate in \eqref{eq:Wpi1_to_l1}, we finally obtain that 
\begin{equation}
\label{eq:bound_diffx2}
\begin{aligned}
W_{\pi, 1}(\Bmu^{\nm}(t),\bmu^\nm(t)) & \leq \bigg(2 \INTSeg{L_{R_r}(s) W_{\pi, 1}(\Bmu^{\nm}(s),\Bmu(s))}{s}{0}{t} + U^\nm_r\bigg) \exp \big( 2\|L_{R_{r}}(\cdot)\|_1 \big) \\
& = \bigg(2 \INTSeg{L_{R_r}(s) W_{\pi, 1}(\Bmu^{\nm}(s),\Bmu(s))}{s}{0}{t} + U^\nm_r\bigg) C_T^2
\end{aligned}
\end{equation}
where
\begin{equation*}
U^\nm_r:= \frac{1}{mn}\sum_{k=1}^n\sum_{\ell=1}^m \INTSeg{\Big| v_\kl^N\Big(t,\Bmu(t),\bx_\kl^\nm(t)\Big) - v_k^n\Big(t,\Bmu(t),\bx_\kl^\nm(t)\Big) \Big|}{t}{0}{T}.
\end{equation*}
We focus now on deriving a vanishing upper bound for $U^\nm_r$. To do so, we observe first that, by Theorem~\ref{thm:CompactnessReach}, there exists a compact set $\Kcalb_r \subset \Pcal_{\pi,1}(\Omega\times B(0,R_r))$ in the $W_{\pi,1}$-topology such that $\Bmu(t) \in \Kcalb_r$ for all times $t \in [0,T]$. We consider then the mapping defined by  
\begin{equation}
\label{eq:VcalParticles}
\Vcalb_r(\omega) : (t,\Bmu, x) \in [0,T] \times \Kcalb_r \times B(0,R_r) \mapsto \vb(t,\Bmu,\omega,x) \in \R^d
\end{equation}
for $\pi$-almost every $\omega \in \Omega$, and we let $\Xcalb_r := L^1([0,T],C^0(\Kcalb_r \times B(0,R_r),\R^d))$ to alleviate the notations. Note that the latter is a separable Banach space by the compactness of $\Kcalb_r \times B(0,R_r)$ (see e.g. \cite[Proposition 1.2.29]{AnalysisBanachSpaces}), and that $\Vcalb_r \in L^{\infty}(\Omega,\Xcalb_r;\pi)$ under Hypotheses \ref{hyp:CL}. Besides, its conditional expectation over the partition $\Ppazo_n=(\Omega_k^n)_{k\in\{1,\cdots,n\}}$ writes
\begin{equation}
\label{eq:CondExpV}
\E_{\Ppazo_n}[\Vcalb_r](\omega) : (t,\Bmu, x) \in [0,T] \times \Kcalb_r \times B(0,R_r) \mapsto \sum_{k=1}^n \bigg( \INTDomdash{\vb(t,\Bmu,\theta,x)}{\Omega_k^n}{\pi(\theta)} \bigg) \mathds{1}_{\Omega_k^n}(\omega) \in \R^d
\end{equation}
for $\pi$-almost every $\omega \in \Omega$. Recalling the expressions of the vector fields $v_i^N,v_k^n : [0,T] \times \Pcal_{\pi,1}(\Omega \times \R^d) \times \R^d \to \R^d$ given in \eqref{eq:viNDef} and \eqref{eq:vknDef} respectively, it then holds that 
\begin{equation*}
\begin{split}
U_r^{n,m} & = \frac{1}{mn} \sum_{k=1}^n  \sum_{\ell=1}^m \INTSeg{\Big| v_\kl^N\Big(t,\Bmu(t),\bx_\kl^\nm(t)\Big) - v_k^n\Big(t,\Bmu(t),\bx_\kl^\nm(t)\Big) \Big|}{t}{0}{T} \\
& = \frac{1}{mn} \sum_{k=1}^n  \sum_{\ell=1}^m \INTSeg{\bigg|  \INTDomdash{\hspace{-0.025cm} \vb \Big(t,\Bmu(t),\omega,\bx_\kl^\nm(t)\Big)}{\Omega^N_\kl} {\pi(\omega)} \\
& \hspace{5cm} - \INTDomdash { \vb \Big(t,\Bmu(t),\theta,\bx_\kl^\nm(t) \Big)}{\Omega^n_k}{\pi(\theta)} \, \bigg|}{t}{0}{T} \\
& \leq \sum_{k=1}^n \sum_{\ell=1}^m \INTDom{ \bigg( \INTSeg{\bigg| \, \vb \Big(t,\Bmu(t),\omega,\bx_\kl^\nm(t) \Big) \\
& \hspace{5cm}- \INTDomdash {\vb \Big( t,\Bmu(t),\theta,\bx_\kl^\nm(t) \Big)}{\Omega^n_k}{\pi(\theta)} \, \bigg|}{t}{0}{T} \bigg)}{\Omega^N_\kl}{\pi(\omega)} \\
& \leq \INTDom{\bigg( \INTSeg{\hspace{-0.025cm} \sup_{(\Bmu,x) \in \Kcalb_r \times B(0,R_r)} \bigg| \, \vb( t,\Bmu,\omega,x) - \sum_{k=1}^n \bigg( \INTDomdash{\vb(t,\Bmu,\theta,x)}{\Omega_k^n}{\pi(\theta)} \bigg) \mathds{1}_{\Omega_k^n}(\omega) \, \bigg|}{t}{0}{T} \bigg)}{\Omega}{\pi(\omega)} \\
& = \INTDom{\Norm{\Vcalb_r(\omega) - \E_{\Ppazo_n}[\Vcalb_r](\omega)}_{\Xcalb_r}}{\Omega}{\pi(\omega)}
\end{split}
\end{equation*}
for each $m,n \geq 1$, where we used Fubini's theorem along with the expression \eqref{eq:CondExpV} of the conditional expectation of $\Vcalb_r \in L^{\infty}(\Omega,\Xcalb_r;\pi)$. In what follows, we thus set
\begin{equation}
\label{eq:UnT}
U^n_r := \INTDom{\Norm{\Vcalb_r(\omega) - \E_{\Ppazo_n}[\Vcalb_r](\omega)}_{\Xcalb_r}}{\Omega}{\pi(\omega)}, 
\end{equation}
and it directly stems from Proposition \ref{prop:Conv_CondExp} above that 
\begin{equation*}
U^n_r = \| \Vcalb_r - \E_{\Ppazo_n}[\Vcalb_r] \|_{L^1(\Omega,\Xcalb_r;\pi)} ~\underset{n \to +\infty}{\longrightarrow}~ 0,
\end{equation*}
which together with \eqref{eq:bound_diffx2}, concludes the proof.
\end{proof}


Coming back to \eqref{eq:WassDecomp}, our goal now is to bound the distance between the curves of measures $\bmu^\nm(\cdot),\bmu^n(\cdot) \in \AC([0,T],\Pcal_{\pi,1}(\Omega \times B(0,R_r)))$, whose constructions are recalled in Definition \ref{def:mumn} above. Note that following Lemma~\ref{lem:SolutionAuxiliary}, both curves solve the auxiliary continuity equation \eqref{eq:Transport-aux} driven by the nonlocal field $\vb^n : [0,T] \times \Pcal_{\pi,1}(\Omega \times \R^d) \times \Omega \times \R^d \to \R^d$ given in \eqref{eq:vn}, which allows to derive the following stability estimate. 

\begin{prop}[Distance estimate between the solutions of the auxiliary dynamics]
\label{prop:munm-mun}
Suppose that the assumptions of Theorem \ref{thm:ConvergenceMicroMacro} are satisfied, and let $\bmu^n(\cdot),\bmu^\nm(\cdot) \in \AC([0,T],\Pcal_{\pi,1}(\Omega\times B(0,R_r)))$ be given as in Definition \ref{def:mumn}. Then, it holds that 
\begin{equation*}
W_{\pi,1}(\bmu^\nm(t),\bmu^n(t)) \leq W_{\pi,1}(\bmu^{0,\nm},\bmu^{0,n}) \exp \Big( (1+C_T)\|L_{R_r}(\cdot)\|_1 \Big)
\end{equation*}
for all times $t\in [0,T]$, where $C_T := \exp \big( \|L_{R_r}(\cdot)\|_1 \big)$.
\end{prop}

\begin{proof}
Since $\bmu^n(\cdot),\bmu^\nm(\cdot) \in \AC([0,T],\Pcal_{\pi,1}(\Omega\times B(0,R_r)))$ are two solutions of \eqref{eq:Transport-aux}, it follows from the stability estimate of Proposition \ref{prop:Stability} that
\begin{equation*}
\begin{split}
&W_{\pi,1}(\bmu^\nm(t),\bmu^n(t)) \\
& \leq C_T \bigg( W_{\pi,1}(\bmu^{0,\nm},\bmu^{0,n}) + \INTSeg{\INTDom{\NormC{\vb^n(s,\bmu^\nm(s),\omega) - \vb^n(s,\bmu^n(s),\omega)}{0}{B(0,R_r),\R^d}}{\Omega}{\pi(\omega)}}{s}{0}{t} \bigg),
\end{split}
\end{equation*}
for all times $t \in [0,T]$, where $C_T > 0$ is defined above. Besides, recalling that $\supp(\bmu^\nm(t)) \cup \supp(\bmu^n(t)) \subset \Omega \times B(0,R_r)$, it follows from Hypothesis \ref{hyp:CL}-$(iii)$ that
\begin{equation*}
\Big|\vb^n(t,\bmu^\nm(t),\omega,x) - \vb^n(t,\bmu^n(t),\omega,x) \Big|\leq L_{R_r}(t) W_{\pi,1}(\bmu^\nm(t),\bmu^n(t))
\end{equation*}
for $\Lcal^1 \times \pi$-almost every $(t,\omega) \in [0,T] \times \Omega$ and all $x \in B(0,R_r)$. Upon combining the previous two estimates and applying Gr\"onwall's lemma, we straightforwardly get
\begin{equation*}
W_{\pi,1}(\bmu^\nm(t),\bmu^n(t)) \leq W_{\pi,1}(\bmu^{0,\nm},\bmu^{0,n}) \exp \Big( (1+C_T) \|L_{R_r}(\cdot)\|_1 \Big)
\end{equation*}
which was the sought-after estimate. 
\end{proof}

Lastly, we derive a bound on the third term in the right-hand side of \eqref{eq:WassDecomp}, which measures the distance between the solution of the auxiliary continuity equation $\bmu^n(\cdot) \in \AC([0,T],\Pcal_{\pi,1}(\Omega \times \R^d))$ and the initial macroscopic curve $\Bmu(\cdot) \in \AC([0,T],\Pcal_{\pi,1}(\Omega \times \R^d))$.

\begin{prop}[Distance estimate between the solutions of the auxiliary and initial dynamics]
\label{prop:mun-mu}
Suppose that the assumptions of Theorem \ref{thm:ConvergenceMicroMacro} are satisfied, and let $\Bmu(\cdot),\bmu^n(\cdot) \in \AC([0,T],\Pcal_{\pi,1}(\Omega\times B(0,R_r)))$ be given as in Definition \ref{def:mumn}. Then, it holds that 
\begin{equation*}
W_{\pi,1}(\bmu^n(t),\Bmu(t)) \leq \Big(W_{\pi,1}(\bmu^{0,n},\Bmu^{0}) + U^n_r\Big) \exp \Big((1+C_T) \|L_{R_r}(\cdot)\|_1 \Big), 
\end{equation*}
for all times $t\in [0,T]$, where $C_T := \exp \big( \|L_{R_r}(\cdot) \|_1 \big)$ and $U^n_r \geq 0$ is given as in Proposition \ref{prop:mumn-mumnbar}.
\end{prop}

\begin{proof}
Recall at first that $\Bmu(\cdot),\bmu^n(\cdot) \in \AC([0,T],\Pcal_{\pi,1}(\Omega\times B(0,R_r)))$ are both solutions of the nonlocal continuity equations \eqref{eq:TransportGraphParticles}, driven respectively by the vector fields $\vb,\vb^n : [0,T] \times \Pcal_{\pi,1}(\Omega \times \R^d) \times \Omega \times \R^d \to \R^d$. Observe in addition that both satisfy Hypotheses \ref{hyp:CL}, which allows us to infer from the stability estimate of Proposition \ref{prop:Stability} that 
\begin{equation*}
\begin{aligned}
& W_{\pi,1}(\Bmu(t),\bmu^n(t)) \\
& \leq C_T \bigg( W_{\pi,1}(\Bmu^0,\bmu^{0,n}) + \INTSeg{\INTDom{ \NormC{\vb(s,\Bmu(s),\omega) - \vb^n(s,\bmu^n(s),\omega)}{0}{B(0,R_r),\R^d}}{\Omega} {\pi(\omega)} }{s}{0}{t} \bigg) \\
& \leq  C_T \bigg( W_{\pi,1}(\Bmu^0,\bmu^{0,n}) + \INTSeg{L_{R_r}(s) W_{\pi,1}(\Bmu(s),\bmu^n(s))}{s}{0}{t} \\
& \hspace{3.65cm} + \INTSeg{\INTDom{ \NormC{\vb(s,\Bmu(s),\omega) - \vb^n(s,\Bmu(s),\omega)}{0}{B(0,R_r),\R^d}}{\Omega} {\pi(\omega)}}{s}{0}{t} \bigg)
\end{aligned}
\end{equation*}
for all times $t \in [0,T]$, where $C_T > 0$ is defined above. Moreover, recall that by Theorem~\ref{thm:CompactnessReach} there exists a compact set $\Kcalb_r \subset \Pcal_{\pi,1}(\Omega\times\R^d)$ for the $W_{\pi,1}$-topology such that $\Bmu(t) \in \Kcalb_r$ for all times $t \in [0,T]$. Upon using this fact in the previous inequality and applying Gr\"onwall's lemma, we get that
\begin{equation*}
\begin{aligned}
W_{\pi,1}(\Bmu(t),\bmu^n(t)) \leq \bigg( W_{\pi,1}(\Bmu^0,\bmu^{0,n}) + \INTSeg{\INTDom{ \sup_{(\Bmu,x) \in \Kcalb_r \times B(0,R_r)} \Big|\, \vb(t,\Bmu,\omega,x) - \vb^n(t,\Bmu,\omega,x) \Big|}{\Omega} {\pi(\omega)}}{t}{0}{T} \bigg) & \\
\times \exp \Big((1+C_T) \|L_{R_r}(\cdot)\|_1 \Big) & 
\end{aligned}
\end{equation*}
for all times $t \in [0,T]$. At this stage, it is enough to recall the definition \eqref{eq:vn} of $\vb^n : [0,T] \times \Pcal_{\pi,1}(\Omega \times \R^d) \times \Omega \times \R^d \to \R^d$ as the conditional expectation of $\vb : [0,T] \times \Pcal_{\pi,1}(\Omega \times \R^d) \times \Omega \times \R^d \to \R^d$ over $\Ppazo_n = (\Omega_k^n)_{k \in \{1,\ldots,n\}}$ and apply Fubini's theorem to infer that
\begin{equation*} 
\begin{aligned}
& \INTSeg{\INTDom{ \sup_{(\Bmu,x) \in \Kcalb_r \times B(0,R_r)} \Big|\, \vb(t,\Bmu,\omega,x) - \vb^n(t,\Bmu,\omega,x) \Big|}{\Omega} {\pi(\omega)}}{t}{0}{T} =~ \NormL{\hspace{-0.01cm}\Vcalb_r - \E_{\Ppazo_n}[\Vcalb_r]}{1}{\Omega,\Xcalb_r;\pi} ~= U^n_r, 
\end{aligned}
\end{equation*}
where $\Xcalb_r := L^1([0,T],C^0(\Kcalb_r \times B(0,R_r),\R^d))$ and $\Vcalb_r \in L^{\infty}(\Omega,\Xcalb_r;\pi)$ was defined via \eqref{eq:VcalParticles} in the proof of Proposition~\ref{prop:mumn-mumnbar} above. In summary, we have shown that
\begin{equation*}
W_{\pi,1}(\bmu^n(t),\Bmu(t)) \leq \Big(W_{\pi,1}(\bmu^{0,n},\Bmu^{0}) + U^n_r\Big) \exp \Big((1+C_T) \|L_{R_r}(\cdot)\|_1 \Big),
\end{equation*}
which closes the proof.  
\end{proof}

All these preliminary results being laid down, we are finally ready to prove Theorem \ref{thm:ConvergenceMicroMacro}.

\begin{proof}[Proof of Theorem~\ref{thm:ConvergenceMicroMacro}]
Putting together the estimates derived in Proposition \ref{prop:mumn-mumnbar}, Proposition \ref{prop:munm-mun} and Proposition \ref{prop:mun-mu} within the inequality \eqref{eq:WassDecomp} above, we have that
\begin{equation*}
\begin{split}
W_{\pi,1}(\Bmu^\nm(t),\Bmu(t)) & \leq  W_{\pi,1}(\Bmu^\nm(t),\bmu^\nm(t)) + W_{\pi,1}(\bmu^\nm(t),\bmu^n(t)) +  W_{\pi,1}(\bmu^n(t),\Bmu(t))\\
& \leq \bigg(2 \INTSeg{L_{R_{r}}(s) W_{\pi,1}(\Bmu^\nm(s),\Bmu(s))}{s}{0}{t} + U^n_r \bigg) C_T^2  \\
& \hspace{0.45cm} + \bigg( W_{\pi,1}(\bmu^{0,\nm},\bmu^{0,n}) + W_{\pi,1}(\bmu^{0,n},\Bmu^{0}) + U^n_r \bigg) \exp \Big((1+C_T) \|L_{R_r}(\cdot)\|_1 \Big)
\end{split}
\end{equation*}
for all times $t \in [0,T]$, where we recall that $C_T := \exp(\|L_{R_r}(\cdot)\|_1)$. From Gr\"onwall's lemma, we therefore obtain that 
\begin{equation}
\label{eq:ParticleConvergenceLast0}
\begin{split}
W_{\pi,1} (\Bmu^\nm(t),\Bmu(t)) & \leq U^n_r \bigg( C_T^2 + \exp \Big((1+C_T) \|L_{R_r}(\cdot)\|_1 \Big) \bigg) \exp\Big( 2 C_T^2 \|L_{R_r}(\cdot)\|_1\Big)\\
& \hspace{0.45cm} + \bigg( W_{\pi,1}(\bmu^{0,\nm},\bmu^{0,n}) + W_{\pi,1}(\bmu^{0,n},\Bmu^{0}) \bigg) \exp \Big((1+C_T+2C_T^2) \|L_{R_r}(\cdot)\|_1 \Big) \\
& \leq \bigg( 2 U^n_r + W_{\pi,1}(\bmu^{0,\nm},\bmu^{0,n}) + W_{\pi,1}(\bmu^{0,n},\Bmu^{0}) \bigg) \\
& \hspace{1.5cm} \times \exp \Big((1+C_T+2C_T^2) \|L_{R_r}(\cdot)\|_1 \Big)
\end{split}
\end{equation}
for all times $t \in [0,T]$. In order to conclude, recall that we have shown in Lemma \ref{lem:conv-init-nm-n} above that 
\begin{equation}
\label{eq:ParticleConvergenceLast1}
\E \Big[W_{\pi,1}(\bmu^{0,\nm},\bmu^{0,n}) \Big] ~\underset{m \to +\infty}{\longrightarrow}~ 0
\end{equation}
uniformly with respect to $n \geq 1$.  Moreover, it follows from Proposition \ref{prop:mumn-mumnbar} that 
\begin{equation}
\label{eq:ParticleConvergenceLast2}
U^n_r ~\underset{n \to +\infty}{\longrightarrow}~ 0,
\end{equation}
whereas Lemma \ref{lem:conv-init-n} entails
\begin{equation}
\label{eq:ParticleConvergenceLast3}
W_{\pi,1}(\bmu^{0,n},\Bmu^{0}) ~\underset{n \to +\infty}{\longrightarrow}~ 0. 
\end{equation}
Thus, by plugging \eqref{eq:ParticleConvergenceLast1}, \eqref{eq:ParticleConvergenceLast2} and \eqref{eq:ParticleConvergenceLast3} in \eqref{eq:ParticleConvergenceLast0}, we finally conclude that 
\begin{equation*}
\lim_{m,n\to+\infty} \E\bigg[\hspace{0.02cm} \sup_{t \in [0,T]}W_{\pi,1}(\Bmu^{n,m}(t),\Bmu(t)) \bigg] = 0,
\end{equation*}
which closes the proof of Theorem \ref{thm:ConvergenceMicroMacro}. 
\end{proof}


\subsection{A quantitative version of Theorem \ref{thm:ConvergenceMicroMacro}}
\label{subsection:QuantitativeConv}

Going back to Definition \ref{def:mumn} and the construction of the various intermediate measure curves detailed therein, it can be seen that, amongst the three terms appearing in the right-hand side of the crucial estimate \eqref{eq:ParticleConvergenceLast0}, one vanishes quantitatively by Lemma \ref{lem:conv-init-nm-n}, while the convergence to zero of the other two hinges upon approximations by finite conditional expectations, and so ultimately on Lebesgue's differentiation theorem. This suggests that it is possible to establish quantitative versions of Theorem \ref{thm:ConvergenceMicroMacro}, at the price of imposing enough regularity in $\omega \in \Omega$ on the initial data and driving fields to have access to quantitative versions of the differentiation theorems. 

For the sake of both practical relevance and simplicity, we prove one such corollary in the particular case in which $(\Omega,\Apazo) := ([0,1],\Lpazo([0,1]))$ is endowed with $\pi \in \Pcal([0,1])$ satisfying
\begin{equation*}
\supp(\pi) = [0,1] \qquad \text{and} \qquad \pi \ll \Lcal^1.
\end{equation*}
In this particular -- albeit fairly standard -- context, we suppose that the maps encoding the initial data and driving field are of \textit{bounded variations}, see e.g. \cite{Ambrosio1990} and \cite[Section 3.2]{AmbrosioFuscoPallara}. 

\begin{Def}[Metric-valued functions of bounded variations]
We say that a map $f : [0,1] \to X$ valued in a Polish space $(X,\dsf_X(\cdot,\cdot))$ has \textnormal{bounded variations} over an interval $[a,b] \subset [0,1]$ provided
\begin{equation*}
\mathrm{Var}(f;[a,b]) := \sup \Bigg\{ \sum_{i=0}^{N-1} \dsf_X(f(\omega_i),f(\omega_{i+1})) ~\,\textnormal{s.t.}~ a = \omega_0 < \ldots < \omega_N = b ~~\text{and}~~ N \geq 1 \Bigg\} < +\infty. 
\end{equation*}
Moreover, given any disjoint partition of $[0,1]$ into subintervals $(I_k^n)_{k \in \{1,\ldots,n\}}$, it holds that
\begin{equation*}
\sum_{k=1}^n \Var(f ; I_k^n) \leq \Var(f ; [0,1]).
\end{equation*}
\end{Def}

\begin{cor}[A quantitative particle approximation result]
\label{cor:QuantitativeConv}
Suppose that the assumptions of Theorem \ref{thm:ConvergenceMicroMacro} are satisfied, and posit in addition that both $\Bmu^0 : \omega \in [0,1] \mapsto \mu_{\omega}^0 \in \Pcal_1(B(0,r))$ and $\Vcalb_r: [0,1] \to \Xcalb_r$ defined in \eqref{eq:VcalParticles} have bounded variations in the above sense. Then, it holds that 
\begin{equation*}
W_{\pi,1}(\bar{\Bmu}^{0,n},\Bmu^0) \leq \frac{\mathrm{Var}(\Bmu^0;[0,1])}{n} \qquad \text{and} \qquad \NormL{\Vcalb_r - \E_{\Ppazo_n}[\Vcalb_r]}{1}{[0,1],\Xcalb_r} ~\leq \frac{\Var(\Vcalb_r;[0,1])}{n} 
\end{equation*}
for each $n \geq 1$. In particular, setting $m := n^{\max\{2,d\}}$ while recalling that $N = mn$, one has that
\begin{equation}
\label{eq:QuantitativeParticleApprox}
\begin{aligned}
& \E\bigg[\hspace{0.02cm} \sup_{t \in [0,T]}W_{\pi,1}(\Bmu^{N}(t),\Bmu(t)) \bigg] \\
& \hspace{0.25cm} \leq \;  
\left\{ 
\begin{aligned}
& D_r \Big(2 \Var(\Vcalb_r;[0,1]) + \Var(\Bmu^0;[0,1]) \Big) N^{-1/3} + r C_d D_r N^{-1/3} & \text{if $d = 1$}, \\
& D_r \Big(2 \Var(\Vcalb_r;[0,1]) + \Var(\Bmu^0;[0,1]) \Big) N^{-1/3} + r C_d D_r N^{-1/3} \ln(1+N^{2/3}) & \text{if $d = 2$},\\
& D_r \Big(2 \Var(\Vcalb_r;[0,1]) + \Var(\Bmu^0;[0,1]) \Big) N^{-1/(d+1)} + r C_d D_r N^{-1/(d+1)} ~~ & \text{if $d > 2$}, \\
\end{aligned}
\right. &
\end{aligned}
\end{equation}
where $D_r := \exp \big((1+C_T+2C_T^2) \|L_{R_r}(\cdot)\|_1 \big)$.  
\end{cor}

\begin{proof}
Under our working assumptions, it can be checked straightforwardly that
\begin{equation*}
\begin{aligned}
\NormL{\Vcalb_r - \E_{\Ppazo_n}[\Vcalb_r]}{1}{\Omega,\Xcalb_r\;\pi} & = \sum_{k=1}^n \INTDom{\Big\|\, \Vcalb_r(\omega) - \mathsmaller{\INTDomdash{\Vcalb_r(\theta)}{I_k^n}{\pi(\theta)}} \, \Big\|_{\Xcalb_r}}{I_k^n}{\pi(\omega)} \\
& \leq \sum_{k=1}^n \INTDom{\INTDomdash{\big\|\, \Vcalb_r(\omega) - \Vcalb_r(\theta) \big\|_{\Xcalb_r}}{I_k^n}{\pi(\theta)}}{I_k^n}{\pi(\omega)} \\
& \leq \frac{1}{n} \sum_{k=1}^n \Var(\Vcalb_r ; I_k^n) \\
& \leq \frac{\Var(\Vcalb_r ; [0,1])}{n}, 
\end{aligned}
\end{equation*}
and upon leveraging the estimate \eqref{eq:DistanceAverageMeasure} from Remark \ref{rmk:ConditionalApproxWass}, it can be shown in the same way that 
\begin{equation*}
W_{\pi,1}(\Bmu^0,\bar{\Bmu}^{0,n}) \leq \frac{\Var(\Bmu^0  ; [0,1])}{n}.
\end{equation*}
Therefore, by plugging both inequalities in \eqref{eq:ParticleConvergenceLast0} and setting $m := n^{\max\{2,d\}}$ while observing consequently that $n = N^{1/\max\{3,d+1\}}$ and $m = N^{\max\{2,d\}/\max\{3,d+1\}}$, we directly get the quantitative estimate \eqref{eq:QuantitativeParticleApprox}. 
\end{proof}

We conclude this section by providing basic examples of vector fields and initial conditions satisfying the assumptions of Corollary \ref{cor:QuantitativeConv}. Concerning the vector field, we go back to the usual nonexchangeable kernel
\begin{equation*}
\vb(t,\Bmu,\omega,x) := \INTDom{w(t,\omega,\theta) \Psi(x,y)}{\Omega \times \R^d}{\Bmu(\theta,y)}
\end{equation*}
with $w(\cdot) \in L^{\infty}([0,T] \times [0,1] \times [0,1],\R;\Lcal^1 \times \pi \times \pi)$ and $\Psi \in \Lip(\R^{2d},\R^d)$. We show that the latter is such that $\omega \in [0,1] \mapsto \Vcalb_r(\omega) \in \Xcalb_r$ has bounded variations provided that $w_{t,\theta} : \omega \in [0,1] \to w(t,\omega,\theta) \in \R$ has bounded variations for $\Lcal^1 \times \pi$-almost every $(t,\theta) \in [0,T] \times [0,1]$, with 
\begin{equation*}
\INTSeg{\INTDom{\Var(w_{t,\theta} ; [0,1])}{[0,1]}{\pi(\theta)}}{t}{0}{T} < +\infty.
\end{equation*}
Indeed, given any subdivision $0 = \omega_0 < \ldots < \omega_N =1$ of $[0,1]$, one has that 
\begin{equation*}
\begin{aligned}
& \sum_{i=0}^{N-1} \Norm{\Vcalb_r(\omega_{i+1})- \Vcalb_r(\omega_i)}_{\Xcalb_r} \\
& = \sum_{i=0}^{N-1} \INTSeg{\hspace{-0.05cm} \sup_{(\Bmu,x) \in \Kcalb_r \times B(0,R_r)} \bigg| \INTDom{\Big( w(t,\omega_{i+1},\theta) - w(t,\omega_i,\theta) \Big) \Psi(x,y)}{[0,1] \times \R^d}{\Bmu(\theta,y)} \bigg|}{t}{0}{T} \\
& \leq \NormC{\Psi}{0}{B(0,R_r)^2,\R^d} \INTSeg{\INTDom{\sum_{i=0}^{N-1} \Big| w(t,\omega_{i+1},\theta) - w(t,\omega_i,\theta) \Big|}{[0,1]}{\pi(\theta)}}{t}{0}{T} \\
& \leq \NormC{\Psi}{0}{B(0,R_r)^2,\R^d} \INTSeg{\INTDom{\Var(w_{t,\theta};[0,1])}{[0,1]}{\pi(\theta)}}{t}{0}{T} <+\infty
\end{aligned}
\end{equation*}
and we may conclude by taking the supremum over all such subdivisions that 
\begin{equation*}
\begin{aligned}
\Var(\Vcalb_r;[0,1]) & = \sup \Bigg\{ \sum_{i=0}^{N-1} \Norm{\Vcalb_r(\omega_{i+1}) - \Vcalb_r(\omega_i)}_{\Xcalb_r} ~\textnormal{s.t.}~ 0 = \omega_0 < \ldots < \omega_N = 1 ~~\text{and}~~ N \geq 1 \Bigg\} \\
& \leq \NormC{\Psi}{0}{B(0,2R_r)^2,\R^d} \INTSeg{\INTDom{\Var(w_{t,\theta};[0,1])}{[0,1]}{\pi(\theta)}}{t}{0}{T} <+\infty.
\end{aligned}
\end{equation*}
Concerning the initial condition, a similar example is provided by measures of the form 
\begin{equation*}
\Bmu^0 := \rho \cdot \Big( \Lcal^1 \times \Lcal^d_{\llcorner B(0,r)} \Big)
\end{equation*}
where $\rho \in L^1([0,1] \times B(0,r),\R_+)$ is such that $(\pfrak_{[0,1]})_{\sharp} \Bmu^0 = \Lcal_{\llcorner [0,1]}^1$ and $\rho_x : \omega \in [0,1] \mapsto \rho(\omega,x) \in \R$ has bounded variations for $\Lcal^d$-almost every $x \in B(0,r)$, with 
\begin{equation*}
\INTDom{\Var(\rho_x \,;[0,1])}{B(0,r)}{x} < +\infty. 
\end{equation*}
Indeed, given any subdivision $0 = \omega_0 < \ldots < \omega_N =1$ of $[0,1]$ and any tuple of functions $\{\phi_i\}_{i=0}^{N-1} \subset \Lip(\R^d,\R)$ with $\Lip(\phi_i\,;\R^d) \leq 1$ and $\phi_i(0) = 0$, one has that 
\begin{equation*}
\begin{aligned}
\sum_{i=0}^{N-1} \bigg| \INTDom{\phi_i(x)}{B(0,r)}{(\mu_{\omega_{i+1}} - \mu_{\omega_i})(x)} \bigg| & = \sum_{i=0}^{N-1} \bigg| \INTDom{\phi_i(x) \Big( \rho(\omega_{i+1},x) - \rho(\omega_i,x) \Big)}{B(0,r)}{x} \, \bigg| \\
& \leq r \INTDom{\sum_{i=0}^{N-1} \big| \rho(\omega_{i+1},x) - \rho(\omega_i,x) \big|}{B(0,r)}{x} \\
& \leq r \INTDom{\Var(\rho_x ;[0,1])}{B(0,r)}{x} < +\infty, 
\end{aligned}
\end{equation*}
and by taking the supremum over all possible subdivisions and tuples of functions satisfying the above conditions, we deduce from the Kantorovich-Rubinstein duality formula \eqref{eq:KantorovichRubinstein} that 
\begin{equation*}
\begin{aligned}
\mathrm{Var}(\Bmu^0;[0,1]) & = \sup \Bigg\{ \sum_{i=0}^{N-1} W_1(\mu_{\omega_i},\mu_{\omega_{i+1}}) ~\,\textnormal{s.t.}~ 0 = \omega_0 < \ldots < \omega_N = 1 ~~\text{and}~~ N \geq 1 \Bigg\} \\
& \leq r \INTDom{\Var(\rho_x ;[0,1])}{B(0,r)}{x} < +\infty, 
\end{aligned}
\end{equation*}
which was precisely the desired estimate. 


\section{Applications of the theory}
\label{section:Applications}

\setcounter{equation}{0} \renewcommand{\theequation}{\thesection.\arabic{equation}}

Due to its generality, the class of dynamics studied throughout the article subsumes many existing frameworks and models for large limits of particle systems, while allowing to consider potential extensions and new models. In what follows, we provide an overview of some of the most relevant applications of the theory, which depend both on the choice of reference measure $\pi \in \Pcal(\Omega)$ and the shape of the driving field $\vb : [0,T] \times \Pcal_{\pi}(\Omega \times \R^d) \times \Omega \times \R^d \to \R^d$. 


\subsection{Particular choices of reference measure}

As explained above, the reference measure $\pi \in \Pcal(\Omega)$ is encoding the distribution of the labels of the agents. It allows in particular to consider populations with heterogeneous weights, in which the tagged particles (or populations) do not have the same influence on the rest of the system. We discuss below some canonical examples in the particular case in which $(\Omega,\Bcal(\Omega)) := ([0,1],\Bcal([0,1]))$.

\paragraph*{Lebesgue measure.} The most common choice of reference measure in meanfield models is to let $\pi := \Lcal^1_{\llcorner [0,1]} \in \Pcal(\Omega)$ be the Lebesgue measure on $[0,1]$. From a modelling point of view, this choice amounts to considering that all particles have essentially the same influence on the rest of the population. As such, the Lebesgue measure appears naturally in large-population limits of equally weighted particle systems, see for instance \cite{ChibaMedvedev19,JabinPoyatoSoler2025,KaliuzhnyiMedvedev2018}. A first and natural variant of this model is to take 
\begin{equation*}
\pi := \rho \cdot \Lcal^1_{\llcorner [0,1]} \in \Pcal([0,1])
\end{equation*}
for some probability density $\rho \in L^1([0,1],\R_+)$. This amounts to assuming that the system is heterogeneous, but in a relatively mild sense. Indeed, although some groups of agents may have a stronger influence on the overall dynamics than others, their individual influence remains vanishingly small.   


\paragraph*{Finite sum of Dirac masses.} Another interesting and standard choice of label space and reference measure is to let $\pi := \frac{1}{N}\sum_{i=1}^N \delta_{\omega_i} \in \Pcal([0,1])$ for some $N \geq 1$ and finite collection $(\omega_1,\ldots,\omega_N) \in [0,1]^N$. With this choice, the dynamics in \eqref{eq:IntroFibredCont} can be rewritten as the system of coupled PDEs 
\begin{equation*}
\partial_t \mu_i(t) + \Div_x \big( v_i(t,\Bmu(t),x) \mu_i(t)\big) = 0, 
\end{equation*}
where the driving fields are given by  $v_i(t,\Bmu,x) := \vb(t,\Bmu,\omega_i,x)$ and $\mu_i : [0,T] \to \Pcal(\R^d)$ represents the spatial density of the $i$-th species for each $i\in\{1,\ldots,N\}$, whose evolution also depends on the densities of other species. Such systems are frequently used to model multispecies interactions, see for instance \cite{DoumicHechtPerthamePeurichard2024} and references therein. Additionally, in the particular case in which the initial measures are Dirac masses for each $i\in\{1,\cdots,N\}$, we obtain a system of coupled ODEs, i.e. a finite-dimensional interacting particle system of the form
\begin{equation*}
\dot{x}_i(t) =  v_i(t,\Bmu(t),x_i(t)).
\end{equation*}
Therein, the measures $\Bmu(t)\in \Pcal([0,1]\times\R^d)$ are the associated nonexchangeable empirical densities, given by $\Bmu(t) := \frac{1}{N}\sum_{i=1}^N \delta_{(\omega_i,x_i(t))}$ for all times $t \in [0,T]$.


\paragraph*{Leader-follower dynamics.}

Lastly, it is possible to consider that  the reference measure is given as the sum of the Lebesgue measure and a finite sum of Dirac masses, that is for instance
\begin{equation*}
\pi := \frac{1}{2M}\sum_{i=1}^M \delta_{\omega_i} + \Lcal^1_{\llcorner[1/2,1]} \in \Pcal([0,1])
\end{equation*}
for some $(\omega_1,\ldots,\omega_M) \in [0,1/2]^M$. This choice allows to model mixed multispecies populations composed of leaders (represented by the Dirac masses), whose influence on the system is strong, and followers (represented by the Lebesgue measure), whose influence is evenly distributed and vanishingly small for each individual agent. A particular case of interest is that in which the leader populations are concentrated on point masses, that is $\Bmu^0 = \Bmu^0_L+\Bmu_F^0 \in \Pcal_{\pi}([0,1] \times \R^d)$ with 
\begin{equation*}
\Bmu_L^0 :=  \frac{1}{2M}\sum_{i=1}^M \delta_{(\omega_i,x_i^0)} \qquad \text{and} \qquad \Bmu_F^0 := \INTDom{\mu^0_{\omega}}{[1/2,1]}{\omega}. 
\end{equation*}
Then, upon choosing a vector field of the form
\begin{equation*}
\begin{aligned}
\vb(t,\Bmu,\omega,x) & := \mathds{1}_{[0,1/2]}(\omega) \bigg(v(t,x) + \INTDom{\mathds{1}_{[1/2,1]}(\theta) \Psi(x-y)}{\Omega \times \R^d}{\Bmu(\theta,y)} + u(t,\omega) \bigg) \\
& \hspace{0.45cm} + \mathds{1}_{[1/2,1]}(\omega) \INTDom{\Psi(x-y)}{\Omega \times \R^d}{\Bmu(\theta,y)}
\end{aligned}
\end{equation*}
for some Lipschitz kernel $\Psi:\R^d\rightarrow \R^d$ and control function $u(\cdot) \in L^{\infty}([0,T] \times [0,1],\R^d;\Lcal^1 \times \pi_{\llcorner [0,1/2]}) \simeq L^{\infty}([0,T],\R^d)^M$, we end up with the coupled leader-follower dynamics 
\begin{equation*}
\left\{
\begin{aligned}
& \dot x_i(t) = v(t,x_i(t)) + \INTDom{\Psi(x_i(t)-y)}{\R^d}{\mu_F(t)(y)} + u_i(t) \\
& \partial_t \Bmu_F(t) + \Div_x \bigg( \Big(\Psi \star_x \Bmu_L(t) +\Psi \star_x \Bmu_F(t) \Big) \Bmu_F(t) \bigg) = 0, 
\end{aligned}
\right. 
\end{equation*}
of the same general form as those studied e.g. in \cite{Bongini2017,FornasierPR2014}, where we introduced the notations $\mu_F(t) := (\pfrak_{\R^d})_{\sharp} \Bmu_F(t)$ for all times $t \in [0,T]$ and $u_i(\cdot) := u(\cdot,\omega_i) \in L^{\infty}([0,T],\R^d)$ for each $i \in \{1,\dots,M\}$. Our framework could clearly handle nonexchangeable versions of this model, up to incorporating explicit label dependences in the dynamics. 


\subsection{Particular choices of vector fields}

Complementarily to the choice of the reference measure, the evolution written in \eqref{eq:IntroFibredCont} encompasses a wide range of different models depending on the form of the vector field $\vb : [0,T] \times \Pcal_{\pi}(\Omega \times \R^d) \times \Omega \times \R^d \to \R^d$. We provide below a non-exhaustive -- albeit quite significant -- overview of popular models of nonexchangeable multiagent dynamics that fit within our framework for a general label space $\Omega$. 


\paragraph*{Nonexchangeable particle systems on dense graphs.}

The prototypical example of structured continuity equation arises from taking the large-population limit of a multiagent system in which the interactions are constrained by an underlying graph $G_N := \langle V_N, E_N, W_N \rangle$, which can be both weighted and directed. Therein, the (fixed) set of vertices $V_N := \{1,\ldots,N\}$ represents the labels of the agents, whereas the sets of weights $W_N := (w_{ij}^N)_{i,j\in\{1,\ldots,N\}} \in \R_+^{N \times N}$ and edges $E_N := \big\{ (i,j)\in \{1,\ldots,N\}^2 ~\, \textnormal{s.t.}~ \; w_{ij}^N \neq 0\big\}$ encode the direction and magnitude of their interactions. The states $(x_i(\cdot))_{i\in\{1,\ldots,N\}}\in C^0([0,T],(\R^d)^N)$ of the agents evolve according to a coupled dynamics of the form
\begin{equation}
\label{eq:ODE-pairwiseinteraction}
\dot x_i(t) = \sum_{j=1}^N w_{ij}^N \Psi(x_i(t),x_j(t))
\end{equation}
for $i\in\{1,\ldots,N\}$, where $\Psi\in \Lip(\R^d\times\R^d, \R^d)$ is an interaction function. The meanfield limit of such nonexchangeable particle systems on graphs was studied in many different frameworks, see for example \cite{ChibaMedvedev19, GkogkasKuehn2022,JabinPoyatoSoler2025,KaliuzhnyiMedvedev2018, KuehnXu2022,Paul2024}. When the sequence of graphs $(G_N)_{N \geq 1}$ is dense, which corresponds to the situation in which $w_{ij}^N = O(1/N)$ for each $i,j \in \{1,\ldots,N\}$, the macroscopic population density $\Bmu(\cdot)\in C^0([0,T],\Pcal_{\pi}(\Omega\times\R^d))$ can be shown to satisfy a nonlocal structured continuity equation of the form \eqref{eq:IntroFibredCont}, in which the velocity field is given explicitly by
\begin{equation*}
\vb(t,\Bmu,\omega,x) := \INTDom{w(\omega,\theta) \Psi(x,y)} {\Omega\times\R^d}{\Bmu(\theta,y)}. 
\end{equation*}
Therein, the map $w \in L^\infty(\Omega\times\Omega,\R)$ is the graphon obtained as the limit as $N \to +\infty$ of the sequence $(G_N)_{N \geq 1}$ of dense graphs, see e.g. \cite{Lovasz2006}.


\paragraph*{Nonexchangeable particle systems on graphs with intermediate density.}

In~\cite{GkogkasKuehn2022} and several related works, the meanfield limit of a system of the form \eqref{eq:ODE-pairwiseinteraction} was derived for a Kuramoto-type dynamics on the torus $\T:=[0,2\pi]$, in the context of graph sequences $(G_N)_{N \geq 1}$ of intermediate density whose limits are called \textit{graphops}, see again the survey \cite{AyiPouradierDuteil2024}. In this model, when the initial datum $\Bmu^0 \in \Pcal_{\pi}(\Omega \times \T)$ has a density $\rho^0 \in L^1(\Omega \times \T,\R_+;\pi \times \Lcal_{\llcorner \T})$, the macroscopic curve $\Bmu : [0,T] \to \Pcal_{\pi}(\Omega \times \T)$ retains a density $\rho(t) \in L^1(\Omega \times \T,\R_+;\pi \times \Lcal^1_{\llcorner \T})$ for all times $t \in [0,T]$, which solves a limit equation of the form \eqref{eq:IntroFibredCont} where the driving field is given by 
\begin{equation*}
\vb(t,\Bmu,\omega,x) := \INTDom{\INTDom{\psi(x,y) \rho(t,\theta,y)}{\Omega}{\nu_\omega(\theta)}}{\T}{y}.
\end{equation*}
Our framework applies to the subclass in which the elements of $\{\nu_\omega\}_{\omega\in\Omega} \subset \Pcal(\Omega)$ associated with the graphop obtained as the limit of the graph sequence are absolutely continuous with respect to the reference measure $\pi \in \Pcal(\Omega)$. 


\paragraph{Nonexchangeable particle systems with pairwise interactions}
A slight generalisation of the microscopic model presented in \eqref{eq:ODE-pairwiseinteraction} above and studied in \cite{Paul2024}, is that in which the pairwise interactions between particles are not decoupled into graph-based interactions and state-based interactions, but fully intertwined. Under suitable assumptions -- including the density of the underlying graphs --, the meanfield limit of such a particle system can be shown to satisfy \eqref{eq:IntroFibredCont} driven by the velocity field 
\begin{equation*}
\vb(t,\Bmu,\omega,x) := \INTDom{\BPsi(\omega,\theta,x,y)}{\Omega\times\R^d}{\Bmu(\theta,y)}, 
\end{equation*}
where the mapping $\BPsi : \Omega \times \Omega \times \R^d \times \R^d \rightarrow \R^d$ is an interaction kernel taking into account the coupled effects of both labels and state variables.


\paragraph{Nonexchangeable particle systems on hypergraphs.}

Another meaningful generalisation of the nonexchangeable dynamics \eqref{eq:ODE-pairwiseinteraction} involves \textit{hypergraphs}, see for instance \cite{AyiPouradierDuteilPoyato2024,Paul2025}. The latter are a new and convenient framework to model multiagent systems in which the interactions between particles are not binary, but occur instead between groups of individuals of arbitrary sizes $\ell\geq 1$. The time-evolution of such systems can be written as
\begin{equation*}
\dot x_i(t) = \sum_{\ell=1}^{N-1} \sum_{j_1=1}^N\cdots\sum_{j_\ell=1}^N  w^N_{ij_1\cdots j_\ell} \Psi_\ell \big(x_i(t),x_{j_1}(t),\cdots, x_{j_\ell}(t) \big)
\end{equation*}
for each $i\in\{1,\ldots,N\}$, wherein the $(w^N_{i j_1,\ldots,j_{\ell}})$ are tensors of dimension $(\ell+1)$ characterising the weight of the $\ell$-dimensional hyperedges, and $\Psi_\ell \in \Lip((\R^d)^{\ell+1},\R^d)$ encodes the interactions between $x_i(t)$ and the $\ell$ other agents $(x_{j_1}(t),\ldots,x_{j_{\ell}}(t))$ at time $t \geq 0$. As shown in \cite{AyiPouradierDuteilPoyato2024}, as the number of particles tends to infinity, the particle density can be shown to converge (under suitable scaling and regularity conditions) to the solution of a nonlocal structured continuity equation of the form \eqref{eq:IntroFibredCont}, in which the velocity field is given by 
\begin{equation*}
\vb(t,\Bmu,\omega,x) = \sum_{\ell=1}^{+\infty} \int_{(\Omega\times\R^d)^\ell} w_\ell(\omega,\omega_1,\cdots, \omega_\ell) \Psi_\ell(x,x_1,\cdots,x_\ell) \,\mathrm{d}\Bmu(\omega_1,x_1) \ldots \mathrm{d}\Bmu(\omega_\ell,x_\ell), 
\end{equation*}
where for each $\ell \geq 1$, the map $w_\ell\in L^\infty(\Omega^{\ell+1};\R_+)$ is the so-called hypergraphon obtained as the limit of the (dense) hypergraph sequence $(w_{ij_1,\ldots j_{\ell}}^N)_{N \geq 1}$ as $N \to +\infty$. 





\paragraph{Cell population with phenotypic heterogeneity.}

In \cite{CarrilloLorenziMacfarlane25}, the authors propose a model for an interacting system of $n \geq 1$ populations of cells, whose densities $\rho_i : [0,T] \times \R^d \to \R_+$ satisfy
\begin{equation*}
\partial_t \rho_i(t,x) - \Div_x \Big( m_i \rho_i(t,x) \nabla_x p(t,x) \Big) = G_i(p(t,x))\rho_i(t,x)
\end{equation*}
for each $i\in\{1,\cdots,n\}$. Therein, $p(t,x) := \sum_{i=1}^n w_i \rho_i(t,x)$ represents the cellular pressure, which is defined as a sum
of the cell densities weighted by coefficients $(w_i)_{i \in \{1,\ldots,n\}}$, and the reaction terms $G_i : \R_+ \to \R$ represent the net growth rate of each subpopulation. In order to model a continuum $\Omega$ of cell phenotypes rather than the discrete set $\{1,\cdots,n\}$, one may let the number of cell populations $n$ tend to infinity, and would expect -- up to neglecting the population growth and setting $G_i := 0$ -- that the limiting cell density $\boldsymbol{\rho} : [0,T] \times \Omega \times \R^d \to \R_+$ satisfies an equation of the form \eqref{eq:TransportGraphNonlocal}, with
\begin{equation*}
\vb(t,\boldsymbol{\rho},\omega,x) = -m(t,\omega)\nabla_x\left(\INTDom {w(\omega) \boldsymbol{\rho}(t,\omega,x)}{\Omega} {\omega}\right).
\end{equation*}


\paragraph{Heterogeneous Michaelis-Menten kinetics}

Enzymatic reactions are usually modelled via the so-called Michaelis-Menten kinetics, which allow to take into account the saturation of the reaction speed for a large concentration of substrate. Such models can also be used to describe the evolution of neuron oscillators \cite{GuRohling19}. Some recent applied works such as \cite{DouglasCarterWills24} have proposed to study heterogeneous enzyme (or neuronal oscillator) populations, in which the Michaelis constant $k$ and the limiting rate $\alpha$ are allowed to vary depending on the enzyme. For a set of $N \geq 1$ enzymes (or neuronal oscillators), such equations take the following form: 
\[
\dot x_i(t) = \frac{1}{N}\sum_{j=1}^N \frac{\alpha_{ij} x_j(t)}{k_{ij}+x_j(t)} + F_i\bigg(\frac{1}{N}\sum_{j=1}^N \beta_j x_j(t)\bigg), 
\]
where $(k_{ij})_{i,j\in\{1,\cdots,N\}}$ represent the heterogeneous Michaelis constants, $(\alpha_{ij})_{i,j\in\{1,\cdots,N\}}$ the heterogeneous limiting reaction rates, and the saturation functions $F_i:\R_+\rightarrow\R_+$ are again given by
$
F_i(r) = g_i r/(a_i+r)
$
for some non-negative coefficients $(g_i)_{i\in\{1,\cdots,N\}}$ and $(a_i)_{i\in\{1,\cdots,N\}}$. Denoting by $\Bmu(t) \in \Pcal_{\pi}(\Omega \times \R_+)$ the probability of finding enzymes (or neuronal oscillators) of type $\omega \in \Omega$ with concentration $x \in \R_+$, the meanfield formulation would be of the type \eqref{eq:TransportGraphNonlocal}, in which 
\begin{equation*}
\vb(t,\Bmu,\omega,x) = \INTDom{\frac{\alpha(\omega,\theta) y}{k(\omega,\theta) + y}}{\Omega\times\R_+} {\Bmu(t)(\theta,y)} + F\bigg(\omega, \INTDom{\beta(\theta) y}{\Omega\times\R_+}{\Bmu(t)(\theta,y)}\bigg),
\end{equation*}
for some $\alpha, k\in L^\infty(\Omega^2;\R_+)$ and $\beta, g, a\in L^\infty(\Omega;\R_+)$, with the function $F:\Omega\times\R_+\rightarrow\R_+$ being given by $F(\omega,r) = g(\omega) r/(a(\omega)+r)$.

{\small
\paragraph*{Acknowledgements:} 
The authors acknowledge the support of the French Agence Nationale de la Recherche (ANR) under the grant ANR-24-CE40-4644-01 (project FISH).
The authors also wish to thank dearly Emmanuel Trélat for the numerous fruitful discussions we had on the topic, notably concerning the quantification of the approximation result of Section \ref{Section:ParticleApprox}. We thank David Poyato for likewise stimulating exchanges and for pointing us towards many precious references on fibred probability spaces, as well as Juli\'an Cabrera-Nyst for reading the first version of this manuscript.
}

{\small
\paragraph*{AI-Assisted Tool Disclosure:} 
In accordance with the Leiden Declaration on AI and Mathematics, we precisely disclose the role of artificial intelligence in the preparation of this work. We began working on this manuscript in 2022, and the first version of this paper was published on-line in November 2025 before the advent of widely available powerful LLMs. As such, the choice of problems, core mathematical ideas, and proof techniques presented in this manuscript were solely chosen  and developed by the authors. In October 2026, the first author performed several iterations of ChatGPT-6 Astra to proof-check the revised version of the manuscript. Through this process, the AI tool helped flag minor technical slips and typos, spotted a misquotation of the quantisation rates of \cite{FournierGuillin}, and helped simplify the argument subtending the converse implication of Proposition \ref{prop:Representative}. It also informed us that a convergence step in the proof of the classical superposition principle from \cite{AmbrosioC2014} was slightly flawed as written, and suggested an adequate fix that we independently verified. The authors take full and exclusive responsibility for the correctness of the results and the contents of the paper.
}


\addcontentsline{toc}{section}{Appendices}
\section*{Appendices}


\setcounter{Def}{0} 
\setcounter{section}{0}
\renewcommand{\thesection}{A} 
\renewcommand{\thesubsection}{A} 

\subsection{Proof of Proposition \ref{prop:Kantorovich}}
\label{section:AppendixKanto}

\setcounter{Def}{0} \renewcommand{\thethm}{A.\arabic{Def}} 
\setcounter{equation}{0} \renewcommand{\theequation}{A.\arabic{equation}}

In this first appendix, we prove the Kantorovich-Rubinstein duality formula for the $1$-fibred Wasserstein distance exposed in Proposition \ref{prop:Kantorovich}.

\begin{proof}[Proof of Proposition \ref{prop:Kantorovich}]
To begin with, note that since $K \subset X$ is a subset of a Polish space, the classical Kantorovich duality formula recalled in \eqref{eq:KantorovichRubinstein} above implies that
\begin{equation*}
\begin{aligned}
W_{\pi,1}(\Bmu,\Bnu) & = \INTDom{W_1(\mu_{\omega},\nu_{\omega})}{\Omega}{\pi(\omega)} \\
& = \INTDom{\sup \bigg\{ \INTDom{\phi(x)}{K}{(\mu_{\omega} - \nu_{\omega})(x)} ~\,\textnormal{s.t.}~ \Lip(\phi\,;K) \leq 1 \bigg\}}{\Omega}{\pi(\omega)}, 
\end{aligned}
\end{equation*}
so the representation formula we are after boils down to exchanging the sup and the integral in the above equation. To this end, we let $\Ipazo_{\Bmu,\Bnu} : \Omega \times C^0(K,\R) \to \R$ be the Carathéodory map defined by  
\begin{equation*}
\Ipazo_{\Bmu,\Bnu}(\omega,\phi) := \INTDom{\phi(x)}{K}{(\mu_{\omega} - \nu_{\omega})(x)}
\end{equation*}
for $\pi$-almost every $\omega \in \Omega$ and each $\phi \in C^0(K,\R)$, and set 
\begin{equation*}
\Scal_{\Bmu,\Bnu}(\omega) := \sup \Big\{\Ipazo_{\Bmu,\Bnu}(\omega,\phi) ~\,\textnormal{s.t.}~ \Lip(\phi\,;K) \leq 1 \Big\}.
\end{equation*}
Note that the latter map is $\pi$-measurable as a consequence e.g. of \cite[Theorem 8.2.11]{Aubin1990}, but also $\pi$-integrable by construction, and such that
\begin{equation*}
\sup \bigg\{ \INTDom{\INTDom{\varphi(\omega,x)}{K}{(\mu_{\omega}-\nu_{\omega})(x)}}{\Omega}{\pi(\omega)} ~\,\textnormal{s.t.}~ \varphi \in \Lip_{\pi,1}(\Omega \times K,\R) \bigg\} \leq \INTDom{\Scal_{\Bmu,\Bnu}(\omega)}{\Omega}{\pi(\omega)}.
\end{equation*}
We then follow the argument e.g. from \cite[Theorem 14.60]{Rockafellar}, and show that for each $\epsilon > 0$, there exists a map $\varphi_{\epsilon} \in \Lip_{\pi,1}(\Omega \times K,\R)$ for which
\begin{equation}
\label{eq:AppendixKantoIneq}
\INTDom{\Scal_{\Bmu,\Bnu}(\omega)}{\Omega}{\pi(\omega)} - \epsilon \leq \INTDom{\INTDom{\varphi_{\epsilon}(\omega,x)}{K}{(\mu_{\omega}-\nu_{\omega})(x)}}{\Omega}{\pi(\omega)}. 
\end{equation}
To do so, take any strictly positive function $\Tpazo \in L^1(\Omega,\R,\pi)$ -- whose existence is ensured e.g. by the fact that $(\Omega,\Apazo,\pi)$ is a probability space --, and given $\delta > 0$, let 
\begin{equation*}
\Scal_{\Bmu,\Bnu}^{\delta} := \Scal_{\Bmu,\Bnu} - \delta \Tpazo.
\end{equation*}
Then, one may easily check that $\Scal_{\Bmu,\Bnu}^{\delta}(\omega) < \Scal_{\Bmu,\Bnu}(\omega)$ for $\pi$-almost every $\omega \in \Omega$, with
\begin{equation*}
\INTDom{\Scal_{\Bmu,\Bnu}^{\delta}(\omega)}{\Omega}{\pi(\omega)} ~\underset{\delta \to 0^+}{\longrightarrow}~ \INTDom{\Scal_{\Bmu,\Bnu}(\omega)}{\Omega}{\pi(\omega)}.
\end{equation*}
In particular, for each $\epsilon >0$, we may choose some $\delta_{\epsilon} > 0$ small enough so that
\begin{equation*}
\INTDom{\Scal_{\Bmu,\Bnu}(\omega)}{\Omega}{\pi(\omega)} - \epsilon \leq \INTDom{\Scal_{\Bmu,\Bnu}^{\delta_{\epsilon}}(\omega)}{\Omega}{\pi(\omega)}. 
\end{equation*}
To conclude, there remains to observe that the mapping
\begin{equation*}
\Jpazo_{\Bmu,\Bnu} : (\omega,\phi) \in \Omega \times C^0(K,\R) \mapsto \left\{
\begin{aligned}
& \INTDom{\phi(x)}{K}{(\mu_{\omega} - \nu_{\omega})(x)} ~~ & \text{if $\Lip(\phi\,;K) \leq 1$}, \\
& - \infty ~~ & \text{otherwise},  
\end{aligned}
\right.
\end{equation*}
is $\pi$-measurable in $\omega \in \Omega$ as well as upper-semicontinuous and concave in $\phi \in C^0(K,\R)$, which, combined with what precedes, implies that the sets 
\begin{equation*}
\Jcal_{\Bmu,\Bnu}(\omega) := \bigg\{\phi \in C^0(K,\R) ~\,\textnormal{s.t.}~ \Lip(\phi\,;K) \leq 1 ~~\text{and}~ \Scal_{\Bmu,\Bnu}^{\delta_{\epsilon}}(\omega) \leq \Jpazo_{\Bmu,\Bnu}(\omega,\phi) \bigg\} 
\end{equation*}
are nonempty and closed for $\pi$-almost every $\omega \in \Omega$. Whence, it follows from the selection principle e.g. of \cite[Theorem 8.2.10]{Aubin1990} that there exists a measurable map $\omega \in \Omega \mapsto \varphi_{\epsilon}(\omega) \in \Jcal_{\Bmu,\Bnu}(\omega)$, which precisely amounts to saying that $\varphi_{\epsilon} \in \Lip_{\pi,1}(\Omega \times K,\R)$ satisfies \eqref{eq:AppendixKantoIneq}, and closes the proof. 
\end{proof}


\setcounter{Def}{0} 
\setcounter{section}{0}
\renewcommand{\thesection}{B} 
\renewcommand{\thesubsection}{B} 

\subsection{Proof of Theorem \ref{thm:Superposition}}
\label{section:AppendixSuperposition}

\setcounter{Def}{0} \renewcommand{\thethm}{B.\arabic{Def}} 
\setcounter{equation}{0} \renewcommand{\theequation}{B.\arabic{equation}}

In this second appendix section, we prove the superposition result for structured transport presented in Theorem \ref{thm:Superposition}. As explained hereinabove, its proof is based on a careful adaptation of the classical arguments detailed in \cite[Section 3]{AmbrosioC2014}. We recall following e.g. \cite[Section 2.1]{AmbrosioFuscoPallara} that the convolution between a vector-valued Radon measure $\Bnu \in \Mcal(\R^d,\R^m)$ and a map $\rho \in C^0_b(\R^d,\R)$ is defined by 
\begin{equation*}
(\Bnu \star \rho)(x) := \INTDom{\rho(x-y)}{\R^d}{\Bnu(y)}
\end{equation*}
for all $x \in \R^d$. We first prove a technical lemma concerning the representation of solutions to local structured continuity equations, which can be seen as a companion result to Proposition \ref{prop:Wellposed}.  

\begin{lem}[Representation formula for structured continuity equations]
\label{lem:AppendixSup}
Fix some $\tau \in [0,T]$ and let $\vb : [0,T] \times \Omega \times \R^d \to \R^d$ be a Carathéodory vector field satisfying 
\begin{equation}
\label{eq:SupLemBound1}
\INTSeg{\sup_{x \in K} |\vb(t,\omega,x)|}{t}{\tau}{T} < +\infty \qquad \text{and} \qquad \INTSeg{\Lip(\vb(t,\omega) \, ; K)}{t}{\tau}{T} < +\infty
\end{equation}
for $\pi$-almost every $\omega \in \Omega$ and each compact set $K \subset \R^d$. Moreover, let $\Bmu : [\tau,T] \to \Pcal_{\pi}(\Omega \times \R^d)$ be a Borel solution of \eqref{eq:FibredCauchyTer} for which
\begin{equation}
\label{eq:SupLemBound2}
\INTSeg{\INTDom{\frac{|\vb(t,\omega,x)|}{1+|x|}}{\Omega \times \R^d}{\Bmu(t)(\omega,x)}}{t}{\tau}{T} < +\infty. 
\end{equation}
Then, the characteristic equations 
\begin{equation*}
\Phi_{(\tau,t)}^{\omega}(x) = x + \INTSeg{\vb \Big( s , \omega , \Phi_{(\tau,s)}^{\omega}(x) \Big)}{s}{\tau}{t}
\end{equation*}
admit a unique globally defined solution for $\Bmu^{\tau}$-almost every $(\omega,x) \in \Omega \times \R^d$, and it holds that
\begin{equation*}
\Bmu(t) = \INTDom{\Phi_{(\tau,t) \sharp \,}^{\omega} \mu_{\omega}^{\tau}}{\Omega}{\pi(\omega)}
\end{equation*}
for all times $t \in [\tau,T]$. 
\end{lem}

\begin{proof}
By Theorem \ref{thm:Equivalence} above, we know that the disintegrated maps $\{\mu_{\omega}(\cdot)\}_{\omega \in \Omega}$ representing the curve $\Bmu : [\tau,T] \to \Pcal_{\pi}(\Omega \times \R^d)$ solve the system of continuity equations
\begin{equation*}
\left\{
\begin{aligned}
& \partial_t \mu_{\omega}(t) + \Div_x (\vb(t,\omega) \mu_{\omega}(t)) = 0, \\
& \mu_{\omega}(\tau) = \mu_{\omega}^{\tau},
\end{aligned}
\right.
\end{equation*}
for $\pi$-almost every $\omega \in \Omega$. Then, it can be checked by a direct adaptation of the arguments detailed in Proposition \ref{prop:Wellposed}  above and \cite[Proposition 8.1.8]{AGS}  (see also \cite[Remark 2.5]{AmbrosioC2014}) that under hypotheses \eqref{eq:SupLemBound1} and \eqref{eq:SupLemBound2}, the characteristic equations 
\begin{equation*}
\Phi_{(\tau,t)}^{\omega}(x) = x + \INTSeg{\vb \Big( s , \omega , \Phi_{(\tau,s)}^{\omega}(x) \Big)}{s}{\tau}{t}
\end{equation*}
admit a unique maximal solution for $\pi$-almost every $\omega \in \Omega$ and all $x \in \R^d$, which is defined up to some time $T(\omega,x) \in [\tau,T]$. By following again the arguments from that same reference, it can in turn be shown that $T(\omega,x) = T$ for $\Bmu^{\tau}$-almost every $(\omega,x) \in \Omega \times \R^d$, so that $\mu_{\omega}(t) = \Phi_{(\tau,t) \sharp \,}^{\omega} \mu_{\omega}^{\tau}$ for $\pi$-almost every $\omega \in \Omega$ and all times $t \in [\tau,T]$. 
\end{proof}


\begin{proof}[Proof of Theorem \ref{thm:Superposition}]
Much as in the original proof of the superposition principle detailed in \cite[Section 3]{AmbrosioC2014}, we are going to split the argument into three steps. We start in Step 1 by regularising the dynamics to generate a family of well-posed solutions of \eqref{eq:TransportGraph}, and then establish the relative compactness of the underlying sequence of superposition measures in Step 2. Finally in Step 3, we show that the limit points are concentrated on the characteristic curves of the original dynamics. 

\paragraph*{Step 1 -- Smoothing by convolution.} To begin with, let us consider for each $\epsilon \in (0,1]$ the family of strictly positive, smooth Gaussian mollifiers given by
\begin{equation}
\label{eq:GaussianKernel}
\rho_{\epsilon}(x) := \frac{\exp(-|x|^2/2\epsilon)}{(2 \pi \epsilon)^{d/2}}
\end{equation}
for all $x \in \R^d$, and consider the regularised measures defined through disintegration as 
\begin{equation}
\label{eq:SmoothedInitial}
\begin{aligned}
\Bmu^{\tau}_{\epsilon} & := \INTDom{(\mu_{\omega}^{\tau} \star \rho_{\epsilon}) \Lcal^d}{\Omega}{\pi(\omega)} \\
& = \INTDom{\mu_{\omega,\epsilon}^{\tau}}{\Omega}{\pi(\omega)}.
\end{aligned}
\end{equation}
Denoting by $\{ \mu_{\omega}(\cdot)\}_{\omega \in \Omega}$ the $\pi$-almost uniquely determined family of curves for which there holds $\Bmu(t) = \INTDom{\mu_{\omega}(t)}{\Omega}{\pi(\omega)}$ for all times $t \in [\tau,T]$, we likewise consider the regularised curves given by  
\begin{equation}
\label{eq:SmoothedCurve}
\begin{aligned}
\Bmu_{\epsilon}(t) & := \INTDom{(\mu_{\omega}(t) \star \rho_{\epsilon}) \Lcal^d}{\Omega}{\pi(\omega)} \\
& = \INTDom{\mu_{\omega,\epsilon}(t)}{\Omega}{\pi(\omega)}, 
\end{aligned}
\end{equation}
for all times $t \in [\tau,T]$, along with the family of smoothed vector fields
\begin{equation}
\label{eq:SmoothedField}
\vb_{\epsilon}(t,\omega,x) := \frac{(\vb(t,\omega) \mu_{\omega}(t)) \star \rho_{\epsilon}(x)}{\mu_{\omega}(t) \star \rho_{\epsilon}(x)}
\end{equation}
for $\Lcal^1 \times \pi$-almost every $(t,\omega) \in [\tau,T] \times \Omega$ and all $x \in \R^d$. By performing a small adaptation of classical computations which can be found e.g. in \cite[Lemma 8.1.9]{AGS}, one may show that the curve $\Bmu_{\epsilon}(\cdot) \in C^0([\tau,T],\Pcal_{\pi}(\Omega \times \R^d))$ defined via \eqref{eq:SmoothedCurve} solves the Cauchy problem
\begin{equation}
\label{eq:SuperpositionCauchy}
\left\{
\begin{aligned}
& \partial_t \Bmu_{\epsilon}(t) + \Div_x(\vb_{\epsilon}(t) \Bmu_{\epsilon}(t)) = 0, \\
& \Bmu_{\epsilon}(\tau) = \Bmu^{\tau}_{\epsilon}. 
\end{aligned}
\right.
\end{equation}
Besides, it may be verified reasoning as in \cite[Lemma 8.1.9]{AGS} that for each $\epsilon > 0$, the vector fields $\vb_{\epsilon} : [\tau,T] \times \Omega \times \R^d \to \R^d$ are such that 
\begin{equation*}
\INTSeg{\sup_{x \in K} |\vb_{\epsilon}(t,\omega,x)|}{t}{\tau}{T} < +\infty \qquad \text{and} \qquad \INTSeg{\Lip(\vb_{\epsilon}(t,\omega) \, ; K)}{t}{\tau}{T} < +\infty
\end{equation*}
for every compact set $K \subset \R^d$ and $\pi$-almost every $\omega \in \Omega$, whereas \cite[Theorem 2.2]{AmbrosioFuscoPallara} implies that
\begin{equation*}
\begin{aligned}
\INTSeg{\INTDom{\frac{|\vb_{\epsilon}(t,\omega,x)|}{1+|x|}}{\Omega \times \R^d}{\Bmu_{\epsilon}}(t)(\omega,x)}{t}{\tau}{T} & = \INTSeg{\INTDom{\INTDom{\frac{|\vb_{\epsilon}(t,\omega,x)|}{1+|x|}}{\R^d}{\mu_{\omega,\epsilon}(t)(x)}}{\Omega}{\pi(\omega)}}{t}{\tau}{T} \\
& \leq \Big( 1+\NormL{\widetilde{\rho}}{1}{\R^d,\R} \Big) \INTSeg{\INTDom{\frac{|\vb(t,\omega,x)|}{1+|x|}}{\Omega \times \R^d}{\Bmu}(t)(\omega,x)}{t}{\tau}{T}, 
\end{aligned}
\end{equation*}
where $\widetilde{\rho} := \vert \cdot \vert \, \rho \in L^1(\R^d,\R)$. It thus follows from Lemma \ref{lem:AppendixSup} above that \eqref{eq:SuperpositionCauchy} admits a unique solution given explicitly by 
\begin{equation*}
\Bmu_{\epsilon}(t) = \INTDom{\Phi_{(\tau,t) \sharp \,}^{\omega,\epsilon} \mu_{\omega,\epsilon}^{\tau}}{\Omega}{\pi(\omega)}, 
\end{equation*}
and it can be checked that the net of superposition measures $(\Beta_{\epsilon}) \subset \Pcal(\Omega \times \R^d \times \Sigma_T)$ given by 
\begin{equation}
 \label{eq:BetaEpsDef}
\Beta_{\epsilon} := \INTDom{ \Big( \Id \, , \Phi_{(\tau,\cdot)}^{\omega,\epsilon} \Big)_{\sharp} \, \mu_{\omega,\epsilon}^{\tau}}{\Omega}{\pi(\omega)}
\end{equation}
for each $\epsilon \in (0,1]$ satisfies $(\BEfrak_t)_{\sharp} \Beta_{\epsilon} = \Bmu_{\epsilon}(t)$ for all times $t \in [\tau,T]$.

\paragraph*{Step 2 -- Compactness argument.}

To begin with, observe that the sequence $(\Bmu^{\tau}_{\epsilon}) \subset \Pcal_{\pi}(\Omega \times \R^d)$ defined in \eqref{eq:SmoothedInitial} converges narrowly towards $\Bmu^{\tau}$ as $\epsilon \to 0^+$ by Lebesgue's dominated convergence theorem. In particular, by Proposition \ref{prop:FibredTightness}, there exists a coercive map $\psi : \R^d \to [0,+\infty]$ such that 
\begin{equation}
\label{eq:Coercivity1}
\sup_{\epsilon \in (0,1]} \INTDom{\psi(x)}{\Omega \times \R^d}{\Bmu^{\tau}_{\epsilon}(\omega,x)} < +\infty.
\end{equation}
Besides, by combining the integrability condition \eqref{eq:IntBound} with De la Vallée-Poussin's theorem (see e.g. \cite[Theorem 4.5.9]{Bogachev}) applied to the Radon measure $\widetilde{\Bnu} \in \Mcal^+([\tau,T] \times \Omega \times \R^d)$ defined by 
\begin{equation}
\label{eq:TildeBnu}
\INTDom{\phi(t,\omega,x)}{[\tau,T] \times \Omega \times \R^d}{\widetilde{\Bnu}(t,\omega,x)} := \INTSeg{\INTDom{\frac{\phi(t,\omega,x)}{1+|x|}}{\Omega \times \R^d}{\Bmu(t)(\omega,x)}}{t}{\tau}{T}, 
\end{equation}
there exists a convex non-decreasing and superlinear map $\Theta : \R_+ \to \R$ for which
\begin{equation*}
\INTSeg{\INTDom{\frac{\Theta(|\vb(t,\omega,x)|)}{1+|x|}}{\Omega \times \R^d}{\Bmu(t)(\omega,x)}}{t}{\tau}{T} < +\infty. 
\end{equation*}
Then, it can be checked by leveraging classical convolution estimates (see e.g. \cite[Page 1202]{AmbrosioC2014}) that 
\begin{equation*}
\begin{aligned}
& \INTDom{\INTSeg{\frac{\Theta(|\dot \sigma(t)|)}{1+|\sigma(t)|}}{t}{\tau}{T}}{\Omega \times \R^d \times \Sigma_T}{\Beta_{\epsilon}(\omega,x,\sigma)} & \\
& \hspace{2cm} = \INTSeg{\INTDom{\INTDom{\frac{\Theta(|\vb_{\epsilon}(t,\omega,x)|)}{1+|x|}}{\R^d}{\mu_{\omega,\epsilon}(t)(x)}}{\Omega}{\pi(\omega)}}{t}{\tau}{T} \\
& \hspace{2cm} \leq \INTSeg{\INTDom{\INTDom{\frac{\Big( \Theta(|\vb(t,\omega,x)|) \mu_{\omega}(t) \Big) \star \rho_{\epsilon}(x)}{1+|x|}}{\R^d}{x}}{\Omega}{\pi(\omega)}}{t}{\tau}{T} \\
& \hspace{2cm} \leq \INTSeg{\INTDom{\INTDom{\bigg( \frac{\Theta(|\vb(t,\omega,x)|)}{1+|x|} \mu_{\omega}(t) \bigg) \star \rho_{\epsilon}(x)}{\R^d}{x}}{\Omega}{\pi(\omega)}}{t}{\tau}{T} \\
& \hspace{2.45cm} + \INTSeg{\INTDom{\INTDom{\frac{1}{1+|x|} \bigg( \frac{\Theta(|\vb(t,\omega,x)|)}{1+|x|} \mu_{\omega}(t) \bigg) \star \widetilde{\rho}_{\epsilon}(x)}{\R^d}{x}}{\Omega}{\pi(\omega)}}{t}{\tau}{T} \\
& \hspace{2cm} \leq \Big( 1+ \NormL{\widetilde{\rho}}{1}{\R^d,\R} \Big) \INTSeg{\INTDom{\INTDom{\frac{\Theta(|\vb(t,\omega,x)|)}{1+|x|}}{\R^d}{\mu_{\omega}(t)(x)}}{\Omega}{\pi(\omega)}}{t}{\tau}{T} \\
& \hspace{2cm} = \Big( 1+ \NormL{\widetilde{\rho}}{1}{\R^d,\R} \Big) \INTSeg{\INTDom{\frac{\Theta(|\vb(t,\omega,x)|)}{1+|x|}}{\Omega \times \R^d}{\Bmu(t)(\omega,x)}}{t}{\tau}{T} < +\infty
\end{aligned}
\end{equation*}
for each $\epsilon \in (0,1]$, where we also used Fubini's theorem. In particular, it follows from combining this latter series of estimates with \eqref{eq:Coercivity1} that  
\begin{equation}
\label{eq:Coercivity2}
\sup_{\epsilon \in (0,1]} \INTDom{\bigg(\psi(x) + \INTSeg{\frac{\Theta(|\dot \sigma(t)|)}{1+|\sigma(t)|}}{t}{\tau}{T} \bigg)}{\Omega \times \R^d \times \Sigma_T}{\Beta_{\epsilon}(\omega,x,\sigma)} < +\infty. 
\end{equation}
At this point, recall that it is well known (see e.g. \cite[Proof of Theorem 3.4]{AmbrosioC2014}) that the functional $\Psi : \R^d \times \Sigma_T \to [0,+\infty]$ defined by 
\begin{equation*}
\Psi(x,\sigma) := \left\{
\begin{aligned}
& \psi(x) + \INTSeg{\frac{\Theta(|\dot \sigma(t)|)}{1+|\sigma(t)|}}{t}{\tau}{T} & \text{if $\sigma(\cdot) \in \AC([\tau,T],\R^d)$}, \\
& + \infty & \text{otherwise}, 
\end{aligned}
\right.
\end{equation*}
has compact sublevels as a consequence of the Ascoli-Arzela theorem (see e.g. \cite[Theorem 11.28]{Rudin1987}). Thence, it follows from Proposition \ref{prop:FibredTightness} together with \eqref{eq:Coercivity2} that the sequence $(\Beta_{\epsilon}) \subset \Pcal_{\pi}(\Omega \times \R^d \times \Sigma_T)$ admits a subsequence that converges narrowly towards some $\Beta \in \Pcal_{\pi}(\Omega \times \R^d \times \Sigma_T)$ as $\epsilon \to 0^+$.   

\paragraph*{Step 3 -- Properties of the limit measure.} In this last step, we prove that the limit point $\Beta \in \Pcal(\Omega \times \R^d \times \Sigma_T)$ is a superposition measure associated with $\Bmu(\cdot) \in C^0([\tau,T],\Pcal_{\pi}(\Omega \times \R^d))$. First, one may notice from the construction in \eqref{eq:SmoothedCurve} that $(\Bmu_{\epsilon}(t)) \subset \Pcal_{\pi}(\Omega \times \R^d)$ converges narrowly towards $\Bmu(t)$ for all times $t \in [\tau,T]$ by Lebesgue's dominated convergence theorem, so that
\begin{equation*}
\begin{aligned}
\INTDom{\varphi(\omega,x)}{\Omega \times \R^d}{\Bmu(t)(\omega,x)} & = \lim_{\epsilon \to 0^+} \INTDom{\varphi(\omega,x)}{\Omega \times \R^d}{\Bmu_{\epsilon}(t)(\omega,x)} \\
& = \lim_{\epsilon \to 0^+} \INTDom{\varphi(\omega,\sigma(t))}{\Omega \times \R^d \times \Sigma_T}{\Beta_{\epsilon}(\omega,x,\sigma)} = \INTDom{\varphi(\omega,\sigma(t))}{\Omega \times \R^d \times \Sigma_T}{\Beta(\omega,x,\sigma)}
\end{aligned}
\end{equation*}
for every $\varphi \in C^0_b(\Omega \times \R^d,\R) \subset C^0_{\pi,b}(\Omega \times \R^d,\R)$, which precisely amounts to the fact that $\Bmu(t) = (\BEfrak_t)_{\sharp} \Beta$ for all times $t  \in [\tau,T]$. 

In order to prove that $\Beta \in \Pcal_{\pi}(\Omega \times \R^d \times \Sigma_T)$ is concentrated on the characteristic curves of \eqref{eq:ODECharac}, our goal is to show that 
\begin{equation*}
\INTDom{\hspace{-0.15cm} \frac{\Big| \sigma(t) - x - \INTSeg{\vb(s,\omega,\sigma(s))}{s}{\tau}{t} \Big|}{1 + \underset{t \in [\tau,T]}{\max} |\sigma(t)|} \,}{\Omega \times \R^d \times \Sigma_T}{\Beta(\omega,x,\sigma)} = 0
\end{equation*}
for all times $t \in [\tau,T]$. To do so, we need to perform yet another regularisation argument to account for the fact that $x \in \R^d \mapsto \vb(t,\omega,x) \in \R^d$ is merely Borel-measurable for $\Lcal^1 \times \pi$-almost every $(t,\omega) \in [\tau,T] \times \Omega$. Thus, let $(\vb_n(\cdot)) \subset C^0_b([0,T] \times \Omega \times \R^d,\R^d)$ be an approximating sequence such that $\vb_n(t,\omega) \in C^0_c(\R^d,\R^d)$ for all $(t,\omega) \in [0,T] \times \Omega$ and 
\begin{equation}
\label{eq:ApproximationTildeBnu}
\INTSeg{\INTDom{\frac{|\vb(t,\omega,x) - \vb_n(t,\omega,x)|}{1+|x|}}{\Omega \times \R^d}{\Bmu(t)(\omega,x)}}{t}{\tau}{T} ~\underset{n \to +\infty}{\longrightarrow}~ 0. 
\end{equation}
The existence of the latter is ensured by the functional version of Lusin's theorem (see e.g. \cite[Theorem 7.1.13]{Bogachev}) applied to the finite Radon measure $\widetilde{\Bnu} \in \Mcal^+([\tau,T] \times \Omega \times \R^d)$ defined in \eqref{eq:TildeBnu}, combined with approximations by smooth cut-offs in the space variable. Then, observe that 
\begin{equation}
\label{eq:CharacteristicIntegral1}
\begin{aligned}
& \INTDom{\hspace{-0.15cm} \frac{\big| \sigma(t) - x - \INTSeg{\vb(s,\omega,\sigma(s))}{s}{\tau}{t} \big|}{1 + \underset{t \in [\tau,T]}{\max} |\sigma(t)|} \,}{\Omega \times \R^d \times \Sigma_T}{\Beta(\omega,x,\sigma)} \\
& \hspace{1cm} \leq \INTDom{\hspace{-0.15cm} \frac{\big| \sigma(t) - x - \INTSeg{\vb_n(s,\omega,\sigma(s))}{s}{\tau}{t} \big|}{1 + \underset{t \in [\tau,T]}{\max} |\sigma(t)|} \,}{\Omega \times \R^d \times \Sigma_T}{\Beta(\omega,x,\sigma)} \\
& \hspace{1.45cm} + \INTSeg{\INTDom{\frac{|\vb(s,\omega,\sigma(s)) - \vb_n(s,\omega,\sigma(s))|}{1+|\sigma(s)|}}{\Omega \times \R^d \times \Sigma_T}{\Beta(\omega,x,\sigma)}}{s}{\tau}{T} \\
& \hspace{1cm} \leq \INTDom{\hspace{-0.15cm} \frac{\big| \sigma(t) - x - \INTSeg{\vb_n(s,\omega,\sigma(s))}{s}{\tau}{t} \big|}{1 + \underset{t \in [\tau,T]}{\max} |\sigma(t)|} \,}{\Omega \times \R^d \times \Sigma_T}{\Beta(\omega,x,\sigma)} \\
& \hspace{1.45cm} + \INTSeg{\INTDom{\frac{|\vb(s,\omega,x) - \vb_n(s,\omega,x)|}{1+|x|}}{\Omega \times \R^d}{\Bmu(s)(\omega,x)}}{s}{\tau}{T}.
\end{aligned}
\end{equation}
The second term in the right-hand side of the previous expression vanishes as $n \to +\infty$ by \eqref{eq:ApproximationTildeBnu}. To estimate the first term in the right-hand side of \eqref{eq:CharacteristicIntegral1}, notice at first that 
\begin{equation*}
\begin{aligned}
& \INTDom{\hspace{-0.15cm} \frac{\big| \sigma(t) - x - \INTSeg{\vb_n(s,\omega,\sigma(s))}{s}{\tau}{t} \big|}{1 + \underset{t \in [\tau,T]}{\max} |\sigma(t)|} \,}{\Omega \times \R^d \times \Sigma_T}{\Beta(\omega,x,\sigma)} \\
& \hspace{2.5cm} = \lim_{\epsilon \to 0^+} \INTDom{\hspace{-0.15cm} \frac{\big| \sigma(t) - x - \INTSeg{\vb_n(s,\omega,\sigma(s))}{s}{\tau}{t} \big|}{1 + \underset{t \in [\tau,T]}{\max} |\sigma(t)|} \,}{\Omega \times \R^d \times \Sigma_T}{\Beta_{\epsilon}(\omega,x,\sigma)}
\end{aligned}
\end{equation*}
along a subsequence by the results of Step 2. At this stage, consider the regularisation by convolution of the sequence $(\vb_n(\cdot)) \subset C^0_b([\tau,T] \times \Omega \times \R^d,\R^d)$ given by 
\begin{equation*}
\vb_{n,\epsilon}(t,\omega,x) := \frac{(\vb_n(t,\omega,x) \mu_{\omega}(t)) \star \rho_{\epsilon}(x)}{\mu_{\omega}(t) \star \rho_{\epsilon}(x)}
\end{equation*}
for all $(t,\omega,x) \in [\tau,T] \times \Omega \times \R^d$, any $n \geq 1$ and each $\epsilon \in (0,1]$, and note that by the definition \eqref{eq:BetaEpsDef} of the sequence $(\Beta_{\epsilon}) \subset \Pcal_{\pi}(\Omega \times \R^d \times \Sigma_T)$, it holds that
\begin{equation}
\label{eq:CharacteristicIntegral2}
\begin{aligned}
& \INTDom{\hspace{-0.15cm} \frac{\big| \sigma(t) - x - \INTSeg{\vb_n(s,\omega,\sigma(s))}{s}{\tau}{t} \big|}{1 + \underset{t \in [\tau,T]}{\max} |\sigma(t)|} \,}{\Omega \times \R^d \times \Sigma_T}{\Beta_{\epsilon}(\omega,x,\sigma)} \\
& \hspace{1cm} = \INTDom{\INTDom{\hspace{-0.15cm} \frac{\big| \Phi_{(\tau,t)}^{\omega,\epsilon}(x) - x  - \INTSeg{\vb_n \big( s, \omega , \Phi_{(\tau,s)}^{\omega,\epsilon}(x) \big)}{s}{\tau}{t} \big|}{1+ \underset{t \in [\tau,T]}{\max} \big| \Phi_{(\tau,t)}^{\omega,\epsilon}(x) \big|} }{\R^d}{\mu_{\omega,\epsilon}^{\tau}(x)}}{\Omega}{\pi(\omega)} \\
& \hspace{1cm} \leq \INTDom{\INTDom{\Bigg( \INTSeg{\frac{\big| \vb_{\epsilon} \big( s,\omega,\Phi_{(\tau,s)}^{\omega,\epsilon}(x) \big) - \vb_n \big( s,\omega,\Phi_{(\tau,s)}^{\omega,\epsilon}(x) \big) \big|}{1+\underset{t \in [\tau,T]}{\max}|\Phi_{(\tau,t)}^{\omega,\epsilon}(x)|}}{s}{\tau}{t} \Bigg)}{\R^d}{\mu_{\omega,\epsilon}^{\tau}(x)}}{\Omega}{\pi(\omega)} \\
& \hspace{1cm} \leq \INTSeg{\INTDom{\INTDom{\frac{\big|\vb_{\epsilon}(s,\omega,x) - \vb_n(s,\omega,x)\big|}{1+|x|}}{\R^d}{\mu_{\omega,\epsilon}(s)}}{\Omega}{\pi(\omega)}}{s}{\tau}{T} \\
& \hspace{1cm} \leq \INTSeg{\INTDom{\INTDom{\frac{\big|\vb_{\epsilon}(s,\omega,x) - \vb_{n,\epsilon}(s,\omega,x)\big|}{1+|x|}}{\R^d}{\mu_{\omega,\epsilon}(s)(x)}}{\Omega}{\pi(\omega)}}{s}{\tau}{T} \\
& \hspace{1.45cm} + \INTSeg{\INTDom{\INTDom{\frac{\big|\vb_{n,\epsilon}(s,\omega,x) - \vb_n(s,\omega,x)\big|}{1+|x|}}{\R^d}{\mu_{\omega,\epsilon}(s)(x)}}{\Omega}{\pi(\omega)}}{s}{\tau}{T} \\
& \hspace{1cm} \leq \Big( 1 + \NormL{\widetilde{\rho}}{1}{\R^d,\R} \Big) \INTSeg{\INTDom{\frac{\big|\vb(s,\omega,x) - \vb_{n}(s,\omega,x)\big|}{1+|x|}}{\Omega \times \R^d}{\Bmu(s)(\omega,x)}}{s}{\tau}{T} \\
& \hspace{1.45cm} + \INTSeg{\INTDom{\INTDom{\INTDom{\big|\vb_n(s,\omega,x) - \vb_n(s,\omega,y)\big|\rho_{\epsilon}(x-y)}{\R^d}{y}}{\R^d}{\mu_{\omega}(s)(x)}}{\Omega}{\pi(\omega)}}{s}{\tau}{T}
\end{aligned}
\end{equation}
for each $\epsilon \in (0,1]$, where we again used the standard convolution estimates of \cite[Page 1202]{AmbrosioC2014}. We now observe that since $\vb_n(s,\omega,\cdot) \in C^0_c(\R^d)$ for each $n \geq 1$, the maps $x \in \R^d \mapsto \vb_n(s,\omega,x) \in \R^d$ are uniformly continuous. In particular, for each $\delta > 0$, there exists $\eta_n(s,\omega) > 0$ such that $|\vb_n(s,\omega,x) - \vb_n(s,\omega,y)| \leq \delta$ whenever $|x-y| \leq \eta_n(s,\omega)$, wherefore
\begin{equation*}
\begin{aligned}
& \INTSeg{\INTDom{\INTDom{\INTDom{\big|\vb_n(s,\omega,x) - \vb_n(s,\omega,y)\big|\rho_{\epsilon}(x-y)}{\R^d}{y}}{\R^d}{\mu_{\omega}(s)(x)}}{\Omega}{\pi(\omega)}}{s}{\tau}{T} \\
& \hspace{0.1cm} \leq \delta(T-\tau) + 2 \NormC{\vb_n}{0}{[\tau,T] \times \Omega \times \R^d,\R^d} \INTSeg{\INTDom{\INTDom{\INTDom{\rho_{\epsilon}(x-y)}{\{y \;\mathrm{s.t.}\, \|x-y\| > \eta_n(s,\omega)\}}{y}}{\R^d}{\mu_{\omega}(s)(x)}}{\Omega}{\pi(\omega)}}{s}{\tau}{T}.
\end{aligned}
\end{equation*}
At this stage, noting that tails of Gaussian distributions vanish as $\epsilon \to 0^+$, it stems from an application of Lebesgue's dominated convergence theorem followed then by letting $\delta \to 0^+$ that
%
\begin{equation*}
\INTSeg{\INTDom{\INTDom{\INTDom{\big|\vb_n(s,\omega,x) - \vb_n(s,\omega,y)\big|\rho_{\epsilon}(x-y)}{\R^d}{y}}{\R^d}{\mu_{\omega}(s)(x)}}{\Omega}{\pi(\omega)}}{s}{\tau}{T}  ~\underset{\epsilon \to 0^+}{\longrightarrow}~ 0
\end{equation*}
for each $n \geq 1$. By inserting this information in the estimates of \eqref{eq:CharacteristicIntegral1} and \eqref{eq:CharacteristicIntegral2}, we obtain that 
\begin{equation*}
\begin{aligned}
& \INTDom{\hspace{-0.15cm} \frac{\big| \sigma(t) - x - \INTSeg{\vb(s,\omega,\sigma(s))}{s}{\tau}{t} \big|}{1 + \underset{t \in [\tau,T]}{\max} |\sigma(t)|} \,}{\Omega \times \R^d \times \Sigma_T}{\Beta(\omega,x,\sigma)} \\
& \hspace{4cm} \leq \Big( 2 + \NormL{\widetilde{\rho}}{1}{\R^d,\R} \Big) \INTSeg{\INTDom{\frac{|\vb(t,\omega,x) - \vb_n(t,\omega,x)|}{1+|x|}}{\Omega \times \R^d}{\Bmu(t)(\omega,x)}}{t}{\tau}{T}
\end{aligned}
\end{equation*}
for all times $t \in [\tau,T]$, which then implies by letting $n \to +\infty$ while using \eqref{eq:ApproximationTildeBnu} that 
\begin{equation*}
\INTDom{\hspace{-0.15cm} \frac{\big| \sigma(t) - x - \INTSeg{\vb(s,\omega,\sigma(s))}{s}{\tau}{t} \big|}{1 + \underset{t \in [\tau,T]}{\max} |\sigma(t)|} \,}{\Omega \times \R^d \times \Sigma_T}{\Beta(\omega,x,\sigma)} = 0, 
\end{equation*}
thereby closing the proof of Theorem \ref{thm:Superposition}. 
\end{proof}

\begin{rmk}[On the proof of the superposition principle for structured continuity equations]
A tempting, and arguably more elegant approach to establishing Theorem \ref{thm:Superposition} would be to consider the family of disintegrated curves $\{\mu_{\omega}(\cdot)\}_{\omega \in \Omega} \subset C^0([\tau,T],\Pcal(\R^d))$ solution of \eqref{eq:DisintegratedTransport}, whose existence is ensured by Theorem \ref{thm:Equivalence}, and then apply the classical superposition principle in order to produce a matching Borel family of superposition measures $\{\Beta_{\omega}\}_{\omega \in \Omega} \subset \Pcal(\R^d \times \Sigma_T)$. While carrying out this approach should be feasible in principle, showing that the mapping $\omega \in \Omega \mapsto \Beta_{\omega} \in \Pcal(\R^d \times \Sigma_T)$ is measurable, which is a mandatory step in order to glue the collection $\{ \Beta_{\omega}\}_{\omega \in \Omega}$ back into a global superposition measure $\Beta \in \Pcal_{\pi}(\Omega \times \R^d \times \Sigma_T)$, was eventually too arduous a task owing to the lack of regularity of $\vb : [\tau,T] \times \Omega \times \R^d \to \R^d$ in the space variable.   
\end{rmk}


\setcounter{Def}{0} 
\setcounter{section}{0}
\renewcommand{\thesection}{C} 
\renewcommand{\thesubsection}{C} 

\subsection{Equivalence of standard probability spaces}
\label{section:AppendixEquivalenceProba}

In this third appendix section, for the sake of completeness, we briefly outline the proof of Proposition \ref{prop:Conv_CondExp}, which is excerpted from the more detailed reference \cite[Appendix D]{Cavagnari2022}. 

\begin{proof}[Proof of Proposition \ref{prop:Conv_CondExp}]
Given $N \geq 1$, let $\widetilde\Ppazo_N:=(I_i^N)_{i\in\{1,\ldots,N\}}$ and $\psi : \Omega \to [0,1]$ be given as in Proposition \ref{prop:partition}, and define the equipartition $\Ppazo_N:=(\Omega_i^N)_{i\in\{1,\ldots,N\}}$ of $\Omega$ by $\Omega_i^N:=\psi^{-1}(I_i^N)$ for all $i\in\{1,\ldots,N\}$.
Fix now any $f\in L^p(\Omega,X;\pi)$ for some $p\in [1,+\infty)$, and consider the mapping $\widetilde{f}:= f \circ\psi^{-1} \in L^p([0,1],X;\Lcal^1)$. Recall lastly that the conditional expectation $\E_{\widetilde{\Ppazo}_N}[\widetilde{f}] \in L^p([0,1],X;\Lcal^1)$ of this map with respect to the partition $\widetilde{\Ppazo}_N$ is given by 
\begin{equation*}
\E_{\widetilde{\Ppazo}_N}[\widetilde{f}](s) := \sum_{i=1}^N \bigg( \INTDomdash{\widetilde{f}(t)}{I_i^N}{t} \bigg) \mathds{1}_{I_i^N}(s).
\end{equation*}
for $\Lcal^1$-almost every $s \in [0,1]$. Then, notice that
\begin{equation*}
\begin{split}
\E_{\Ppazo_N}[f] \circ \psi^{-1}(s) & = \sum_{i=1}^N \bigg( \INTDomdash{f(\theta)}{\Omega_i^N}{\pi(\theta)} \bigg) \mathds{1}_{\Omega_i^N} \circ \psi^{-1}(s) \\
& = \sum_{i=1}^N \bigg( N \INTDom{f(\theta)}{\Omega_i^N}{(\psi^{-1}_{\sharp}\Lcal^1)(\theta)} \bigg) \mathds{1}_{\psi(\Omega_i^N)}(s) \\
& = \sum_{i=1}^N \bigg( N \INTDom{\widetilde{f}(s)}{I_i^N}{s} \bigg) \mathds{1}_{I_i^N}(s) \\
& = \E_{\widetilde{\Ppazo}_N}[\widetilde f](s)
\end{split}
\end{equation*}
from which we may deduce that
\begin{equation*}
\begin{split}
\NormL{f-\E_{\Ppazo_N}[f]}{p}{\Omega,X;\pi}^p & = \INTDom{ \big\|\, f(\omega)-\E_{\Ppazo_N}[f](\omega) \,\big\|_X^p}{\Omega}{\pi(\omega)} \\
& = \INTDom{ \big\| \, f(\omega)-\E_{\Ppazo_N}[f](\omega) \,\big\|_X^p }{\Omega}{(\psi^{-1}_{\sharp}\Lcal^1)(\omega)} \\
& = \INTDom{ \big\| \, f \circ\psi^{-1}(s)-\E_{\Ppazo_N}[f]\circ\psi^{-1}(s) \, \big\|_X^p}{[0,1]}{s} \\
& = ~\NormL{\widetilde f - \E_{\widetilde{\Ppazo}_N}[\widetilde f]}{p}{[0,1],X;\Lcal^1}^p.
\end{split}
\end{equation*}
Then, according to \cite[Lemma D3]{Cavagnari2022}, it follows from Lebesgue's differentiation and Vitali's dominated convergence theorems that
\begin{equation*}
\NormL{\widetilde f - \E_{\widetilde{\Ppazo}_N}[\widetilde f]}{p}{[0,1],X;\Lcal^1} ~\underset{N \to +\infty}{\longrightarrow}~ 0,
\end{equation*}
which yields the desired result. 

\end{proof}


\setcounter{section}{0} 
\renewcommand{\thesection}{D} 
\renewcommand{\thesubsection}{D} 
\subsection{Well-posedness of the particle systems}\label{sec:AppendixWellPosednessMic}

In this fourth and last appendix section, we detail the proof of Proposition \ref{prop:ExistUniq-Mic}.

\begin{proof}
For the sake of conciseness, we solely focus on the original particle system~\eqref{eq:Mic}, the proof being identical for the auxiliary system~\eqref{eq:Mic-aux}. Observe first that the growth estimate \eqref{eq:AuxSystemBound} and Lipschitz regularity \eqref{eq:AuxSystemLip} of the vector fields $v^N_i : [0,T] \times \Pcal_{\pi,1}(\Omega \times \R^d) \times \R^d \to \R^d$ -- which we recall are defined in \eqref{eq:viNDef} --, are directly inherited from those of $\vb : [0,T] \times \Pcal_{\pi,1}(\Omega \times \R^d) \times \Omega \times \R^d \to \R^d$ posited in Hypothesis \ref{hyp:CL}. Moreover, given two vectors $(x_i^\nm)_{i\in\{1,\ldots,N\}}, (y_i^\nm)_{i\in\{1,\ldots,N\}}\in (\R^d)^N$ with associated empirical measures 
\begin{equation*}
\Bmu^{\nm} = \sum_{k=1}^n \pi_{\llcorner \Omega_k^n} \times \bigg( \frac{1}{m} \sum_{\ell=1}^m \delta_{x_{(k-1)m+\ell}^\nm} \bigg) \qquad \text{and}\qquad \Bnu^{\nm} = \sum_{k=1}^n \pi_{\llcorner \Omega_k^n} \times \bigg( \frac{1}{m} \sum_{\ell=1}^m \delta_{y_{(k-1)m+\ell}^\nm} \bigg), 
\end{equation*}
it can be shown by repeating the computations from \eqref{eq:Wpi1_to_l1} that
\begin{equation*}
\begin{split}
 W_{\pi,1}(\Bmu^\nm,\Bnu^\nm) \leq \frac{1}{N} \sum_{i=1}^N |x_{i}^\nm-y_{i}^\nm|.
\end{split}
\end{equation*}
Then, upon assuming that $x_i^\nm,y_i^\nm \in B(0,R)$ for each $i\in\{1,\ldots,N\}$, which in turn implies that $\Bmu^\nm, \Bnu^\nm\in\Pcal_{\pi}(\Omega\times B(0,R))$, it holds that 
\begin{equation*}
\begin{split}
\frac{1}{N} \sum_{i=1}^N \big|v_{i}^N(t,\Bmu^\nm,x_i^\nm)-v_{i}^N(t,\Bnu^\nm,y_i^\nm) \big| & \leq \frac{1}{N} \sum_{i=1}^NL_{R}(t) \big( W_{\pi,1}(\Bmu^\nm,\Bnu^\nm) + |x^\nm_i-y^\nm_i| \big)\\
& \leq 2L_{R}(t) \frac{1}{N} \sum_{i=1}^N \big|x_{i}^\nm-y_{i}^\nm\big|.
\end{split}
\end{equation*}
Hence, the vector fields $(v_i^N(\cdot))_{i\in\{1,\ldots,N\}}$ are both sublinear and locally Lipschitz with respect to $(x_i^\nm)_{i\in\{1,\ldots,N\}}$, which implies the existence and uniqueness of a local solution to the particle system \eqref{eq:Mic}. We focus now on proving the boundedness of the particle trajectories. To this end, note that
\begin{equation*}
\begin{split}
|x_i^\nm(t)| & \leq |x_i^\nm(0)| + \int_0^t \Big| v_{i}^N \Big( s,\Bmu^\nm(s),x_i^\nm(s) \Big)\Big| \, \mathrm{d} s \\
& \leq |x_i^\nm(0)| + \int_0^t \INTDomdash{\Big| \, \vb \Big(s,\Bmu^\nm(s),\omega,x^\nm_i(s) \Big)\Big|}{\Omega^N_i}{\pi(\omega)} \, \mathrm{d} s  \\
& \leq |x_i^\nm(0)| + \int_0^t  m(s) \bigg(1+ \big|x^\nm_i(s)\big| + \Mpazo_{\pi,1}(\Bmu^\nm(s)) \bigg) \, \mathrm{d}s 
\end{split}
\end{equation*}
for all times $t \in [0,T]$ and each $i \in \{1,\ldots,N\}$, where $\Bmu^\nm(\cdot) \in \AC([0,T],\Pcal_{\pi,1}(\Omega \times B(0,R_r)))$ is the curve of empirical measures defined in  \eqref{eq:emp_Mic}. From the elementary observation that
\begin{equation*}
\begin{split}
\Mpazo_{\pi,1}(\Bmu^\nm(s)) & = \int_\Omega \int_{\R^d} |x|\, \mathrm{d}\mu^\nm_\omega(s)(x) \, \mathrm{d} \pi(\omega) \\
& \leq \max_{i \in \{1,\ldots,N\}}|x^\nm_i(s)|, 
\end{split}
\end{equation*}
we straightforwardly deduce from an application of Gr\"onwall's lemma that
\begin{equation*}
\max_{i\in\{1,\ldots,N\}} |x_i^\nm(t)| \leq \bigg(\max_{i\in\{1,\ldots,N\}} |x_i^{0,\nm}| + \|m(\cdot)\|_1 \bigg) \exp \big(2 \|m(\cdot)\|_1 \big) 
\end{equation*}
for all times $t \in [0,T]$. As a consequence, the curves $(x_i^\nm(\cdot))_{i\in\{1,\ldots,N\}}$ solution of \eqref{eq:Mic} are globally defined on the interval $[0,T]$.
\end{proof}


\bibliographystyle{alpha}
{\footnotesize
\bibliography{../FibredWassBib}
}

\end{document}